\documentclass[final,12pt]{colt2025} 

\usepackage{physics}
\usepackage{mathtools}
\usepackage{txfonts}
\usepackage[capitalise]{cleveref}

\newtheorem{assumption}{Assumption}

\crefname{assumption}{assumption}{assumptions}

\usepackage{bm}

\DeclareMathOperator*{\argmin}{arg\,min}
\newcommand{\R}{\mathbb{R}}
\newcommand{\N}{\mathbb{N}}
\newcommand{\Z}{\mathbb{Z}}

\newcommand{\E}{\mathbb{E}}
\renewcommand{\P}{\mathbb{P}}
\renewcommand{\l}{\left}
\renewcommand{\r}{\right}

\newcommand{\inv}{^{-1}}
\newcommand{\reff}[2]{\hyperref[#1]{#2 \ref*{#1}}}
\newcommand{\sumi}[1]{\sum_{i=1}^{#1}}

\newcommand{\newpar}{\vspace{8pt}}
\newcommand{\lip}{\langle}
\newcommand{\rip}{\rangle}
\newcommand{\ind}{\mathbb{I}}
\DeclareMathSymbol{\indep}{\mathrel}{symbolsC}{121}

\newcommand{\sil}{\stackrel{(i)}{\leq}}
\newcommand{\siil}{\stackrel{(ii)}{\leq}}
\newcommand{\siiil}{\stackrel{(iii)}{\leq}}

\newcommand{\sie}{\stackrel{(i)}{=}}

\def\bG{\mathbf{G}}
\def\bH{\mathbf{H}}

\def\bP{\mathbf{P}}

\def\bU{\mathbf{U}}
\def\bV{\mathbf{V}}

\def\bX{\mathbf{X}}
\def\bY{\mathbf{Y}}
\def\bZ{\mathbf{Z}}

\def\bq{\mathbf{q}}
\def\be{\bm{e}}
\def\bv{\bm{v}}
\def\bu{\bm{u}}
\def\bw{\mathbf{w}}
\def\by{\mathbf{y}}

\def\bone{\mathbf{1}}

\def\cA{\mathcal{A}}
\def\cB{\mathcal{B}}
\def\cC{\mathcal{C}}
\def\cI{\mathcal{I}}
\def\cF{\mathcal{F}}
\def\cS{\mathcal{S}}

\def\cH{\mathcal{H}}
\def\cL{\mathcal{L}}
\def\cM{\mathcal{M}}
\def\cN{\mathcal{N}}
\def\cP{\mathcal{P}}
\def\cR{\mathcal{R}}
\newcommand{\mean}{\mathbb{E}}
\newcommand{\Var}{\text{\rm Var}}
\newcommand{\Cov}{\text{\rm Cov}}

\newcommand{\tagaligneq}{\refstepcounter{equation}\tag{\theequation}}

\DeclareMathOperator{\msum}{\medmath\sum}
\newcommand{\mmax}{\max\nolimits}
\newcommand{\msup}{\sup\nolimits}
\newcommand{\mmin}{\min\nolimits}
\newcommand{\minf}{\inf\nolimits}
\newcommand{\tvec}{\text{\rm vec}}

\def\bg{\mathbf{g}}
\def\bh{\mathbf{h}}

\usepackage{nccmath}
\usepackage{scalerel}

\newcommand{\argdot}{{\,\vcenter{\hbox{\tiny$\bullet$}}\,}}

\DeclareMathOperator{\mint}{\scaleobj{.8}{\int}}

\usepackage{tikz}
\usetikzlibrary{decorations.pathreplacing}
\usetikzlibrary{decorations.pathmorphing, decorations.markings}

\usepackage{enumitem}
\usepackage{soul}

\title[Universality of High-Dimensional Logistic Regression \& CGMT under Dependence
]{Universality of High-Dimensional Logistic Regression and a Novel CGMT under {Dependence} with Applications to Data Augmentation}
\usepackage{times}

\coltauthor{%
 \Name{Matthew {Esmaili Mallory}} \thanks{denotes equal contribution} \Email{matthewmallory@fas.harvard.edu}\\
 \addr Department of Statistics, Harvard University
 \AND
 \Name{Kevin Han {Huang}} \footnotemark[1] \Email{han.huang.20@ucl.ac.uk}\\
 \addr Gatsby Unit, University College London
 \AND
 \Name{Morgane Austern} \Email{maustern@fas.harvard.edu}\\
 \addr Department of Statistics, Harvard University
}

\renewenvironment{proof}[1][Proof]%
{%
\par\noindent{\bfseries\upshape #1\ }%
}%
{\jmlrQED}

\begin{document}

\maketitle

\begin{abstract}
Over the last decade, a wave of research has characterized the exact asymptotic risk of many high-dimensional models in the proportional regime. Two foundational results have driven this progress: Gaussian universality, which shows that the asymptotic risk of estimators trained on non-Gaussian and Gaussian data is equivalent, and the convex Gaussian min-max theorem (CGMT), which characterizes the risk under Gaussian settings. However, these results rely on the assumption that the data consists of independent random vectors—an assumption that significantly limit its applicability to many practical setups. In this paper, we address this limitation by generalizing both results to the dependent setting. More precisely, we prove that Gaussian universality still holds for high-dimensional logistic regression under block dependence, {$m$-dependence and special cases of $\beta$-mixing}, and establish a novel CGMT framework that accommodates for correlation across both the covariates and observations. Using these results, we establish the impact of data augmentation, a widespread practice in deep learning, on the asymptotic risk.
\end{abstract}


\section{Introduction} \label{sec:intro}

Over the past decade, landmark results such as Gaussian universality and the convex Gaussian min-max theorem (CGMT) have been extended and applied to analyze the asymptotic risk of various high-dimensional feature models. They have led to a deeper understanding of matters such as the impact of regularization and hyperparameters on the risk \citep{salehi2019impact, deng2022model} and the double descent phenomenon \citep{mei2022generalization,hastie2022surprises, belkin2019double}.

Broadly speaking, Gaussian universality is the observation that the risk of many high dimensional estimators depends on the data distribution only through its first two moments \citep{montanari2022universality,montanari2023universality,dandi2024universality,gerace2022gaussian,korada2011applications, han2023universality, hu2022universality}. Consequently, for these estimators, their risks can be studied by analyzing the risk for Gaussian data with matching mean and variance. This unlocks the many useful tools developed for the Gaussian case, including Approximate Message Passing \citep{donoho2009message}, the Cavity Method \citep{opper2001naive} and the CGMT \citep{gordon1985some,thrampoulidis2014gaussian}. Among them, the CGMT is a framework that converts a complex optimization problem on Gaussian data to a much more analytically tractable auxiliary problem. The auxiliary optimization is often further simplified into a deterministic equation involving only a few scalars, and under the CGMT, its solution completely characterizes that of the original problem. 

A general pipeline of analysis built on universality and the CGMT entails the following:
\begin{enumerate}[topsep=0.2em, parsep=0em, partopsep=0em, itemsep=0.1em, leftmargin=2em]
    \item[(i)] Consider a high-dimensional model, such as generalized linear regression or random feature models, with data following some pre-specified distribution;
    \item[(ii)] Equate our estimation problem to that of the same model on Gaussian data via universality;
    \item[(iii)] Simplify the Gaussian optimization problem via the CGMT into a format that can be more readily solved, either analytically or computationally.
\end{enumerate}
One substantial limitation of existing Gaussian universality and CGMT analyses is that the data must consist of independent—and often also identically distributed—vectors, which is not realistic for many applications. Several forms of dependence are commonly observed in practice:
\begin{itemize}[topsep=0.2em, parsep=0em, partopsep=0em, itemsep=0.1em, leftmargin=2em]
    \item \emph{Block dependence. } An important example of dependence in machine learning is found in data augmentation\footnote{The definition of data augmentation in machine learning differs from its use in statistics. In the latter, data augmentation often refers to the introduction of latent variables to the model, e.g.~in the EM algorithm.}, a technique that synthetically expands a training dataset by applying random transformations to existing data and incorporating the transformed data back into the dataset \citep{taqi2018impact,shorten2021text,volkova2024DA}. In machine learning practice, data augmentation has become one of the most widely adopted methods, especially in the presence of invariance (e.g. symmetries) or an underlying structure (e.g sparsity) \citep{lyle2020benefits}. Theoretically, however, the dependence arising from multiple transformed copies of the same observation makes the effect of data augmentation challenging to analyze. 
    \item \emph{$m$-dependence.} Another common form of dependence manifests through a finite dependency neighborhood: Under spatial moving average models \citep{cressie1993statistics}, the observation at a given point is dependent on a local neighborhood of observations but no further. Similar examples are ubiquitous in time series, graph and spatial analysis \citep{cryer1986time,brock1992simple,schweinberger2015local,wackernagel2003multivariate};
    \item {\emph{$\beta$-mixing. } Data can also depend on infinitely many variables, with strong short-range dependence and decaying long-range dependence. This is typically described by mixing conditions \citep{billingsley2017probability,bradley2005basic}, and is also found in many common time series and spatial models \citep{deo1973note,tuan1985some,tsay2005analysis,gelfand2010handbook}.}
     
\end{itemize}
This paper, for the first time, extends
the Gaussian universality principle beyond the independence assumption to encompass dependent vectors $(X_i)$ in the context of high-dimensional logistic regression. {Universality results are provided for block dependence, $m$-dependence, as well as specific $\beta$-mixing processes.}
Moreover, 
we develop a novel CGMT framework, 
accommodating
dependence both between covariates and observations under a certain ``low-rank" 
assumption. Leveraging these two new {tools}, 
we precisely characterize the impact of data augmentation on the risk. We notably investigate the effectiveness of data augmentation when the invariance or structure of the problem is only partially known, as is often the case in practice \citep{benton2020learning,yang2023generative}.

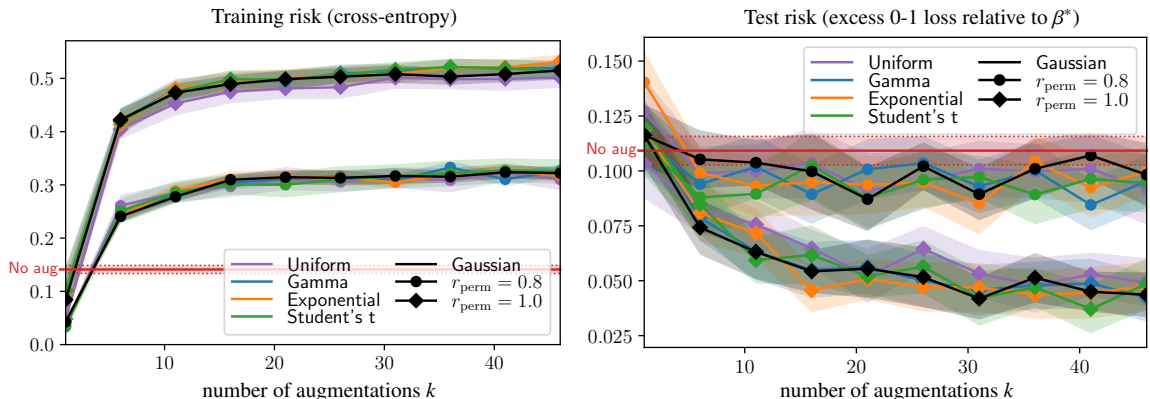
\begin{figure}[t]
    \centering 
    \begin{tikzpicture}

        \node[inner sep=0pt] at (-3.8,0){\includegraphics[trim={.5em .5em 0em .5em},clip,width=.5\linewidth]{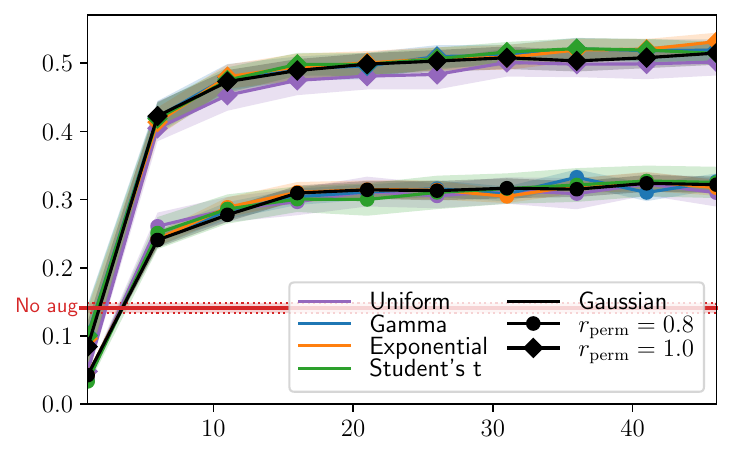}};
        \node[inner sep=0pt] at (3.8,0){\includegraphics[trim={.5em .5em .5em .5em},clip,width=.5\linewidth]{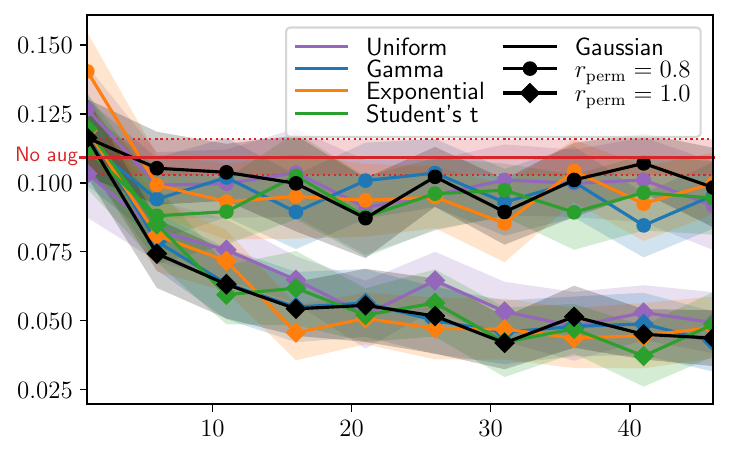}};

        \node[inner sep=0pt] at (-3.2, 2.5) {\scriptsize Training risk (cross-entropy)};
        \node[inner sep=0pt] at (4.4, 2.5) {\scriptsize Test risk (excess 0-1 loss relative to $\beta^*$)};

        \node[inner sep=0pt] at (-3.4, -2.45) {\scriptsize number of augmentations $k$};
        \node[inner sep=0pt] at (4.2, -2.45) {\scriptsize number of augmentations $k$};
    \end{tikzpicture}
    \vspace{-2em}
    \centering
    \caption{Universality of risks of a logistic regressor, trained with different number and amount of random permutations. See \Cref{sec:DA} and \Cref{appendix:simulation:details} for the detailed setup.}
    \label{fig:perm}
    \vspace{-1.5em}
\end{figure}

\subsection{Model Overview}\label{modeloverview}

We observe high-dimensional data $(X_i, y_i)_{i=1}^n$ with covariates $X_i \in \R^p$ and labels $y_i \equiv y_i(X_i) \in \{0, 1\}$. We consider the \textit{proportional regime}, where the signal dimension $p$ grows linearly with the sample size $n$. {In our main universality result (\Cref{sec:universality}), which is used for analyzing data augmentation, } the data $\bX \coloneqq (X_i)_{i \leq n}$
--- 
not assumed to be identically distributed
--- are block dependent: \vspace{-.5em}
\begin{align}
    (X_i, y_i) \indep (X_j, y_j)~ \text{ if }~ j \notin\mathcal B_i := \left\{k\lfloor \tfrac{i-1}{k}\rfloor + 1, \ldots, k\lfloor \tfrac{i-1}{k}\rfloor + k\right\}.
    \label{eq:blockdepdef}
\end{align}
We also consider data that satisfy $m$-dependence and $\beta$-mixing; see \Cref{sec:universality:extend} for the precise definitions. To relate the labels to their covariates, we assume there is a true signal $\beta^* \in \R^p$ such~that 
\begin{align}
    \P\l(y_i = 1\mid \bX\r) = \sigma\big(X_i^\intercal\beta^*\big),\quad\quad \sigma(t) := (1 + e^{-t})\inv.\label{eq:datageneration}
\end{align}
The signal is estimated via  a penalized and weighted logistic regression:  
\begin{align}
    \hat\beta(\bX) &\coloneqq \textstyle\argmin_{\beta\in \R^p} \mfrac1n\msum_{i=1}^n\omega_i\l(\log(1+e^{X_i^\intercal\beta}) - y_iX_i^\intercal\beta\r) + \mfrac{\lambda}{2n}\|\beta\|^2,\label{eq:betahat}
\end{align}
where $(\omega_i) \in [0,1]^{\mathbb{N}}$ are deterministic weights.
If 
the weights are all set as $1$, \eqref{eq:betahat} recovers the traditional penalized logistic regression. More generally, $\omega_i$'s can be chosen to be different to accommodate potential heterogeneity such as heteroskedasticity (e.g., \cite{shalizi2013advanced}), or unique cases like data augmentation (\Cref{sec:DA}). Examples of setups that can be handled by this model include:
\begin{itemize}[topsep=0.2em, parsep=0em, partopsep=0em, itemsep=0.4em, leftmargin=2em]
    \item \emph{
   {Dependent } $(X_i)$ and conditionally independent $(y_i)$:} 
    {Dependent } covariates commonly arise in various applications. For instance, in biological experiments on mice, the littermate effect introduces dependence between the behaviors
    of mice from the same litter \citep{haseman1979analysis}. Similarly, local dependence is prevalent in genomic data \citep{yu2017bio}. In those settings, while the covariates are dependent, each response variable $y_i$ has direct dependence only on $X_i$ and not on the other covariates.  
    \item \emph{
       {Block Dependent } $(X_i)$ and $(y_i)$}\label{remark:dependence}. In many other practical settings, the response variable $y_i$ can depend on 
        a set of multiple covariates $\{X_j : j \in \mathcal B_i\}$. 
    In ICU settings, for example, predicting 24-hour mortality is improved 
    by incorporating past data on the same patient
    \citep{plate2019incorporating}. To model these setups,
      we can assume that there exists a matrix $A \in \R^{n\times n}$ that models the dependencies between observations, so that $y_i$ depends on the linear combination $Z_i \coloneqq \sum_{j\in\mathcal B_i}a_{i,j}X_j$  of predictors in the same block \citep{wu1979use}. 
    In this case,
       $(Z_i,y_i)$ is 
      still a block dependent process that 
      satisfies all the assumptions of our setup.
    \item \emph{Data Augmentation}. If the original dataset $(Z_i)$ are independent, then the augmented data will exhibit block dependence within each set of augmented copies of $X_i$, {whereas the response $y_i$ typically depends only on the original variable $Z_i$}; more details in \Cref{sec:DA}.
\end{itemize}

\begin{remark} \label{remark:model:DA}
In \eqref{eq:betahat}, the logistic regression is performed on the same variables that $y_i$ depends on. \\
\noindent
We have chosen this presentation for simplicity. \Cref{appendix:general:model} includes a more general model, which allows for the label to depend on the entire block $\{X_j : j\in\mathcal B_i\}$ while $X_i$ is regressed only on a subset of those observations. This hence allows for the regression to be misspecified. Data augmentation, for example, implicitly assumes that the label of the transformed data depends only on the untransformed data, which makes this generalization necessary.
\end{remark}

\noindent
Detailed assumptions on our data-generating process and model are presented in \Cref{defs_assumps,sec:universality}. 

\vspace{-.5em}

\subsection{Summary of Results}

The main contributions of our paper are as follows:

\begin{enumerate}[topsep=0.2em, parsep=0em, partopsep=0em, itemsep=0.4em, leftmargin=2em]
    \item[(i)] \textit{Universality}.  {Under mild conditions, we prove a set of dependent Gaussian universality results for the training and test risks, which address block dependence (\Cref{mainthm} in \Cref{sec:universality}), $m$-dependence, and specific $\beta$-mixing processes (\Cref{mainthm:extend} in \Cref{sec:universality:extend}). } To the best of our knowledge, this constitutes the first results demonstrating that universality holds in the proportional regime for estimators trained with dependent observations. A key consequence is that if the data is uncorrelated, even if dependent, the asymptotic risk is the same as in the independent setting. Hence previously derived results for logistic regression still hold (see \Cref{related_work}). To tackle the case where the data is correlated, we propose a novel CGMT result. 
    \item[(ii)] \textit{CGMT}. We introduce a novel extension of the CGMT for Gaussian matrices with a ``low-rank" dependence structure (\Cref{thm:CGMT_2}) in \Cref{sec:cgmt}. In particular, this result accommodates dependence across both columns and rows. This significantly broadens the applicability of the CGMT approach, which, until now, required either the rows or the columns to be independent.  
    \item[(iii)] \textit{Data Augmentation}. Using our universality result and the dependent CGMT, we exactly characterize the asymptotic risks of logistic regression under different forms of data augmentation, such as random permutations when the covariates are partially exchangeable and sign flipping when $\beta^*$ is sparse. We observe that when the structure of the problem is fully known, data augmentation significantly decreases the test risk. However, when it is only partially known, the effect of data augmentation can be negligible. See \Cref{sec:DA}.
\end{enumerate}
The remainder of the paper consists of a literature overview in \Cref{related_work}, an overview of proof techniques in \Cref{proof_overview}, and a discussion of future directions in \Cref{conclusion}.

\section{Definitions }\label{defs_assumps}

In this section, we define various quantities that will be used throughout the paper.
We first define the empirical risk of an estimator as \\[-1.5em]
\begin{align*}
    \hat R_n(\beta; \bX) \;\coloneqq\; \mfrac1n\msum_{i=1}^n\omega_i\l(\log(1+e^{X_i^\intercal\beta}) - y_iX_i^\intercal\beta\r) + \frac{\lambda}{2n}\|\beta\|^2.
\end{align*} \\[-1.5em]
The performance of our estimator is then evaluated on a new observation $X_{\rm new}$, which we do not assume to have the same observation as any of the training points. The test risk is hence defined as  \\[-1.5em]
\begin{align*}
    R_{\text{test}}(\hat \beta(\bX)) \;\coloneqq\; 
    \mean\big[  
   \ell_{\rm test}\big( X_{\text{new}}^\intercal \hat \beta(\bX), X_{\text{new}}^\intercal \beta^* \big)
        \,\big|\, 
    \hat \beta(\bX) \big]
    \;,
\end{align*} \\[-1.5em]
where the expectation is taken over the mean-zero random vector $X_{\text{new}}$ that is independent of the trained estimator $\hat \beta = \hat\beta(\bX)$, 
and where $\ell_{\rm test}$ is a generic locally Lipschitz function. 
For our simulations, $\ell_{\rm test}$ will be the 0-1 loss, for which we also verify our results (see \Cref{proof_second_eq}). To compare the distribution of training risk on Gaussian and non-Gaussian data, we use the metric given by  \\[-1.5em]
\begin{align*}
    d_{\cH}(X, Y) \;\coloneqq\; \msup_{h\in\mathcal H}\E\l[h(X) - h(Y)\r],
\end{align*} \\[-1.5em]
where $\mathcal H$ is the set of differentiable functions $h$ with Lipschitz derivative satisfying $\|h\|_\infty, \|h'\|_\infty \leq 1$; see \cite{montanari2022universality} for why this distance metrizes convergence in distribution.

We shall establish universality with respect to the Gaussian surrogates $G_i \sim \mathcal N\l(0, \Var(X_i)\r)$, where each block $(G_j)_{j\in\mathcal B_i}$ is jointly normal. The corresponding dataset $\l(G_i, y_i(G_i)\r)_{i=1}^n$ satisfies the same assumptions as $\l(X_i, y_i(X_i)\r)_{i=1}^n$ in \Cref{sec:universality}. We also write the Gaussian counterpart of the test risk $R_{\rm test}$ as $R^G_{\rm test}(\hat\beta(\bG)) \coloneqq \mean\big[  
   \ell_{\rm test}\big( G_{\text{new}}^\intercal \hat \beta(\bG), G_{\text{new}}^\intercal \beta^* \big)
        \,|\, 
    \hat \beta(\bG) \big]$,
where $G_{\text{new}} \sim \cN\l(0, \Sigma_{\text{new}}\r)$ for $\Sigma_{\text{new}} \coloneqq \Var(X_{\text{new}})$ is a substitute for $X_{\rm new}$.

In our proofs, we will restrict our minimization problem to a particular set of the form
\vspace{-.3em}
\begin{align}
    \mathcal S_p := \Big\{\beta \in \R^p : \|\beta\|_2 \leq \textsf L\hspace{-2pt}\sqrt p,\ \|\beta\|_\infty \leq \textsf Lp^{\frac{1-r}{2}}\Big\}\label{eq:sp}
\end{align} \\[-1.2em]
for fixed constants $\textsf L > 0$ and $r \in (0, \tfrac18)$. This can be viewed as the set of parameter vectors $\beta$ which cannot align too strongly with a particular direction to ensure pointwise normality, and is widely used in proving universality results (e.g., \cite{lahiry2023universality, han2023universality, montanari2022universality}). This restriction becomes equivalent to the unconstrained minimization when one proves that $\hat\beta \in \mathcal S_p$ with high probability, which can be done on a case-by-case basis. See Appendix \ref{the_discussion} for a more detailed discussion on the set $\mathcal S_p$.

\section{Universality of the Risks {Under Block Dependence}} \label{sec:universality}
We first state the various assumptions we place on our data generating process and the model. We postpone the discussion of those assumptions to \Cref{sec:discuss:assumptions} after the result is stated.

\begin{assumption}[Block-dependence]\label{blockdepass}  
There exists $k \geq 1$ such that $(X_i, y_i)$ is independent of $(X_j, y_j)$ whenever $j \notin\mathcal B_i = \left\{k\lfloor \tfrac{i-1}{k}\rfloor + 1, \ldots, k\lfloor \tfrac{i-1}{k}\rfloor + k\right\}$.
\end{assumption}

\begin{assumption}[Logistic Model]\label{logmodel}
The labels are generated as $y_i(X_i) = \mathbb I\big(X_i^\intercal\beta^* -\varepsilon_i > 0\big)$, where each $\varepsilon_i \sim \text{Logistic}(0, 1)$.
\end{assumption}

\begin{assumption}[Scaling \& Sub-Gaussianity]\label{scalinggg}
$\E[X_{i}] =0$ and $\E[ X_i X_i^\intercal ] = \Sigma_i$. Moreover, each $X_i$ is sub-Gaussian, and there exists $\textnormal{\textsf K}_X > 0$ such that $\sup_{1\leq i\leq n}\|X_i\|_{\psi_2} \leq \textnormal{\textsf{K}}_X/\sqrt n$, where $\| \argdot \|_{\psi_2}$ denotes the sub-Gaussian norm.
\end{assumption}

\begin{assumption}[Signal Size]\label{signalsizinghehe}
$\beta^* \in \mathcal S_p$ as in \eqref{eq:sp}, and there exists $\kappa\in(0, \infty)$ such that \label{ts}$\frac{p}{n} = \frac{p(n)}{n}\to\kappa$.
\end{assumption}

\begin{assumption}[Gaussian Approximation]\label{gauss_approx} For the hypersphere $\mathcal S^{k-1} := \{\bm x \in \R^k : \|\bm x\|_2 = 1\}$,
\begin{align*}
   \sup_{f\in \mathcal{F}}
   \;\, \sup_{\substack{\beta_1, \hdots, \beta_k  \in\mathcal S_p}}
   \;\, \sup_{\theta\in\mathcal{S}^{k-1}\,,\,i \leq n-k}
   \;\,
   \Big| 
   \E\Big[
    f \Big(\msum_{r=1}^k\theta_r X_{i+r}^\intercal\beta_r\Big) - f\Big(\msum_{r=1}^{k}\theta_r G_{i+r}^\intercal\beta_r
    \Big)\Big] \Big| \;\rightarrow\; 0 \;,
\end{align*}  
where $\mathcal{F} \coloneqq\{f : \R \to \R\mid f \in \mathcal C_1, \|f\|_{\infty} < \infty, \|\partial f\|_{\infty} \leq 1 \}$.
\end{assumption}


Assumptions \ref{blockdepass}-\ref{gauss_approx} are used to establish universality of the training risk. For universality of the test risk, we require two additional assumptions: one on the distribution of $X_{\text{new}}$, and one on the geometry of the training risk. Below, we denote $\tilde{\mathcal{F}} \coloneqq \{f : \R^2 \to \R \mid f \in \cC_1  , \|f\|_{\infty} < \infty, \|\partial f\|_{\infty}\le 1\}.$

\begin{assumption}[Gaussian Approximation of $X_{\text{new}}$]\label{courage} We have
\begin{align*}
    \msup_{f\in \mathcal{\tilde F}} \; \msup_{\substack{\beta \in\mathcal S_p}} \; \big|& \E\big[f\big( X_{\rm{new}}^T\beta, X_{\rm{new}}^T\beta^* \big)- f\big( G_{\rm{new}}^T\beta, G_{\rm{new}}^T\beta^*\big) \big]\big| \;\rightarrow\; 0 \;.
\end{align*}
\end{assumption}

\begin{assumption}  \label{assumption:test:risk} There exist constants $\bar \chi, \chi_* > 0$ such that for every fixed $\epsilon > 0$,
\begin{align*} 
    \P \, \bigg(    
    \min_{\beta \in \cS_p \,,\, | (\beta^\intercal \Sigma_{\text{new}} \beta)^{1/2} \,-\, \bar \chi | > \epsilon}\hspace{-7pt} \hat R_n(\beta; \bG)
    \;>\;  
    \min_{\beta \in \cS_p} \hat R_n(\beta; \bG)
    \bigg) 
    \;\rightarrow\; 
    1
    \qquad \text{ and } \qquad
    {\beta^{*}}^\intercal \Sigma_{\text{new}} \beta^{*}\;\rightarrow\; \chi_*^2
    \;.
\end{align*}\end{assumption}

Under these assumptions, the following theorem holds:

\begin{theorem}[Block Dependent Universality]\label{mainthm} Let $\l(X_i, y_i(X_i)\r)_{i=1}^n$ and $\l(G_i, y_i(G_i)\r)_{i=1}^n$ be generated under Assumptions \ref{blockdepass}-\ref{gauss_approx}, where each $G_i \sim \mathcal N\l(0, \Var(X_i)\r)$. Then 
\begin{align}
    d_{\mathcal H} \big(\textstyle\min_\beta \hat R_n(\beta; \bX), \textstyle\min_\beta \hat R_n(\beta; \bG)\big) \to 0.\label{eq:main_first}
\end{align}
Moreover, if Assumptions \ref{courage} and \ref{assumption:test:risk} also hold, then
\begin{align}
      | \, R_{\rm test}(\hat \beta(\bX)) - R^G_{\rm test}(\hat \beta(\bG)) \, |
    \;\xrightarrow{\P}\; 
    0
    \;.\label{eq:main_second}
\end{align}
\end{theorem} 

The proof of \Cref{mainthm} is deferred to \Cref{app_b} and \ref{proof_second_eq}, with a proof sketch given in \Cref{proof_overview}. This result allows us to better understand the properties of the risk—notably, we observe that it only depends on the distribution of the data through its first two moments. Consequently, the dependence among the observations influences the risk only via the covariance, rather than through a more intricate relationship. In other words, even if the data exhibits dependence, as long as it is uncorrelated, the asymptotic behavior of the risk is the same as the independent case,
and allows it to be analyzed using the existing extensive literature. In scenarios where the data is not uncorrelated, the risk of $\hat{\beta}(\bX)$ still simplifies to the risk of $\hat{\beta}(\bG)$, and such cases can be studied using our novel dependent CGMT approach (see \Cref{sec:cgmt}), as long as a ``low rank" dependence assumption holds. 

\subsection{Discussion of Assumptions}\label{sec:discuss:assumptions}

Assumptions \ref{logmodel}–\ref{signalsizinghehe} are
standard in high-dimensional settings. Assumptions \ref{gauss_approx}–\ref{assumption:test:risk} are mirror conditions required to establish universality of the risk in the independent case 
(e.g.~\citet{montanari2022universality,han2023universality}). In particular, our \Cref{gauss_approx} is closely related to Assumption 5 of \cite{montanari2022universality}. However, ours is slightly stronger, 
as it requires the joint convergence of $(X_{i_1}^T\beta_1, \dots, X_{i_k}^T\beta_k)$ to a Gaussian limit for all $\beta_1, \dots, \beta_k \in \mathcal{S}_p$.
This is a direct consequence of \Cref{blockdepass}, which relaxes the independence assumption to block dependence. It can hence be seen as a multivariate version of the pointwise normality assumption in \citet{montanari2022universality}.

To establish the Gaussian universality of the testing risk, in addition to the training risk, it is necessary to introduce further assumptions, specifically Assumptions \ref{courage} and \ref{assumption:test:risk}. \Cref{courage} closely resembles \Cref{gauss_approx}, but applies to $X_{\rm new}$ rather than our original data. Note that we did not require $X_{\rm new}$ to share the same distribution as any of the $(X_i)$. 
\Cref{assumption:test:risk} is a stronger condition:
informally, it states that in the Gaussian case, the optimizer should be concentrated on a small subset of $\mathcal{S}_p$. However, since this assumption pertains to the Gaussian data rather than $\bX$, it can be proven via our 
dependent CGMT framework, provided that \Cref{assumption:CGMT:low:rank} is satisfied. In the independent setting this is notably established in \citet[Eq.~92]{salehi2019impact}, \citet[Eq.~74]{dhifallah2021inherent} and \citet[Eq.~B.11]{thrampoulidis2016recovering}. \cite{montanari2022universality} does not impose such a condition, but instead studied a modified notion of the test risk. More formally \cite{montanari2022universality} proved the universality of $\min_{\beta\in \mathcal{\tilde S}(\bX)} R_{\rm test}(\beta)$ where $\mathcal{\tilde S}(\bX)\subset\mathcal{S}_p$ is a subset defined using the empirical risk (see Theorem 2 of \cite{montanari2022universality}). 

\section{Extension to $m$-dependent and Specific $\beta$-Mixing Processes} \label{sec:universality:extend}

In this section, we show that the universality result of \Cref{mainthm} also extends to data with $m$-dependence and specific $\beta$-mixing processes. In the definition of mixing, we shall temporarily make the dimension dependence explicit in the data $(X_i, y_i) \equiv (X^{(p)}_i, y^{(p)}_i)$ and $\bX \equiv \bX^{(p)}$. Following \cite{bradley2005basic}, for every $p \in \N$, we define the $\beta$-mixing coefficient of $(X^{(p)}_i, y^{(p)}_i)_{i \in \N}$ as 
\begin{align*}
    \beta_{\rm mix}(N) \;\coloneqq\; \sup_{p \in \N} \sup_{t \in \N} \,  \sup_{\cA \in \cP^{(p)}_{\leq t}, \, \cB \in \cP^{(p)}_{\geq t+N}  } \, \mfrac{1}{2} \msum_{A \in \cA} \msum_{B \in \cB} \big| \P( A \cap B ) - \P( A) \P(B) \big|  \,,
\end{align*}
where $\cP^{(p)}_{ \leq t}$ is the set of all finite partitions of $\sigma( (X^{(p)}_i, y^{(p)}_i) \,|\, i \leq t)$ and $\cP^{(p)}_{ \geq t + N}$ is the set of all finite partitions of $\sigma( (X^{(p)}_i, y^{(p)}_i) \,|\, i \geq t+N)$. We use this definition to state our $\beta$-mixing requirement for the triangular array $(X^{(p)}_i, y^{(p)}_i)_{i,p \in \N}$ below in \Cref{general} $(ii)$.


 \begin{assumption}[$m$-dependence or specific $\beta$-mixing processes]\label{general}  
One of the following holds:
\begin{enumerate}
    \item[(i)] $(X_i, y_i)$ is independent of $(X_j, y_j)$ whenever $|i-j| > m$; 
    \item[(ii)] There exists some $r \in (0,1)$ such that $\sum_{l=1}^\infty \beta_{\rm mix}(l)^r < \infty$.
    Moreover, there exists a process $(Z^{(p)}_i)_{i, p \in \N}$ of centered independent random vectors with bounded sub-Gaussian norms such that, for every $i, p \in \N$, we can express 
    \begin{align*}
        X^{(p)}_i \; =\; \msum_{j\in \N} c^{(p)}_{i,j} \, Z^{(p)}_j
    \end{align*}
    for some constants $(c^{(p)}_{i,j})$. We also assume that there exists an universal constant $\underline{c}>0$ such that $\sqrt{p}\lambda_{\min}(\Var(Z^{(p)}_j))\ge \underline{c}$ for all $j,p\in \mathbb{N}$.
\end{enumerate}

\end{assumption}

\Cref{general} substitutes the block-dependence \Cref{blockdepass}.
\Cref{general}(ii) is restrictive, but already covers important $\beta$-mixing processes of practical interests: For example, any autoregressive model \citep{tsay2005analysis,gelfand2010handbook}
can be expressed in the form of \Cref{general}(ii).

\begin{assumption}[Strong Gaussian Approximation]\label{gauss_approx2} For all $d\in\mathbb{N}$, we have
\begin{align*}
   \sup_{f\in \mathcal{F}}
   \;\, \sup_{\substack{\beta_1, \hdots, \beta_d  \in\mathcal S_p}}
   \;\, \sup_{\theta\in\mathcal{S}^{d-1}\,,\,i \leq n-d}
   \;\,
   \Big| 
   \E\Big[
    f \Big(\msum_{r=1}^d\theta_r X_{i+r}^\intercal\beta_r\Big) - f\Big(\msum_{r=1}^{d}\theta_r G_{i+r}^\intercal\beta_r
    \Big)\Big] \Big| \;\rightarrow\; 0 \;,
\end{align*}  
where $\mathcal{F} \coloneqq\{f : \R \to \R\mid f \in \mathcal C_1, \|f\|_{\infty} < \infty, \|\partial f\|_{\infty} \leq 1 \}$.
\end{assumption}

\Cref{gauss_approx2} is a stronger form of \Cref{gauss_approx}: Instead of requiring the Gaussian approximation to hold for $k$-many vectors, where $k$ is the dependency block size, we now require this to hold for every fixed $d \in \N$. Note however that we do not require $d$ to grow with $n$; it suffices that for every fixed $d$, the convergence holds as $n \rightarrow \infty$. By replacing Assumptions \ref{blockdepass} \& \ref{gauss_approx} by Assumptions \ref{general} \& \ref{gauss_approx2}, we may state our extended universality result for $m$-dependence and specific mixing processes.


\begin{theorem}[Universality under $m$-dependence or mixing] \label{mainthm:extend} 
    Let $\l(X_i, y_i(X_i)\r)_{i=1}^n$ and $\l(G_i, y_i(G_i)\r)_{i=1}^n$ be generated under Assumptions \ref{logmodel}-\ref{signalsizinghehe}, where each $G_i \sim \mathcal N\l(0, \Var(X_i)\r)$. Assume in addition that Assumptions \ref{general} and \ref{gauss_approx2} hold. Then 
\begin{align}
    d_{\mathcal H}\big( \textstyle\min_\beta \hat R_n(\beta; \bX), \textstyle\min_\beta \hat R_n(\beta; \bG)\big) \to 0.\label{eq:main_first_mdep}
\end{align}
Moreover, if Assumptions \ref{courage} and \ref{assumption:test:risk} also hold, then
\begin{align}
      | R_{\rm test}(\hat \beta(\bX)) - R^G_{\rm test}(\hat \beta(\bG)) |
    \;\xrightarrow{\P}\; 
    0
    \;.\label{eq:main_second_mdep}
\end{align}
\end{theorem}

The proof of \Cref{mainthm:extend} is an adaptation of the block-dependent case of \Cref{mainthm}, and is { included in \Cref{lorelai,lorelai2}}. To give an overview, the first step of the proof is to apply the classical technique of representing the data as an alternating sequence of big blocks and small blocks of random vectors, where the small blocks are then ignored 
\citep{bernstein1927extension,ibragimov1975independent,davidson1992central}. In the $m$-dependent case, the big blocks become independent provided that each small block is of size at least $m$, and the block-dependent result of \Cref{mainthm} applies directly. In the mixing case, the big blocks are only approximately independent and, hence, \Cref{mainthm} does not directly apply. To be able to use the results we derived in the independent case, we use the embedding result of \cite{yu1994rates}. This result allows one to compare the expectation of functions of $\beta$-mixing block with that of exactly independent random variables. However, a difficulty is that in our setting, the sizes of the blocks are not allowed to depend on $n$, and the number of blocks that can be approximated by independent ones cannot depend on $n$. Hence, we cannot apply the embedding result of \cite{yu1994rates} directly to the risk. We instead use \cite{yu1994rates} to bound how much the risk changes along the path of interpolation between $\bG$ and $\bX$. Finally, note that a key ingredient of the proof is being able to control the largest eigenvalue of $\bX^T\bX$. The proof relies on being able to prove a Bernstein inequality for $\sum_{i=1}^n(X_i^T\beta)^2$ for all $\beta\in \mathcal{S}_p$. Under the mixing setting, \Cref{general}(ii)  allows us to do so by rewriting $\sum_{i=1}^n(X_i^T\beta)^2$ as a quadratic form over the independent process $(Z_j).$ In general, this assumption --- that $(X_i)$ can be rewritten as an infinite sum of independent processes --- can be weakened to processes for which one can control the moments of $\lambda_{\rm max}(\bX^T\bX).$



\section{Dependent CGMT}\label{sec:cgmt}

Under general conditions, {\theoremref{mainthm,mainthm:extend}} allows us to study the risk of $\hat\beta(\bX)$ via that of $\hat\beta(\bG)$. When
$X_i$'s are isotropic and uncorrelated, $\bG = (G_i)_{i \leq n}$ can be viewed as an $\R^{p \times n}$ matrix with i.i.d.~standard Gaussian entries. In this case, the risk of $\hat\beta(\bG)$ can be studied via the CGMT method.
Broadly speaking, the classical CGMT method first relates $\min_\beta \hat R_n(\beta; \bG)$ to the optimization \\[-1.4em]
\begin{align*}
    \Psi_{\mathcal{S}_w, \mathcal{S}_u}\;\coloneqq\;
    \mmin_{w \in \mathcal{S}_w} \, \mmax_{u \in \mathcal{S}_u} \, L_\Psi(w,u)
    \qquad 
    \text{ with } 
    \qquad
    L_\Psi(w,u) 
    \;\coloneqq\;
    w^\intercal \bH u 
    + f(w,u)\;,
    \tagaligneq \label{eq:CGMT:PO}
\end{align*} \\[-1.4em]
where $\cS_w \subset \R^p$ and $\cS_u \subset \R^n$ are compact and convex, $f:\mathcal{S}_w\times \mathcal{S}_u\rightarrow\mathbb{R}$ is a convex-concave function, and $\bH$ is typically a suitably projected version of $\bG$. The main result of CGMT is that $\Psi_{\mathcal{S}_w, \mathcal{S}_u}$ is equivalent to a simpler optimization involving only two Gaussian vectors (see \Cref{appendix:statement:CGMT}).

In the dependent setting, however, $\bH$ can exhibit both correlated columns and rows, and the standard CGMT framework is generally not applicable.
 To address this, 
 we develop a more 
general CGMT framework that accommodates 
 a ``low-rank assumption" on the 
 dependence structure of $\bH$. 
 
\begin{assumption}
    [Low-rank Dependence]  \label{assumption:CGMT:low:rank}
    Let $\bH$ be an $\R^{p \times n}$ Gaussian matrix. There exist $M\in \mathbb{N}$ and symmetric positive  semi-definite matrices  $(\Sigma^{(l)},\tilde\Sigma^{(l)})_{l\le M}$, with $\Sigma^{(l)} \in \R^{p \times p}$ and $\tilde \Sigma^{(l)} \in \R^{n \times n}$, s.t.
    \begin{align*}
        \Cov[ \bH_{ji}, \bH_{j'i'}] \;=\; \msum_{l=1}^M  \Sigma^{(l)}_{jj'} \tilde \Sigma^{(l)}_{ii'} 
        \;\qquad\;
         \text{ for all } 
         i,i' \leq n 
         \text{ and }
         j,j' \leq p
     \;.
\end{align*}
\end{assumption} 

\begin{remark} For $M=1$, \Cref{assumption:CGMT:low:rank} is equivalent to the assumption that $\bH = \big(\Sigma^{(1)}\big)^{1/2} \bH' \big(\tilde \Sigma^{(1)}\big)^{1/2}$ for some Gaussian matrix $\bH'$ with i.i.d.~standard normal entries. For $M > 1$, this says that the dependence is captured by some sum of covariance matrices that ``factorise". In \Cref{appendix:low:rank:discussion}, we discuss how this is a natural assumption for the dependence structure that arises in data augmentation.
\end{remark}

Denote $\| v \|_{\Sigma'} = \sqrt{v^\intercal \Sigma' v}$. Under \Cref{assumption:CGMT:low:rank}, we shall compare $\Psi_{\mathcal{S}_w,\mathcal{S}_u}$ to the risk
\vspace{-.5em}
\begin{align*}
    &\;\psi_{\mathcal{S}_w,\mathcal{S}_u}
    \;\coloneqq\;
    \min_{w \in \mathcal{S}_w} \, \max_{u \in \mathcal{S}_u} \, L_\psi(w,u)
    \;,
    \\
    &\qquad\textrm{ where }
    \quad 
    L_\psi(w,u) 
    \;\coloneqq\;
    \msum_{l=1}^M \Big\{ 
        \| w \|_{\Sigma^{(l)}} \bh_l^\intercal \big( \tilde \Sigma^{(l)} \big)^{1/2} u + w^\intercal \big( \Sigma^{(l)} \big)^{1/2} \bg_l \| u \|_{\tilde \Sigma^{(l)}}
    \Big\} 
   +
    f(w,u)
    \;.
    \tagaligneq \label{eq:CGMT:AO}
\end{align*} 
$(\bh_l,\bg_l)_{l \leq M}$ are independent standard Gaussians respectively in $\R^n$ and $\R^p$. Our next result formalizes the equivalence of $\Psi_{\mathcal{S}_w,\mathcal{S}_u}$ and $\psi_{\mathcal{S}_w,\mathcal{S}_u}$, and additionally controls $\hat w_\Psi \in \cS_p$, the minimizer of $\Psi_{\mathcal{S}_w, \mathcal{S}_u}$.

\begin{theorem}[Dependent CGMT] \label{thm:CGMT_2} Suppose $\mathcal{S}_w $ and $\mathcal{S}_u$ are compact and convex, and $f$ is continuous and convex-concave on $\mathcal{S}_w \times \mathcal{S}_u$. Under \Cref{assumption:CGMT:low:rank}, the following statements hold:
\begin{enumerate}[topsep=0.2em, parsep=0em, partopsep=0em, itemsep=0em, leftmargin=2em]
    \item[(i)] For all $c \in \R$, 
    \begin{align*}
         \P(  \Psi_{\mathcal{S}_w, \mathcal{S}_u} \leq c ) 
         \;\leq&\;  2^M  \, \P(  \psi_{\mathcal{S}_w, \mathcal{S}_u} \leq c )
         &\text{ and }&&
         \P(  \Psi_{\mathcal{S}_w, \mathcal{S}_u} \geq c ) \,\leq\,  2^M  \, \P(  \psi_{\mathcal{S}_w, \mathcal{S}_u} \geq c )\;;
    \end{align*} 
    \item[(ii)]  Let $\cA_p$ be an arbitrary open subset of $\mathcal{S}_w$ and $\cA_p^c \coloneqq \mathcal{S}_w \setminus \cA_p$. If there exist constants $\bar \phi_{\mathcal{S}_w}$, $\bar \phi_{\cA^c_p}$ and $\eta, \epsilon > 0$ such that $\bar \psi_{\cA^c_p} \geq \bar \psi_{\mathcal{S}_w} + 3 \eta$, \;\;$\P( \psi_{\mathcal{S}_w, \mathcal{S}_u} \leq \bar \psi_{\mathcal{S}_w} + \eta  ) \,\geq\, 1 - \epsilon$ \; and $\P( \psi_{\cA^c_p, \mathcal{S}_u} \geq \bar \psi_{\cA^c_p} - \eta  ) \,\geq\, 1 - \epsilon$, then 
    \begin{align*}
        \P( \hat w_\Psi  \,\in\, \cA_p ) \,\geq\, 1 - 4 \epsilon \;.
    \end{align*}
\end{enumerate}
\end{theorem}

\begin{remark} Convexity is not required for the first bound of (i); see \theoremref{thm:CGMT:full} for the full theorem.
\end{remark}

Notably, \theoremref{thm:CGMT_2} implies an asymptotic concentration result for the minimizer $\hat w_\Psi$ in $\cA_p$: 

\begin{corollary}[Asymptotic CGMT] \label{cor:CGMT}   Let $\cA_p$ be an arbitrary open subset of $\cS_w$ and $\cA_p^c \coloneqq \cS_w \setminus \cA_p$. If there exists constants $\bar \psi < \bar \psi^c$ such that $\psi_{\cS_w, \cS_u} \xrightarrow{\P} \bar \psi$ and $\psi_{\cA^c_p, \cS_u}  \xrightarrow{\P} \bar \psi^c$, then \\[-1.2em]
\begin{align*}
    \P(  \hat w_\Psi \in \cA_p ) \;\rightarrow\; 1 \;.
\end{align*}
\end{corollary}

\noindent
\Cref{thm:CGMT_2} and \corollaryref{cor:CGMT} have several important implications:

\vspace{.5em}

\noindent
\emph{Simplifying the analysis of $\hat R_n(\beta; \bG)$.} \Cref{thm:CGMT_2} reduces the analysis of $\hat R_n(\beta; \bG)$, which involves a high-dimensional and correlated Gaussian matrix, to a loss involving only Gaussian vectors. This substantially simplifies the analysis of the asymptotic risk, as one can avoid invoking random matrix theory. Indeed in the isotropic case, the conversion of $\psi_{\cS_w, \cS_u}$ into a deterministic, low-dimensional problem has been performed in many models through algebraic calculations and the min-max theorem \citep{thrampoulidis2016recovering,salehi2019impact,dhifallah2021inherent}. In our case, the terms in $\psi_{\cS_w, \cS_u}$ depend on the $2M$ covariance matrices, and the complexity of these calculations grows with $M$. We present the calculations for special cases of data augmentation in \Cref{sec:DA} and \Cref{appendix:additional}.

\vspace{.5em}

\noindent 
\emph{Pipeline of analysis for dependent data. } Together with our dependent universality result (\theoremref{mainthm,mainthm:extend}), \Cref{thm:CGMT_2} extends the pipeline of analysis discussed in \Cref{sec:intro} to dependent
data in logistic regression. Since \Cref{thm:CGMT_2} is model-independent, we also expect it to be valuable to other setups, provided that an analogous dependent universality result is established.

\vspace{.5em}

\noindent 
\emph{Universality of test risk. } Our CGMT also helps with verifying \Cref{assumption:test:risk}, required for the universality of test risk in \theoremref{mainthm,mainthm:extend}. To see this, let us identify $\cA_p$ in \Cref{thm:CGMT_2}$(ii)$ as the set $\{ \, \beta \in \cS_p \mid |(\beta^\intercal \Sigma_{\text{new}} \beta)^{1/2} - \bar \chi | \leq \epsilon_n \}$. Under this notation, \Cref{assumption:test:risk} is a comparison between the training risks of two optimizations on $\cS_w \setminus \cA_p$ and $\cS_w$ respectively. \Cref{thm:CGMT_2} allows us to perform this comparison on the simpler auxiliary optimizations instead, which additionally allows for computing the value of $ \bar \chi$; see \Cref{appendix:DA:EQs}. 



\begin{remark}[Comparison to existing CGMT results] 
For comparison, \Cref{thm:CGMT_2} recovers the standard CGMT  with $\Sigma^{(1)} = I_p$, $\tilde \Sigma^{(1)} = I_n$ and $M=1$. It also recovers the multivariate CGMT of \citet{dhifallah2021inherent} by setting $\Sigma^{(l)}$ and $\tilde \Sigma^{(l)}$ as block diagonal matrices with $M$ equal-sized subblocks, such that the $l$-th subblock is identity and the other blocks are zero.
 \citet{akhtiamov2024novel} generalizes the block diagonal setup to allow non-identity subblocks, which is a special case of our \Cref{assumption:CGMT:low:rank}, but they also allow for transforming $w$ and $u$, which we do not address here.
\end{remark}



\section{Applications to Data Augmentation (DA)}\label{sec:DA}

As an example application of \Cref{mainthm,thm:CGMT_2}, we analyze the effect of data augmentation, which introduces a simple yet ubiquitous form of block dependence in machine learning. To see how the dependence arises, let $(Z_i)_{i \leq m}$ be i.i.d.~mean-zero random vectors, which correspond to our original data, and let ${\phi_{1},\ldots,\phi_{mk}}$ be $n=mk$ i.i.d.~$\R^p \rightarrow \R^p$ transformations, which are the augmentations. Note that the coordinates of $Z_i$ may be locally dependent. DA synthesizes an artificial dataset $(X_i, y_i)_{i \leq n}$ by setting $X_i \coloneqq \phi_i(Z_{\lceil i/k\rceil})$, i.e.~each observation is augmented $k$ times, and setting $y_i = y_i(Z_{\lceil i/k\rceil})$ to retain the label of the original observation. The estimator $\hat \beta$ is then fitted on the augmented dataset through the minimization \\[-1.2em]
\begin{align*}
    \textstyle\min_{\beta \in \cS_p} \,
    \mfrac{1}{n} \msum_{i=1}^n \, 
    \big( 
        \log \big( 1 + e^{X_i^\intercal \beta} \big) 
        -
        y_i(Z_{\lceil i/k\rceil}) \times X_i^\intercal \beta
    \big)
    +    
    \mfrac{\lambda}{2n} \| \beta \|^2_2 
    \;.
    \tagaligneq \label{eq:DA:logistic}
\end{align*}\\[-1.2em]
See \remarkref{remark:model:DA} for how this relates to the model \eqref{eq:betahat}.
For simplicity, we assume that the test data $X_{\rm new}$ is identically distributed as the unaugmented $Z_1$. The practical heuristic behind DA is that, if $\phi_i$'s are chosen to reflect certain structures of the problem well, DA may improve the test risk of $\hat \beta$ despite the dependence introduced. While the benefits of DA are empirically observed across a large body of ML literature, limited theoretical attempts have provided an exact theoretical quantification, especially in the case of a classification task; see \Cref{related_work}. Here, we analyze several DA schemes:

\vspace{.2em}

\noindent
\textbf{Random permutations under a group structure. } Suppose $Z_1$ can be broken down into $N$ groups of coordinates, each of size $p_t$, with $p_1 + \ldots + p_N = p$. Namely, $Z_1 = ( (Z^{(1)}_1)^\intercal \,,\, \ldots \,,\, (Z^{(N)}_1)^\intercal )^\intercal$
such that $\{Z^{(t)}_1\}_{t \leq N}$ are independent vectors and each $Z^{(t)}_i$ is $\R^{p_t}$-valued with i.i.d.~coordinates. The i.i.d.~structure within each group motivates one to augment the data by permuting coordinates within each group. As the full permutation group is exponentially large in $p$, a practical question is how much permutation should one perform. This concerns both the number of random permutations $k$ as well as the proportion of coordinates to permute. For simplicity, we fix $r_{\rm perm} \in [0,1]$, a proportionality parameter that may be chosen by practitioners, so that within each $t$-th group, we only consider permuting the top $\lceil r_{\rm perm} p_t \rceil$ coordinates. Each augmentation $\phi_i$ is a uniformly random permutation that permutes the top $\lceil r_{\rm perm} p_t \rceil$ coordinates within each $t$-th group.

\vspace{.2em}

\noindent
\textbf{Random sign flipping under a sparsity structure. } 
Consider a sparsity structure in $\beta^*$: A proportion $\rho^* \in [0,1]$ of the $p$ entries of  $\beta^*$ are non-zero, whereas the remaining $(1-\rho^*) p$ entries are zero. The positions of the non-zero coordinates are unknown in general, and $\rho^*$ may be known or unknown. This motivates the use of random sign flipping to shrink the estimate $\hat \beta$ at locations where the entries of $\beta^*$ may be zero. We fix some $\lceil r_{\rm flip}  p \rceil$ entries of $p$, where $r_{\rm flip} \in [0,1]$ is a parameter chosen by users. Each  $\phi_i$ is a random diagonal matrix, generated by drawing the fixed $\lceil r_{\rm flip}  p \rceil$ entries of its diagonal as i.i.d.~Rademacher variables, and setting the remaining entries of the diagonal to $1$. 

\vspace{.2em}

\noindent
\textbf{Random cropping under a sparsity structure.} Suppose $\beta^*$ has the same sparsity structure as above. Another way to encode our guess of the zero entries  is by randomly removing coordinates in the data. Here, each $\phi_i$ is  a random diagonal matrix generated by randomly setting $\lceil r_{\rm crop}  p \rceil$ entries of its diagonal to zero and leaving the remaining entries as $1$. 

\vspace{.2em}

As $y_i$'s depend on the unaugmented data instead of the augmented,
we apply a generalization of \theoremref{mainthm} 
(\Cref{appendix:universality:general}) to show that universality holds under the same assumptions. The key condition to verify is \Cref{gauss_approx}. In \Cref{check:augmentation}, we 
verify this
for sign flipping, cropping and
noise injection. 
For permutations, we 
verify this
under general conditions on the group sizes $p_t$'s for both a fixed and small number of groups $N$ and a growing number of groups $N=N(n) \rightarrow \infty$. Meanwhile, our CGMT result (\Cref{thm:CGMT_2}) applies to sign flipping and cropping above with $\Var[Z_1] = \frac{1}{p} I_p$, as well as permutations (\Cref{appendix:DA:special:cases}). Under simplifying conditions that hold for permutations and sign flipping, we also derive a set of $10$ deterministic and scalar 
equations \eqref{EQs} in \Cref{appendix:DA:EQs}, which explicitly characterize the test risk of logistic regressors. More general augmentations can be accommodated but at the expense of a more complicated system of equations. While \eqref{EQs} is complicated to state, we verify that in the isotropic case with no augmentation, it recovers exactly the characterizing equation 
by \cite{salehi2019impact}.

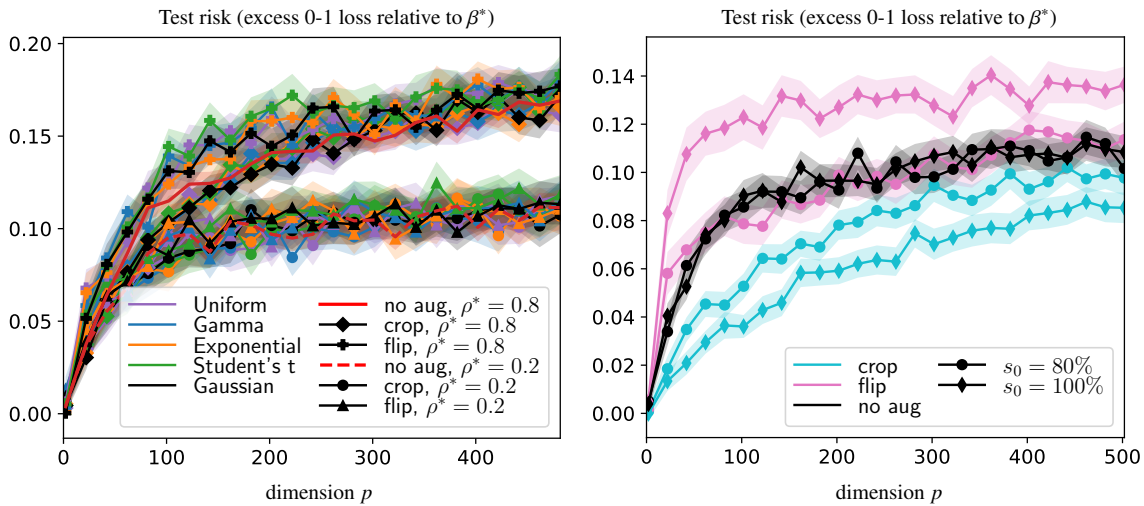
\begin{figure}[t]
    \centering 
    \begin{tikzpicture}
        \node[inner sep=0pt] at (-3.8,0){\includegraphics[trim={.5em .5em 0em .5em},clip,width=.5\linewidth]{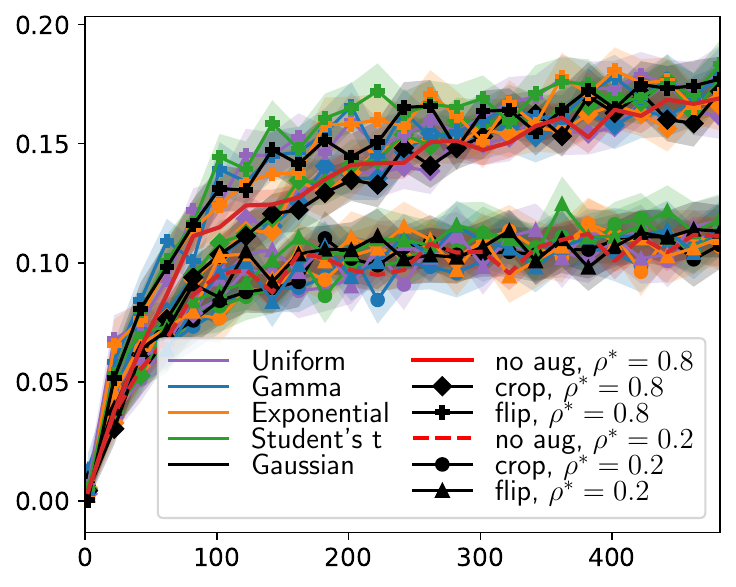}};
        \node[inner sep=0pt] at (3.8,0){\includegraphics[trim={.5em .5em .5em .5em},clip,width=.49\linewidth]{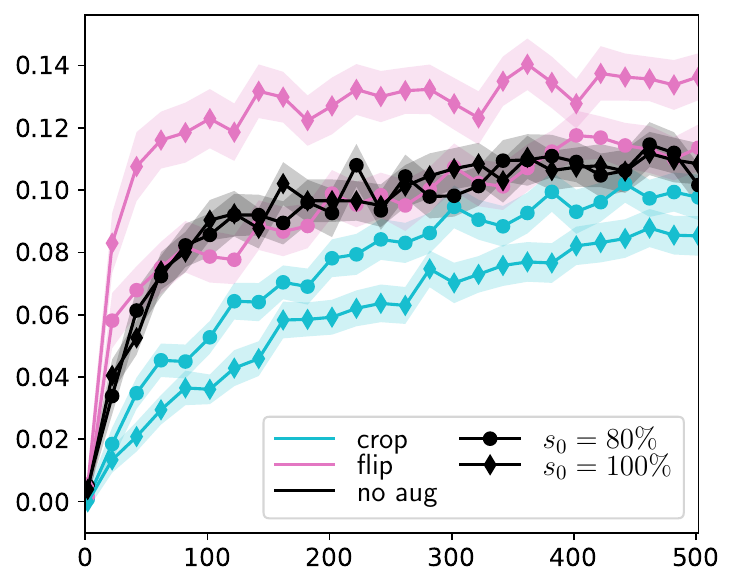}};

        \node[inner sep=0pt] at (-3.3, 3.1) {\scriptsize Test risk (excess 0-1 loss relative to $\beta^*$)};
        \node[inner sep=0pt] at (4.1, 3.1) {\scriptsize Test risk (excess 0-1 loss relative to $\beta^*$)};

        \node[inner sep=0pt] at (-3.4, -3.2) {\scriptsize dimension $p$};
        \node[inner sep=0pt] at (4.1, -3.2) {\scriptsize dimension $p$};
    \end{tikzpicture}
    \vspace{-1.0em}
    \centering
    \caption{Test risks under random cropping and sign flipping. \emph{Left}. Same setup as \Cref{sec:DA}. \emph{Right}. Signal ratio $\rho^*=0.2$ and the bottom $\lceil s_0(1 - \rho^*)p \rceil$ coordinates are known to be zero. }
    \vspace{-2em}
    \label{fig:crop:flip}
\end{figure}

\renewcommand{\thefootnote}{\fnsymbol{footnote}} 
These results enable us to understand the effects of DA through two possible approaches: To simulate a logistic regression on Gaussians, i.e.~a high-dimensional convex optimization with simple distributions, or to solve the nonlinear equations \eqref{EQs}, i.e.~a low-dimensional but highly non-convex optimization. We use the former approach with $m=200$ synthetic data and present results in Fig.~\ref{fig:perm} ($p=500$) and \ref{fig:crop:flip} ($k=30$); see \Cref{appendix:simulation:details} for full simulation details\footnote{Code is available at \href{https://github.com/KevHH/Dependent_Universality.git}{\texttt{github.com/KevHH/Dependent\_Universality}}.}.~A few observations:

\emph{Full permutations alleviate overfitting under a group structure}. Fig.~\ref{fig:perm} considers a high dimensional regime ($p/m = 2.5$), where logistic regression with no augmentations is expected to overfit. This typically manifests through a low training risk but a high test risk.  Fig.~\ref{fig:perm} shows that a full permutation  ($r_{\rm perm}=1.0$) of the i.i.d.~coordinates guards against this overfitting: The test risk improves substantially and as more and more augmentations are used. 

\emph{Full knowledge of the problem structure can be crucial}. A surprising observation from Fig.~\ref{fig:perm} is that using only a slightly smaller subgroup of permutations ($r_{\rm perm}=0.8$) results in test risks that are within error margins from that of no augmentations. This suggests that, at least within our model, exploiting the full set of permutation invariance is  critical for obtaining noticeable improvements. On the other hand, the sparsity setup for cropping and sign flipping does not allow the knowledge of the full structure by design, as it would otherwise imply that we know exactly which coordinates of $\beta$ to exclude from the regression. In the left plot of Fig.~\ref{fig:crop:flip}, perhaps surprisingly, we see that sign flipping and cropping both yield indistinguishable test risks from that under no augmentation.
For comparison, we perform Gaussian simulation in an artificial setup in the right plot of Fig.~\ref{fig:crop:flip}, where some
portion of the null entries of $\beta^*$ are known and on which 
 cropping and sign flipping are always performed. The remaining amount of cropping and sign flipping are applied to the rest of the coordinates. There, cropping ensures that the known
 null coordinates never enter the regression and outperforms no augmentation and sign flipping. In summary, these observations send a cautionary message: the benefits of data augmentation may be concretely visible only when the full problem structure is known, 
 which is too stringent for many practical setups.

\section{Related Literature}\label{related_work}
\emph{Universality}. Universality has been extensively studied in the probability, statistics, and ML literature. In statistics, the risk of a wide range of penalized linear models and the behavior of the approximate message passing (AMP) algorithm have been demonstrated to be universal  \citep{korada2011applications, montanari2017universality,abbasi2019universality,oymak2018universality,han2023universality,dudeja2023universality,wang2024universality,chen2021universality}. Beyond linear regression, it has been proven that generalized linear models, perceptron models, max-margin classifiers, random feature models, and others obtained via empirical risk minimization exhibit universal behavior \citep{montanari2022universality,montanari2023universality,dandi2024universality,gerace2022gaussian,korada2011applications, han2023universality, hu2022universality}.
~Those works either assume that the covariates are independent or that the observations projected on a wide range of directions are asymptotically normal (e.g., \cite{montanari2022universality}).
\cite{lahiry2023universality} further proved it for regularized linear regression if the covariates within each vector are block-dependent. 
However, they still assumed the rows of the design matrix were independent. \cite{huang2022data} moved beyond this condition and showed that under certain stability conditions, machine learning estimators trained with data augmentation satisfy Gaussian universality. Those conditions are, however, hard to verify and this paper does not cover overparameterized logistic regression. 

\emph{CGMT \& Exact Asymptotics}. 
The exact asymptotic risk of many high-dimensional models has been extensively studied using a variety of techniques. These include AMP \citep{donoho2009message}, the cavity method \citep{opper2001naive}, the Gaussian Min-Max Theorem \citep{gordon1985some}, and the CGMT \citep{thrampoulidis2014gaussian}. Since its introduction, the CGMT has been successfully applied to analyze the risk of numerous high-dimensional models (e.g.~\citet{stojnic2013framework, stojnic2013upper, thrampoulidis2015regularized, thrampoulidis2018precise, akhtiamov2024regularized, aolaritei2022performance, javanmard2022precise, mignacco2020role}). Further developments have extended the method to settings with independent but non-identically distributed rows \citep{akhtiamov2024novel,dhifallah2021inherent}. We, for the first time, extend CGMT to dependent rows and columns.

\emph{Data Augmentation}. Data augmentation is a widely utilized practice in machine learning, particularly in deep learning (e.g.~ \citet{taqi2018impact,shorten2019survey,shorten2021text,volkova2024DA}). Given its critical role, a number of studies have investigated its theoretical properties (e.g.~ \citet{hanin2020data,huang2022data,chen2020group,lin2024good}). The first work that applies CGMT to the study of data augmentation is \citet{dhifallah2021inherent}, which examines the impact of noise injection on logistic regression, demonstrating that it serves as an implicit regularization. However, their results and analysis are limited to noise injection, a data augmentation strategy that preserves the independence of the covariates, which simplifies their study. In this paper, we develop a novel universality and CGMT result that allows us to significantly broaden the scope of data augmentations we can study.

\emph{Logistic Regression}. 
Recently, substantial progress has been made in understanding the exact asymptotics of high-dimensional logistic regression in the proportional regime \citep{sur2019modern, deng2022model, kini2021label}. 
\citet{salehi2019impact} successfully adapted the CGMT framework to the logistic regression setting, enabling the analysis of regularized logistic regression. The issue of dependence in logistic regression, motivated by applications to biology and sociology, has also been well-studied \citep{bonney1987logistic, prentice1988correlated, reboussin2008locally, zorn2001generalized}. The dependence was notably modeled through mixed effects or latent variable models. In high dimensions, recent work has taken inspiration from the Ising model to propose a model in which  $y_i$ depends not only on $X_i$ but also other labels $(y_j)_{j\ne i}$, exhibiting network dependence \citep{mukherjee2021high}. Beyond the locally dependent setting, a recent wave of papers has studied properties of estimators trained on dependent data \citep{nagaraj2020least,zou2009generalization}. Beyond logistic regression, a number of papers have studied the asymptotics of Lasso estimators trained on time series data and in VAR models (e.g \cite{wong2020lasso,basu2015regularized,wong2016regularized, nicholson2020jmlr, guo2016HD}).


\section{Proof Overview} \label{proof_overview}
\emph{Universality.} The proof for the training risk builds upon a variant of the Lindeberg method (see e.g.~\citet{chatterjee2005concentration}) introduced by \citet{montanari2022universality}. One crucial difference, however, is that we need to account for the dependence between the observations. In \citet{montanari2022universality}, the independence of the data allows one to reduce the proof to showing that the mean of a particular function of $X_i^\intercal \beta$
approaches zero for all $\beta \in \cS_p$. This is done by exploiting the asymptotic normality of $X_i^\intercal \beta$. However, in the presence of dependence, this reduction is no longer valid. Instead, we must control the mean of a function of $(X_{i_1}^\intercal \beta_1, \dots, X_{i_k}^\intercal \beta_k)$ for all $\beta_1, \dots, \beta_k \in \mathcal{S}_p$. This requires establishing its joint asymptotic normality and a more careful analysis.

The proof for the test risk builds on the observation that under \Cref{courage},
the test risk  depends asymptotically only on $\hat \beta(\bX)^T\Sigma_{\rm new}\hat\beta(\bX)$. We further exploit universality of the training risk and \Cref{assumption:test:risk} to obtain that $\hat \beta(\bX)^T\Sigma_{\rm new}\hat\beta(\bX)$
converges in probability to a deterministic constant $\bar \chi^2$, which is the same limit as in the Gaussian case. The desired result then directly follows.


\emph{ CGMT for Data Augmentation (DA).} In DA, the covariance of a set of augmented data $\{ \phi_1(Z_1),$ $\ldots, \phi_k(Z_1)\}$ is completely described by the variance of the individual data points $\Var[\phi_1(Z_1)]$ and the covariance between two differently transformed data $\Cov[\phi_1(Z_1), \phi_2(Z_1)]$. As a result, this satisfies the low-rank dependence assumption of our CGMT (\Cref{assumption:CGMT:low:rank}) with $M=2$. However, the actual application of the CGMT is more subtle since the logistic regression \eqref{eq:betahat} is not a priori in the form of the primary optimization \eqref{eq:CGMT:PO}. Similar to \citet{thrampoulidis2016recovering, salehi2019impact, dhifallah2021inherent}, we first move the data $X_i$ outside the logarithm in \eqref{eq:betahat} via Lagrange multipliers. This yields a formulation similar to \eqref{eq:CGMT:PO} involving a high-dimensional $\R^{p \times n}$ matrix $\bX=(X_1, \ldots, X_n)$. Due to the presence of the labels $y_i$ and their nonlinear dependence on the data, CGMT can only be applied to a suitably projected version of $\bX$, say $P \bX$, that is uncorrelated with both the labels $y_i$ and the remainder $(I_p-P)\bX$. In the isotropic and independent case, $PX_i$ and $(I_p-P)X_j$ are uncorrelated for any projection matrix $P$, whereas $y_i$ depends on $X_i$ only through $X_i^\intercal \beta^*$, so one may choose $P$ to project onto the subspace orthogonal to $\beta^*$. In the dependent case, $PX_i$ and $(I_p-P)X_j$ may still be correlated, as $\Cov[X_i, X_j]$ does not necessarily commute with $P$. In DA, this is further complicated by the fact that $y_i$ depends on the unaugmented data $Z_i$ instead of $X_i$. 
Choosing a suitable projection $P$ is thus highly non-trivial, and we develop such a $P$ for DA in \Cref{appendix:DA:cgmt:results}.

\section{Conclusion and Future Work}\label{conclusion}

We have shown that, for high-dimensional logistic regression, the pipeline of analysis of Gaussian universality and CGMT extends readily to dependent 
data, and the asymptotic risks are again completely characterized by the mean and the variance of the data. This has many useful implications, from allowing us to perform Gaussian simulations in lieu of the actual data, to obtaining low-dimensional scalar equations that capture the behavior of the estimator, as we have demonstrated in simple examples of data augmentation. In fact, the majority of our analysis is not exclusive to logistic regression and can be directly extended to any classification algorithm such as SVM that relies on $(X_i)$ solely through its one-dimensional projections $(X_i^T\beta)$ (see \Cref{proof_overview}). Moreover, our dependent CGMT is not tied to the logistic model and relies only on a low-rank dependence assumption. An interesting future line of work would be to extend our analysis to high-dimensional models such as random feature models and to extend our 
dependency assumptions 
to a more general mixing condition \citep{bradley2005basic}.
As demonstrated in our plots, there is also no reason that universality should be a uniquely sub-Gaussian phenomenon, as opposed to a proof artifact. Extending this to other distributions would constitute a valuable extension of our work. 

In our simulations, we also observed \emph{non-universality of the training trajectories}. In Fig.~\ref{fig:perm} and \ref{fig:crop:flip}, most estimates $\hat \beta$ are obtained via gradient descent with a learning rate $0.1$ until either convergence or $10^6$ steps are completed. Two exceptions are the t-distribution and the uniform distribution in Fig.~\ref{fig:perm}: Numerically, we find that different learning rates are required to converge to the global minimum within $10^6$ steps. Indeed, our results establish the universality of the \emph{global minima}, but do not answer whether the \emph{training trajectories to reach these minima} are universal, since the latter question is specific to the optimization methods employed. In Fig.~\ref{fig:train:trajectory} in the appendix, we observe that with the same learning rate, the training loss curves differ for uniform and t distributions, but agree for the remaining distributions. An interesting follow-up question to investigate is whether universality holds for training trajectories under different optimization methods.

\clearpage
\acks{
MEM is supported by the National Science Foundation Graduate Research Fellowship under Grant No. DGE 2140743. KHH is supported by the Gatsby Charitable Foundation. 
}

\bibliography{refs}

\clearpage
\appendix

\noindent
The appendix is organized as follows:
\begin{itemize}[topsep=0.2em, parsep=0em, partopsep=0em, itemsep=0em, leftmargin=1em]
    \item \Cref{appendix:definition:notation} presents additional definitions and notation used throughout the appendix.
    \item \Cref{appendix:additional} state additional results. This includes the characterizing \eqref{EQs} for selected data augmentation in \Cref{appendix:DA:EQs}, a comparison of our full dependent CGMT to the classical CGMT in \Cref{appendix:statement:CGMT}, a generalized logistic model in \Cref{appendix:general:model}, {and a brief discussion on \Cref{assumption:CGMT:low:rank} in \Cref{appendix:low:rank:discussion}}.
    \item \Cref{appendix:simulation:details} includes simulation details.
    \item \Cref{app_b} proves training risk universality in \Cref{mainthm} for the more general model in \Cref{appendix:general:model} under block dependence. The generalized result is stated in \Cref{appendix:universality:general}. Note that in \Cref{app_b,important_lemmas}, we temporarily convert the $0/1$ labels to $\pm1$ as it simplifies the proofs; the equivalence between the two label schemes is performed in \Cref{appendix:label:equivalence}.
    \item \Cref{lorelai} proves training risk universality in \Cref{mainthm:extend} under \Cref{general}(i) for $m$-dependent processes.
    \item \Cref{lorelai2} proves training risk universality in \Cref{mainthm:extend} under \Cref{general}(ii) for mixing processes.
    \item \Cref{proof_second_eq} proves test risk universality in \Cref{mainthm,mainthm:extend}.
    \item \Cref{important_lemmas} collects important lemmas used for the proofs in \Cref{app_b,proof_second_eq}.
    \item \Cref{equivalence_sec} contains a result that establishes the equivalence of training losses under deletion of small blocks, which is utilised in the big-block-small-block technique applied in \Cref{lorelai,lorelai2}. 
    \item \Cref{check:augmentation} verifies \Cref{gauss_approx} for different augmentation schemes.
    \item \Cref{appendix:auxiliary:lemma} collects auxiliary lemmas  used for the proofs in \Cref{proof_second_eq,check:augmentation}.
    \item \Cref{appendix:proof:CGMT} proves the dependent CGMT.
    \item \Cref{appendix:DA:cgmt:results} present all intermediate optimizations used for applying CGMT to analyze data augmentation. We also include results that verify the CGMT conditions for different augmentations.
    \item \Cref{appendix:proof:DA} proves all results in \Cref{appendix:DA:cgmt:results}.
\end{itemize}

\section{Additional Definitions and Notations} \label{appendix:definition:notation}

Our results hold under the assumption that the random vectors $(X_i)$ are sub-Gaussian. We present here a formal definition of sub-Gaussianity.
\begin{definition}\label{ui}
    We say that a random vector $Y \in \R^p$ is sub-Gaussian with constant $\textnormal{\textsf K}$ if, for all vectors $\textbf{v} \in \R^p$, we have $$\mathbb{E}\l[\exp\l(\lambda\lip\textbf{v},Y-\mathbb{E}(Y)\rip\r)\r]\le \exp\l(\textsf C\lambda^2\|\textbf{v}\|_2^2\textnormal{\textsf K}^2\r)\quad \text{ for all } \lambda > 0,$$ for some absolute constant $\textsf C > 0$. 
\end{definition} If $Y$ is sub-Gaussian, then the norm of its covariance matrix is well controlled (see \lemmaref{cabris} for more details). Furthermore, a number of results assume that the data is locally dependent, as defined in \cite{ross2011fundamentals}.

\begin{definition}
    Let $(X_{i})_{i\le p} \in \R^p$ be a random vector. We say that it is locally dependent if for all $i\le p$ there exists a subset $\mathcal{N}_{i}\subset[p]$ such that $X_{i}$ is independent from $(X_{k})_{k\not\in \mathcal{N}_i}$. We call $\mathcal{N}_i$ the dependency neighborhood of $X_i$.
\end{definition}
A similar definition can be made for random arrays:
\begin{definition}
    Let $(X_{i,j})_{i\le p_1,j\le p_2}$ be a random array. We say that it is locally dependent if for all $i\le p_1$ and $j\le p_2$ there exists a subset $\mathcal{N}_{i,j}\subset[p_1]\times [p_2]$ such that $X_{i,j}$ is independent from $(X_{k,l})_{(k,l)\not\in\mathcal{N}_{i,j}}$. 
\end{definition}
Throughout the appendix we use the following notation:
\begin{itemize}[topsep=0.2em, parsep=0em, partopsep=0em, itemsep=0em, leftmargin=1em]
    \item For a sequence $(W_i)$ and a set $B\subset\mathbb{N}$, we let $W_B$ designate $(W_i)_{i\in B}$.
    \item Recall the definition of the blocks $\mathcal B_i$ in \eqref{eq:blockdepdef}. In our block dependent proofs, we may assume $\mathcal  B_i = \{i, i+1, \hdots, i+k-1\}$ for notational simplicity, without any loss of generality. 
    \item For a matrix $A \in \R^{n\times n}$, $a_{\mathcal B_i}$ denotes $(a_{ij})_{j\in\mathcal B_i} \in \R^k$. 
    \item For a set $\mathcal S$ and $\delta > 0$, we let $\mathcal S_\delta$ designate a minimal $\delta\sqrt p$-net of $\mathcal S$. 
    \item We use $x \vee y := \max\{x, y\}$ and $x \wedge y := \min\{x, y\}$. 
    \item The function $\bm 1^\pm(x)$ will be used to denote the sign function $\text{sgn}(x) = \mathbb I(x \geq 0) - \mathbb I(x < 0)$.
    \item Any constant in sans serif font such as $\textsf{C}_1$, $\textsf{C}_2$, and so on, depends on at most the constants $\textsf{L}$, $\textsf{K}_X$, and $\kappa$ given in our assumptions. If it further depends on $\delta$ for example, then it will be written as $\textsf{C}_\delta$ or $\textsf C(\delta)$. 
    \item We write $\hat\beta(\bX)$ as simply $\hat\beta$ when it is clear from context.
    \item For a matrix $\mathbf Y \in \R^{n\times p}$ and a given set $\mathcal S \subseteq \R^p$, the ``scaled" operator norm on $\mathcal S$ is defined via
    \begin{align}
        \|\bm Y\|_{\mathcal S} := \sup_{\beta\in\mathcal S}\|\mathbf Y\beta\|.\label{eq:goodcatch}
    \end{align}
\end{itemize}

\section{Additional Results}\label{appendix:additional}

\subsection{Characterizing equations for selected data augmentations}\label{appendix:DA:EQs}

As discussed in \Cref{sec:cgmt}, the dependent CGMT allows us to derive explicitly a set of deterministic, low-dimensional equations that capture the asymptotic behavior of a logistic regression under data augmentations. As an example, we compute this explicitly under a further simplifying assumption on the covariance structure of the augmented data. To state the assumption, let $\Sigma_o \coloneqq \Var[Z_1] \in \R^{p \times p}$, $\Sigma \coloneqq \Var[X_1] = \Var[\phi_1(Z_1)] \in \R^{p \times p}$, and $\Sigma_o^\dagger, \Sigma^\dagger$ be their respective pseudo-inverses.

\begin{assumption} \label{assumption:CGMT:var} Write $ \Sigma_* \coloneqq ( \Sigma^\dagger )^{1/2} \, \Cov[ \phi_1(Z_1)\,,\, \phi_2(Z_1) ] \,  ( \Sigma^\dagger )^{1/2}$. Assume that
\begin{align*}
    (i) \; \Sigma_* \;=&\; 
    ( \Sigma^\dagger )^{1/2}\, \Cov[ \phi_1(Z_1)\,,\, Z_1 ]  ( \Sigma_o^\dagger )^{1/2}
    &\text{ and }&&
    (ii) \; \Sigma_*^2 \;=&\; \Sigma_*\;.
\end{align*} 
\end{assumption}
Since $\phi_1$ and $\phi_2$ are i.i.d.~transformations, \Cref{assumption:CGMT:var}(i) holds for example under an invariance assumption, $Z_1 \overset{d}{=} \phi_1(Z_1)$. \Cref{assumption:CGMT:var}(ii) requires that the eigenvalues of $\Sigma_*$ consist of only zeros and ones. Note that $\Sigma_*$ is symmetric and idempotent, and therefore a projection matrix; this property is exploited throughout the CGMT formula computation in \Cref{appendix:proof:CGMT}. 

We verify \Cref{assumption:CGMT:var} for the cases of no augmentation, random permutation and random sign flipping in \Cref{appendix:DA:special:cases}. Note that \Cref{assumption:CGMT:var} is more restrictive than necessary: While it does not cover random cropping, the CGMT theorem does apply to random cropping and the only difference is that one cannot use \Cref{assumption:CGMT:var} to simplify certain algebraic calculations, resulting in a more complicated set of equations than \eqref{EQs}. We clarify this in \Cref{appendix:DA:special:cases}.

To apply the CGMT to obtain a deterministic set of equations, one needs to establish the equivalence of multiple optimization problems. We state them in \Cref{appendix:DA:cgmt:results}, which includes the original optimization \eqref{OO} (i.e.~\eqref{eq:DA:logistic} in \Cref{sec:DA}), the primary optimization \eqref{PO} (i.e.~$\Psi_{\cS_w, \cS_u}$ in \Cref{thm:CGMT_2}), the auxiliary optimization \eqref{AO} (i.e.~$\psi_{\cS_w, \cS_u}$ in \Cref{thm:CGMT_2}), a low-dimensional scalar optimization \eqref{SO} and a low-dimensional deterministic optimization \eqref{DO}, whose solutions are characterized by \eqref{EQs}. We state the main result here.

\begin{theorem}[Effect of data augmentation on the test risk] \label{thm:DA:risk} Let $\hat \beta(\bX, \bX^\Phi)$ be the estimator fitted via \eqref{OO} with $S=\cS_p$. Assume that Assumptions \ref{blockdepass} -- \ref{courage} hold, that the minimizer-maximizers of \eqref{DO} are within the interior of the domain of optimization and \Cref{assumption:CGMT:var} holds. Then
\begin{align*}
    | R_{\rm test}( \hat \beta(\bX, \bX^\Phi)) -  R^G_{\rm test}( \hat \beta(\bG, \bG^\Phi))  | \;\overset{\P}{\rightarrow}&\; 0 
    &\text{ and }&&
     | 
     R_{\rm test}( \hat \beta(\bX, \bX^\Phi)) -
     \bar R_{\rm test}(\bar \chi^{r, \theta, \sigma, \tau}_2)
     |
     \;\overset{\P}{\rightarrow}&\; 0 
     \;,
\end{align*}
where $(r, \theta, \sigma, \tau)$ solves the system of equations \eqref{EQs}, $ \bar \chi^{r, \theta, \sigma, \tau}_2$ is defined in \eqref{DO}, and 
\begin{align*}
    \bar R_{\rm test}(\bar \chi) \;\coloneqq&\; 
    \mean_{\eta \sim \cN(0,1)}\big[ 
        \ell_{\rm test} 
        \big( \sqrt{\bar \chi}\, \eta \,,\, \| \Sigma_{\rm new}^{1/2} \beta^* \| \eta \big)
    \big]
    \;.
\end{align*}
\end{theorem}

\Cref{thm:DA:risk} shows that the test risk is completely characterized by a 1d quantity $\bar \chi^{r, \theta, \sigma, \tau}_2$. This quantity is completely determined by the parameters $(\alpha, \sigma_1, \sigma_2, \tau_1, \tau_2, \nu_1, \nu_2, r_1, r_2, \theta)$, defined as solutions to the system of 10 non-linear equations
\begin{align}
    \begin{cases}
        0
        \;=\;
        \theta \bar \kappa_*^2 
        - 
        \frac{ \alpha \bar \kappa_*^2}{\sigma_2 \tau_2}
        -
        \frac{r_2 \nu_2 \bar \kappa_* }{k}
        \mean \,
        \Big[
            \bar Z_1  
            \bone_k^\intercal
            u_{\bar Z, \varepsilon_1, \eta}
        \Big] 
        +
        r_2 \nu_2 \alpha \bar \kappa_*^2 
        \;,
        \\
        0
        \;=\;
        -
        \frac{1}{2 \tau_1} 
        -
        \partial_{\sigma_1} \bar \chi_1^{r, \theta, \sigma, \tau} 
        +
        \frac{r_1 \nu_1}{ k} \, 
        \mean\Big[
            \eta^\intercal \, \big(I_k - \frac{1}{k} \bone_{k \times k}\big)
             \, u_{\bar Z, \varepsilon_1, \eta} 
        \Big] 
        +
        \frac{r_1 \nu_1 \sigma_1 (k-1)}{ k} 
        +
        \frac{r_2 \nu_2 }{ k}
        \mean\Big[
            \eta^\intercal
            \frac{1}{k} \bone_{k \times k}
            u_{\bar Z, \varepsilon_1, \eta} 
        \bigg]
        \\
        \qquad\;
        +
        \frac{r_2 \nu_2\sigma_1 }{ k}
        \;,
        \\
        0
        \;=\; 
        - 
        \frac{1}{2 \tau_2}
        +
        \frac{\alpha^2 \bar \kappa_*^2}{2 \sigma_2^2 \tau_2}
        -
        \partial_{\sigma_2} \bar \chi_1^{r, \theta, \sigma, \tau} 
        +
        \frac{r_2 \nu_2 }{ k}
        \mean\Big[
            \bar Z_2
            \bone_k^\intercal
                u_{\bar Z, \varepsilon_1, \eta} 
        \Big]
        +
        r_2 \nu_2 \sigma_2 
        \;,
        \\
        0 
        \;=\;
        \frac{\sigma_1}{2 \tau_1^2}
        -
        \partial_{\tau_1} \bar \chi_1^{r, \theta, \sigma, \tau} 
        \;,
        \\
        0 
        \;=\;
        \frac{\sigma_2}{2 \tau_2^2}
        +
        \frac{\alpha^2 \bar \kappa_*^2}{2 \sigma_2 \tau_2^2}
        -
        \partial_{\tau_2} \bar \chi_1^{r, \theta, \sigma, \tau} 
        \;,
        \\
        0
        \;=\;
        - 
        \frac{r_1}{2 \nu_1^2}
        +
        \frac{r_1}{2 k} \,  \mean\Big[ \big\| 
            \big(I_k - \frac{1}{k} \bone_{k \times k}\big)
            ( u_{\bar Z, \varepsilon_1, \eta} + \sigma_1 \eta )
            \big\|^2 \Big]
        \;,
        \\
        0
        \;=\;
        -
        \frac{r_2}{2 \nu_2^2}
        +
        \frac{1}{4 r_2 \nu_2^2}
        +
        \mfrac{r_2}{2 k}
        \mean\bigg[
        \Big\| 
            \mfrac{1}{k} \bone_{k \times k}
            \Big(
                u_{\bar Z, \varepsilon_1, \eta} 
                -
                \mfrac{1}{r_2 \nu_2} \bar Y
                \bone_k 
                -
                \alpha \bar \kappa_* \bar Z_1 
                \bone_k
                +
                \sigma_1 \eta 
                +
                \sigma_2 \bar Z_2 \bone_k 
            \Big)
        \Big\|^2
        \bigg]\\
        \;\qquad+
        \mfrac{1}{\nu_2 k}
        \mean\bigg[
            \bar Y
            \Big(
                \bone_k^\intercal
                u_{\bar Z, \varepsilon_1, \eta} 
                -
                \mfrac{k}{r_2 \nu_2} \bar Y
                -
                k \alpha \bar \kappa_* \bar Z_1 
            \Big)
        \bigg]
        \;,
        \\
        0
        \;=\;
        \frac{1}{2\nu_1} 
        -
        \partial_{r_1} \bar \chi_1^{r, \theta, \sigma, \tau} 
        +
        \frac{\nu_1}{2 k} \,  \mean\Big[ \big\| 
        \big(I_k - \frac{1}{k} \bone_{k \times k}\big)
        ( u_{\bar Z, \varepsilon_1, \eta} + \sigma_1 \eta )
        \big\|^2 \Big]
        \;,
        \\
        0
        \;=\;
        \frac{1}{2\nu_2} 
        +
        \frac{1}{4 r_2^2 \nu_2}
        -
        \partial_{r_2} \bar \chi_1^{r, \theta, \sigma, \tau} 
        +
        \mfrac{\nu_2}{2 k}
        \mean\bigg[
        \Big\| 
            \mfrac{1}{k} \bone_{k \times k}
            \Big(
                u_{\bar Z, \varepsilon_1, \eta} 
                -
                \mfrac{1}{r_2 \nu_2} \bar Y
                \bone_k 
                -
                \alpha \bar \kappa_* \bar Z_1 
                \bone_k
                +
                \sigma_1 \eta 
                +
                \sigma_2 \bar Z_2 \bone_k 
            \Big)
        \Big\|^2
        \bigg]
        \\
        \;\qquad
        +
        \mfrac{1}{r_2 k}
        \mean\bigg[
            \bar Y
            \Big(
                \bone_k^\intercal
                u_{\bar Z, \varepsilon_1, \eta} 
                -
                \mfrac{k}{r_2 \nu_2} \bar Y
                -
                k \alpha \bar \kappa_* \bar Z_1 
            \Big)
        \bigg]
        \;,
        \\
        0 
        \;=\;
        \alpha \bar \kappa_*^2 
        -
        \partial_\theta \bar \chi_1^{r, \theta, \sigma, \tau} 
        \;.
    \end{cases}
     \tag{EQs} \label{EQs}
\end{align}
Here, $\bar Z = (\bar Z_0, \bar Z_1, \bar Z_2)$ and $\eta = (\eta_1, \ldots, \eta_k)$ are both independent low-dimensional standard Gaussians, $\varepsilon_1$ is an independent {\rm Logistic}-$(0,1)$ variable, $\bar Y \coloneqq  \ind_{\geq 0}\{ \bar \kappa_o \bar Z_0 + \bar \kappa_* \bar Z_1 - \varepsilon_1  \} 
=\ind\{ \bar \kappa_o \bar Z_0 + \bar \kappa_* \bar Z_1 - \varepsilon_1 \geq 0 \} 
$, and $\bar \kappa_*$, $\bar \kappa_0$ and $\bar \chi_1^{r,\theta,\sigma,\tau}$ are limits  defined in \eqref{DO} that are related to $\beta^*$ and the covariances of the original data as well as the augmented data. $u_{\bar Z,\varepsilon_1,\eta}$ can be viewed as a generalization of the proximal operator used in \cite{salehi2019impact}, in the sense that its is defined as an minimizer of the low-dimensional, random optimization problem 
\begin{align*}
    \min_{\tilde u \in \R^k}
    &\;
    \mfrac{1}{k} \bone_k^\intercal \rho(\tilde u)
    + 
    \mfrac{r_1 \nu_1}{2 k} \, 
    \big\| 
        \big(I_k - \mfrac{1}{k} \bone_{k \times k}\big)
        ( \tilde u  + \sigma_1 \eta )
    \big\|^2
    \;
    \\
    &\;
    +
    \mfrac{r_2 \nu_2}{2 k}
    \Big\| 
        \mfrac{1}{k} \bone_{k \times k}
        \Big(
            \tilde u 
            -
            \mfrac{1}{r_2 \nu_2} \bar Y
            \bone_k 
            -
            \alpha \bar \kappa_* \bar Z_1 
            \bone_k
            +
            \sigma_1 \eta 
            +
            \sigma_2 \bar Z_2 \bone_k 
        \Big)
    \Big\|^2
    \;.
    \tagaligneq \label{eq:EQs:proximal:define}
\end{align*}
While the system of equations is rather complicated, we show in \lemmaref{lem:EQs:iso} that it recovers exactly the system of $6$ non-linear equations in the case of isotropic data with no augmentation, derived in \citet{salehi2019impact}. As part of our proof, we also observe that $\sigma_1$, $\tau_1$, $\nu_1$ and $r_1$ are the additional parameters that arise due to augmentation.

\vspace{.5em}

\begin{proof}[Proof of \theoremref{thm:DA:risk}] By \theoremref{mainthm_2}, the conclusion of \theoremref{mainthm} \eqref{eq:main_first} holds for the logistic model \eqref{logmodel_two}. This in particular includes the data augmentation model \eqref{eq:DA:logistic} by considering an $m(k+1)$ dataset $(Z_{i'}, \phi_{i'(k-1)+1}(Z_{i'}), \ldots, \phi_{i'k}(Z_{i'}) )_{i' \leq m}$, setting the weights $\omega_i$ of the loss to be $1$ for all augmented data and $0$ for the unaugmented data, and setting the weights $a_{ij}$ in the labels such that the labels of $\phi_{i'(k-1)+1}(Z_{i'}), \ldots, \phi_{i'k}(Z_{i'})$ all depend only on $Z_{i'}$. Meanwhile, notice that the proof of \theoremref{mainthm} \eqref{eq:main_second} in \Cref{proof_second_eq} does not depend on the choice of the logistic model as long as training risk universality is estbalished. Therefore the test risk universality in \theoremref{mainthm} \eqref{eq:main_second} would hold for data augmentation and the stated assumptions, if \Cref{assumption:test:risk} is verified. 

To verify \Cref{assumption:test:risk}, we set $\bar \chi =  \bar \chi^{r, \theta, \sigma, \tau}_2$ in  \Cref{assumption:test:risk}, and combine  \lemmaref{lem:GO:PO}, \lemmaref{lem:PO:SO} and \lemmaref{lem:SO:DO} to relate \eqref{GO} to \eqref{DO}. By assumption, the minimizer-maximizers of \eqref{DO} are within the interior of the domain of optimization, so the converged risk of \eqref{DO} changes by $\Theta(\epsilon^2)$ depending on whether the optimization domain of $\beta$ requires $| (\beta^\intercal \Sigma_{\rm new} \beta)^{1/2} - (\bar \chi^{r, \theta, \sigma, \tau}_2)^{1/2} | > \epsilon$. This verifies \Cref{assumption:test:risk} and proves the universality of the test risk. The deterministic approximation then follows by substituting $\bar \chi =  (\bar \chi^{r, \theta, \sigma, \tau}_2)^{1/2}$  and applying \lemmaref{lem:EQs} to obtain \eqref{EQs}.
\end{proof}

\subsection{Dependent and classical CGMT results}\label{appendix:statement:CGMT}

The next result states the full version of our dependent CGMT, for which \theoremref{thm:CGMT_2} is a direct corollary. $\Psi_{\cS_w, \cS_u}$ and  $\psi_{\cS_w, \cS_u}$ are the risks under the primary optimization \eqref{eq:CGMT:PO} and the auxiliary optimization \eqref{eq:CGMT:AO} respectively, both defined in \Cref{sec:cgmt}; $\hat w_\Psi \in \cS_w$ is the minimizer of  $\Psi_{\cS_w, \cS_u}$.

\begin{theorem}[Dependent CGMT] \label{thm:CGMT:full} Suppose $\cS_w $ and $\cS_u$ are compact and $f$ is continuous on $\cS_w \times \cS_u$. Then the following statements hold:
\begin{enumerate}
    \item[(i)] For all $c \in \R$, 
    \begin{align*}
            \P(  \Psi_{\cS_w, \cS_u} \leq c ) \,\leq\,  2^M \P(  \psi_{\cS_w, \cS_u} \leq c )\;.
    \end{align*}
    \item[(ii)] If additionally $\cS_w$ and $\cS_u$ are convex and $f$ is convex-concave on $\cS_w \times \cS_u$, then for all $c \in \R$, 
    \begin{align*}
        \P(  \Psi_{\cS_w, \cS_u} \geq c ) \,\leq\,  2^M \P(  \psi_{\cS_w, \cS_u} \geq c )\;,
    \end{align*}
    and in particular, for all $\mu \in \R$ and $t > 0$, 
    \begin{align*}
        \P(  |  \Psi_{\cS_w, \cS_u} - \mu | \geq t ) \,\leq\,  2^M  \, \P(  | \psi_{\cS_w, \cS_u} - \mu | \geq t )\;.
    \end{align*}
    \item[(iii)] Assume the conditions of (ii). Let $\cA_p$ be an arbitrary open subset of $\cS_w$ and $A_p^c \coloneqq \cS_w \setminus \cA_p$. If there exists constants $\bar \psi_{\cS_w}$, $\bar \psi_{\cA^c_p}$ and $\eta, \epsilon > 0$ such that
    \begin{align*}
        \bar \psi_{\cA^c_p} \geq \bar \psi_{\cS_w} + 3 \eta \;,
        \quad 
        \P( \psi_{\cS_w, \cS_u} \leq \bar \psi_{\cS_w} + \eta  ) \,\geq\, 1 - \epsilon\;,
        \quad 
        \P( \psi_{\cA^c_p, \cS_u} \geq \bar \psi_{\cA^c_p} - \eta  ) \,\geq\, 1 - \epsilon\;,
    \end{align*}
    then $
        \P( \hat w_\Psi  \,\in\, \cA_p ) \,\geq\, 1 - 4 \epsilon
    $.
\end{enumerate}
\end{theorem}

As a comparison, we remark that the standard CGMT in the isotropic, independent case is exactly the same as above with $M=1$, and stated for the loss 
\begin{align*}
    \Psi_{\mathcal{S}_w, \mathcal{S}_u}\;\coloneqq&\;
    \min_{w \in \mathcal{S}_w} \, \max_{u \in \mathcal{S}_u} \, w^\intercal \bH u 
    + f(w,u) \\
   \tilde \psi_{\mathcal{S}_w,\mathcal{S}_u}
    \;\coloneqq&\;
    \min_{w \in \mathcal{S}_w} \, \max_{u \in \mathcal{S}_u} \, \|w\|_2 \bh^\intercal u+w^\intercal \bg \|u\| + f(w,u)
    \;.
\end{align*}
In this case, $\bH$ is an $\R^{p \times n}$ matrix with i.i.d.~standard Gaussian entries, and $\bh$ and $\bg$ are again independent standard Gaussian vectors in $\R^n$ and $\R^p$ respectively. We refer interested readers to \cite{thrampoulidis2016recovering} for a detailed overview of CGMT and their Theorem 3.3.1 for the standard CGMT result.

\subsection{Generalizing the Model}\label{appendix:general:model}

Recall that the model stated in \Cref{modeloverview} assumes that $y_i$ is only a function of its own covariates:
\begin{align*}
    \P(y_i = 1\mid X_i) = \sigma\l(X_i^\intercal\beta^*\r).
\end{align*}
However, this formulation is quite limiting with regards to the types of dependence it can handle. Recall that a key property of data augmentation, for example, is that any transformation we apply to the covariates should not alter the associated label (meaning $y_1 = y_2 = \cdots = y_k$). This suggests that our model must be able to account for both the classical setup of logistic regression, and that of data augmentation, repeated measurements, and more. Thus, for the rest of our block-dependent results and proofs given in the appendix, we alter \Cref{logmodel} in the following way:

\begin{assumption}[Generalized Model]\label{logmodel_two}
There exists a block diagonal matrix $A = (a_{ij}) \in [0,1]^{n\times n}$ satisfying $a_{ij} = 0$ if $j\not\in\mathcal B_i$ and $\sum_{j\in\mathcal B_i}a_{ij} = 1$ for all  $1\leq i \leq n$, such that
\begin{align*}
    \P\l(y_i = 1\r) = \sigma\Big(\textstyle\sum_{j\in\mathcal B_i}a_{ij}X_j^\intercal\beta^*\Big).
\end{align*}
\end{assumption}
Recalling that our training risk is given by 
\begin{align*}
    \hat R_n(\beta, \bX) = \frac1n\sum_{i=1}^n\omega_i\l(\log(1 + e^{X_i^\intercal\beta}) - y_iX_i^\intercal\beta\r) + \frac{\lambda}{2n}\|\beta\|^2,
\end{align*}
we can specify certain values of our matrix $A$ and the weights $\omega:= (\omega_1, \hdots, \omega_n)$ to obtain relevant setups:
\begin{enumerate}
    \item[(i)] When $A = I_n$ and $\omega = (1, \ldots, 1)$, we recover the classic logistic regression framework.
    \item[(ii)] When $a_{ij} = \mathbb I(j = \min \mathcal B_i)$ and $\omega_i = \mathbb I(i \neq \min\mathcal B_i)$, we obtain the data augmentation framework as utilized in \cite{hu2022universality}. Note that this implies that the first element of each block is the original data point which defines the labels, but is not considered in the regression.
    \item[(iii)] If $(a_{ij})_{j\in\mathcal B_i} = (\tfrac1k, \hdots, \tfrac1k)$, the label is defined by an equally weighted sum of the block, which can be utilized for situations such as repeated measurements and peer effects.
\end{enumerate}

\subsection{Discussion on \Cref{assumption:CGMT:low:rank}}
\label{appendix:low:rank:discussion}

The low-rank covariance structure in  \Cref{assumption:CGMT:low:rank} is natural for setups with data augmentation, as illustrated by the next lemma.

\begin{lemma} \label{lem:low:rank:verify} Fix $m$ such that $m$ divides $n$ and write $k \coloneqq m/n$. Let $X_1, \ldots, X_m$ be i.i.d.~random vectors in $\R^p$, let $\phi_1, \ldots, \phi_n$ be i.i.d.~random $\R^d \rightarrow \R^d$ transformations, and let $\bG$ be an $\R^{p \times n}$-valued Gaussian matrix that matches the mean and covariance of the augmented data matrix 
\begin{align*}
    (\phi_1(X_1), \ldots, \phi_k(X_1), \ldots, \phi_{(m-1)k+1}(X_m), \ldots, \phi_{mk}(X_m))
    \;.
\end{align*}
Also denote $\Sigma'_1 \coloneqq \Var[\phi_1(X_1)]$ and $\Sigma'_2 \coloneqq \Cov[\phi_1(X_1), \phi_2(X_1)] $. Then $\bG$ satisfies \Cref{assumption:CGMT:low:rank} with 
\begin{align*}
    \Cov[\bG_{ji}, \bG_{j'i'}]
    \;=\;
    (\Sigma_1 - \Sigma_2) \ind\{i = i' \}
    +
    \Sigma_2 \ind\{ i \in N(i') \}
    \;\qquad\;
     \text{ for all } 
     i,i' \leq n 
     \text{ and }
     j,j' \leq p
     \;,
\end{align*}
where $N(i') \coloneqq \{ \lfloor (i'-1)/k \rfloor + 1, \ldots, \lfloor (i'-1)/k \rfloor + k  \} $ is the set of indices that correspond to differently augmented versions of the same data vector. Moreover, $\Sigma_1 - \Sigma_2$ is positive semi-definite.
\end{lemma}

\begin{proof}[Proof of \lemmaref{lem:low:rank:verify}] The proof is identical to that of \lemmaref{lem:cov:GPstar} in \Cref{appendix:proof:DA} by replacing $P_*^\perp$ with identity.  
\end{proof}

We do note that, however, in the actual application of CGMT, the Gaussian matrix considered in \Cref{assumption:CGMT:low:rank} is typically a slightly modified version of $\bH$, although most of the dependence structure is typically inherited. In our examples, this modification comes from a projection matrix analogous to those used in the i.i.d.~version of CGMT. A precise formulation is included in \Cref{appendix:DA:cgmt:results}, and the corresponding verification of \Cref{assumption:CGMT:low:rank}
 is presented in \lemmaref{lem:cov:GPstar} of \Cref{appendix:proof:DA}.

\section{Simulation Details} \label{appendix:simulation:details}

We present some additional simulation details on top of the setups described in \Cref{sec:DA} here. The regularization parameter is held at $\lambda = 0.01$, and the test loss is computed as the difference between the 0-1 loss achieved by $\hat \beta$ and that achieved by the oracle $\beta^*$. Below, we denote $\cN$, ${\rm Unif}$, $\Gamma_2$, ${\rm Exp}$ and $t_3$ respectively as a standard normal, a uniform distribution, a gamma distribution with shape $2$, an exponential distribution and a Student's t distribution with $3$ degrees of freedom, all shifted and rescaled to have zero mean and $1/p$ variance. We also write $\tilde t_3$ as $t_3$ rescaled to have unit variance.

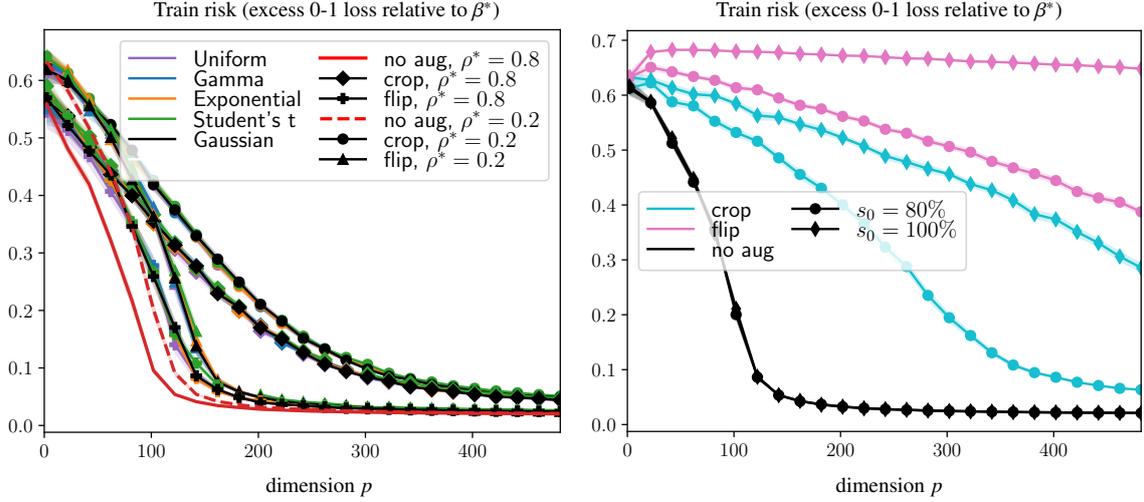
\begin{figure}[t]
    \centering 
    \begin{tikzpicture}
        \node[inner sep=0pt] at (-3.8,0){\includegraphics[trim={.5em .5em 0em .5em},clip,width=.5\linewidth]{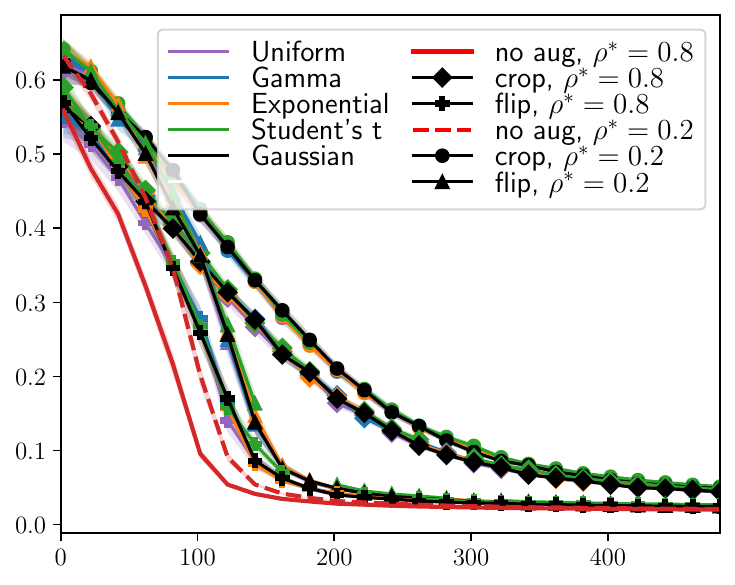}};
        \node[inner sep=0pt] at (3.8,0){\includegraphics[trim={.5em .5em .5em .5em},clip,width=.49\linewidth]{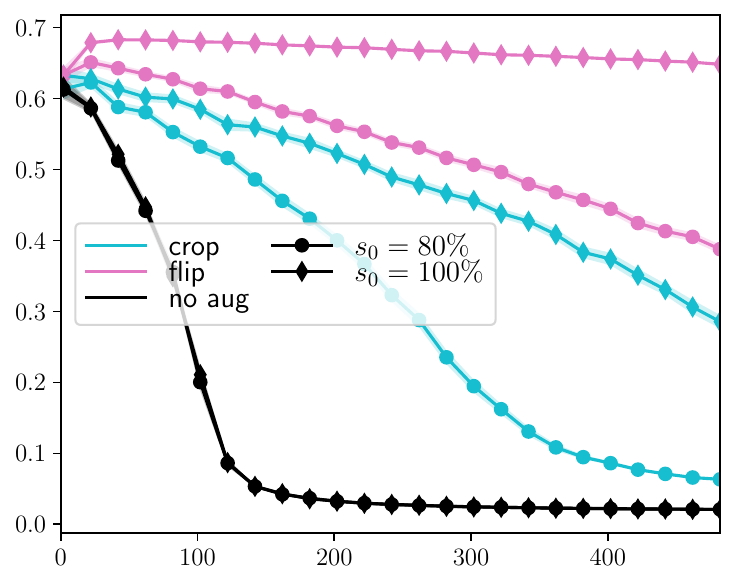}};

        \node[inner sep=0pt] at (-3.3, 3.1) {\scriptsize Train risk (excess 0-1 loss relative to $\beta^*$)};
        \node[inner sep=0pt] at (4.1, 3.1) {\scriptsize Train risk (excess 0-1 loss relative to $\beta^*$)};

        \node[inner sep=0pt] at (-3.4, -3.2) {\scriptsize dimension $p$};
        \node[inner sep=0pt] at (4.1, -3.2) {\scriptsize dimension $p$};
    \end{tikzpicture}
    \vspace{-1.0em}
    \centering
    \caption{Universality of training risks under cropping and sign flipping. The left and the right plots are the training risk analogues of the left and the right plots of Fig.~\ref{fig:crop:flip} respectively. }
    \vspace{-2em}
    \label{fig:crop:flip:train}
\end{figure}

\noindent
\textbf{Details for random permutations.} Fig.~\ref{fig:perm} concerns the performance of random permutations across different proportion $r_{\rm perm}$ of coordinates to permute and different number of augmentations $k$. Results are collected over $50$ random trials for augmented data and over $200$ random trials for unaugmented data. The dimension is fixed as $p=500$, $50$ groups are considered, and all group sizes are kept the same with $p_1 = \ldots = p_{50} = 10$.  In each trial,  $\beta^* \in \R^p$ is generated by concatenating $50$ groups of $10$ identical entries each, where the $50$ different entries are generated i.i.d.~from $\tilde t_3$. Every group of coordinates of the data are generated i.i.d.~according to $\cN$, ${\rm Unif}$, $\Gamma_2$, ${\rm Exp}$ and $t_3$, but additionally rescaled by a random group-dependent parameter drawn from $\Gamma(0.5,1)$. The choices of $r_{\rm perm}=0.8$ and $r_{\rm perm}=1.0$ correspond to random permutations performed respectively on the top $8$ and $10$ coordinates of each group. Plots with no augmentation are generated under Gaussian data.

\vspace{.5em}

\noindent
\textbf{Details for random cropping and sign flipping.} Fig.~\ref{fig:crop:flip} concerns the performance of random cropping and random sign flipping across different signal ratio $\rho^*$ and different data dimension $p$. Results are collected over $100$ random trials for augmented data and over $200$ random trials for unaugmented data. The number of augmentations is fixed as $k=30$. Plots with no augmentation are generated under Gaussian data.
\begin{itemize}[topsep=0.2em, parsep=0em, partopsep=0em, itemsep=0.2em, leftmargin=1em]
    \item For the setup without knowledge of zero coordinates (left plot of Fig.~\ref{fig:crop:flip} and left plot of Fig.~\ref{fig:crop:flip:train}), $\beta^*$ is generated such that a uniformly random subset of $\lceil (1-\rho^*)p \rceil$ coordinates are zero and the remaining entries are drawn i.i.d.~from $\tilde t_3$, and random cropping and sign flipping are performed on $r_{\rm flip} = r_{\rm crop} = 20 \%$ of the coordinates. Data are generated coordinate-wise i.i.d.~according to $\cN$, ${\rm Unif}$, $\Gamma_2$, ${\rm Exp}$ and $t_3$.
    \item For the setup where the bottom $\lceil s_0(1 - \rho^*)p \rceil$ coordinates are known to be zero (right plot of Fig.~\ref{fig:crop:flip} and right plot of Fig.~\ref{fig:crop:flip:train}),  the remaining coordinates of $\beta^*$ are generated such that a random subset of $\lceil (1-\rho^*)p \rceil - \lceil (1-\rho^*)p \rceil$ coordinates are zero and the rest are again drawn i.i.d.~from $\tilde t_3$. Cropping and sign flipping are always performed on the bottom  $\lceil s_0(1 - \rho^*)p \rceil$ coordinates, as well as also on $\lceil r\rceil - \lceil s_0(1 - \rho^*)p \rceil$ of the remaining coordinates, where $r=r_{\rm flip}=r_{\rm crop}=0.2$. Data are generated coordinate-wise i.i.d.~according to $\cN$.
\end{itemize}
 We also remark that even with knowledge of the coordinates, sign flipping does not outperform no augmentation: Unlike cropping, sign flipping does not explicitly leave out the zero coordinates.

\vspace{.5em}

\noindent
\textbf{Universality of risks}. Notice that the simulations are performed over different distributions on the coordinates of $Z_i$'s, shifted and scaled to have zero mean and the same variance. Notably, the uniform distribution obeys the sub-Gaussianity in \Cref{scalinggg}, the exponential and gamma distributions only satisfy sub-exponential tails, and the t-distribution is chosen with $3$ degrees of freedom, i.e.~with unbounded third moments. Universality behavior is observed across all distributions. Indeed in our proof, sub-Gaussianity is only assumed for convenience, and we conjecture that this is not a necessary assumption for \theoremref{thm:CGMT_2}. 

\vspace{.5em}

\noindent
\textbf{Non-universality of training trajectories as observed by the requirement of different learning rates.} In both Fig.~\ref{fig:perm} and \ref{fig:crop:flip}, gradient descent is employed to optimize the logistic regressor either until convergence or until $10^6$ steps are exhausted. Learning rate is chosen as ${\rm LR}=0.1$ across all simulations with three exceptions: ${\rm LR}=1$ for $t_3$ in Fig.~\ref{fig:perm} under $r_{\rm perm}=0.8$, ${\rm LR}=0.5$ for uniform distribution in Fig.~\ref{fig:perm} under $r_{\rm perm}=1.0$ and ${\rm LR}=0.8$ for uniform distribution in Fig.~\ref{fig:perm} under $r_{\rm perm}=0.8$. We find that for these three setups, ${\rm LR}=0.1$ does not lead to convergence within $10^5$ steps. We conjecture that this arises due to the lack of universality of the training trajectories, as illustrated in \Cref{fig:train:trajectory} and as discussed towards the end of \Cref{sec:DA}.

\begin{figure}[t]
    \centering 
    \begin{tikzpicture}

        \node[inner sep=0pt] at (0,0){\includegraphics[trim={.5em .5em 0em .5em},clip,width=.6\linewidth]{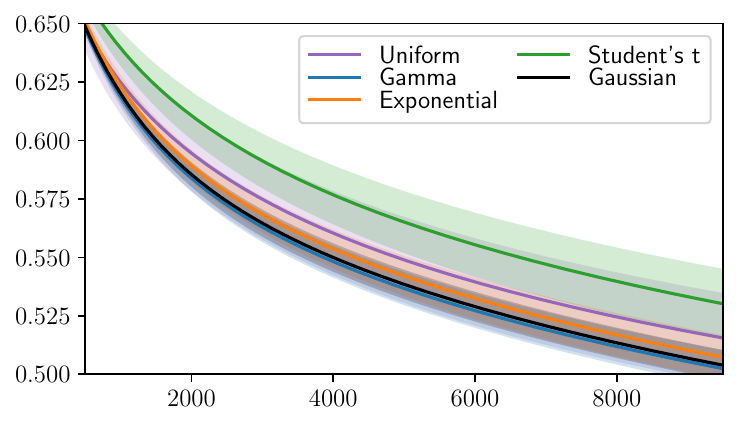}};

        \node[inner sep=0pt,rotate=90] at (-4.8, 0.1) {\scriptsize Training risk (cross-entropy)};

        \node[inner sep=0pt] at (0.3, -2.8) {\scriptsize number of training iterations};
    \end{tikzpicture}
    \vspace{-.5em}
    \centering
    \caption{Initial training loss curves for the random permutation setup in Fig.~\ref{fig:perm} with $\rho_{\rm perm} = 0.8$, $k=11$ and learning rate ${\rm LR}=0.1$.}
    \vspace{-2em}
    \label{fig:train:trajectory}
\end{figure}



\section{Proof of \theoremref{mainthm} \eqref{eq:main_first}: Training risk universality}\label{app_b}

\subsection{Restricting to $\mathcal S_p$}\label{the_discussion}

Recall that in Equation \eqref{eq:sp} we defined the set we are taking our minimum over as 
\begin{align*}
    \mathcal S_p := \l\{\beta \in \R^p : \|\beta\|_2 \leq \textsf L\hspace{-2pt}\sqrt p,\ \|\beta\|_\infty \leq \textsf Lp^{\frac{1-r}{2}}\r\}.
\end{align*}
This restriction is very commonly applied in many results on universality, such as \cite{lahiry2023universality, han2023universality, montanari2022universality}. In our case, the Euclidean norm is trivial, since by virtue of being a minimizer we know that 
\begin{align*}
    \hat R_n(\hat\beta) \leq \hat R_n(0) &\implies  \mfrac1n\msum_{i=1}^n\omega_i\l(\log(1+e^{X_i^\intercal\hat\beta}) - y_iX_i^\intercal\hat\beta\r) + \frac{\lambda}{2n}\|\hat\beta\|^2 \leq \log(2)\\
    &\implies \|\hat\beta\| \leq \sqrt{\frac{\log(4)}{\lambda}}\sqrt n.
\end{align*}
However, establishing the infinity norm bound with high probability is much more challenging. Such a bound is often proven through a leave-one-out approach, as in \cite{karoui2013, han2023universality}, in which one defines a new minimizer
\begin{align*}
    \hat\beta^{(s)} &:= \argmin_{\substack{\beta\in\mathbb R^p \\ \beta_s = 0}}\hspace{2pt}\frac1n\sum_{i=1}^n\log(1 + e^{-y_iX_i^\intercal\beta}) + \frac{\lambda}{2n}\|\beta\|_2^2
\end{align*}
for each $1 \leq s \leq p$. For a well-behaved design, it should be the case that $\|\hat\beta - \hat\beta^{(s)}\|$ is small, which leads to a bound on $\|\hat\beta\|_\infty$. This idea was even generalized in \cite{lahiry2023universality} for block dependence within observations, where the new minimizer assumes that $\beta_{\mathcal B_i} = 0$ for an entire block $\mathcal B_i$.

Since our dependence is also across observations, these approaches are not able to leverage the independence between coordinates which plays a crucial role in such bounds. A more viable approach is taken in Theorem 5 \& Lemma 13 of \cite{montanari2022universality}: they prove that under certain conditions, if a minimizer with bounded $\ell_2$ norm exists (with high probability), then there also exists a minimizer with both bounded $\ell_2$ and $\ell_\infty$ norms (with high probability). The only drawback to such an approach is that it requires a lower bound on the smallest singular value of a matrix, of the form 
\begin{align*}
    \sigma_{\min}(\bX\bX^\intercal) \geq \frac{p}{\textsf C}.
\end{align*}
This can be nearly impossible in some dependent set-ups: in the worst-case scenario, where all the rows in a neighborhood are identical, we know $\sigma_{\min}(\bX)$ is identically zero and such a bound is unattainable. However, certain DA schemes possess a nice enough structure to make this result hold and thus also the $\ell_\infty$ bound. In noise injection, for example, we can decompose the augmented matrix into a matrix of identical rows and the independent Gaussian noise. This latter matrix is asymptotically free from the first one, and thus we can obtain control on its smallest singular value. To extend this into a proof covering all DA schemes will be left to future work, and for now we will work under the mild constraint that $\hat\beta \in \mathcal S_p$. 

\subsection{Generalized Theorem} \label{appendix:universality:general}

In this section, we will prove \eqref{eq:main_first} from \theoremref{mainthm}. However, we actually prove a slightly more general result:
\begin{theorem}\label{mainthm_2} Let $\l(X_i, y_i(X_i)\r)_{i=1}^n$ and $\l(G_i, y_i(G_i)\r)_{i=1}^n$ be generated under Assumptions \ref{blockdepass}, \ref{scalinggg}-\ref{gauss_approx}, and \ref{logmodel_two}, where each $G_i \sim \mathcal N\l(0, \Var(X_i)\r)$. Then for any $\tilde{\mathcal S}\subseteq \mathcal S_p$,
\begin{align*}
    d_{\mathcal H}\l(\min_{\beta\in \tilde{\mathcal S}} \hat R_n(\beta; \bX), \min_{\beta\in \tilde{\mathcal S}} \hat R_n(\beta; \bG)\r) \to 0.
\end{align*}
\end{theorem}
We remark that if we successfully establish \theoremref{mainthm_2}, then the claim in \eqref{eq:main_first} of \theoremref{mainthm} directly follows by setting $\tilde{\mathcal S} = \mathcal S_p$.

\subsection{Converting the Loss}\label{appendix:label:equivalence} 

Before continuing, we will convert our labels from $\{0, 1\}$ to $\{-1, 1\}$, as this combines two of the terms in the training risk to significantly simplify calculations. To be specific, noting that $y_i \in \{0, 1\}$, we can define $\tilde y_i := 2 y_i - 1 \in \{-1, 1\}$, which still satisfies 
\begin{align*}
    \P\l(\tilde y_i = 1 \mid X_i\r) = \P\l(y_i = 1\mid X_i\r) = \sigma(a_{\mathcal B_i}^\intercal X_{\mathcal B_i}\beta^*).
\end{align*}
Then the loss evaluated at each data point $(X_i, y_i)$ can be re-expressed as 
\begin{align*}
    \log\big( 1 + e^{ X_i^\intercal \beta} \big) - y_i X_i^\intercal \beta
    \;=&\;
    \Big( \log\big( 1 + e^{ X_i^\intercal \beta} \big) - X_i^\intercal \beta \Big)  \, \ind\{ y_i = 1 \}
    +
    \log\big( 1 + e^{ X_i^\intercal \beta} \big)  \, \ind\{ y_i = 0 \}
    \\
    \;=&\;
    \log\Big( \mfrac{1 + e^{X_i^\intercal \beta}}{ e^{X_i^\intercal \beta}} \Big)  \, \ind\{ y_i = 1 \}
    +
    \log\big( 1 + e^{ X_i^\intercal \beta} \big)  \, \ind\{  y_i = 0 \}
    \\
    \;=&\;
    \log\big( 1 + e^{- X_i^\intercal \beta} \big)  \, \ind\{ \tilde y_i = 1 \}
    +
    \log\big( 1 + e^{ X_i^\intercal \beta} \big)  \, \ind\{\tilde y_i = - 1 \}
    \\
    \;=&\;
    \log\big( 1 + e^{- \tilde y_i X_i^\intercal \beta} \big) 
    \;.
\end{align*}
Thus, renaming our labels as $y_i \in \{-1, 1\}$, for the rest of this section and also \Cref{important_lemmas}, we use the training risk
\begin{align*}
    \hat R_n(\beta, \bX) = \frac1n\sum_{i=1}^n\omega_i\log(1 + e^{-y_iX_i^\intercal\beta}) + \frac{\lambda}{2n}\|\beta\|^2.
\end{align*}

\subsection{Definitions}

To complete the proof, we must first introduce the various terminology and techniques that are used throughout, from smoothing the labels and minimum function to the continuous Lindeberg interpolation.

\subsubsection{Smoothing the Labels}

First we will define the way in which we smooth our labels and subsequently the risk function. To do so, let us define the mollifier $\zeta_\gamma : \R \to \R$ for $\gamma \in (0, 1)$ as
\begin{align*}
    \zeta_\gamma(x) := \textsf{C}\cdot\exp\l(\frac{\gamma^2}{x^2 - \gamma^2}\r) \cdot \mathbb I(|x| < \gamma), 
\end{align*}
where $\textsf{C}$ is chosen such that $\int_{\R}\zeta_\gamma( x)\,\text{d} x = 1$. Then for a given function $f:\R\to\R$ we define 
\begin{align*}
    f_\gamma := f \ast \zeta_\gamma
\end{align*}
as the convolution of $f$ with $\zeta_\gamma$, noting that this makes $f$ smooth. We can then define a smoothed version of the labels as
\begin{align}
    \eta_i := \bm 1^\pm_\gamma\l(a_{\mathcal B_i}^\intercal X_{\mathcal B_i}\beta^* - \varepsilon_i\r).\label{eq:smooth_like_butter}
\end{align}
From here we can define the new smoothed risk as 
\begin{align*}
    \hat R_n^\gamma(\beta; \bX) := \frac{1}{n}\sumi{n}\omega_i\log\l(1 + e^{-\eta_iX_i^\intercal\beta}\r) + \frac{\lambda}{2n}\|\beta\|_2^2,
\end{align*}
where we have replaced each $y_i$ with its smoothed counterpart $\eta_i$. 

\subsubsection{Smoothing the Minimum}

Next we will define the function that will be used to approximate the minimum over our parameter space. For the set $\mathcal{\tilde S}$, define the smoothed minimum 
\begin{align*}
    f_\delta : \R^{>0}\times \R^{n\times p} \to \R,\ \ \ f_\delta(\alpha, \bX) := \frac{-1}{n\alpha}\log\l(\sum_{\beta\in \mathcal{\tilde S}_{\delta}}\exp\l[-n\alpha\hat R_n^\gamma(\beta; \bX)\r]\r),
\end{align*}
where the sum is over the minimal $\delta\sqrt p$-net $\mathcal {\tilde S}_{\delta}$.
When $\alpha$ is fixed or understood from context, we will refer to $f_\delta(\alpha, \bX)$ as simply $f_\delta(\bX)$. 

\subsubsection{Interpolation Technique}

Finally, we define the interpolation that we will use for the proof of our main theorem. For $t \in [0, \tfrac\pi2]$, let 
\begin{align*}
    \bU^t := \sin(t)\bX + \cos(t)\bG.
\end{align*}
When $t$ is fixed or understood from context, we will refer to $\bU^t$ as simply $\bU$. Now, for each $i = 1, \hdots, n$, define the weight functions
\begin{align*}
    w_\gamma(\beta) := \frac{e^{-n\alpha \hat R_n^{\gamma}(\beta, \bU)}}{\sum_{\beta'\in\mathcal {\tilde S}_{\delta}}e^{-n\alpha \hat R_n^{\gamma}(\beta', \bU)}},\quad\quad\quad w_\gamma^{i,k}(\beta) := \frac{e^{-n\alpha \hat R_n^{\gamma, i, k}(\beta, \bU)}}{\sum_{\beta'\in\mathcal {\tilde S}_{\delta}}e^{-n\alpha \hat R_n^{\gamma, i, k}(\beta', \bU)}}.
\end{align*}
Also define expectation with respect to the density induced by these weights as 
\begin{align*}
    \lip\textsf{g}(\beta)\rip := \sum_{\beta\in\mathcal {\tilde S}_{\delta}}\hspace{-5pt}w_\gamma(\beta)\hspace{1pt}\textsf{g}(\beta),\quad\quad\quad \lip\textsf{g}(\beta)\rip_{i,k} := \sum_{\beta\in\mathcal {\tilde S}_{\delta}}\hspace{-5pt}w_\gamma^{i,k}(\beta)\hspace{1pt}\textsf{g}(\beta),
    \tagaligneq \label{eq:angular:bracket}
\end{align*}
where
\begin{align*}
    \hat R_n^{\gamma, i, k}(\beta; \bm\eta, \bU) := \frac{1}{n}\sum_{j\notin\mathcal B_i}\omega_j\log\l(1 + e^{-\eta_jU_j^\intercal\beta}\r) + \frac{\lambda}{2n}\|\beta\|_2^2
\end{align*}
represents the risk taken only over the points outside the block $\mathcal B_i$ containing $X_i$. Also for each $i = 1, \hdots, n$ define the conditional expectation
\begin{align*}
    \E_{(i, k)}[\ \cdot\ ] := \E[\ \cdot\ \mid \bU^{ik}],
\end{align*}
where $\bU^{ik}$ is used to denote the interpolation matrix without $\mathcal B_i$:
\begin{align*}
    \bU^{ik} := \l(U_1, \hdots, U_{i-1},  0, \hdots,  0, U_{i+k}, \hdots, U_n\r),
\end{align*}
noting that this forces $\bU^{ik} \indep U_{\mathcal B_i}$. Lastly, we also define a ``gradient" term 
\begin{align*}
    \mathcal D_{i}(U_{\mathcal B_i}, \beta) := \left(\eta_i\omega_i\sigma_{i\beta}\right)\beta + \left(\sum_{j\in\mathcal B_i}\omega_j\sigma_{j\beta}\eta_j'a_{ji}U_j^\intercal\beta \right)\beta^* \in \R^p,
\end{align*}
where $\sigma_{i\beta} := \sigma(-\eta_i{U_i}^\intercal\beta)$ and $\eta_i' := {\bm 1^{\pm}_\gamma}'(a_{\mathcal B_i}^\intercal{U_{\mathcal B_i}}\beta^* - \varepsilon_i)$. When the data matrix is clear from context, we will write $\mathcal D_{i}(U_{\mathcal B_i}, \beta)$ as simply $\mathcal D_{i}(\beta)$. Lastly, we use the shorthand
\begin{align*}
    \ell(a, b) := \log(1+e^{-ab}),\quad\quad\ell_i(\beta) := \ell(\eta_i, U_i^\intercal\beta).
\end{align*}
With these, we are ready to begin the proof.

\subsection{Proof of the Theorem \ref{mainthm_2}}\label{proof_main_theorem}

In this subsection we finally prove \theoremref{mainthm_2} which, from our previous remarks, immediately proves \eqref{eq:main_first} of \theoremref{mainthm}.\newpar

\begin{proof}[Proof of \theoremref{mainthm_2}] For ease, let us refer to the quantity of interest as 
\begin{align*}
    d_{\mathcal H}\l(\min_\beta \hat R_n(\beta; \bX), \min_\beta \hat R_n(\beta; \bG)\r) &= (\star).
\end{align*}
Let $\alpha, \delta, \gamma, \tau > 0$. We may  first bound 
\begin{align}
    (\star) &\leq 
    d_{\mathcal H}\l(\min_\beta \hat R_n(\beta; \bX), \min_\beta \hat R_n^\gamma(\beta; \bX)\r) + d_{\mathcal H}\l(\min_\beta \hat R_n^\gamma(\beta; \bX), \min_\beta \hat R_n^\gamma(\beta; \bG)\r)\\&+ d_{\mathcal H}\l(\min_\beta \hat R_n^\gamma(\beta; \bG), \min_\beta \hat R_n(\beta; \bG)\r)\nonumber\\
    &\stackrel{(i)}{\leq} 2\textsf{C}_1\sqrt{k\gamma} + d_{\mathcal H}\l(\min_\beta \hat R_n^\gamma(\beta; \bX), \min_\beta \hat R_n^\gamma(\beta; \bG)\r),\label{eq:wefirstbound}
\end{align}
where $(i)$ follows from applying \lemmaref{smooth_risk} to the first and third summands. Then, we use \lemmaref{smoothmin} to bound
\begin{align}
    d_{\mathcal H}\l(\min_\beta \hat R_n^\gamma(\beta; \bX), \min_\beta \hat R_n^\gamma(\beta; \bG)\r) \leq d_{\mathcal H}\l(f_\delta(\alpha, \bX), f_\delta(\alpha, \bG)\r) + \textsf{C}_2\l(\sqrt k \delta
   + \frac1\alpha\log\l(\frac{1}{\delta}\r)\r).\label{eq:wesecondbound}
\end{align}
From here, we will prove universality for $f_\delta(\bX)$, and then show why this is sufficient. Recall from above the interpolator
\begin{align*}
    \bU^t := \sin(t)\bX + \cos(t)\bG,\ \ \ t \in [0,\tfrac\pi2].
\end{align*}
By the fundamental theorem of calculus, since $\bU^0 =\bG$ and $\bU^{\pi/2} = \bX$, we may bound
\begin{align*}
    \abs{\E\l[h\l(f_\delta(\bX)\r) - h\l(f_\delta(\bG)\r)\r]} &\leq \int_0^{\pi/2}\abs{\E\l[\partial _t h\l(f_\delta(\bU^t)\r)\r]}\,\text{d}t.
\end{align*}
Using the chain rule we may expand
\begin{align*}
    \partial_th\l(f_\delta(\bU)\r) = \frac{-h'\l(f_\delta(\bU)\r)}{n}\sumi{n}\lip \tilde U_i^\intercal\mathcal D_{i}(\beta)\rip
\end{align*}
where we set $$\tilde{\bU}^t := \partial_t \bU^t = \cos(t)\bX - \sin(t)\bG.$$ Using \lemmaref{cabristop} we obtain that 
\begin{align}
    \lim_{n\to\infty}\abs{\E\l[h\l(f_\delta(\bX)\r) - h\l(f_\delta(\bG)\r)\r]} &\stackrel{(i)}{\leq} \int_0^{\pi/2}\limsup_{n\to\infty}\abs{\E\l[\partial _t h\l(f_\delta(\bU^t)\r)\r]}\,\text{d}t \;\stackrel{(ii)}{\leq} \frac{\pi}{2}\tau,\label{eq:wethirdbound}
\end{align}
where $(i)$ is from the Dominated Convergence Theorem with the dominating function given by the bound on this derivative in \Cref{second_mom}, and $(ii)$ is from \lemmaref{cabristop}. Combining \eqref{eq:wethirdbound} with \eqref{eq:wefirstbound} and \eqref{eq:wesecondbound}, we conclude that 
\begin{align*}
    \lim_{n\to\infty}(\star) \leq \textsf C\l[\sqrt{k\gamma} + \sqrt{k}\delta + \frac{1}{\alpha}\log\l(\frac{1}{\delta}\r) + \tau\r].
\end{align*}
Noting that the left-hand side is independent of our four parameters $(\alpha, \delta, \gamma, \tau)$, we can take our limits in the proper order to conclude that 
\begin{align*}
    \lim_{n\to\infty}(\star) \leq \lim_{\delta\to0}\lim_{\substack{\tau, \gamma\to0 \\ \alpha\to\infty}}\textsf C\l[\sqrt{k\gamma} + \sqrt{k}\delta + \frac{1}{\alpha}\log\l(\frac{1}{\delta}\r) + \tau\r] = 0.
\end{align*}
\end{proof}

\section{Proof of Theorem \ref{mainthm:extend} \eqref{eq:main_first_mdep}  under Assumption \ref{general}$(i)$: Training risk universality under $m$-dependence}\label{lorelai}
The next result restates Theorem \ref{mainthm:extend} \eqref{eq:main_first_mdep} under Assumption \ref{general}$(i)$, i.e. the universality universality of the training risk in the $m$-dependent setting:

\begin{theorem}[Training risk universality under $m$-dependence]\label{second_thm_m} Let $\l(X_i, y_i(X_i)\r)_{i=1}^n$ and $\l(G_i, y_i(G_i)\r)_{i=1}^n$ be generated under Assumptions \ref{logmodel}-\ref{signalsizinghehe}, \ref{general}(i), and \ref{gauss_approx2}, where each $G_i \sim \mathcal N\l(\bm0, \Var(X_i)\r)$. Then 
\begin{align}
    d_{\mathcal H}\l(\min_\beta \hat R_n(\beta; \bX), \min_\beta \hat R_n(\beta; \bG)\r) \to 0.\label{mdepfirst}
\end{align}
\end{theorem}

\begin{proof}[Proof of \Cref{second_thm_m} ]
Let $M \in \Z^+$ be fixed. Define new matrices $\bX^M, \bG^M \in \R^{n' \times p}$ as
\begin{align*}
    \bX^M &:= (X_1, \hdots, X_M, X_{M+m+1}, \hdots, X_{2M+m}, X_{2M+2m+1}, \hdots)^\intercal\\
    \bG^M &:= (G_1, \hdots, G_M, G_{M+m+1}, \hdots, G_{2M+m}, G_{2M+2m+1}, \hdots)^\intercal,
\end{align*}
noting that 
\begin{align}
n' \in [n-(r+1)m+1,\ n - rm] \subset \l[n\frac{M}{M+m} - m,\ n\frac{M}{M+m} + m\r] = [nq -m, nq + m], \label{eq:nprime}
\end{align}
where $r := \lfloor\frac{n}{M+m} \rfloor$ and $q := \frac{M}{M+m}$. By construction, $\bX^M$ and $\bG^M$ are block dependent with block size $M$. We may also define
\begin{align*}
    \bX^{m} &:=(X_{M+1},\dots, X_{M+m},X_{2M+m+1},\dots,X_{2M+2m},\dots)\\
    \bG^{m} &:=(G_{M+1},\dots, G_{M+m},G_{2M+m+1},\dots,G_{2M+2m},\dots)
\end{align*}
so that every vector $X_i$ is either in $\bX^M$ or $\bX^m$. For simplicity we can also write these indexing sets as 
\begin{align*}
    B_M &:= \{1, \hdots, M, M+m+1, \hdots, 2M + m, \hdots, \}\\
    B_m &:= [n] \setminus B_M.
\end{align*}
If we once again refer to our quantity of interest as 
\begin{align*}
    d_{\mathcal H}\l(\min_\beta \hat R_n(\beta; \bX), \min_\beta \hat R_n(\beta; \bG)\r) = (\star),
\end{align*}
then we may bound
\begin{align*}
    (\star) &\leq \underbrace{d_{\mathcal H}\l(\min_\beta \hat R_n(\beta; \bX), \min_\beta \hat R_{n'}(\beta; \bX^M)\r)}_{(a)} + \underbrace{d_{\mathcal H}\l(\min_\beta \hat R_{n'}(\beta; \bX^M), \min_\beta \hat R_{n'}(\beta; \bG^M)\r)}_{(b)} \\&+ \underbrace{d_{\mathcal H}\l(\min_\beta \hat R_{n'}(\beta; \bG^M), \min_\beta \hat R_n(\beta; \bG)\r)}_{(c)}.
\end{align*}
By \theoremref{mainthm}, we know that $(b) \to 0$ as $n\to\infty$. Hence we have
\begin{align*}
    \limsup_{n\to\infty}\ (\star) &\leq \limsup_{n\to\infty}\ (a)  +\limsup_{n\to\infty}\ (c).
\end{align*}
We may now apply \theoremref{equivalence} with $\tilde m=m$ to say that for some $\textsf{C}_d > 0$, 
\begin{align}
    \limsup_{n\to\infty}\ (\star) \leq 2\textsf{C}_d\frac{m}{M}\sqrt{M+m} = \mathcal O\l(M^{-1/2}\r). 
\end{align}
As this holds for every $M \in \Z^+$, we take $M\to\infty$ to obtain the result.
\end{proof}


\section{Proof of Theorem \ref{mainthm:extend} \eqref{eq:main_first_mdep}  under Assumption \ref{general}$(ii)$: Training risk universality under mixing}\label{lorelai2}
We first prove the universality of the train risk if the data is made of blocks of size $k$ that are almost independent. We will then combine this result with \cref{equivalence} to get the desired result. \begin{assumption}[$\beta_{\rm{mix}}$ almost independent blocks of size $k$]\label{blockdepass2}  
For all $i,j$ such that $j\ge i+k$ we have that if we define $\mathcal{A}:=\sigma\big((X_\ell, y_\ell)_{\ell \le i}\big)$ and $\mathcal{B}:=\sigma\big((X_\ell, y_\ell)_{\ell > j}\big)$ then $\beta\Big(\mathcal{A},\mathcal{B}\Big)\le \beta_{\rm{mix}}$, where $\beta\Big(\mathcal{A},\mathcal{B}\Big)$ is the $\beta-$mixing coefficient between the sigma-algebras $\mathcal{A}$ and $\mathcal{B}$ (see \cite{bradley2005basic} for a definition).
\end{assumption} 

Under this assumption we can establish the universality of the train risk.

\begin{lemma}\label{neuchatel2}
Assume that $(\bX_i,y_i)$ and $(\bG_i,y_i(\bG_i))$ satisfy Assumption \ref{logmodel}-\ref{gauss_approx}.
Assume that $\rm{Var}((\bG_i))=\rm{Var}((\bX_i)).$ Moreover assume that  \Cref{blockdepass2} also holds. Then we have that there is a constant $\textsf C$ such that 
\begin{align*}
    \limsup_{n\rightarrow\infty} d_{\mathcal H}\l(\min_\beta \hat R_n(\beta; \bX), \min_\beta \hat R_n(\beta; \bG)\r) &\le C\beta_{\rm mix}^{1/15}
\end{align*}
\end{lemma}

\begin{proof} For ease, let us refer to the quantity of interest as 
\begin{align*}
    d_{\mathcal H}\l(\min_\beta \hat R_n(\beta; \bX), \min_\beta \hat R_n(\beta; \bG)\r) &= (\star).
\end{align*}
Let $\alpha, \delta, \gamma, \tau > 0$. We may  first bound 
\begin{align}
    (\star) &\leq 
    d_{\mathcal H}\l(\min_\beta \hat R_n(\beta; \bX), \min_\beta \hat R_n^\gamma(\beta; \bX)\r) + d_{\mathcal H}\l(\min_\beta \hat R_n^\gamma(\beta; \bX), \min_\beta \hat R_n^\gamma(\beta; \bG)\r)\nonumber\\&+ d_{\mathcal H}\l(\min_\beta \hat R_n^\gamma(\beta; \bG), \min_\beta \hat R_n(\beta; \bG)\r)\nonumber\\
    &\stackrel{(i)}{\leq} 2\textsf{C}_1\sqrt{\gamma}\frac{\sqrt{\max(\mathbb{E}\|\bX\|_{\mathcal{S}_p}^2,\mathbb{E}\|\bG\|_{\mathcal{S}_p}^2)}}{\sqrt{n}}+ d_{\mathcal H}\l(\min_\beta \hat R_n^\gamma(\beta; \bX), \min_\beta \hat R_n^\gamma(\beta; \bG)\r),\label{eq:wefirstbound2}
\end{align}
where $(i)$ follows from applying \lemmaref{smooth_risk} to the first and third summands. Then, we use \lemmaref{smoothmin} to bound 
\begin{align}&
    d_{\mathcal H}\l(\min_\beta \hat R_n^\gamma(\beta; \bX), \min_\beta \hat R_n^\gamma(\beta; \bG)\r)\\& \leq d_{\mathcal H}\l(f_\delta(\alpha, \bX), f_\delta(\alpha, \bG)\r) + \textsf{C}_2\l(\frac{\sqrt{\max(\mathbb{E}\|\bX\|_{\mathcal{B}(0,\sqrt{p})}^2,\mathbb{E}\|\bG\|_{\mathcal{B}(0,\sqrt{p})}^2)}}{\sqrt{n}} \delta
   + \frac1\alpha\log\l(\frac{1}{\delta}\r)\r).
\end{align}  {Using  the bound in \corollaryref{dep_norm_bound} under \Cref{general}(ii) we know that there exists a constant ${\textsf C}_3>0$ such that $$\frac{\sqrt{\max(\mathbb{E}\|\bX\|_{\mathcal{S}_p}^2,\mathbb{E}\|\bG\|_{\mathcal{S}_p}^2)}}{\sqrt{n}},\frac{\sqrt{\max(\mathbb{E}\|\bX\|_{\mathcal{B}(0,\sqrt{p})}^2,\mathbb{E}\|\bG\|_{\mathcal{B}(0,\sqrt{p})}^2)}}{\sqrt{n}}\le{\textsf C}_3\mathcal{S}. $$Hence we have 
\begin{align}
   (\star)\leq d_{\mathcal H}\l(f_\delta(\alpha, \bX), f_\delta(\alpha, \bG)\r) + \textsf{C}_2\l({\textsf C}_3\mathcal{S}\delta
   + \frac1\alpha\log\l(\frac{1}{\delta}\r)\r)+2\textsf C_1\textsf C_3\sqrt{\gamma}\mathcal{S}.\label{eq:wesecondbound2}
\end{align}}
From here, we will prove universality for $f_\delta(\bX)$, and then show why this is sufficient. Recall from above the interpolator
\begin{align*}
    \bU^t := \sin(t)\bX + \cos(t)\bG,\ \ \ t \in [0,\tfrac\pi2].
\end{align*}
By the fundamental theorem of calculus, since $\bU^0 =\bG$ and $\bU^{\pi/2} = \bX$, we may bound
\begin{align*}
    \abs{\E\l[h\l(f_\delta(\bX)\r) - h\l(f_\delta(\bG)\r)\r]} &\leq \int_0^{\pi/2}\abs{\E\l[\partial _t h\l(f_\delta(\bU^t)\r)\r]}\,\text{d}t.
\end{align*}
Using the chain rule we may expand
\begin{align*}
    \partial_th\l(f_\delta(\bU)\r) = \frac{-h'\l(f_\delta(\bU)\r)}{n}\sumi{n}\lip \tilde U_i^\intercal\mathcal D_{i}(\beta)\rip
\end{align*} 
\textcolor{black}{To further bound this, we will replace specific blocks $(X_{\mathcal B_i})$ with independent blocks, allowing us to use the results established in \cref{lemm:bliock}. More precisely, we will show that for an average $i\le n$ the expectation $\mathbb{E}\Big[h'\l(f_\delta(\bU)\r)\lip \tilde U_i^\intercal\mathcal D_{i}(\beta)\rip\Big]$ is approximately the same if the block containing the i-th observation $\Big(X_j,y_j(X_j)\Big)_{j\in \mathcal{B}_i}$ is independent from the others. We also use the notation $j > \cB_i$ to denote the set of indices after the last index of the block $\cB_i$, and analogously use $j < \cB_i$ for those before $\cB_i$. } 

{ In this goal,  for $i\le n$ we write $(W_j^i)$ the process such that $\Big(W_j^i,y_j(W_j^i)\Big)_{j\in \mathcal{B}_i}$, $\Big(W_j^i,y_j(W_j^i)\Big)_{j>\mathcal{B}_i}$ and $\Big(W_j^i,y_j(W_j^i)\Big)_{j<\mathcal{B}_i}$ are independent and have the same marginals as $\Big(X_j,y_j(X_j)\Big)_{j\in \mathcal{B}_i}$, $\Big(X_j,y_j(X_j)\Big)_{j>\mathcal{B}_i}$ and $\Big(X_j,y_j(X_j)\Big)_{j<\mathcal{B}_i}$. We define similarly $\Big(G^i_j,y_j(G^i_j)\Big)_{j\in \mathcal{B}_i}$, $\Big(G^i_j,y_j(G^i_j)\Big)_{j>\mathcal{B}_i}$ and $\Big(G^i_j,y_j(G^i_j)\Big)_{j<\mathcal{B}_i}$. The interpolated process between $W^i$ and $\bG^i$ is written as $${U}^{(i)}_j:=\rm{sin}(t)W_j^i+\rm{cos}(t)G^i_j,\qquad \tilde{U}^{(i)}_j:=\rm{cos}(t)W_j^i-\rm{sin}(t)G^i_j.$$ Denote $\mathcal D^{(i)}_{j}(\beta)$  the version of $\mathcal{D}_i(\beta)$ for $\bU^{(i)}$ and $\tilde \bU^{(i)}$. Similarly we write $\lip\cdot\rip_{(i)}$ the version of $\lip\cdot\rip$ for $U^{(i)}$ (see \eqref{eq:angular:bracket} for the definition of $\langle \cdot \rangle$).}


Choose $L>0$ to be a real. Define $$\mathcal{D}_{i,L}(\beta):=\mathcal{D}_{i}(\beta)~\mathbb{I}(|\tilde U_i^T\beta|,|U_i^T\beta|\le L)$$ and $$\mathcal{D}^{(i)}_{i,L}(\beta):=\mathcal{D}^{(i)}_{i}(\beta)\mathbb{I}(|\tilde (U^{(i)}_i)^T\beta|,|(U^{(i)}_i)^T\beta|\le L).$$ We first remark that we can switch $\mathcal D_{i}(\beta)$ for the truncated $\mathcal D_{i,L}(\beta)$  without changing too much the value of the expectation $\mathbb{E}\Big({-h'\l(f_\delta(\bU)\r)}\lip \tilde U_i^\intercal\mathcal D_{i}(\beta)\rip\Big)$. Indeed we have  \begin{align*}&
   \frac{1}{n} \sum_{i=1}^n\Big|\mathbb{E}\Big({-h'\l(f_\delta(\bU)\r)}\lip \tilde U_i^\intercal\mathcal D_{i,L}(\beta)\rip\Big)-\mathbb{E}\Big({-h'\l(f_\delta(\bU)\r)}\lip \tilde U_i^\intercal\mathcal D_{i}(\beta)\rip\Big)\Big|
   \\&\le   \frac{1}{n} \sum_{i=1}^n\mathbb{E}\Big(\lip |\tilde U_i^\intercal\mathcal D_{i}(\beta)|\mathbb{I}(|\tilde U_i^T\beta|>L)\rip\Big)
   +\frac{1}{n} \sum_{i=1}^n\mathbb{E}\Big(\lip |\tilde U_i^\intercal\mathcal D_{i}(\beta)|\mathbb{I}(|U_i^T\beta|>L)\rip\Big)\
   \\&\le  \frac{1}{n} \sum_{i=1}^n\mathbb{E}\Big(\lip \Big\{|\tilde U_i^\intercal \beta|+| U_i^\intercal \beta||\tilde U_i^\intercal \beta^*|\Big\}~\mathbb{I}(|\tilde U_i^T\beta|>L)\rip\Big)
   + \frac{1}{n} \sum_{i=1}^n\mathbb{E}\Big(\lip \Big\{|\tilde U_i^\intercal \beta|+| U_i^\intercal \beta||\tilde U_i^\intercal \beta^*|\Big\}~\mathbb{I}(|U_i^T\beta|>L)\rip\Big)
   \\&\overset{(i)}{\le}  \frac{1}{nL^{1/4}} \sum_{i=1}^n\mathbb{E}\Big(\lip \Big\{|\tilde U_i^\intercal \beta|^{5/4}+| U_i^\intercal \beta||\tilde U_i^\intercal \beta|^{1/4}|\tilde U_i^\intercal \beta^*|\Big\}\rip\Big)
   \\&\qquad+  \frac{1}{nL^{1/4}} \sum_{i=1}^n\mathbb{E}\Big(\lip \Big\{|\tilde U_i^\intercal \beta|{|U_i^T\beta|^{1/4}}+| U_i^\intercal \beta|^{5/4}|\tilde U_i^\intercal \beta^*|\Big\}\rip\Big)
    \\&\overset{(ii)}{\le}\frac{1}{nL^{1/4}} \sum_{i=1}^n\mathbb{E}\Big(\lip |\tilde U_i^\intercal \beta|^{5/4}\rip\Big)+L^{-1/4}\max_{i\le n}\| \tilde U_i^\intercal \beta\|_6\sqrt{\frac{1}{n}\sumi{n}\mathbb{E}(\lip (U_i^T\beta)^{2}\rip)}\Big\{{\frac{1}{n}\sumi{n}\mathbb{E}(\lip (U_i^T\beta)^{3/4}\rip)}  \Big\}^{3/4} 
    \\&\qquad +\frac{1}{L^{1/4}} \sqrt{\frac{1}{n}\sumi{n}\mathbb{E}(\lip (\tilde U_i^T\beta)^{2}\rip)}\sqrt{\frac{1}{n}\sumi{n}\mathbb{E}(\lip |\tilde U_i^T\beta|^{3/4}\rip^{2/3})}+\max_i\|\tilde U_i^T\beta^*\|_{3}\sqrt{\frac{1}{n}\sumi{n}\mathbb{E}(\lip ( U_i^T\beta)^{2}\rip)}
    \\&\le \frac{1}{L^{1/4}} \Big\{\mathbb{E}\Big(n^{-1}\|\tilde U\|_{\mathcal{S}_p}^2\Big)^{5/8}+\max_{i\le n}\| \tilde U_i^\intercal \beta\|_6\sqrt{\mathbb{E}\Big(n^{-1}\| \bU\|_{\mathcal{S}_p}^2\Big)}\mathbb{E}\Big(n^{-1}\|\tilde U\|_{\mathcal{S}_p}^2\Big)^{9/32}
    \\&\qquad +\sqrt{\mathbb{E}\Big(n^{-1}\|\tilde U\|_{\mathcal{S}_p}^2\Big)}\Big({\mathbb{E}\Big(n^{-1}\|U\|_{\mathcal{S}_p}^2\Big)}\Big)^{1/8}+\max_i\|\tilde U_i^T\beta^*\|_{3}\sqrt{\mathbb{E}\Big(n^{-1}\| U\|_{\mathcal{S}_p}^2\Big)}\Big\}.
\end{align*} where (i) is a consequence of the fact that $\mathbb{I}(|\tilde U_i^T\beta|\ge L)\le \frac{|\tilde U_i^T\beta|^{1/4}}{L^{1/4}}$, (ii) comes from Jensen inequality combined with H\"{o}lder inequality. Hence using the bound in \corollaryref{dep_norm_bound} under \Cref{general}(ii) and using \Cref{scalinggg} we have that there exists a constant $\textsf C_4$ such that\begin{align}&\label{rory2}
   \frac{1}{n} \sum_{i=1}^n\Big|\mathbb{E}\Big({-h'\l(f_\delta(\bU)\r)}\lip \tilde U_i^\intercal\mathcal D_{i,L}(\beta)\rip\Big)-\mathbb{E}\Big({-h'\l(f_\delta(\bU)\r)}\lip \tilde U_i^\intercal\mathcal D_{i}(\beta)\rip\Big)\Big|\le \textsf C_4 L^{-1/4}.\end{align} 
Moreover according to \lemmaref{tur}  we have that there is a constant $\textsf C_5$ such that 
\begin{align}&\label{rory}
    \frac{1}{n}\sum_{i=1}^n\Big|\mathbb{E}\Big({-h'\l(f_\delta(\bU)\r)}\lip \tilde U_i^\intercal\mathcal D_{i,L}(\beta)\rip\Big)-\mathbb{E}\Big({-h'\l(f_\delta(\bU^{(i)})\r)}\lip ({\tilde \bU}^{(i)})^\intercal\mathcal D^{(i)}_{i,L}(\beta)\rip_{(i)}\Big)\Big|
    \\&\overset{(i)}{\le} 4({3\beta_{\rm{mix}}})^{1/3}\frac{1}{n}\sum_{i=1}^{n}\Big\{L(1+\|\tilde U_i^T\beta^*\|_{3/2}\Big\}\nonumber
    \\&\overset{(ii)}{\le}L\textsf C_5({\beta_{\rm{mix}}})^{1/3}\nonumber
\end{align} 
where (i) is due to the fact that $|\tilde U_i\mathcal{D}_{i,L}(\beta)|\le |\tilde U_i\beta|~\mathbb{I}(|\tilde U_i\beta|\le L)+ \mathbb{I}(| U_i\beta|\le L)|~\tilde U_i\beta^*||U_i^T\beta|$  combined with the triangle inequality (ii)  is due to 
\cref{scalinggg} which imply that $\limsup_n\sup_i\|U_i^T\beta^*\|_{3/2}<\infty$. Hence combining \cref{cabristop} with \cref{rory,rory2} we obtain that \begin{align}\label{fin}\frac{1}{n}\sum_{i=1}^n\Big|\mathbb{E}\Big(\frac{-h'\l(f_\delta(\bU)\r)}{n}\sumi{n}\lip \tilde U_i^\intercal\mathcal D_{i}(\beta)\rip\Big)\Big|\le L\textsf C_5\beta_{\rm mix}^{1/3}+\tau+\textsf C_4 L^{-1/4}.\end{align}

This directly implies that 
\begin{align}
    \lim_{n\to\infty}\abs{\E\l[h\l(f_\delta(\bX)\r) - h\l(f_\delta(\bG)\r)\r]} &\stackrel{(i)}{\leq} \int_0^{\pi/2}\limsup_{n\to\infty}\abs{\E\l[\partial _t h\l(f_\delta(\bU^t)\r)\r]}\,\text{d}t \\&\;\stackrel{(ii)}{\leq} \frac{\pi}{2}\Big[L\textsf C_5\beta_{\rm mix}^{1/3}+\tau+\textsf C_4 L^{-1/4}\Big]\label{eq:wethirdbound2}
\end{align}
where $(i)$ is from the Dominated Convergence Theorem with the dominating function given by the bound on this derivative in \Cref{second_mom}, and $(ii)$ is from \cref{fin}. Combining \eqref{eq:wethirdbound2} with \eqref{eq:wefirstbound2} and \eqref{eq:wesecondbound2}, we conclude that 
\begin{align*}
    \lim_{n\to\infty}(\star) \leq \textsf C\l[\sqrt{k\gamma} + \sqrt{k}\delta + \frac{1}{\alpha}\log\l(\frac{1}{\delta}\r) + \tau +L\beta_{\rm mix}^{1/3}+L^{-1/4}\r].
\end{align*}
Noting that the left-hand side is independent of our four parameters $(\alpha, \delta, \gamma, \tau)$, we can take our limits in the proper order to conclude that 
\begin{align*}
    \lim_{n\to\infty}(\star) \leq \lim_{\delta\to0}\lim_{\substack{\tau, \gamma\to0 \\ \alpha\to\infty}}\textsf C\l[\sqrt{k\gamma} + \sqrt{k}\delta + \frac{1}{\alpha}\log\l(\frac{1}{\delta}\r) + \tau\r] = C\l[L\beta_{\rm mix}^{1/3}+L^{-1/4}\r].
\end{align*}
Optimizing over $L\ge 0$ gives us the desired result.

\end{proof} 

The next result restates Theorem \ref{mainthm:extend} \eqref{eq:main_first_mdep} under Assumption \ref{general}$(ii)$, i.e.~training risk universality under mixing:

\begin{theorem}[Universality of the training risk under Assumption \ref{general}(ii)]\label{second_thm} Let $\l(X_i, y_i(X_i)\r)_{i=1}^n$ and \\ $\l(G_i, y_i(G_i)\r)_{i=1}^n$ be generated under Assumptions \ref{logmodel}-\ref{signalsizinghehe}, \ref{general}(ii), and \ref{gauss_approx2}, where each $G_i \sim \mathcal N\l(\bm0, \Var(X_i)\r)$. Then 
\begin{align}
    d_{\mathcal H}\l(\min_\beta \hat R_n(\beta; \bX), \min_\beta \hat R_n(\beta; \bG)\r) \to 0.\label{mixingversion}
\end{align}
\end{theorem}

\begin{proof}[Proof of \Cref{second_thm}]
Let $M ,\tilde m\in \Z^+$ be fixed. Define new matrices $\bX^M, \bG^M \in \R^{n' \times p}$ as
\begin{align*}
    \bX^M &:= (X_1, \hdots, X_M, X_{M+\tilde m+1}, \hdots, X_{2M+\tilde m}, X_{2M+2\tilde m+1}, \hdots)^\intercal\\
    \bG^M &:= (G_1, \hdots, G_M, G_{M+\tilde m+1}, \hdots, G_{2M+\tilde m}, G_{2M+2\tilde m+1}, \hdots)^\intercal,
\end{align*}
noting that 
\begin{align}
n' \in [n-(r+1)\tilde m+1,\ n - r\tilde m] \subset \l[n\frac{M}{M+\tilde m} -\tilde m,\ n\frac{M}{M+\tilde m} + \tilde m\r] = [nq -\tilde m, nq + \tilde m], \label{eq:nprime2}
\end{align}
where $r := \lfloor\frac{n}{M+\tilde m} \rfloor$ and $q := \frac{M}{M+\tilde m}$. By construction, $\bX^M$ and $\bG^M$ satisfy Assumption \ref{blockdepass2} with $\beta_{\rm mix}=\beta(\tilde m)$. We may also define
\begin{align*}
    \bX^{\tilde m} &:=(X_{M+1},\dots, X_{M+\tilde m},X_{2M+\tilde m+1},\dots,X_{2M+2\tilde m},\dots)\\
    \bG^{\tilde m} &:=(G_{M+1},\dots, G_{M+\tilde m},G_{2M+\tilde m+1},\dots,G_{2M+2\tilde m},\dots)
\end{align*}
so that every vector $X_i$ is either in $\bX^M$ or $\bX^{\tilde m}$. For simplicity we can also write these indexing sets as 
\begin{align*}
    B_M &:= \{1, \hdots, M, M+{\tilde m}+1, \hdots, 2M + {\tilde m}, \hdots, \}\\
    B_{\tilde m} &:= [n] \setminus B_M.
\end{align*}
If we once again refer to our quantity of interest as 
\begin{align*}
    d_{\mathcal H}\l(\min_\beta \hat R_n(\beta; \bX), \min_\beta \hat R_n(\beta; \bG)\r) = (\star),
\end{align*}
then we may bound
\begin{align*}
    (\star) &\leq \underbrace{d_{\mathcal H}\l(\min_\beta \hat R_n(\beta; \bX), \min_\beta \hat R_{n'}(\beta; \bX^M)\r)}_{(a)} + \underbrace{d_{\mathcal H}\l(\min_\beta \hat R_{n'}(\beta; \bX^M), \min_\beta \hat R_{n'}(\beta; \bG^M)\r)}_{(b)} \\&+ \underbrace{d_{\mathcal H}\l(\min_\beta \hat R_{n'}(\beta; \bG^M), \min_\beta \hat R_n(\beta; \bG)\r)}_{(c)}.
\end{align*}
By \theoremref{neuchatel2}, we know that there is $\textsf C>0$ such that $\limsup~(b)\le \textsf C \beta(\tilde m)^{1/15} $ as $n\to\infty$. Hence we have
\begin{align*}
    \limsup_{n\to\infty}\ (\star) &\leq \limsup_{n\to\infty}\ (a)  +\limsup_{n\to\infty}\ (c)+\textsf C \beta(\tilde m)^{1/15}.
\end{align*}
We may now apply \theoremref{equivalence} with $\tilde m$ to say that for some $\textsf{C}_d > 0$, 
\begin{align}
    \limsup_{n\to\infty}\ (\star) \leq 2\textsf{C}_d\frac{m}{M}\sqrt{M+m} +\textsf C \beta(\tilde m)^{1/15}= \mathcal O\l(M^{-1/2}+\textsf C \beta(\tilde m)^{1/15}\r). 
\end{align}
As this holds for every $M \in \Z^++$, we take $M\to\infty$ to obtain that \begin{align}
    \limsup_{n\to\infty}\ (\star) \leq \textsf C \beta(\tilde m)^{1/15}. 
\end{align} Finally noting that this holds for an arbitrary $\tilde m$ gives us the desired result.
\end{proof}

\section{Proof of \Cref{mainthm} \eqref{eq:main_second} and \Cref{mainthm:extend} \eqref{eq:main_second_mdep}: Test risk universality}\label{proof_second_eq}

In this section, we prove the second equation of both \theoremref{mainthm,mainthm:extend} concerning test risk universality, as both share the same proof once training risk universality is proved. We focus on presenting the proof for the 0-1 loss, i.e.~the test risk $R_{\rm test}$ is defined with
\begin{align*}
    \ell_{\rm test}( X_{\rm new}^\intercal \hat \beta \,,\, X_{\rm new}^\intercal \beta^*  )
    \;\coloneqq\; 
    \ind\Big\{  \;  \ind\Big\{ \sigma(X_{\text{new}}^\intercal \hat \beta) \geq \mfrac{1}{2} \Big\} = \ind \big\{ X_{\text{new}}^\intercal \beta^* - \varepsilon_{\text{new}} \geq 0 \big\}  \; \Big\} 
\end{align*}
for both $\hat \beta = \hat \beta(\bX)$ and $\hat \beta = \hat \beta(\bG)$. As our proof strategy relies on approximating $\ell_{\rm test}$ by the $1$-Lipschitz functions in $\tilde \cF$, the same proof also works if  $\ell_{\rm test}$ is already Lipschitz. Therefore the result also applies to any locally Lipschitz $\ell_{\rm test}$, which is Lipschitz over the compact set $\cS_p$.

\vspace{8pt}

\begin{proof}[Proof of \theoremref{mainthm} \eqref{eq:main_second} and \Cref{mainthm:extend} \eqref{eq:main_second_mdep}] 
Recall that $\sigma(x) = \frac{1}{1+e^{-x}}$. Our test loss can then be re-expressed as 
\begin{align*}
    R_{\text{test}}( \hat \beta) 
    \;=&\; 
    \mean\Big[ 
        \,
        \ind\Big\{  \ind\big\{ \sigma(X_{\text{new}}^\intercal   \hat \beta) \geq \mfrac{1}{2} \big\} = \ind\{ X_{\text{new}}^\intercal \beta^* - \varepsilon_{\text{new}} \geq 0 \} \Big\} 
    \,\Big|\,   \hat \beta \Big]
    \\
    \;=&\; 
    \mean\Big[ 
        \,
        \ind\Big\{  \ind\big\{ X_{\text{new}}^\intercal   \hat  \beta \geq 0\big\} = \ind\{ X_{\text{new}}^\intercal \beta^* - \varepsilon_{\text{new}} \geq 0 \} \Big\} 
    \,\Big|\,  \hat  \beta \Big]
    \\
    \;=&\;
    \mean\Big[ 
        \,
        \ind\Big\{ 
            X_{\text{new}}^\intercal  \hat  \beta \geq 0
            \,,\,
            X_{\text{new}}^\intercal \beta^* - \varepsilon_{\text{new}} \geq 0 
        \Big\}
        \,
        +
        \,
        \ind\Big\{ 
            X_{\text{new}}^\intercal  \hat \beta < 0
            \,,\,
            X_{\text{new}}^\intercal \beta^* - \varepsilon_{\text{new}} < 0 
        \Big\}
        \,
        \Big|\,  \hat  \beta \Big]
    \;,
    \tagaligneq \label{eq:0:1:risk:initial}
\end{align*}
and by a similar argument, 
\begin{align*}
    R_{{\text{test}}}^G(  \hat  \beta) 
    \;=\;
    \mean\Big[ 
        \,
        \ind\Big\{ 
            G_{\text{new}}^\intercal  \hat \beta \geq 0
            \,,\,
            G_{\text{new}}^\intercal \beta^* - \varepsilon_{\text{new}} \geq 0 
        \Big\}
        \,
    +
        \,
        \ind\Big\{ 
            G_{\text{new}}^\intercal \hat \beta < 0
            \,,\,
            G_{\text{new}}^\intercal \beta^* - \varepsilon_{\text{new}} < 0 
        \Big\}
        \,
        \Big|\,  \hat  \beta \Big]
    \;.
\end{align*}
For convenience, we denote the random $\R^2$ vectors 
\begin{align*}
    V_X \;\coloneqq&\;
    ( X_{\text{new}}^\intercal \hat \beta 
    \,,\,
    X_{\text{new}}^\intercal \beta^* - \varepsilon_{\text{new}} )^\intercal
    &\text{ and }&&
    V_G \;\coloneqq&\;
    ( G_{\text{new}}^\intercal \hat \beta 
    \,,\,
    G_{\text{new}}^\intercal \beta^* - \varepsilon_{\text{new}} )^\intercal
    \;.
\end{align*}
We first perform a standard smoothing of the indicator function. By Lemma 34 of \citet{huang2023high}, for any $\tau \in \R$ and $\delta >0$, there exists a continuously differentiable function $h_{\tau;\delta}$ such that $h_{\tau+ \delta;\delta}(x) \leq \ind\{ x \geq \tau \} \leq h_{\tau;\delta}(x)$ for all $x \in \R$ and that $\partial h_{\tau;\delta}$ is bounded in norm by $\delta^{-1}$. Moreover $h_{\tau;\delta}$ takes value in $[0,1]$. We use this to construct the $\R^2 \rightarrow \R$ function $\tilde h_{\tau;\delta}(x,y) \coloneqq h_{\tau;\delta}(x) \, h_{\tau;\delta}(y)$, which satisfies 
\begin{align*}
    \tilde h_{\delta;\delta}(x,y) 
    \;\leq\; 
    \ind\{ x \geq 0, y \geq 0\} 
    \;=\;  
    \ind\{ x \geq 0 \} \, \ind\{ y \geq 0\} 
    \;\leq\; 
    \tilde h_{0;\delta}(x,y) \;.
\end{align*}
This implies that for every $\delta > 0$, almost surely
\begin{align*}
    &\;
    \mean\Big[ 
        \,
        \ind\Big\{ 
            X_{\text{new}}^\intercal \hat \beta \geq 0
            \,,\,
            X_{\text{new}}^\intercal \beta^* - \varepsilon_{\text{new}} \geq 0 
        \Big\}
        \,
    - 
        \,
        \ind\Big\{ 
            G_{\text{new}}^\intercal \hat \beta \geq 0
            \,,\,
            G_{\text{new}}^\intercal \beta^* - \varepsilon_{\text{new}} \geq 0 
        \Big\}
        \,
    \Big|\, \hat \beta
    \Big] 
    \\
    &\;\leq\; 
    \mean\big[\tilde h_{0;\delta}(V_X)  -  \tilde h_{\delta;\delta}(V_G) \,\big|\, \hat \beta \big]
    \\
    &\;\leq\;
    \mean\big[\tilde h_{0;\delta}(V_X) - \tilde h_{0;\delta}(V_G) + \tilde h_{0;\delta}(V_G) - \tilde h_{\delta;\delta}(V_G) \,\big|\, \hat \beta \big]
    \\
    &\;\leq\;
    \mean\big[ \tilde h_{0;\delta}(V_X) -  \tilde h_{0;\delta}(V_G) 
    +
    \ind\{ 
        (V_G)_1 \geq - \delta
        \,,\,
        (V_G)_2 \geq - \delta
    \}  
    -  
    \ind\{ 
        (V_G)_1 \geq \delta 
        \,,\,
        (V_G)_2 \geq \delta
    \} \,\big|\, \hat \beta \big]
    \\
    &\;\leq\;
    \mean\big[\tilde h_{0;\delta}(V_X)-  \tilde h_{0;\delta}(V_G) \,\big|\, \hat \beta \big]
    + 
    \P\big(  (V_G)_1\in [-\delta,\delta)  \,\big|\, \hat \beta \big)
    +
    \P\big(  (V_G)_2 \in [-\delta,\delta)  \,\big|\, \hat \beta \big)
    \;.
\end{align*}
In the last inequality, we have noted that if $(V_G)_1, (V_G)_2 \geq - \delta$ is true and yet $(V_G)_1, (V_G)_2 \geq \delta$ is false, we must have either $(V_G)_1\in [-\delta,\delta)$ or $ (V_G)_2 \in [-\delta,\delta)$. Now let $\delta \in (0,1]$. Notice that $\delta \, \tilde h_{0;\delta} \in \tilde \cF$, where $\tilde \cF$ is defined in \Cref{courage}. Also note that $\hat \beta$ is independent of $X_{\rm new}$ and $G_{\rm new}$ in $V_X$ and $V_G$. This implies according to \Cref{courage} that 
\begin{align*}&
    \,\big|\,  \mean\big[\tilde h_{0;\delta}(V_X) \,\big|\, \hat \beta \big] - \mean\big[ \tilde h_{0;\delta}(V_G)  \,\big|\, \hat \beta \big]  \,\big|\,
    \;\\
    \leq&\;
    \mfrac{1}{\delta}
    \,
    \msup_{f\in \mathcal{\tilde F}} 
    \Big| \mean\Big[ 
        f (
            X_{\text{new}}^\intercal \hat \beta 
        , X_{\text{new}}^\intercal \beta^* - \varepsilon_{\text{new}})
        \,\Big|\, \hat \beta
     \Big] 
        -  
        \mean\Big[ 
            f (G_{\text{new}}^\intercal \hat \beta 
,G_{\text{new}}^\intercal \beta^* - \varepsilon_{\text{new}}) 
    \,\Big|\, \hat \beta
     \Big] 
    \Big|
    \\
    \;\leq&\;
    \mfrac{1}{\delta}
    \,
    \msup_{f\in \mathcal{\tilde F}} 
    \, 
    \msup_{\beta \in \cS_p}
    \Big| \mean\Big[ 
        f (
            X_{\text{new}}^\intercal \hat\beta 
        , X_{\text{new}}^\intercal \beta^* - \varepsilon_{\text{new}})
     \Big] 
        -  
        \mean\Big[ 
            f (G_{\text{new}}^\intercal \hat\beta 
,G_{\text{new}}^\intercal \beta^* - \varepsilon_{\text{new}}) 
     \Big] 
    \Big|
    \\
    \;\overset{(a)}\le  &  \mfrac{1}{\delta}
    \,
    \msup_{f\in \mathcal{\tilde F}} 
    \, 
    \msup_{\beta \in \cS_p}
    \Big| \mean\Big[ 
        f (
            X_{\text{new}}^\intercal \hat\beta 
        , X_{\text{new}}^\intercal \beta^* 
        )
     \Big] 
        -  
        \mean\Big[ 
            f (G_{\text{new}}^\intercal \hat\beta 
,G_{\text{new}}^\intercal \beta^* )  \Big] 
    \Big|
    \;\eqqcolon\; 
    \mfrac{1}{\delta}\Delta_n 
\end{align*}
In $(a)$, we have used a conditioning on $\varepsilon_{\text{new}}$, moved the suprema and the norm inside the expectation over  $\varepsilon_{\text{new}}$ and observed that the function $f( \argdot, \argdot - \varepsilon_{\text{new}} ) \in \mathcal{\tilde F}$ almost surely.
~Substituting this into the above yields that, almost surely  
\begin{align*}
    &\;
    \mean\Big[ 
        \,
        \ind\Big\{ 
            X_{\text{new}}^\intercal \hat \beta \geq 0
            \,,\,
            X_{\text{new}}^\intercal \beta^* - \varepsilon_{\text{new}} \geq 0 
        \Big\}
        \,
    - 
        \,
        \ind\Big\{ 
            G_{\text{new}}^\intercal \hat \beta \geq 0
            \,,\,
            G_{\text{new}}^\intercal \beta^* - \varepsilon_{\text{new}} \geq 0 
        \Big\}
        \,
    \Big|\, \hat \beta
    \Big] 
    \\
    &\;\leq\;
    \mfrac{1}{\delta}\Delta_n 
    + 
    \P\big(  (V_G)_1\in [-\delta,\delta)  \,\big|\, \hat \beta \big)
    +
    \P\big(  (V_G)_2 \in [-\delta,\delta)  \,\big|\, \hat \beta \big)
    \;.
\end{align*}
By a similar argument, we can obtain that almost surely
\begin{align*}
    &\;
    \mean\Big[ 
        \,
        \ind\Big\{ 
            G_{\text{new}}^\intercal \hat \beta \geq 0
            \,,\,
            G_{\text{new}}^\intercal \beta^* - \varepsilon_{\text{new}} \geq 0 
        \Big\}
        \,
    -
        \,
        \ind\Big\{ 
            X_{\text{new}}^\intercal \hat \beta \geq 0
            \,,\,
            X_{\text{new}}^\intercal \beta^* - \varepsilon_{\text{new}} \geq 0 
        \Big\}
        \,
    \Big|\, 
        \hat \beta
    \Big] 
    \\
    &\;\leq\; 
    \mean\big[\tilde h_{0;\delta}(V_G)  - \tilde h_{\delta;\delta}(V_X) \,\big|\, \hat \beta \big]
    \\
    &\;\leq\;
    \mean\big[\tilde h_{\delta;\delta}(V_G) 
    -
    \tilde h_{\delta;\delta}(V_X) 
    +
    \tilde h_{0;\delta}(V_G) 
    -
    \tilde h_{\delta;\delta}(V_G) \,\big|\, \hat \beta \big]
    \\
    &\;\leq\;
    \mean\big[ \tilde h_{\delta;\delta}(V_X)
    -
    \tilde h_{\delta;\delta}(V_G) \,\big|\, \hat \beta \big]
    + 
    \P\big(  (V_G)_1\in [-\delta,\delta)  \,\big|\, \hat \beta \big)
    +
    \P\big(  (V_G)_2 \in [-\delta,\delta)  \,\big|\, \hat \beta \big)
    \\
    &\;\leq\; 
  \frac{1}{\delta}\Delta_n 
    + 
    \P\big(  (V_G)_1\in [-\delta,\delta)  \,\big|\, \hat \beta \big)
    +
    \P\big(  (V_G)_2 \in [-\delta,\delta)  \,\big|\, \hat \beta \big)
    \;.
\end{align*}
Combining the two bounds implies that, almost surely,
\begin{align*}
    (\star)
    \;\coloneqq\;
    &\;\Big| \, \mean\Big[ 
        \,
        \ind\Big\{ 
            X_{\text{new}}^\intercal \hat \beta \geq 0
            \,,\,
            X_{\text{new}}^\intercal \beta^* - \varepsilon_{\text{new}} \geq 0 
        \Big\}
        \,
        -
        \,
        \ind\Big\{ 
            G_{\text{new}}^\intercal \hat \beta \geq 0
            \,,\,
            G_{\text{new}}^\intercal \beta^* - \varepsilon_{\text{new}} \geq 0 
        \Big\}
        \,
    \Big|\, 
        \hat \beta
    \Big] 
    \,\Big|
    \\
    &\;\leq\;
    \mfrac{2}{\delta}\Delta_n 
    + 
    \P\big(  (V_G)_1\in [-\delta,\delta)  \,\big|\, \hat \beta \big)
    +
    \P\big(  (V_G)_2 \in [-\delta,\delta)  \,\big|\, \hat \beta \big)
    \;.
\end{align*}
To control the probability terms, notice that conditioning on $\hat \beta$, $(V_G)_1 = G_{\text{new}}^\intercal \hat \beta \,|\, \hat \beta \sim \cN(0, \hat \beta^\intercal \Sigma_{\rm new} \hat \beta )$ and $(V_G)_2 \,|\, \varepsilon_{\text{new}}, \hat\beta\, \sim \cN( -\varepsilon_{\text{new}},  {\beta^{*}}^\intercal \Sigma_{\rm new} \beta^{*} )$. Therefore by a standard anti-concentration result for Gaussians (see e.g.~\citet{carbery2001distributional}), there is an absolute constant $C'' >0$ such that, almost surely,
\begin{align*}
    \P \big(  (V_G)_1\in [-\delta,\delta) \,\big|\, \hat \beta \big)
    +
    \P \big(  (V_G)_2 \in [-\delta,\delta) \,\big|\, \hat \beta \,,\, \varepsilon_{\rm new} \big)
    \;\leq\;
    C'' \delta \, \bigg( 
        \mfrac{1}{\hat \beta^\intercal \Sigma_{\rm new }\hat \beta } 
        +
        \mfrac{1}{{\beta^*}^\intercal \Sigma_{\rm new} \beta^* } 
    \bigg)
    \;.
\end{align*}
Meanwhile by \Cref{assumption:test:risk}, for every $\epsilon > 0$,
\begin{align*}
    \P( D_\epsilon(\bG) > 0 ) \;\rightarrow\; 1\;,
    \qquad 
    \text{ where }\quad
    D_\epsilon(\bG) \;\coloneqq\;
    \min_{\beta \in \cS_p \,,\, | (\beta^\intercal \Sigma_{\text{new}} \beta)^{1/2} \,-\, \bar \chi | > \epsilon}\hspace{-7pt} \hat R_n(\beta; \bG)
    -
    \min_{\beta \in \cS_p} \hat R_n(\beta; \bG)
    \;,
\end{align*}
and by the universality of the training risk (\Cref{mainthm} \eqref{eq:main_first} or \Cref{mainthm:extend} \eqref{eq:main_first_mdep}), we also have $\P( D_\epsilon(\bX) > 0) \rightarrow 1$. This implies that  
\begin{align*}
    \P\big( \,  \big|( \hat \beta(\bX)^\intercal \Sigma_{\text{new}} \hat \beta(\bX))^{1/2}  - \bar \chi \big| \leq \epsilon \, \big)
    \;=&\;
    \P( D_\epsilon(\bX) > 0 )
    \;\rightarrow\; 
    1\;,
    \\
    \P\big( \,  \big|( \hat \beta(\bG)^\intercal \Sigma_{\text{new}} \hat \beta(\bG))^{1/2}  - \bar \chi \big| \leq \epsilon \, \big)
    \;=&\;
    \P( D_\epsilon(\bG) > 0 )
    \;\rightarrow\; 
    1\;.
    \tagaligneq \label{eq:0:1:conv:chi}
\end{align*}
In other words, for both $\hat \beta = \hat \beta(\bX)$ and $\hat \beta = \hat \beta(\bG)$, $\hat \beta^\intercal \Sigma_{\rm new }\hat \beta \xrightarrow{\P} \bar \chi^2$ in probability. Moreover, $\bar \chi > 0$ and ${\beta^*}^\intercal \Sigma_{\rm new} \beta^* \xrightarrow{\P} \chi_*^2 > 0$ by \Cref{assumption:test:risk}. This allows us to consider a rare event 
\begin{align*}
    E_\chi 
    \;\coloneqq&\; 
    \Big\{ \hat \beta^\intercal \Sigma_{\rm new }\hat \beta < \mfrac{\bar \chi^2}{2} \,,\,  {\beta^*}^\intercal \Sigma_{\rm new} \beta^* < \mfrac{\chi_*^2}{2} \Big\}
    \;\qquad \;
    \text{ such that }
    \quad 
    \P(E_\chi ) \;\rightarrow\; 0\;.
\end{align*}
Denoting $E_\chi^c$ as the complement of $E_\chi$, we obtain that for any $\epsilon' > 0$,
\begin{align*}
    \P( |(\star)| > \epsilon' )
    \;\leq&\;
    \P\bigg( 
        \,
        \mfrac{2}{\delta} \Delta_n + C'' \delta \, \Big( 
            \mfrac{1}{\hat \beta^\intercal \Sigma_{\rm new }\hat \beta } 
            +
            \mfrac{1}{{\beta^*}^\intercal \Sigma_{\rm new} \beta^* } 
        \Big)
        \,>\,
        \epsilon'
        \,
    \bigg)
    \\
    \;\leq&\;
    \P\bigg( 
        \,
        \mfrac{2}{\delta} \Delta_n + C'' \delta \, \Big( 
            \mfrac{1}{\hat \beta^\intercal \Sigma_{\rm new }\hat \beta } 
            +
            \mfrac{1}{{\beta^*}^\intercal \Sigma_{\rm new} \beta^* } 
        \Big)
        \,>\,
        \epsilon'
        \,
        \Big| \, E_\chi^c
    \bigg)
    +
    \P( E_\chi )
    \\
    \;\leq&\;
    \ind\Big\{ 
        \,
        \mfrac{2}{\delta} \Delta_n + C'' \delta \, \Big( 
            \mfrac{2}{\bar \chi^2} 
            +
            \mfrac{2}{\chi_*^2} 
        \Big)
        \,>\,
        \epsilon'
    \Big\}
    +
    \P( E_\chi )
    \;.
\end{align*}
By \Cref{courage}, $\Delta_n \rightarrow 0$. Since the above is valid for any $\delta$, whose choice is independent of $\epsilon'$, we can choose $\delta = \sqrt{\Delta_n}$, which implies that the above converge to zero. In other words, we have shown that 
\begin{align*}
    \Big| \, \mean\Big[ 
        \,
        \ind\Big\{ 
            X_{\text{new}}^\intercal \hat \beta \geq 0
            \,,\,
            X_{\text{new}}^\intercal \beta^* - \varepsilon_{\text{new}} \geq 0 
        \Big\}
        \,
        -
        \,
        \ind\Big\{ 
            G_{\text{new}}^\intercal \hat \beta \geq 0
            \,,\,
            G_{\text{new}}^\intercal \beta^* - \varepsilon_{\text{new}} \geq 0 
        \Big\}
        \,
    \Big|\, 
        \hat \beta
    \Big] 
    \,\Big|
    \;\xrightarrow{\P}\; 
    0
\end{align*}
for both $\hat \beta=\hat \beta(\bX)$ and $\hat \beta = \hat \beta(\bG)$. By an exactly analogous argument, we have 
\begin{align*}
    \Big| \, \mean\Big[ 
        \,
        \ind\Big\{ 
            X_{\text{new}}^\intercal \hat \beta < 0
            \,,\,
            X_{\text{new}}^\intercal \beta^* - \varepsilon_{\text{new}} < 0 
        \Big\}
        \,
    \Big] 
    - 
    \mean\Big[ 
        \,
        \ind\Big\{ 
            G_{\text{new}}^\intercal \hat \beta < 0
            \,,\,
            G_{\text{new}}^\intercal \beta^* - \varepsilon_{\text{new}} < 0 
        \Big\}
        \,\Big|\, \hat \beta
    \Big] 
    \,\Big|
    \;\rightarrow\; 0
    \;.
\end{align*}
In view of \eqref{eq:0:1:risk:initial}, we can use a triangle inequality to obtain that
\begin{align*}
    | R_{\text{test}}( \hat \beta(\bX) ) - R_{{\text{test}}}^G(\beta(\bX)) | 
    \;\xrightarrow{\P}&\;
     0
    &\text{ and }&&
    | R_{\text{test}}( \hat \beta(\bG) ) - R_{{\text{test}}}^G(\beta(\bG)) | 
    \;\xrightarrow{\P}&\;
     0
     \;.
\end{align*}
Meanwhile, note that $ R_{{\text{test}}}^G(\hat \beta)$ depends on $\hat \beta$ only through the mean-zero conditionally Gaussian variable $G_{\rm new}^\intercal \hat \beta$, which is completely characterized by $\Var[G_{\rm new}^\intercal \hat \beta \,|\, \hat \beta ] = \hat \beta^\intercal \Sigma_{\rm new} \hat \beta$. In view of \eqref{eq:0:1:conv:chi}, both $\hat \beta(\bX)^\intercal \Sigma_{\rm new} \hat \beta(\bX)$ and $\hat \beta(\bG)^\intercal \Sigma_{\rm new} \hat \beta(\bG)$ converge in probaility to the same constant $\bar \chi^2$. This implies
\begin{align*}
    | R_{\text{test}}^G( \hat \beta(\bX) ) - R_{{\text{test}}}^G(\beta(\bG)) | 
    \;\xrightarrow{\P}&\;
     0
     \;,
\end{align*}
which in particular implies the desired statement that  $| R_{\text{test}}( \hat \beta(\bX) ) - R_{{\text{test}}}^G(\beta(\bG)) | 
\,\xrightarrow{\P}\,
 0$
\end{proof}

\section{Important Lemmas}\label{important_lemmas}

In this section, we present the statements and proofs of the various lemmas used to prove our main theorems.

\subsection{Auxiliary Lemmas}

The lemma and its corollary aim to bound the expectation of the maximum possible norm of our signal $\bX\beta$, conditional on the fact that a Bernstein-like Inequality holds. 

\begin{lemma}[Operator Norm Bound]\label{ess_pee_two} Let $(\mathbf{Y}_i)$ be a sequence of $\R^p$-valued random vectors, and let $\textnormal{\textsf R} > 0$. Suppose that $\frac{n}{p}\to\kappa$ and that there exist constants $\textnormal{\textsf K}, \textnormal{\textsf C}_1, \textnormal{\textsf c}_2, \textnormal{\textsf C}_3$ such that 
\begin{enumerate}
    \item $\sup_{i\le n}\|\textnormal{Var}(\mathbf Y_i)\|_{\textnormal{op}}\le \tfrac1p\textnormal{\textsf K}\ $.
    \item For all $\beta \in \mathcal S := \mathcal B_p(0,\textnormal{\textsf R}\sqrt{p})$ and $t > 0$, we have 
    \begin{align}
    \P\l(\frac1n\Big|\sum_{i= 1}^n(\mathbf{Y}_i^\intercal\beta)^2-\mathbb{E}((\mathbf{Y}_i^\intercal\beta)^2)\Big|\ge t\r)\le \textnormal{\textsf C}_1\exp\l(-\textnormal{\textsf c}_2n\l(\frac{t}{\textnormal{\textsf C}_3\textnormal{\textsf R}^2} \wedge \frac{t^2}{\textnormal{\textsf C}_3^2\textnormal{\textsf R}^4}\r)\r).\label{eq:bernstein_again}
\end{align}
\end{enumerate}
Then there exists $\textnormal{\textsf{C}}_{\textnormal{\textsf{R}}} > 0$ depending on $\textnormal{\textsf R}$ and the constants above such that for $n$ sufficiently large, $$\E\l[\|\mathbf Y\|_{\mathcal S}^2\r] \leq \textnormal{\textsf{C}}_{\textnormal{\textsf{R}}}p,$$
where $\|\mathbf Y\|_{\mathcal S}$ is as defined in \eqref{eq:goodcatch}.
\end{lemma}

\begin{proof}
Let $\beta \in \mathcal S$, meaning by definition $\|\beta\|_2 \leq \textsf R\sqrt p$. Note that by the first statement of the lemma, we have that
\begin{align*}
  \E\l[\|\mathbf Y\beta\|^2\r] = n\beta^\intercal\overline{\Sigma}_n\beta \leq n \textsf R^2\textsf K,
\end{align*}
where $\overline{\Sigma}_n:=\frac{1}{n}\sum_{i=1}^n\text{Var}(\mathbf Y_i)$. Now, we can note that for $s > 0$, 
\begin{align*}
    \P\l(\|\mathbf Y\beta\|^2 \geq s\r) = \P\l(\frac1n\sum_{i=1}^nW_i \geq \frac{s - n\beta^\intercal\overline{\Sigma}_n\beta}{n}\r),
\end{align*}
where $W_i:= (\mathbf Y_i^\intercal\beta)^2 -\beta^\intercal\text{Var}(\mathbf Y_i)\beta$.~
If we let $\textsf{M}_\textsf{R} := \textsf R^2(\textsf K + \textsf C_3)$, 
then by \eqref{eq:bernstein_again}, we obtain that for $s > 0$,
\begin{align}
    \P\l(\|\mathbf Y\beta\|^2 \geq n\textsf{M}_\textsf{R}(s+1)\r) \leq \textsf C_1\cdot\exp\l[-{\textsf c_2n}\l(s\wedge {s^2}\r)\r]\label{eq:johnlennon},
\end{align}
which follows from noting that 
\begin{align*}
    \frac{n\textsf{M}_\textsf{R}(s+1) - n\beta^\intercal \overline{\Sigma}_n\beta}{n\textsf C_3\textsf R^2} \geq \frac{\textsf{M}_\textsf{R}(s+1) - \textsf{M}_\textsf{R}}{\textsf{M}_\textsf{R}} = s.
\end{align*}
Now let $\varepsilon >0$, and define $\mathcal S_{\varepsilon}$ to be a minimal $\varepsilon\sqrt p$-net of $\mathcal S$. Also, for $t > 0$, define the quantity $$\eta_t := \frac{\textsf{C}_3\sqrt{p} + t}{\sqrt n}\quad\text{for}\quad \textsf{C}_3 := \sqrt{\frac{1}{\textsf c_2}\log\l(1 + \frac{3\textsf R}{\varepsilon}\r)}.$$ Then if we set $s=\eta_t \vee \eta^2_t$ in \eqref{eq:johnlennon}, by a union bound we obtain 
\begin{align*}
    \P\l(\sup_{\beta\in\mathcal S_{\varepsilon}}\|\mathbf Y\beta\|^2 \geq n\textsf{M}_\textsf{R}\l((\eta_t\vee\eta_t^2) + 1\r)\r) &\stackrel{(i)}{\leq} \textsf C_1\abs{\mathcal S_{\varepsilon}}\exp\l[-\textsf c_2n\cdot\eta_t^2\r]\\
    &\stackrel{(ii)}{\leq} \textsf C_1\l(\frac{3\textsf R}{\varepsilon}\r)^p\exp\l[-\textsf c_2(\textsf{C}_3^2p +t^2)\r]\\
    &\stackrel{(iii)}{\leq} \textsf C_1e^{-\textsf c_2t^2},
\end{align*}
where $(i)$ is via the fact that for any $x \geq0$, we have 
\begin{align*}
    (x \vee x^2) \wedge (x \vee x^2)^2 = x^2,
\end{align*}
$(ii)$ is via Corollary 4.2.13 of \cite{vershyninhighdimprob} bounding the cardinality of a minimal $\varepsilon$-net, and $(iii)$ is from the definition of $\eta_t$ and $\textsf C_3$. Now we may bound the error between the supremum on the whole space and the supremum on the $\varepsilon$-net by applying a similar technique to Lemma 4.4.1 of \cite{vershyninhighdimprob}, which gives
\begin{align*}
    \sup_{\beta\in\mathcal S}\|\mathbf Y\beta\|^2  \leq \frac{1}{1-2\varepsilon}\sup_{\beta\in\mathcal S_{\varepsilon}}\|\mathbf Y\beta\|^2 .
\end{align*}
We conclude that, for $\textsf A := \sqrt{\textsf{M}_\textsf{R}/(1-2\varepsilon)}$, 
\begin{align}
    \P\l(\|\mathbf Y\|_{\mathcal S}^2 \geq n\textsf{A}^2\l((\eta_t\vee\eta_t^2) + 1\r)\r) \leq \textsf C_1e^{-\textsf c_2t^2}.\label{eq:paulmacca}
\end{align}
Now, let us define the two events 
\begin{align*}
    \mathcal E_1 := \l\{\frac{\|\mathbf Y\|_{\mathcal S}^2}{n\textsf{A}^2} - 1 \leq \eta_t \vee \eta_t^2\r\},\quad\quad \mathcal E_2 := \l\{\frac{\|\mathbf Y\|_{\mathcal S}}{\sqrt{n}\textsf{A}} \leq 1\r\},
\end{align*}
where we note that $\mathcal E_1$ is exactly the high-probability event of \eqref{eq:paulmacca}. Then we can first see that, on the event $\mathcal E_1 \cap \mathcal E_2$, we have
\begin{align}
    \mathcal E_1 \cap \mathcal E_2 \implies \|\bY\|_{\mathcal S} \leq \textsf{A}\sqrt{n},\label{duhbro}
\end{align}
which is simply from the definition of the event $\mathcal E_2$. On the other hand, for the event $\mathcal E_1 \cap \mathcal E_2^c$, we have 
\begin{align*}
    \l(\frac{\|\bY\|_{\mathcal S}}{\sqrt{n}\textsf{A}} - 1\r)^2 \vee\  \abs{\frac{\|\bY\|_{\mathcal S}}{\sqrt{n}\textsf{A}} - 1} &\stackrel{(i)}{\leq} \abs{\frac{\|\bY\|_{\mathcal S}^2}{n\textsf{A}^2} - 1}\\
    &\stackrel{(ii)}{=} \frac{\|\bY\|_{\mathcal S}^2}{n\textsf{A}^2} - 1\\
    &\stackrel{(iii)}{\leq} \eta_t \vee \eta_t^2,
\end{align*}
where $(i)$ is from the fact that $(x - y)^2 \vee |x-y| \leq |x^2-y^2|$ for $x , y > 0$ and $x + y \geq 1$, $(ii)$ follows from $\mathcal E_2^c$, and $(iii)$ from $\mathcal E_1$. Since we also know that 
\begin{align*}
    (x \vee x^2) \leq (y \vee y^2) \implies x \leq y\quad \text{ for }\quad x , y \geq 0,
\end{align*}
we can say that
\begin{align}
    \mathcal E_1 \cap \mathcal E_2^c \implies \abs{\frac{\|\bY\|_{\mathcal S}}{\sqrt{n}\textsf{A}} - 1} \leq \eta_t &\implies \|\bY\|_{\mathcal S} \leq \textsf{A}\l(\sqrt n + \textsf C_3 \sqrt p + t\r).\label{duhbroagain}
\end{align}
Combining \eqref{duhbro} and \eqref{duhbroagain}, we conclude that
\begin{align*}
    \mathcal E_1 = (\mathcal E_1 \cap \mathcal E_2) \cup (\mathcal E_1 \cap \mathcal E_2^c) &\implies \|\bY\|_{\mathcal S_p} \leq \tilde{\textsf{A}}\l(\sqrt n + \sqrt p + t\r),
\end{align*}
where $\tilde{\textsf{A}} = \textsf{A}(\textsf{C}_3 + 1)$. Thus we have
\begin{align*}
    \P\l(\|\bY\|_{\mathcal S} \geq \tilde{\textsf{A}}\l(\sqrt n + \sqrt p + t\r)\r) &\leq \P\l(\mathcal E_1^c\r)\\
    &\leq \textsf C_1e^{-\textsf c_2t^2}.
\end{align*}
If we set the variable $y = \tilde{\textsf{A}}^2\l(\sqrt n + \sqrt p + t\r)^2$, then we have that
\begin{align*}
    \P\l(\|\bY\|_{\mathcal S} \geq \sqrt{y}\r) \leq \textsf C_1\cdot\exp\l[-\textsf c_2\l(\frac{\sqrt y}{\tilde{\textsf A}} - \sqrt n - \sqrt p \r)^2\r]
\end{align*}
whenever $\sqrt y \geq \tilde{\textsf A}\l(\sqrt n + \sqrt p\r)$. 
This lets us conclude via the tail-integral formula that, after setting $\varepsilon = 1/4$ for example, 
\begin{align*}
    \E\|\bY\|_{\mathcal S}^2 &= \int_0^\infty\P\l(\|\bY\|_{\mathcal S}^2 \geq y\r)\,\text{d}y\\
    &= \int_0^\infty\P\l(\|\bY\|_{\mathcal S} \geq \sqrt{y}\r)\,\text{d}y\\
    &= \int_0^{\tilde{\textsf{A}}^2(\sqrt{n} + \sqrt{p})^2}\P\l(\|\bY\|_{\mathcal S} \geq \sqrt{y}\r)\,\text{d}y + \int_{\tilde{\textsf{A}}^2(\sqrt{n} + \sqrt{p})^2}^\infty \P\l(\|\bY\|_{\mathcal S} \geq \sqrt{y}\r)\,\text{d}y\\
    &\leq \int_0^{\tilde{\textsf{A}}^2(\sqrt{n} + \sqrt{p})^2}1\,\text{d}y + \textsf C_1\int_{\tilde{\textsf{A}}^2(\sqrt{n} + \sqrt{p})^2}^\infty \exp\l[-\textsf c_2\l(\frac{\sqrt y}{\tilde{\textsf A}} - \sqrt n - \sqrt p \r)^2\r]\,\text{d}y\\
    &\leq \tilde{\textsf{A}}^2(\sqrt{n} + \sqrt{p})^2 + \tilde{\textsf A}^2\l(\frac{1}{\textsf c_2} + \frac{\sqrt{\pi}(\sqrt n + \sqrt p)}{\sqrt{\textsf c_2}}\r)\\
    &\leq \textsf C_{\textsf R} \cdot p,
\end{align*}
for 
\begin{align*}
    \textsf C_{\textsf R} \propto \textsf R^2\log(1 + 12\textsf R). 
\end{align*}

\end{proof}

Now that we have given a bound on this scaled operator norm, we can use this to obtain a bound in the specific cases of block dependence, $m$-dependence, and mixing.
\begin{corollary}[Dependent Norm Bound]\label{dep_norm_bound} Let $(\mathbf{X}_i)$ be a sequence of $\R^p$-valued random vectors satisfying Assumptions \ref{logmodel}-\ref{signalsizinghehe}. Then setting $\mathcal S$ in \lemmaref{ess_pee_two} as either $\mathcal{B}_p(0,\sqrt{p})$ or $\mathcal S_p$ as in \eqref{eq:sp}, the following holds for $n$ sufficiently large: 
\begin{enumerate}
    \item[$(1)$] If $(\mathbf{X}_i)$ satisfies Assumption \ref{blockdepass} (block-dependence), then $\E[\|\mathbf X\|_{\mathcal S}^2] \leq \textnormal{\textsf C}kp$. 
    \item[$(2)$] If $(\mathbf{X}_i)$ satisfies Assumption \ref{general}(i) ($m$-dependence), then $\E[\|\mathbf X\|_{\mathcal S}^2] \leq \textnormal{\textsf C}mp$. 
    \item[$(3)$] If $(\mathbf{X}_i)$ satisfies Assumption \ref{general}(ii) ($\beta$-mixing) and there exists $\varepsilon > 0$ such that 
    \begin{align*}
        \textnormal{\textsf S} := \sum_{\ell=1}^\infty \beta(\ell)^{\frac{\varepsilon}{2+\varepsilon}} < \infty,
    \end{align*}
    then $\E[\|\mathbf X\|_{\mathcal S}^2] \leq \textnormal{\textsf C}\textnormal{\textsf S}p$. Furthermore, for $(\mathbf G_i)$ also satisfying Assumptions \ref{logmodel}-\ref{signalsizinghehe} with $\mathbf G_i \sim \mathcal N\l(0, \textnormal{Var}(\mathbf X_i)\r)$ and $\textnormal{cov}(\mathbf G) = \textnormal{cov}(\mathbf X)$, we have $\E[\|\mathbf G\|_{\mathcal S}^2] \leq \textnormal{\textsf C}\textnormal{\textsf S}p.$
\end{enumerate}
\end{corollary}

\begin{proof}
We display the proofs for $\mathcal S_p$, as those for the ball of radius $\sqrt p$ are identical. 

\vspace{8pt}\noindent$(1)$: Note that $(1)$ follows from $(2)$, as block-dependence with block parameter $k$ is a special case of $m$-dependence with $m = k$. 

\vspace{8pt}\noindent $(2)$: We check that the two conditions of \lemmaref{ess_pee_two} hold. For the first one, note that by the sub-Gaussian Assumption \ref{scalinggg} and \lemmaref{cabris}, we have 
\begin{align*}
    \|\text{Var}(\mathbf X_i)\|_{\text{op}} \leq \frac{\textsf C_1\textsf K_X^2}{n} \leq \frac{\textsf B}{p}
\end{align*}
for some $\textsf C_1, \textsf B> 0$ and $n$ sufficiently large. For the second condition, by sub-Gaussianity and the definition of $\mathcal S_p$, we have that 
\begin{align}
    \|(\mathbf X_i^\intercal\beta)^2\|_{\psi_1} \sie \|\mathbf X_i^\intercal\beta\|_{\psi_2}^2 \leq \frac{\textsf K_X^2\|\beta\|^2}{n} \leq \textsf C_2\textsf L^2 
\end{align}
for $n$ sufficiently large, where $(i)$ is by Lemma 2.7.6 of \cite{vershyninhighdimprob}. Thus we may apply \lemmaref{bern_dep} with $Z_i := (X_i^\intercal\beta^2) - \E\l[(X_i^\intercal\beta^2)\r]$ and $\textsf K = \textsf C_2\textsf L^2$. This gives us our second condition, namely \eqref{eq:bernstein_again}, with the constants $(\textsf R, \textsf K, \textsf C_1, \textsf c_2, \textsf C_3)$ set as $(\textsf L, \textsf B, 4, cm\inv, \textsf C_2)$. For the Gaussin case, the first condition holds for the exact same reason. For the second condition, we again use \lemmaref{lem:lemmethink}, which allows us to satisfy \eqref{eq:bernstein_again} with the constants $(\textsf R, \textsf K, \textsf C_1, \textsf c_2, \textsf C_3)$ set as $(\textsf L, \textsf B, 2, c\textsf S\inv, \textsf K_X\underline{c}\inv)$, respectively. 

\vspace{8pt}\noindent $(4)$: As above, the first condition of \lemmaref{ess_pee_two} holds for the same reason as $(2)$ and $(3)$. The second condition holds by the second statement of \lemmaref{lem:lemmethink}, which allows us to satisfy \eqref{eq:bernstein_again} with the constants $(\textsf R, \textsf K, \textsf C_1, \textsf c_2, \textsf C_3)$ set as $(\textsf L, \textsf B, 2, c\textsf S\inv, c_2\l(\tilde C_2 \wedge \tilde C_2'\r))$.
\end{proof}

Next, recall that we have smoothed our labels $y_1, \hdots, y_n$ by taking a convolution of the sign function with a mollifier $\zeta_\gamma$ as in \eqref{eq:smooth_like_butter}. This next lemma shows that the derivative of this convolution grows at a rate inversely proportional to the smoothing factor, $\gamma$. 

\begin{lemma}[Smoothed Label Derivative Bound]\label{desmos_lol}
    Define $\eta_i' := {\bm 1_\gamma^\pm}'(a_{\mathcal B_i}^\intercal X_{\mathcal B_i}\beta^*- \varepsilon_i)$. Then $$\abs{\eta_i'} \leq 3\gamma\inv.$$
\end{lemma}

\begin{proof}
Recall by the properties of convolution that a derivative can be ``absorbed" into the convolution like so:
\begin{align*}
\eta_i' := {\bm 1^{\pm}_\gamma}'(a_{\mathcal B_i}^\intercal X_{\mathcal B_i}\beta^* - \varepsilon_i) &= (\bm 1^\pm * \zeta_\gamma)'(a_{\mathcal B_i}^\intercal X_{\mathcal B_i}\beta^*- \varepsilon_i)\\
&= (\bm 1^\pm * \zeta_\gamma')(a_{\mathcal B_i}^\intercal X_{\mathcal B_i}\beta^* - \varepsilon_i)\\
&= \int_{-\gamma}^\gamma \bm 1^\pm(a_{\mathcal B_i}^\intercal X_{\mathcal B_i}\beta^*- \varepsilon_i-t)\zeta_\gamma'(t)\,\text{d}t.
\end{align*}
We may thus bound 
\begin{align*}
    \abs{\eta_i'} &= \abs{\int_{-\gamma}^\gamma \bm 1^\pm(a_{\mathcal B_i}^\intercal X_{\mathcal B_i}\beta^* - \varepsilon_i-t)\zeta_\gamma'(t)\,\text{d}t}\\
    &\leq \int_{-\gamma}^\gamma  \abs{\zeta_\gamma'(t)}\,\text{d}t\\
    &= 2\textsf{C}\gamma^2\int_{-\gamma}^\gamma\frac{|t|}{(t^2-\gamma^2)^2}\exp\l(\frac{\gamma^2}{t^2-\gamma^2}\r)\,\text{d}t\\
    &\stackrel{(i)}{\leq} 4\textsf{C}\gamma^2\int_{0}^\gamma\frac{t}{(t^2-\gamma^2)^2}\exp\l(\frac{\gamma^2}{t^2-\gamma^2}\r)\,\text{d}t\\
    &= \frac{2}{e}\textsf C,
\end{align*}
where $(i)$ is from the fact that the integrand is even. Thus it suffices to bound $\textsf C$, which is the integrating constant of our mollifier $\zeta_\gamma$. We may lower bound the integral like so:
\begin{align*}
    \textsf C\inv &= \int_{-\gamma}^\gamma\exp\l(\frac{\gamma^2}{t^2-\gamma^2}\r)\,\text{d}t\\
    &= 2\int_{0}^\gamma\exp\l(\frac{\gamma^2}{t^2-\gamma^2}\r)\,\text{d}t\\
    &\stackrel{(i)}{\geq} 2\int_{0}^{\gamma/2}\exp\l(\frac{\gamma^2}{t^2-\gamma^2}\r)\,\text{d}t\\
    &\stackrel{(ii)}{\geq} 2\int_{0}^{\gamma/2}\exp\l(\frac{\gamma^2}{(\gamma/2)^2-\gamma^2}\r)\,\text{d}t\\
    &= \gamma\cdot e^{-4/3}\\
    &\geq \frac{\gamma}{4},
\end{align*}
where $(i)$ follows from the fact that the integrand is non-negative, and $(ii)$ from the fact that it is a decreasing function, as it has derivative
\begin{align*}
    \frac{-2t\gamma^2}{(t^2-\gamma^2)^2}\exp\l(\frac{\gamma^2}{t^2-\gamma^2}\r) < 0\quad \text{ for } \quad0 < t < \gamma.
\end{align*}
We conclude that 
\begin{align*}
    \abs{\eta_i'} \leq \frac{2}{e}\cdot\frac{4}{\gamma} \leq \frac{3}{\gamma}. 
\end{align*}
\end{proof}
\subsection{General lemmas}\subsubsection{Replacing the True Minimum with a Smoothed \& Discretized Minimum}

To prove that the two minimum risks are close in distribution to one another, we must first smooth the labels, and then also discretize the parameter space that we are taking the minimum over. The following two lemmas show that the error incurred by these two approximations is negligible in the limit.  

\begin{lemma}[Smoothing the Risk]\label{smooth_risk} Suppose that $\bX$ and $\bG$ satisfy Assumptions \ref{logmodel}--\ref{signalsizinghehe}. Let $\gamma \in (0, 1)$. Then there exists $\textnormal{\textsf C} > 0$ such that for $n$ sufficiently large, 
\begin{align*}
    d_{\mathcal H}\l(\min_{\beta\in \mathcal{\tilde S}} \hat R_n(\beta; \bX), \min_{\beta\in \mathcal{\tilde S}}\hat R_n^\gamma(\beta; \bX)\r) &\leq \textnormal{\textsf{C}}\frac{\sqrt{\mathbb{E}\|\bX\|_{\mathcal{S}_p}^2}}{\sqrt{n}}\sqrt{\gamma}\\
    d_{\mathcal H}\l(\min_{\beta\in \mathcal{\tilde S}} \hat R_n(\beta; \bG), \min_{\beta\in \mathcal{\tilde S}} \hat R_n^\gamma(\beta; \bG)\r) &\leq \textnormal{\textsf{C}}\frac{\sqrt{\mathbb{E}\|\bG\|_{\mathcal{S}_p}^2}}{\sqrt{n}}\sqrt{\gamma}.
\end{align*}
\end{lemma}

\begin{proof}
We show the proof only for $\bX$, and note that the exact same technique holds for $\bG$. Since $h \in \mathcal H$ is Lipschitz, we know that if $\tilde\beta$ and $\tilde\beta_\gamma$ are the minimizers of $\hat R_n$ and $\hat R_n^\gamma$ on $\tilde{\mathcal S}$ respectively, then
\begin{align*}
    d_{\mathcal H}\l(\min_\beta \hat R_n(\beta; \bX), \min_\beta \hat R_n^\gamma(\beta; \bX)\r) &= \sup_{h\in\mathcal H}\E\l[h\l(\hat R_n(\tilde\beta; \bX)\r) - h\l(\hat R_n^\gamma(\tilde\beta_\gamma; \bX)\r)\r]\\
    &\leq \sup_{h\in\mathcal H}\|h'\|_{\infty}\E\abs{\hat R_n(\tilde\beta; \bX) - \hat R_n^\gamma(\tilde\beta_\gamma; \bX)}\\
    &\leq \E\abs{\hat R_n(\tilde\beta; \bX) - \hat R_n^\gamma(\tilde\beta_\gamma; \bX)},
\end{align*}
where the last line is by definition of $\mathcal H$. To control this term, we can note that 
\begin{align*}
    \abs{\hat R_n(\tilde\beta; \bX) - \hat R_n^\gamma(\tilde\beta_\gamma; \bX)} \leq
    \max\Big\{&\underbrace{\abs{\hat R_n(\tilde\beta; \bX) - \hat R_n^\gamma(\tilde\beta; \bX)}}_{(a)},\ \  \underbrace{\abs{\hat R_n(\tilde\beta_\gamma;  \bX) - \hat R_n^\gamma(\tilde\beta_\gamma; \bX)}}_{(b)} \Big\}.
\end{align*}
We bound the first term $(a)$ like so:
\begin{align*}
    \abs{\hat R_n(\tilde\beta; \bX) - \hat R_n^\gamma(\tilde\beta; \bX)} &\overset{(i)}{\leq }\frac{1}{n}\sumi{n}\abs{\log\l(1 + e^{-y_iX_i^\intercal\tilde\beta}\r) - \log\l(1 + e^{-\eta_iX_i^\intercal\tilde\beta}\r)}\\
    &\stackrel{(ii)}{\leq} \frac{1}{n}\sumi{n}\abs{X_i^\intercal\tilde\beta}|y_i - \eta_i|\\
    &\stackrel{(iii)}{\leq} \frac{1}{n}\|\bX\tilde\beta\|\ \|\bm y -\bm\eta\|\\
    &\stackrel{(iv)}{\leq} \frac{1}{n}\|\bX\|_{\mathcal S_p}\|\bm y - \bm\eta\|,
\end{align*}
where $(i)$ uses that $0\le \omega_i\le 1$ for all $i \leq n$, $(ii)$ comes from treating the loss as a function of the label and Taylor expanding, $(iii)$ is from Cauchy-Schwarz, and $(iv)$ is from the definition of $\|\bX\|_{\mathcal S_p}$. Applying Cauchy-Schwarz once more we see that
\begin{align*} 
    \E[(a)] &\leq \frac{1}{n}\sqrt{\E\|\bX\|_{\mathcal S_p}^2\E\|\bm y -\bm\eta\|^2}.
\end{align*}
From here we note that
\begin{align*}
    \E\|\bm y -\bm\eta\|^2 &= \sum_{i=1}^n\E[(y_i-\eta_i)^2]\\
    &\stackrel{(i)}{=} \sum_{i=1}^n\E[(y_i-\eta_i)^2\bm 1_{|a_{\mathcal B_i}^\intercal X_{\mathcal B_i}\beta^* - \varepsilon_i| \leq \gamma}]\\
    &\stackrel{(ii)}{\leq} \sum_{i=1}^n\E\l[\P\l(|a_{\mathcal B_i}^\intercal X_{\mathcal B_i}\beta^* - \varepsilon_i| \leq \gamma \mid X_{\mathcal B_i}\r)\r]\\
    &\stackrel{(iii)}{\leq} n\gamma,
\end{align*}
where the indicator in $(i)$ is introduced as 
\begin{align*}
    y_i - \eta_i \neq 0 \iff |a_{\mathcal B_i}^\intercal X_{\mathcal B_i}\beta^* - \varepsilon_i| \leq \gamma,
\end{align*}
$(ii)$ is because $(y_i - \eta_i)^2 \leq 1$, and $(iii)$ from
\begin{align*}
    \P\l(|a_{\mathcal B_i}^\intercal X_{\mathcal B_i}\beta^*- \varepsilon_i| \leq \gamma \mid X_{\mathcal{B}_i}\r) &= \sigma(a_{\mathcal B_i}^\intercal X_{\mathcal B_i}\beta^*+ \gamma) - \sigma(a_{\mathcal B_i}^\intercal X_{\mathcal B_i}\beta^*- \gamma)\\
    &\leq 2\gamma\|\sigma'\|_\infty\\
    &\leq \gamma.
\end{align*}
We conclude that 
\begin{align*}
    \E[(a)] \leq \frac{\sqrt{\mathbb{E}\|\bX\|_{\mathcal{S}_p}^2}}{n}\sqrt{n\gamma} \leq \frac{\sqrt{\mathbb{E}\|\bX\|_{\mathcal{S}_p}^2}}{\sqrt{n}}\sqrt{\gamma} 
\end{align*}
for $n$ sufficiently large. To finish, note that this exact string of inequalities also holds for $\E[(b)]$. 
\end{proof}

Now that we have bounded the difference between the original risk and smoothed risk in terms of the smoothing parameter $\gamma$, we can show that universality for the smoothed risk reduces to universality for the smooth minimum $f_\delta$. 

\begin{lemma}[Discretization \& Smooth-Min]\label{smoothmin} Let $\alpha, \delta > 0$. Suppose that $\bX$ and $\bG$ satisfy Assumptions \ref{logmodel}-\ref{signalsizinghehe}. Write $\mathcal{S}:=\mathcal{B}(0,\sqrt{p})$. Then there exists $\textnormal{\textsf C} > 0$ such that for $n$ sufficiently large, 
\begin{align*}&
     d_{\mathcal H}\l(\min_\beta \hat R_n^\gamma(\beta; \bX), \min_\beta \hat R_n^\gamma(\beta; \bG)\r) \\&\le d_{\mathcal H}\l(f_\delta(\alpha, \bX), f_\delta(\alpha, \bG)\r) + \textsf C \Big(\delta +\frac{\delta}{\sqrt{n}}\sqrt{\E\l[\|\textbf{X}\|^2_{\mathcal{S}}\r]\vee\E\l[\|\textbf{G}\|^2_{\mathcal{S}}\r]} + \frac1\alpha\log\l(\frac{1}{\delta}\r)\Big)
\end{align*} 

\end{lemma}

\begin{proof}
    By the Triangle Inequality, we know that 
\begin{align*}
    d_{\mathcal H}\l(\min_\beta \hat R_n^\gamma(\beta; \bX), \min_\beta \hat R_n^\gamma(\beta; \bG)\r) &\leq d_{\mathcal H}\l(\min_{\beta\in\tilde{\mathcal S}} \hat R_n^\gamma(\beta; \bX), \min_{\beta \in \tilde{\mathcal{S}}_{\delta}} \hat R_n^\gamma(\beta; \bX)\r)\\
    &+ d_{\mathcal H}\l(\min_{\beta \in \tilde{\mathcal{S}}_{\delta}} \hat R_n^\gamma(\beta; \bX), f_\delta(\alpha, \bX)\r)\\ 
    &+ d_{\mathcal H}\l(f_\delta(\alpha, \bX), f_\delta(\alpha, \bG)\r)\\
    &+ d_{\mathcal H}\l(f_\delta(\alpha, \bG), \min_{\beta \in \tilde{\mathcal{S}}_{\delta}} \hat R_n^\gamma(\beta; \bG)\r)\\ 
    &+ d_{\mathcal H}\l(\min_{\beta \in \tilde{\mathcal{S}}_{\delta}} \hat R_n^\gamma(\beta; \bG), \min_{\beta\in\tilde{\mathcal S}} \hat R_n^\gamma(\beta; \bG)\r)\\
    &=: \mathbb D_1 + \mathbb D_2 + \mathbb D_3 + \mathbb D_4 + \mathbb D_5.
\end{align*}
For $\mathbb D_1$, let $\tilde\beta$ be the minimizer of the risk on $\mathcal{\tilde S}$, $\tilde\beta_\delta$ the closest point to it on the $\delta\sqrt p$-net $\mathcal{\tilde S}_{\delta}$, and $\tilde\beta'$ the minimizer of the risk on $\mathcal{\tilde S}_\delta$. Then we have that   
\begin{align*}
    \mathbb D_1 &:= \sup_{h\in\mathcal H}\abs{\E\l[h\l(\hat R_n^\gamma(\tilde\beta; \bX)\r) - h\l(\hat R_n^\gamma(\tilde\beta'; \bX)\r)\r]}\\
    &\leq \sup_{h\in\mathcal H}\E\abs{h\l(\hat R_n^\gamma(\tilde\beta; \bX)\r) - h\l(\hat R_n^\gamma(\tilde\beta'; \bX)\r)}\\
    &\stackrel{(i)}{\leq} \E\abs{\hat R_n^\gamma(\tilde\beta; \bX) - \hat R_n^\gamma(\tilde\beta'; \bX)}\\
    &\stackrel{(ii)}{\leq} \E\abs{\hat R_n^\gamma(\tilde\beta; \bX) - \hat R_n^\gamma(\tilde\beta_\delta; \bX)}\\
    &\leq \E\l[\sup_{\nu}\abs{\lip \nabla \hat R_n^\gamma(\nu; \bX), \tilde\beta - \tilde\beta_\delta\rip}\r],
\end{align*}
where $(i)$ is from the definition of $\mathcal H$ and $(ii)$ is because 
\begin{align*}
    \hat R_n^\gamma(\tilde\beta_\delta) \geq \hat R_n^\gamma(\tilde\beta')
\end{align*}
by definition. The gradient of our risk has coordinates
\begin{align*}
    \frac{\partial}{\partial \nu_j}\hat R_n^\gamma(\nu; \bX) = \frac{\lambda}{n}\nu_j - \frac{1}{n}\sumi{n}\omega_i\eta_iX_{ij}\sigma_{i\nu},
\end{align*}
and thus

{\begin{align*}
    \mathbb D_1 &\leq \frac{\lambda}{n}\E\l[\sup_\nu\abs{\lip \nu, \tilde\beta - \tilde\beta_\delta \rip}\r]
    + \frac{1}{n}\E\l[\sup_{\nu}\abs{\l\lip \sumi{n}\eta_i\omega_iX_{i}\sigma_{i\nu}, \tilde\beta - \tilde\beta_\delta\r\rip}\mathbb{I}(\bX\in E_n)\r]\\
    &\stackrel{(i)}{\leq} \frac{\lambda \textsf{L}\delta p}{n}+ \frac{1}{n}\E\l[\sup_{\nu}\abs{\sumi{n}\eta_i\omega_i\sigma_{i\nu}\l\lip X_i, \tilde\beta - \tilde\beta_\delta\r\rip}\mathbb{I}(\bX\in E_n)\r]\\
    &\overset{(ii)}{\le} \frac{\lambda \textsf{L}\delta p}{n}+ \frac1n\E\l[\sup_{\nu}\sqrt{\sumi{n} \eta_i^2\sigma_{i\nu}^2\omega_i^2}\sqrt{\sumi{n} \l\lip X_{i}, \tilde\beta - \tilde\beta_\delta\r\rip^2}\r]\\& \stackrel{(iii)}{\le} \frac{\lambda \textsf{L}\delta p}{n}+ \frac{1}{\sqrt n}\E\l[\sqrt{\sumi{n} \l\lip X_{i}, \tilde\beta - \tilde\beta_\delta\r\rip^2}\r] 
   \\&\overset{(iv)}{\le} \frac{\lambda \textsf{L}\delta p}{n}+ \frac{\delta}{\sqrt{n}}\E\l[\|\textbf{X}\|_{\mathcal{S}}\r]
\end{align*}
}
Above, $(i)$ is from the fact that
\begin{align*}
    \abs{\lip \nu, \tilde\beta - \tilde\beta_\delta \rip} \leq \|\nu\|\ \|\tilde\beta - \tilde\beta_\delta\| \leq \textsf L\sqrt p\delta \sqrt p = \textsf L\delta p,
\end{align*}
$(ii)$ is a result of Cauchy-Schwarz, $(iii)$ uses that $\eta_i, \omega_i,\sigma_{i,v}\le 1$ for all $i\le n$, $(iv)$ is via the definition of $\|\bX\|_{\mathcal S}$ and the fact that $\delta^{-1}(\tilde \beta-\tilde \beta_{\delta})\in\mathcal{S}$.
 \ To bound $\mathbb D_2$, we observe that 
\begin{align}
    \abs{f_{\delta}(\alpha, \bX) -\min_{\beta\in \mathcal{\tilde S}_{ \delta}}\hat R_n^\gamma(\beta; \bX)} &= \abs{\frac{-1}{n\alpha}\log\l(\sum_{\beta\in \mathcal{\tilde S}_{\delta}}\exp\l[-n\alpha\hat R_n^\gamma(\beta; \bX)\r]\r) - \hat R_n^\gamma(\tilde\beta; \bX)} \nonumber\\
    &= \frac{1}{n\alpha}\abs{\log\l(\sum_{\beta\in \mathcal{\tilde S}_{\delta}}\exp\l[-n\alpha\hat R_n^\gamma(\beta; \bX)\r]\r) - \log\l(\exp\l[-n\alpha\hat R_n^\gamma(\tilde\beta; \bX)\r]\r)}\nonumber\\
    &= \frac{1}{n\alpha}\abs{\log\l(\sum_{\beta\in \mathcal{\tilde S}_{\delta}}\exp\l[-n\alpha\l(\hat R_n^\gamma(\beta; \bX)-\hat R_n^\gamma(\tilde\beta; \bX)\r)\r]\r)}\nonumber\\
    &<
    \frac{1}{n\alpha}\log\abs{\mathcal{\tilde S}_{\delta}},\label{eq:dee_two_bound}
\end{align}
where the last line follows from the fact that $\hat R_n^\gamma(\beta; \bX)-\hat R_n^\gamma(\tilde\beta; \bX)$ is always non-negative for $\beta \in \mathcal {\tilde S}_{\delta}$ by definition, and is zero at least once, meaning when we multiply by $-n\alpha$ they will always be either zero or strictly negative. This means the sum inside of the logarithm lies in $\l(1, \abs{\mathcal {\tilde S}_{\delta}}\r)$. By Proposition 4.2.12 of \cite{vershyninhighdimprob}, we can say that since $\mathcal S_p \subseteq B_{\R^p}(\bm 0, \textsf{L}\sqrt{p})$, then
\begin{align*}
    \abs{\mathcal {\tilde S}_{\delta}} \leq \l(\frac{3\textsf{L}\sqrt{p}}{\delta\sqrt p}\r)^p = \l(\frac{3\textsf{L}}{\delta}\r)^p,
\end{align*}
and so combining this with \eqref{eq:dee_two_bound} we have
\begin{align*}
    \mathbb D_2 \leq \frac{p}{n\alpha}\log\l(\frac{3\textsf{L}}{\delta}\r) \leq \textsf{C}_3\frac{1}{\alpha}\log\l(\frac{1}{\delta}\r)
\end{align*}
for $n$ sufficiently large. We finish by noting that $\mathbb D_4$ and $\mathbb D_5$ have the exact same bounds as $\mathbb D_2$ and $\mathbb D_1$, respectively. 
\end{proof}

\subsection{Main lemmas under independence under partial block independence.}\label{lemm:bliock}
Let $i\le n$ be an index. Throughout this subsection we will assume that the block $(\bX_j, y_j(\bX_j))_{j\in \mathcal{B}_i}$ is independent from  $(\bX_j, y_j(\bX_j))_{j\not\in \mathcal{B}_i}$ where $\mathcal{B}_i$ is the block of size $k$ that contains $i$.

\subsubsection{Main Lemma}

In the next result, we recall the notation that $\tilde{\bU}_i \equiv \tilde{\bU}_i(t) = \cos(t) \bX_i - \sin(t) \bG_i$ for $t \in [0,1]$, where we abbreviate the dependence on $t$ for the simplicity of presentation in the proof.

\begin{lemma}\label{cabristop}Let $i\le n$. Let $(\bX_j, y_j(\bX_j))$ and $(\bG_j, y_j(\bG_j))$ be generated under Assumptions \ref{logmodel}-\ref{gauss_approx} and that $\rm{Var}(\bG)=\rm{Var}(\bX)$. In addition, suppose that the block $(\bX_j, y_j(\bX_j))_{j\in \mathcal{B}_i}$ is independent from  $(\bX_j, y_j(\bX_j))_{j\not\in \mathcal{B}_i}$. For every $t\in (0,\frac{\pi}{2})$, we have

    \begin{align}\label{eq:taubound}
     \limsup_{n\rightarrow \infty}   \abs{\E\l[ -h'\l(f_\delta(\bU)\r)\lip \tilde U_i^\intercal\mathcal D_{i}(\beta)\rip\r]}\le \tau 
    \end{align}
and
 \begin{align}
     \limsup_{n\rightarrow \infty}   \abs{\E\l[ -h'\l(f_\delta(\bU)\r)\lip \tilde U_i^\intercal\mathcal D_{i,L}(\beta)\rip\r]}\le \tau 
    \end{align}
    where we defined $\mathcal{D}_{i,L}(\beta):=\mathcal{D}_i(\beta)\mathbb{I}(|\tilde U_i^T\beta|,|\tilde U_i^T\beta|\le L)$
\end{lemma}
\begin{proof} We note that it is enough to show the desired result for $\mathcal{D}_{i,L}$ as \cref{eq:taubound} directly follows by taking $L=\infty$. 

Firstly by adding and subtracting the quantity $h'(f_\delta(\bU^{ik}))$ we obtain 
\begin{align*}&
          \abs{\E\l[-h'\l(f_\delta(\bU)\r)\lip \tilde U_i^\intercal\mathcal D_{i,L}(\beta)\rip\r]} \\ &\leq \underbrace{\E\abs{\l(h'(f_\delta(\bU)) - h'(f_\delta(\bU^{ik}))\r)\lip \tilde U_i^\intercal\mathcal D_{i,L}(\beta)\rip}}_{(a)} + \underbrace{\abs{\E \l[h'(f_\delta(\bU^{ik}))\lip \tilde U_i^\intercal\mathcal D_{i,L}(\beta)\rip\r]}}_{(b)}.
\end{align*}
For the term $(a)$, we use Cauchy-Schwarz to say that 
\begin{align*}
    \E\abs{\l(h'(f_\delta(\bU)) - h'(f_\delta(\bU^{ik}))\r)\lip \tilde U_i^\intercal\mathcal D_{i,L}(\beta)\rip} &\leq \underbrace{\sqrt{\E\l[h'(f_\delta(\bU)) - h'(f_\delta(\bU^{ik}))\r]^2}}_{(a_1)}\cdot\underbrace{\sqrt{\E\lip \tilde U_i^\intercal\mathcal D_{i,L}(\beta)\rip^2}}_{(a_2)}.
\end{align*}
To control these two terms, we first apply \lemmaref{lindy} to $(a_1)$ to obtain
\begin{align}
    (a_1) \leq \sqrt{\frac{\textsf{C}_1k^2}{n}} = \frac{\textsf{C}_2k}{\sqrt n},\label{eq:cee}
\end{align}
and then apply \lemmaref{second_mom} to $(a_2)$ to say that
\begin{align}
    \E\lip \tilde U_i^\intercal\mathcal D_{i,L}(\beta)\rip^2 &\stackrel{(i)}{=} \E \l\lip\frac{\tilde U_i^\intercal \mathcal D_{i,L}(\beta)e^{-\alpha\sum_{j\in\mathcal B_i}\ell_j(\beta)}}{\lip e^{-\alpha\sum_{j\in\mathcal B_i}\ell_j(\beta)}\rip_{i,k}}\r\rip_{i,k}^2\nonumber\\
    &\stackrel{(ii)}{\leq} \E \l\lip\l(\frac{\tilde U_i^\intercal \mathcal D_i(\beta)e^{-\alpha\sum_{j\in\mathcal B_i}\ell_j(\beta)}}{\lip e^{-\alpha\sum_{j\in\mathcal B_i}\ell_j(\beta)}\rip_{i,k}}\r)^2\r\rip_{i,k}\nonumber\\
    &= \E \l\lip\E_{(i,k)}\l[\l(\frac{\tilde U_i^\intercal \mathcal D_i(\beta)e^{-\alpha\sum_{j\in\mathcal B_i}\ell_j(\beta)}}{\lip e^{-\alpha\sum_{j\in\mathcal B_i}\ell_j(\beta)}\rip_{i,k}}\r)^2\r]\r\rip_{i,k}\nonumber\\
    &\stackrel{(iii)}{\leq} \E\l[\sup_{\beta\in\mathcal {\tilde S}_{\delta}}\E_{(i, k)}\l[\l(\frac{\tilde U_i^\intercal \mathcal D_i(\beta)e^{-\alpha\sum_{j\in\mathcal B_i}\ell_j(\beta)}}{\lip e^{-\alpha\sum_{j\in\mathcal B_i}\ell_j(\beta)}\rip_{i,k}}\r)^2\r]\r]\nonumber\\
    &\stackrel{(iv)}{\leq} \textnormal{\textsf C}_1(k, \alpha, \gamma),\label{eq:dee}
\end{align}
where $(i)$ comes from noting that, for a general function $\textsf g$, 
\begin{align}
    \lip\textsf{g}(\beta)\rip = \sum_{\beta\in\mathcal {\tilde S}_{\delta}}\hspace{-5pt}w_\gamma(\beta)\hspace{1pt}\textsf{g}(\beta)\nonumber
    = \frac{\sum_{\beta\in\mathcal {\tilde S}_{\delta}}e^{-n\alpha\hat R_n^\gamma(\beta)}\textsf{g}(\beta)}{\sum_{\beta'\in\mathcal{\tilde S}_{\delta}}e^{-n\alpha\hat R_n^\gamma(\beta')}}&= \frac{\sum_{\beta\in\mathcal S_{\delta}}e^{-n\alpha\hat R_n^{\gamma, i, k}(\beta)}\textsf{g}(\beta)e^{-\alpha\sum_{j\in\mathcal B_i}\omega_j\ell_j(\beta)}}{\sum_{\beta'\in\mathcal {\tilde S}_{\delta}}e^{-n\alpha\hat R_n^\gamma(\beta')}e^{-\alpha\sum_{j\in\mathcal B_i}\omega_j\ell_j(\beta')}}\nonumber\\
    &= \l\lip\frac{\textsf{g}(\beta)e^{-\alpha\sum_{j\in\mathcal B_i}\omega_j\ell_j(\beta)}}{\lip e^{-\alpha\sum_{j\in\mathcal B_i}\omega_j\ell_j(\beta)}\rip_{i,k}}\r\rip_{i,k},\label{eq:rewriteliprip}
\end{align}
$(ii)$ follows from Jensen's Inequality combined with the fact that $|\tilde U_i^T\mathcal{D}_{i,L}(\beta)|\overset{a.s}{\le}\tilde U_i^T\mathcal{D}_{i}(\beta)$, $(iii)$ from the fact that 
\begin{align}
    \lip \textsf g(\beta)\rip \leq \msup_{\beta\in\mathcal {\tilde S}_{\delta}}g(\beta),\label{eq:replacewithsup}
\end{align}
and $(iv)$ is exactly the statement of \lemmaref{second_mom} for some $\textsf C_1(k,\alpha,\gamma)$. Combining \eqref{eq:cee} and \eqref{eq:dee}, we conclude that
\begin{align}
    (a) \leq \frac{\textsf C_2k\textsf C_1(k,\alpha,\gamma)}{\sqrt n}    =: \frac{\textsf{C}_2(k,\alpha,\gamma)}{\sqrt n}.\label{eq:aisbounded}
\end{align}
For the term $(b)$, we first apply \lemmaref{poly_approx}, which says that there exist $\textsf D = \textsf D(k, \alpha, \gamma, \tau)$ and real coefficients $b_0, \hdots, b_{\textsf D}$ such that
\begin{align}
    (b) \leq \tau + \sum_{\ell=0}^{\textnormal{\textsf D}}|b_\ell|\sup_{\substack{\beta_0, \hdots, \beta_\ell \\ \in \mathcal {\tilde S}}}\abs{\E\l[\tilde U_i^\intercal \mathcal D_{i,L}(\beta_0)\exp\Big(-\alpha\sum_{r=0}^\ell\sum_{j\in\mathcal B_i} \omega_j\ell(\eta_j, U_j^\intercal\beta_r)\Big)\r]},\label{eq:taupluspoly}
\end{align}
noting by definition that $\abs{b_\ell} = {{\textsf D} \choose \ell} \leq \textsf D!$ is bounded. Now, we define the Gaussian interpolator 
\begin{align*}
    \bV := \sin(t)\tilde{\bG} + \cos(t)\bG
\end{align*}
with $\tilde{\bG}$ as defined in \lemmaref{many_betas} being an identical copy of $\bG$. Moreover note that according to \cref{lem:lipschitz:approx} for all $\epsilon>0$ there exists $g_\epsilon(\cdot)$ such that $\|g_\epsilon\|_\infty\le 1$ and $\|g_\epsilon\|_{\rm Lipchitz}\le2 \epsilon^{-1}$ and such that  $$\big|g_\epsilon(|\tilde U_i^T\beta|,|U_i^T\beta|)-\mathbb{I}(|\tilde U_i^T\beta|,|U_i^T\beta|\le L)\big|\le \mathbb{I}\Big[|\tilde U_i^T\beta|\in [L-\epsilon,L]\Big]+\mathbb{I}\Big[| U_i^T\beta|\in [L-\epsilon,L]\Big].$$ We note that\begin{align*}
    \limsup_{n\to\infty}\sup_{\substack{\beta_0, \hdots, \beta_\ell \\ \in \mathcal {\tilde S}}}&\Big|\E\l[\tilde U_i^\intercal\mathcal D_{i}(U_{\mathcal B_i}, \beta_0)g_\epsilon(|\tilde U_i^T\beta_0|,|U_i^T\beta_0|)\exp\Big(-\alpha\sum_{r=0}^\ell\sum_{j\in\mathcal B_i} \omega_j\ell(\eta_j,  U_j^\intercal\beta_r)\Big)\r]\\&-\E\l[\tilde U_i^\intercal\mathcal D_{i,L}(U_{\mathcal B_i}, \beta_0)\exp\Big(-\alpha\sum_{r=0}^\ell\sum_{j\in\mathcal B_i} \omega_j\ell(\eta_j,  U_j^\intercal\beta_r)\Big)\r]\Big|\\\le \limsup_{n\to\infty}\sup_{\substack{\beta_0 \in \mathcal {\tilde S}}}&\E\l[\Big|\tilde U_i^\intercal\mathcal D_{i}(U_{\mathcal B_i}, \beta_0)\Big|\mathbb{I}\big(|U_{i}^T\beta_0|\in [L-\epsilon,L])\r]\\&+\E\l[\Big|\tilde U_i^\intercal\mathcal D_{i}(U_{\mathcal B_i}, \beta_0)\Big|\mathbb{I}\big(|\tilde U_{i}^T\beta_0|\in [L-\epsilon,L])\r]
    \\\le \limsup_{n\to\infty}\sup_{\substack{\beta_0 \in \mathcal {\tilde S}}}&\E\l[\Big|\tilde U_i^\intercal\mathcal D_{i}(U_{\mathcal B_i}, \beta_0)\Big|^2\r]^{1/2}\sqrt{P\big(|U_{i}^T\beta_0|\in [L-\epsilon,L]\big)+P\big(|\tilde U_{i}^T\beta_0|\in [L-\epsilon,L]\big)}
\end{align*}
Using \cref{scalinggg} we know that $$\limsup_{n\to\infty}\sup_{\substack{\beta_0 \in \mathcal {\tilde S}}}\E\l[\Big|\tilde U_i^\intercal\mathcal D_{i}(U_{\mathcal B_i}, \beta_0)\Big]\r]<\infty.$$ Moreover we note that \begin{align*}
    P\big(|U_{i}^T\beta_0|\in [L-\epsilon,L]\big)&=\mathbb{E}\Big(\Phi^c\Big(\frac{L-\epsilon-\rm{sin}(t)X_i^T\beta_0}{\rm{cos}(t)}\Big)\Big)-\mathbb{E}\Big(\Phi^c\Big(\frac{L-\rm{sin}(t)X_i^T\beta_0}{\rm{cos}(t)}\Big)\Big)
    \\&\le \frac{\epsilon}{\rm{cos}(t)}\mathbb{E}\Big(\varphi\Big(\frac{L-\epsilon-\rm{sin}(t)X_i^T\beta_0}{\rm{cos}(t)}\Big)\Big)\\&\le \frac{\epsilon}{\rm{cos}(t)}
\end{align*} Similarly we can obtain that \begin{align*}
    P\big(|\tilde U_{i}^T\beta_0|\in [L-\epsilon,L]\big)\le \frac{\epsilon}{\rm{sin}(t)}.
\end{align*} Hence we obtain that there is a constant $\textsf C_3>0$ such that \begin{align}\label{krio}
    \limsup_{n\to\infty}\sup_{\substack{\beta_0, \hdots, \beta_\ell \\ \in \mathcal {\tilde S}}}&\Big|\E\l[\tilde U_i^\intercal\mathcal D_{i}(U_{\mathcal B_i}, \beta_0)g_\epsilon(|\tilde U_i^T\beta_0|,|U_i^T\beta_0|)\exp\Big(-\alpha\sum_{r=0}^\ell\sum_{j\in\mathcal B_i} \omega_j\ell(\eta_j,  U_j^\intercal\beta_r)\Big)\r]\\&-\E\l[\tilde U_i^\intercal\mathcal D_{i,L}(U_{\mathcal B_i}, \beta_0)\exp\Big(-\alpha\sum_{r=0}^\ell\sum_{j\in\mathcal B_i} \omega_j\ell(\eta_j,  U_j^\intercal\beta_r)\Big)\r]\Big|\\\le 
\textsf C_3 \sqrt{ \epsilon }.  
\end{align}Moreover we note that 
\begin{align}
    &\limsup_{n\to\infty}\sup_{\substack{\beta_0, \hdots, \beta_\ell \\ \in \mathcal {\tilde S}}}\abs{\E\l[\tilde U_i^\intercal\mathcal D_i(U_{\mathcal B_i}, \beta_0)g_\epsilon(|\tilde U_i^T\beta_0|,|U_i^T\beta_0|)\exp\Big(-\alpha\sum_{r=0}^\ell\sum_{j\in\mathcal B_i} \omega_j\ell(\eta_j,  U_j^\intercal\beta_r)\Big)\r]}\nonumber\\
    &\stackrel{(i)}{\leq} \limsup_{n\to\infty}\sup_{\substack{\beta_0, \hdots, \beta_\ell \\ \in \mathcal 
    {\tilde S}}}\abs{\E\l[\tilde V_i^\intercal\mathcal D_i(V_{\mathcal B_i}, \beta_0)g_\epsilon(|\tilde V_i^T\beta_0|,|V_i^T\beta_0|)\exp\Big(-\alpha\sum_{r=0}^\ell\sum_{j\in\mathcal B_i} \omega_j\ell(\eta_j,  V_j^\intercal\beta_r)\Big)\r]}\nonumber\\&=\limsup_{n\to\infty}\sup_{\substack{\beta_0, \hdots, \beta_\ell \\ \in \mathcal 
    {\tilde S}}}\abs{\E\l[\mathbb{E}\Big(g_\epsilon(|\tilde V_i^T\beta_0|,|V_i^T\beta_0|)\tilde V_i\Big|(V_j)\Big)^\intercal\mathcal D_i(V_{\mathcal B_i}, \beta_0)\exp\Big(-\alpha\sum_{r=0}^\ell\sum_{j\in\mathcal B_i} \omega_j\ell(\eta_j,  V_j^\intercal\beta_r)\Big)\r]}\nonumber\\
    &\overset{(ii)}{=} 0.\label{eq:limsupiszero}
\end{align}
Above, $(i)$ follows from the second statement of \lemmaref{many_betas} coupled with \Cref{gauss_approx}, since we know that the function 
\begin{align*}
    g(X_{\mathcal B_i}B, G_{\mathcal B_i}B) := \tilde U_i^\intercal\mathcal D_i(U_{\mathcal B_i}, \beta_0)g_\epsilon(|\tilde V_i^T\beta_0|,|V_i^T\beta_0|)\exp\Big(-\alpha\sum_{r=0}^\ell\sum_{j\in\mathcal B_i} \omega_j\ell(\eta_j,  U_j^\intercal\beta_r)\Big),
\end{align*}
where $B := (\beta_0, \hdots, \beta_\ell) \in \R^{p\times(\ell+1)}$ is locally Lipschitz as all of its components are, and is also square integrable as 
\begin{align}
    \sup_B\E\l[g(X_{\mathcal B_i}B, G_{\mathcal B_i}B)^2\r] \leq \sup_B\E\l[\l(\tilde U_i^\intercal\mathcal D_i(\beta_0)\r)^2\r] \leq \textsf C(k, \gamma)
\end{align}
by the argument in \eqref{eq:uitdb} and \eqref{eq:see_five} of \lemmaref{second_mom}. Further, $(ii)$ follows from independence of $\tilde V_i$ from $(V_j)$, as
\begin{align*}
    E[\tilde V_iV_j^\intercal] = \sin(t)\cos(t)\E[X_iX_j^\intercal] - \sin(t)\cos(t)\E[G_iG_j^\intercal] = 0,
\end{align*}
and we know that two zero-covariance Gaussians are necessarily independent of each other, combined with the fact that $x\rightarrow g_\epsilon(x^T\beta_0,V_i^T\beta_0)x$ is a symmetric function.  We may thus combine \eqref{eq:aisbounded}, \eqref{eq:taupluspoly}, \eqref{krio} and \eqref{eq:limsupiszero} to conclude 
\begin{align}
    \limsup_{n\to\infty}\abs{\E\l[ -h'\l(f_\delta(\bU)\r)\lip \tilde U_i^\intercal\mathcal D_{i,L}(\beta)\rip\r]} &\leq \limsup_{n\to\infty}\frac{\textsf C_2(k,\alpha,\gamma)}{\sqrt n} + \tau +\textsf C_3\sqrt{\epsilon}= \tau+\textsf C_3\sqrt{\epsilon},
\end{align}
As the dependence on $\epsilon$ is arbitrary we get the desired result.
\end{proof}
\subsubsection{Upper Bounds on Expectations \& Approximations}
The next set of lemmas are used to bound various expectations that appear in the proofs of our main results. Throughout this subsection $i\le n$ designates an index. We assume that $(\bX_j, y_j(\bX_j))$ and $(\bG_j, y_j(\bG_j))$ are generated under Assumptions \ref{logmodel}-\ref{gauss_approx} and that $\rm{Var}(\bG)=\rm{Var}(\bX)$. In addition, suppose that the block $(\bX_j, y_j(\bX_j))_{j\in \mathcal{B}_i}$ is independent from  $(\bX_j, y_j(\bX_j))_{j\not\in \mathcal{B}_i}$. We begin with a Lindeberg bound. 

\begin{lemma}[Lindeberg Difference Bound]\label{lindy}
Let $\alpha, \delta> 0$. Then there exists $\textnormal{\textsf C} > 0$ such that for each $t \in [0, \tfrac\pi2]$ and $n$ sufficiently large, 
\begin{align*}
    \E\l(\l[h'(f_\delta(\bU)) - h'(f_\delta(\bU^{ik}))\r]^2\r) \leq \frac{\textnormal{\textsf{C}}k^2}{n}.
\end{align*}
\end{lemma}

\begin{proof}
By the Lipschitzness of $h'$, we first have
\begin{align*}
    \abs{h'(f_\delta(\bU)) - h'(f_\delta(\bU^{ik}))} &\leq \|h'\|_{\text{Lip}}\abs{f_\delta(\bU) - f_\delta(\bU^{ik})}\\
    &\stackrel{(i)}{\leq} \frac{\textsf C_1}{n\alpha}\abs{\log\l(\frac{\sum_{\beta}\exp\l[-n\alpha\hat R_n^\gamma(\beta; \bm \eta, \bU)\r]}{\sum_{\beta}\exp\l[-n\alpha\hat R_n^\gamma(\beta; \bm \eta, \bU^{ik})\r]}\r)}\\
    &\stackrel{(ii)}{\leq} \frac{\textsf C_1}{n\alpha}\l(k\alpha\log(2)+\Big\lip\alpha \sum_{j\in\mathcal B_i}\omega_j\ell(\eta_j, U_j^\intercal\beta)\Big\rip_{i,k}\r)\\
    &\leq \frac{\textsf C_1}{n}\l(k + \Big\lip \sum_{j\in\mathcal B_i}\omega_j\ell(\eta_j, U_j^\intercal\beta)\Big\rip_{i,k}\r),
\end{align*}
where $(i)$ is because $h \in \mathcal H$, and $(ii)$ is via Jensen's inequality on the natural logarithm and the fact that $\ell(a, 0) = \log(2)$ for any $a \in \R$. Thus, we have
\begin{align*}
    \E\l[h'(f_\delta(\bU)) - h'(f_\delta(\bU^{ik}))\r]^2 \leq \frac{\textsf C_1}{n^2}\l(\E\Big\lip \sum_{j\in\mathcal B_i}\omega_j\ell(\eta_j, U_j^\intercal\beta)\Big\rip_{i,k}^2 + k\E\Big\lip \sum_{j\in\mathcal B_i}\omega_j\ell(\eta_j, U_j^\intercal\beta)\Big\rip_{i,k} + k^2\r).
\end{align*}
From here, we can first notice that
\begin{align*}
    \log(1+ e^{-x}) \leq |x| + 1 \quad\ \ \  \text{ for all } x \in \R,
\end{align*}
and thus, using $\abs{\eta_j} = 1, \omega_j \leq 1$, and Cauchy-Schwarz, we can say that
\begin{align*}
    \sum_{j\in\mathcal B_i}\omega_j\ell(\eta_j, U_j^\intercal\beta)&\leq \sum_{j\in\mathcal B_i}\l(\abs{\eta_jU_j^\intercal\beta} + 1\r)\\
    &\leq k + \|\beta\|\sum_{j\in\mathcal B_i}\|U_j\|,
\end{align*}
and so
\begin{align*}
    \E\Big\lip \sum_{j\in\mathcal B_i}\omega_j\ell(\eta_j, U_j^\intercal\beta)\Big\rip_{i,k} &\leq k + \sum_{j\in\mathcal B_i}\E\Big\lip \|\beta\|\ \|U_j\|\Big\rip_{i,k}\\
    &\leq k + \sum_{j\in\mathcal B_i}\E\l[\|U_j\|\sup_{\beta\in\mathcal {\tilde S}_{\delta}}\|\beta\|\r]\\
    &\leq k + \textsf{L}\sqrt{p}\sum_{j\in\mathcal B_i}\E\l[\|U_j\|\r] \\
    &\stackrel{(i)}{\leq} k + \textsf{L}\sqrt{p}\sum_{j\in\mathcal B_i}\sqrt{\sum_{k=1}^p\E[U_{jk}^2]}\\
    &\stackrel{(ii)}{\leq} k + \textsf{L}k\sqrt{p}\sup_j\sqrt{\text{Tr}(\Sigma_j)}\\
    &\stackrel{(iii)}{\leq} \textsf{C}_2k\sqrt{p},
\end{align*}
where $(i)$ is via Jensen's Inequality, $(ii)$ is because  
\begin{align*}
    \E\l[U_{jk}^2\r] = \E\l[\sin^2(t)X_{jk}^2 + \cos^2(t)G_{jk}^2 + \sin(t)\cos(t)X_{jk}G_{jk}\r] = (\Sigma_j)_{k,k}\Big(\sin^2(t) + \cos^2(t)\Big) = (\Sigma_j)_{k,k},
\end{align*}
and $(iii)$ is via \lemmaref{cabris} and the fact that
\begin{align*}
   \text{Tr}(\Sigma_j) = \sum_{i=1}^p\lip \Sigma_j e_i, e_i\rip \leq \sum_{i=1}^p\|\Sigma_j\|_{\text{op}} \leq p\|\Sigma_j\|_{\text{op}}= O(1),
\end{align*} where the last inequality comes from the scaling of $\bU$.
Using an identical argument and the fact that
\begin{align*}
    \E\Big\lip \sum_{j\in\mathcal B_i}\omega_j\ell(\eta_j, U_j^\intercal\beta)\Big\rip_{i,k}^2 \leq \E\Big\lip \Big(\sum_{j\in\mathcal B_i}\omega_j\ell(\eta_j, U_j^\intercal\beta)\Big)^2\Big\rip_{i,k},
\end{align*}
we similarly obtain 
\begin{align*}
    \E\Big\lip \sum_{j\in\mathcal B_i}\omega_j\ell(\eta_j, U_j^\intercal\beta)\Big\rip_{i,k}^2 \leq \textsf{C}_3k^2p,
\end{align*}
and thus 
\begin{align*}
    \E\l(\l[h'(f_\delta(\bU)) - h'(f_\delta(\bU^{ik}))\r]^2\r)\leq \frac{\textsf C_1}{n^2}\l( \textsf{C}_3k^2p + \textsf{C}_2k^2\sqrt{p} + k^2\r) \leq \frac{\textsf{C}k^2}{n}
\end{align*}
for $n$ sufficiently large. 
\end{proof}

\begin{lemma}[Second Moment Bound]\label{second_mom} There exists $\textnormal{\textsf C}(k, \alpha, \gamma) > 0$ such that for every $t \in \l[0, \frac{\pi}{2}\r]$ and $i = 1, \hdots, n$, we have 
\begin{align*}
    \sup_{\beta\in\mathcal {\tilde S}}\E_{(i, k)}\l[\l(\frac{\tilde U_i^\intercal\mathcal D_{i}(\beta)e^{-\alpha\sum_{j\in\mathcal B_i}\omega_j\ell_j(\beta)}}{\lip e^{-\alpha\sum_{j\in\mathcal B_i}\omega_j\ell_j(\beta)}\rip_{i,k}}\r)^2\r] \leq \textnormal{\textsf C}(k, \alpha, \gamma).
\end{align*}

\end{lemma}

\begin{proof}
Fix some $\beta\in\mathcal{\tilde S}$. We may bound
\begin{align*}
   \E_{(i, k)}\l[\l(\frac{\tilde U_i^\intercal \mathcal D_i(\beta)e^{-\alpha\sum_{j\in\mathcal B_i}\omega_j\ell_j(\beta)}}{\lip e^{-\alpha\sum_{j\in\mathcal B_i}\omega_j\ell_j(\beta)}\rip_{i,k}}\r)^2\r] &\stackrel{(i)}{\leq} \E_{(i, k)}\l[\l(\tilde U_i^\intercal \mathcal D_i(\beta)\r)^4\r]^{1/2}\E_{(i, k)}\l[\l(\frac{e^{-\alpha\sum_{j}\omega_j\ell_j(\beta)}}{\lip e^{-\alpha\sum_{j}\omega_j\ell_j(\beta)}\rip_{i,k}}\r)^4\r]^{1/2}\\
   &\stackrel{(ii)}{\leq} \E\l[\l(\tilde U_i^\intercal \mathcal D_i(\beta)\r)^4\r]^{1/2}\Big\lip\E\l[e^{4\alpha\sum_{j}\omega_j\ell_j(\beta)}\r]\Big\rip_{i,k}^{1/2}\\
   &\leq \underbrace{\E\l[\l(\tilde U_i^\intercal \mathcal D_i(\beta)\r)^4\r]^{1/2}}_{(a)}\underbrace{\l(\sup_\beta\E\l[e^{4\alpha\sum_{j}\omega_j\ell_j(\beta)}\r]\r)^{1/2}}_{(b)},
\end{align*}
where $(i)$ is via Cauchy-Schwarz, and $(ii)$ is from the fact that $\E_{(i, k)}[\ \cdot\ ]$ and $\lip\ \cdot\ \rip_{i, k}$ commute due to $w_{\gamma}^{i, k}(\beta)$ only being a function of $(U_j)_{j\notin\mathcal B_i}$, and from applying Jensen's Inequality to the convex function $x^{-4}$, using the fact that 
\begin{align*}
    \alpha\sum_{j\in\mathcal N_i}\omega_j\ell_j(\beta) \geq 0 \implies \exp\l(-\alpha\sum_{j\in\mathcal N_i}\omega_j\ell_j(\beta)\r) \leq 1.
\end{align*}
Also in $(ii)$, note that the conditional expectation has disappeared due to block dependence. For the term $(a)$, we can notice that 
\begin{align}
    \abs{\tilde U_i^\intercal\mathcal D_{i}(\beta)} &= \abs{\eta_i\omega_i\sigma_{i\beta}\tilde U_i^\intercal\beta + \sum_{j\in\mathcal B_i}\omega_j\sigma_{j\beta}\eta_j'a_{ji}U_j^\intercal\beta \tilde U_i^\intercal\beta^*}\nonumber\\
    &\stackrel{(i)}{\leq} \abs{\tilde U_i^\intercal\beta} + \frac{3}{\gamma}\abs{\tilde U_i^\intercal\beta^*}\sum_{j\in\mathcal B_i}\abs{U_j^\intercal\beta},\label{eq:uitdb}
\end{align}
where $(i)$ is from $\eta_i, \omega_i, \sigma_{i\beta}, a_{ji} \leq 1$ and \lemmaref{desmos_lol} bounding $\abs{\eta_j'}$. From here, we know that since the rows of $\bX$ and $\bG$ are sub-Gaussian by \Cref{scalinggg}, so must those of $\bU$ and $\tilde{\bU}$ be as well, since 
\begin{align*}
    \|U_i\|_{\psi_2} = \|\sin(t)X_i + \cos(t)G_i\|_{\psi_2} \leq \abs{\sin(t)}\|X_i\|_{\psi_2} + \abs{\cos(t)}\|G_i\|_{\psi_2} \leq \frac{\sqrt2\textsf K_X}{\sqrt n},
\end{align*}
and similarly for $\tilde{\bU}$. This further implies that for any $\beta \in \mathcal S_p$ we have 
\begin{align}
    \|U_i^\intercal\beta\|_{\psi_2} &\leq \frac{\sqrt2\textsf K_X\|\beta\|}{\sqrt n}\nonumber\\
    &\leq \frac{\sqrt2\textsf K_X\textsf L \sqrt p}{\sqrt n}\nonumber\\
    &\leq \textsf C_1\label{eq:lol_see_one}
\end{align}
for $n$ sufficiently large, and again this holds for $\tilde {U}_i^\intercal\beta$ as well. We conclude by a multinomial expansion that 
\begin{align}
    \E\l[\l(\tilde U_i^\intercal\mathcal D_i(\beta)\r)^4\r] &\stackrel{(i)}{\leq} \sum_{\ell=0}^4{4 \choose \ell}\E\l[\abs{\tilde U_i^\intercal\beta}^{4-\ell}\frac{3^\ell}{\gamma^\ell}\abs{\tilde U_i^\intercal\beta^*}^\ell\l(\sum_{j\in\mathcal B_i}\abs{U_j^\intercal\beta}\r)^\ell\r]\nonumber\\
    &\stackrel{(ii)}{\leq} \frac{486}{\gamma^4}\sum_{\ell=0}^4\E\l[\abs{\tilde U_i^\intercal\beta}^{4-\ell}\abs{\tilde U_i^\intercal\beta^*}^\ell\l(\sum_{j\in\mathcal B_i}\abs{U_j^\intercal\beta}\r)^\ell\r]\nonumber\\
    &\stackrel{(iii)}{\leq} \frac{162}{\gamma^4}\sum_{\ell=0}^4\E\l[\abs{\tilde U_i^\intercal\beta}^{12-3\ell} + \abs{\tilde U_i^\intercal\beta^*}^{3\ell} + \l(\sum_{j\in\mathcal B_i}\abs{U_j^\intercal\beta}\r)^{3\ell}\r]\nonumber\\
    &\stackrel{(iv)}{\leq} \frac{162}{\gamma^4}\sum_{\ell=0}^4l\l(\textsf C_2\sqrt{12-3\ell}\r)^{12-3\ell} + \l(\textsf C_3\sqrt{3\ell}\r)^{3\ell} + \l(\textsf C_4k\sqrt{3\ell}\r)^{3\ell}\nonumber\\
    &\stackrel{(v)}{\leq}\textsf C_5\gamma^{-4}k^{12},\label{eq:see_five}
\end{align}
where $(i)$ is from \eqref{eq:uitdb} and binomial expansion, $(ii)$ is because $\ell\leq4$ and $\max_\ell{4\choose\ell} = 6$, $(iii)$ is the AM-GM Inequality for $n = 3$, $(iv)$ is from Proposition 2.5.2 in \cite{vershyninhighdimprob} on equivalent properties of sub-Gaussian random variables, and $(v)$ is from the fact that $$3\ell \vee (12-3\ell) \leq 12\quad \text{ for all } \quad0 \leq \ell \leq 4.$$ For the term $(b)$, we once again use that $|\omega_i| \leq 1$ and $\log(1+e^{-x}) \leq |x| + 1$ for $x\in\R$ to say that 
\begin{align}
    \E\l[e^{4\alpha\sum_{j}\omega_j\ell_j(\beta)}\r] &\leq \E\l[e^{4\alpha\sum_{j}|U_j^\intercal\beta|+1}\r] \nonumber\\
    &\stackrel{(i)}{\leq} e^{4k\alpha}\E\l[e^{4\alpha\sum_{j}|U_j^\intercal\beta|}\r] \nonumber\\
    &\leq  e^{4k\alpha}\E\l[\prod_{j\in\mathcal B_i}e^{4\alpha|U_j^\intercal\beta|}\r]\nonumber\\
    &\stackrel{(ii)}{\leq} e^{4k\alpha}\l(\prod_{j\in\mathcal B_i}\E\l[e^{4k\alpha|U_j^\intercal\beta|}\r]\r)^{1/k}\nonumber\\
    &\stackrel{(iii)}{\leq} e^{4k\alpha(\mu+1)}\l(\prod_{j\in\mathcal B_i}\E\l[e^{4k\alpha(|U_j^\intercal\beta|-\mu)}\r]\r)^{1/k} \nonumber\\
    &\stackrel{(iv)}{\leq} e^{4k\alpha(\mu+1)}e^{\textsf C_1k^2\textsf{K}_X^2\alpha^2}\nonumber\\
    &\leq e^{\textsf{C}_2k^2\alpha^2},\label{eq:see_six}
    \end{align}
where $(i)$ is via $\abs{\mathcal B_i} = k$, $(ii)$ is via H\"older's Inequality, $(iii)$ is from adding and subtracting $\mu := \E[|U_j^\intercal\beta|]$ in the exponent, which satisfies
\begin{align*}
    \mu &\leq \sqrt{\E\l[(U_j^\intercal\beta)^2\r]} = \sqrt{\beta^\intercal\Sigma_j\beta} \leq \|\beta\|\ \|\Sigma_j\|_{\text{op}}^{1/2} \leq \frac{\textsf C\textsf K_X\textsf L\sqrt p}{\sqrt n} \leq \textsf C_3
\end{align*}
for $n$ sufficiently large, by Jensen's Inequality and \lemmaref{cabris}, and $(iv)$ is via sub-Gaussianity of the centered version of $U_j^\intercal\beta$, which is sub-Gaussian by Lemma 2.6.8 of \cite{vershyninhighdimprob}, and thus satisfies Condition $(v)$ of Proposition 2.5.2 of the same text. Since this holds for all $\beta\in\tilde{\mathcal S}$, we conclude by combining \eqref{eq:see_five} and \eqref{eq:see_six} that 
\begin{align*}
    \sup_{\beta\in\mathcal {\tilde S}}\E_{(i, k)}\l[\l(\frac{\tilde U_i^\intercal\mathcal D_{i}(\beta)e^{-\alpha\sum_{j\in\mathcal B_i}\omega_j\ell_j(\beta)}}{\lip e^{-\alpha\sum_{j\in\mathcal B_i}\omega_j\ell_j(\beta)}\rip_{i,k}}\r)^2\r] \leq \textsf C_5\gamma^{-2}k^{6}\exp\l(\textsf{C}_2k^2\alpha^2\r)
 =: \textnormal{\textsf C}(k, \alpha, \gamma).
\end{align*}
\end{proof}

The following lemma employs a technique developed in \cite{montanari2022universality}, which will allow us to convert a complicated term involving the inverse function $1/x$ into one involving a polynomial that is much more straightforward to control. 

\begin{lemma}[Polynomial Approximation]\label{poly_approx} Let $\alpha, \delta, \gamma, \tau>0$. Then there exists $\textnormal{\textsf D} = \textnormal{\textsf D}(k, \alpha, \gamma, \tau)$ and coefficients $b_0, \hdots, b_{\textnormal{\textsf D}} \in \R$ such that
\begin{align*}&
    \abs{\E \l[h'(f_\delta(\bU^{ik}))\lip \tilde U_i^\intercal\mathcal D_{i}(\beta)\rip\r]} \\&\hspace{50pt}\leq \tau + \sum_{\ell=0}^{\textnormal{\textsf D}}|b_\ell|\sup_{\substack{\beta_0, \hdots, \beta_\ell \\ \in \mathcal {\tilde S}}}\abs{\E\l[\tilde U_i^\intercal \mathcal D_i(\beta_0)\exp\Big(-\alpha\sum_{r=0}^\ell\sum_{j\in\mathcal B_i}\omega_j \ell(\eta_j, U_j^\intercal\beta_r)\Big)\r]}.
\end{align*}
\end{lemma}

\begin{proof}
Since $\bU^{ik}$ was constructed to be independent of $U_{\mathcal B_i}$ and $\tilde U_{\mathcal B_i}$, we first expand this quantity as 
\begin{align}
    &\abs{\E \l[h'(f_\delta(\bU^{ik})\lip \tilde U_i^\intercal\mathcal D_{i}(\beta)\rip\r]}\nonumber\\ &\hspace{30pt}\stackrel{(i)}{=} \abs{\E \l[\E_{(i, k)}\l[h'(f_\delta(\bU^{ik})\lip \tilde U_i^\intercal\mathcal D_{i}(\beta)\rip\r]\r]}\nonumber\\
    &\hspace{30pt}\stackrel{(ii)}{=} \abs{\E \l[h'(f_\delta(\bU^{ik})\E_{(i, k)}\lip \tilde U_i^\intercal\mathcal D_{i}(\beta)\rip\r]}\nonumber\\
    &\hspace{30pt}\leq \|h'\|_\infty \cdot\E \abs{\E_{(i, k)}\lip \tilde U_i^\intercal\mathcal D_{i}(\beta)\rip}\nonumber\\
    &\hspace{30pt}\stackrel{(iii)}{\leq} \E \abs{\E_{(i, k)}\l\lip \frac{\tilde U_i^\intercal\mathcal D_{i}(\beta)e^{-\alpha\sum_{j\in\mathcal B_i}\omega_j\ell(\eta_j, U_j^\intercal\beta)}}{\lip e^{-\alpha\sum_{j\in\mathcal B_i}\omega_j\ell(\eta_j, U_j^\intercal\beta)}\rip_{i, k}}\r\rip_{i, k}}\nonumber\\
    &\hspace{30pt}\stackrel{(iv)}{=} \E\abs{\l\lip \E_{(i, k)}\l(\frac{\tilde U_i^\intercal\mathcal D_{i}(\beta)e^{-\alpha\sum_{j\in\mathcal B_i}\omega_j\ell(\eta_j, U_j^\intercal\beta)}}{\lip e^{-\alpha\sum_{j\in\mathcal B_i}\omega_j\ell(\eta_j, U_j^\intercal\beta)}\rip_{i, k}}\r)\r\rip_{i, k}}\nonumber\\
    &\hspace{30pt}\stackrel{(v)}{\leq} \E\Bigg[\sup_{\beta_0 \in \mathcal{\tilde S}} \underbrace{\Bigg|\E_{(i, k)}\l(\frac{\tilde U_i^\intercal\mathcal D_{i}(\beta_0)e^{-\alpha\sum_{j\in\mathcal B_i}\omega_j\ell(\eta_j, U_j^\intercal\beta_0)}}{\lip e^{-\alpha\sum_{j\in\mathcal B_i}\omega_j\ell(\eta_j, U_j^\intercal\beta)}\rip_{i, k}}\r)\Bigg|}_{(a)}\Bigg],\label{eq:supremumoverbetazero}
\end{align}
where $(i)$ is via the law of total expectation, $(ii)$ is by the independence mentioned above, $(iii)$ is by definition of $h \in \mathcal H$ and the ability to rewrite $\lip \tilde U_i^\intercal\mathcal D_{i}(\beta)\rip$ as done previously in \eqref{eq:rewriteliprip}, $(iv)$ is because $\E_{(i, k)}[\ \cdot\ ]$ and $\lip\ \cdot\ \rip_{i, k}$ commute with each other, and $(v)$ is the same as in \eqref{eq:replacewithsup}. 

Now, we will approximate the inverse function $x\inv$ by a polynomial by defining the functions
\begin{align*}
    Q_{\textnormal{\textsf D}}(x) := \sum_{\ell=0}^\textnormal{\textsf D}(1-x)^\ell = \sum_{\ell=0}^\textnormal{\textsf D}b_\ell x^\ell,\quad\quad R_\textnormal{\textsf D}(x) := \frac{1}{x} - Q_{\textnormal{\textsf D}}(x)
\end{align*}
for some degree $\textsf D$ and $x \in (0, 1]$. Recall from \lemmaref{second_mom} that there exists $\textnormal{\textsf C}(k, \alpha, \gamma)$ such that
\begin{align}
    \sup_{\beta\in\mathcal {\tilde S}}\E_{(i, k)}\l[\l(\frac{\tilde U_i^\intercal\mathcal D_{i}(\beta_0)e^{-\alpha\sum_{j\in\mathcal B_i}\omega_j\ell_j(\beta_0)}}{\lip e^{-\alpha\sum_{j\in\mathcal B_i}\omega_j\ell_j(\beta)}\rip_{i,k}}\r)^2\r] \leq \textnormal{\textsf C}(k, \alpha, \gamma).\label{eq:secondmomforpoly}
\end{align}
Thus we may choose the degree of our polynomial to be the exact $\textsf D = \textsf D\l(k, \alpha, \tau^2 / \textnormal{\textsf C}(k, \alpha, \gamma)\r)$ such that, by \lemmaref{lem:remaindersquared}, we have 
\begin{align}
    \E_{(i, k)}\l[R_\textnormal{\textsf D}\l(\lip e^{-\alpha\sum_{j\in\mathcal B_i}\omega_j\ell(\eta_j, U_j^\intercal\beta)}\rip_{i, k}\r)^2\r] < \frac{\tau^2}{\textnormal{\textsf C}(k, \alpha, \gamma)}.\label{eq:tauover_seekay}
\end{align}
We conclude that 
\begin{align}
    (a) &\overset{(i)}{\leq} \abs{\E_{(i, k)}\l[\tilde U_i^\intercal\mathcal D_{i}(\beta_0)e^{-\alpha\sum_{j\in\mathcal B_i}\omega_j\ell_j(\beta_0)}Q_{\textsf D}\l(\lip e^{-\alpha\sum_{j\in\mathcal B_i}\omega_j\ell_j(\beta)}\rip_{i, k}\r)\r]}\nonumber\\
    &\quad\quad + \abs{\E_{(i, k)}\l[\tilde U_i^\intercal\mathcal D_{i}(\beta_0)e^{-\alpha\sum_{j\in\mathcal B_i}\omega_j\ell_j(\beta_0)}R_{\textsf D}\l(\lip e^{-\alpha\sum_{j\in\mathcal B_i}\omega_j\ell_j(\beta)}\rip_{i, k}\r)\r]}\nonumber\\
    &\overset{(ii)}{\leq} \abs{\E_{(i, k)}\l[\tilde U_i^\intercal\mathcal D_{i}(\beta_0)e^{-\alpha\sum_{j\in\mathcal B_i}\omega_j\ell_j(\beta_0)}Q_{\textsf D}\l(\lip e^{-\alpha\sum_{j\in\mathcal B_i}\omega_j\ell_j(\beta)}\rip_{i, k}\r)\r]}\nonumber\\
    &\quad\quad + \E_{(i, k)}\l[\l(\tilde U_i^\intercal\mathcal D_{i}(\beta_0)e^{-\alpha\sum_{j\in\mathcal B_i}\omega_j\ell_j(\beta_0)}\r)^2\r]^{1/2}\E_{(i, k)}\l[R_{\textsf D}\l(\lip e^{-\alpha\sum_{j\in\mathcal B_i}\omega_j\ell_j(\beta)}\rip_{i, k}\r)^2\r]^{1/2}\nonumber\\
    &\stackrel{(iii)}{\leq}  \abs{\E_{(i, k)}\l[\tilde U_i^\intercal\mathcal D_{i}(\beta_0)e^{-\alpha\sum_{j\in\mathcal B_i}\omega_j\ell_j(\beta_0)}Q_{\textsf D}\l(\lip e^{-\alpha\sum_{j\in\mathcal B_i}\omega_j\ell_j(\beta)}\rip_{i, k}\r)\r]} + \tau,\label{eq:firsttermlikeso}
\end{align}
where (i) is via the Triangle Inequality, (ii) via Cauchy-Schwarz, and $(iii)$ and \eqref{eq:secondmomforpoly} with \eqref{eq:tauover_seekay}. To finish, we rewrite the first term in \eqref{eq:firsttermlikeso} like so: recall that if $X_1, \hdots, X_\ell$ are $\ell$ i.i.d. random variables with the same distribution as some random variable $X$, then 
\begin{align*}
    \E[e^X]^\ell = \E[e^{X_1}]\cdots\E[e^{X_\ell}] = \prod_{r=1}^\ell\E[e^{X_r}] = \E\l[\prod_{r=1}^\ell e^{X_r}\r] = \E\l[e^{\sum_{r=1}^\ell X_r}\r],
\end{align*}
which means that we can say 
\begin{align}
    \abs{\E_{(i, k)}\l[\tilde U_i^\intercal\mathcal D_{i}(\beta_0)e^{-\alpha\sum_{j\in\mathcal B_i}\omega_j\ell_j(\beta_0)}Q_{\textsf D}\l(\lip e^{-\alpha\sum_{j\in\mathcal B_i}\omega_j\ell_j(\beta)}\rip_{i, k}\r)\r]}\nonumber\\
    &\hspace{-160pt}=\abs{\E_{(i, k)}\l[\tilde U_i^\intercal\mathcal D_i(\beta_0)e^{-\alpha\sum_{j\in\mathcal B_i}\omega_j\ell_j(\beta_0)}\sum_{\ell=0}^{\textsf D}b_\ell\lip e^{-\alpha\sum_{j\in\mathcal B_i}\omega_j\ell_j(\beta)}\rip_{i, k}^\ell\r]}\nonumber\\
    &\hspace{-160pt}\leq \sum_{\ell=0}^\textsf{D}\abs{b_\ell}\abs{\E_{(i, k)}\l[\tilde U_i^\intercal\mathcal D_i(\beta_0)e^{-\alpha\sum_{j\in\mathcal B_i}\omega_j\ell_j(\beta_0)}\lip e^{-\alpha\sum_{j\in\mathcal B_i}\omega_j\ell(\eta_j, U_j^\intercal\beta)}\rip_{i, k}^\ell\r]}\nonumber\\
    &\hspace{-160pt}\leq \sum_{\ell=0}^\textsf{D}\abs{b_\ell}\abs{\l\lip\E_{(i, k)}\l[ \tilde U_i^\intercal\mathcal D_i(\beta_0)e^{-\alpha\sum_{r=0}^\ell\sum_{j\in\mathcal B_i}\omega_j\ell_j(\beta_r)}\r]\r\rip_{i, k, \ell}}\nonumber\\
    &\hspace{-160pt}\leq  \sum_{\ell=0}^\textsf{D}\abs{b_\ell}\sup_{\substack{\beta_1, \hdots, \beta_\ell \\ \in \mathcal {\tilde S}}}\abs{\E\l[ \tilde U_i^\intercal\mathcal D_i(\beta_0)e^{-\alpha\sum_{r=0}^\ell\sum_{j\in\mathcal B_i}\omega_j\ell_j(\beta_r)}\r]},\label{eq:ihatethis}
\end{align}
where the final expectation is unconditional due to independence, and $\lip\ \cdot\ \rip_{i, k, \ell}$ represents the $\ell$-dimensional joint expectation with marginals following $\lip\ \cdot\ \rip_{i, k}$. Combining \eqref{eq:firsttermlikeso} and \eqref{eq:ihatethis}  with the supremum over $\beta_0$ in \eqref{eq:supremumoverbetazero}, we conclude the result. 
\end{proof}

The following lemma allows us to convert the statement about Gaussian approximation from \Cref{gauss_approx}, which involves $k$ terms, to one that involves arbitrarily many, which will be important when combined with the polynomial derived from the previous lemma. 

\begin{lemma}[$\bm k$-to-many Betas]\label{many_betas} 
Let $\ell \geq 1$ and $g: \R^{2\ell k} \to \R$. Then
\begin{enumerate}
    \item If $g$ is bounded Lipschitz and 
    \begin{align*}
        K \;\coloneqq\; 
        \sup_{\substack{f\in \mathcal{F}\\\theta\in \mathcal{S}^{k-1}}}\sup_{\substack{(\beta_1, \hdots, \beta_k) \\ \in\mathcal S_p^k}}\abs{\E f\l(\sum_{i=1}^k\theta_iX_i^\intercal\beta_i\r) - \E f\l(\sum_{i=1}^k\theta_iG_i^\intercal\beta_i\r) }
        \;\leq\; 1\;,
    \end{align*}
    then there exists some $\textnormal{\textsf C}_{k,l} > 0$ such that
\begin{align*}&
     \sup_{\substack{B = (\beta_1, \hdots, \beta_\ell) \\ \in\mathcal {\tilde S}^\ell}}\abs{\E\l[g(X_{\mathcal B_i}B, \tilde G_{\mathcal B_i}B) - g( G_{\mathcal B_i}B, \tilde G_{\mathcal B_i}B)\r]}
     \\
     &\hspace{15pt} \leq \textnormal{\textsf C}_{k,l} \,
    \Big(
    \| g \|_{\rm Lip} 
    + 
    \| g \|_\infty
    \,
    \big( 
    \mean[ \| X_{\mathcal B_i}B \|_\infty + \| G_{\mathcal B_i}B \|_\infty + 2 \| \tilde G_{\mathcal B_i}B \|_\infty  ]
    + 
    1 \big) 
    \Big)\, K^{\frac{1}{8kl-4}} 
\end{align*}
where $\tilde{\bG}$ is an independent copy of $\bG$ and $\mathcal{F}$ is as in \Cref{gauss_approx}.

\item If $g$ is locally Lipschitz \& square-integrable and \Cref{gauss_approx} holds, then 
\begin{align*}
    \sup_{\substack{B = (\beta_1, \hdots, \beta_\ell) \\ \in\mathcal {\tilde S}^\ell}}\abs{\E\l[g(X_{\mathcal B_i}B, \tilde G_{\mathcal B_i}B) - g( G_{\mathcal B_i}B, \tilde G_{\mathcal B_i}B)\r]} \xrightarrow{n\to\infty} 0.
\end{align*}
\end{enumerate}
\end{lemma}

\begin{proof}
To prove the first statement, let $g: \R^{2\ell k} \to \R$ be a bounded Lipschitz function. For notational simplicity, we may assume WLOG that the block $\mathcal B_i$ begins at index $i$, meaning 
\begin{align*}
    X_{\mathcal B_i} := \begin{pmatrix}
    \rule[.5ex]{2.5ex}{0.5pt} & X_i & \rule[.5ex]{2.5ex}{0.5pt}\\
    \  & \vdots & \ \\
    \rule[.5ex]{2.5ex}{0.5pt} & X_{i+k-1} & \rule[.5ex]{2.5ex}{0.5pt}
    \end{pmatrix} \in \R^{k\times p},\quad\quad B := \begin{pmatrix}
    \vert & \  & \vert\\
    \beta_1 & \cdots & \beta_\ell\\
    \vert & \  & \vert
    \end{pmatrix} \in \mathcal {\tilde S}^\ell \subseteq \R^{p \times \ell}.
\end{align*}
Let us define the following quantities:
\begin{align*}
    \mathcal M := B^\intercal \otimes I_{2k} \in \R^{2k\ell \times 2kp}\ \ \ \ \ \ \bm{u} := \textsf{vec}{X_{\mathcal B_i}\choose \tilde G_{\mathcal B_i}},\ \bm{v} := \textsf{vec}{G_{\mathcal B_i}\choose \tilde G_{\mathcal B_i}} \in \R^{2kp\times 1},
\end{align*}
where $\textsf{vec}(A)$ is the vectorized version of a matrix $A$. This allows us to say that 
\begin{align*}
    g(X_{\mathcal B_i}B, \tilde G_{\mathcal B_i}B) - g( G_{\mathcal B_i}B, \tilde G_{\mathcal B_i}B) = g(\mathcal M\bm{u}) - g(\mathcal M\bm{v}).
\end{align*}
Now, let $\sigma>0$ and define $\bm Z \sim \mathcal N(0, \sigma^2I_{2k\ell})$ independent of all other quantities. Then by the Triangle Inequality,
\begin{align}
\abs{\E\l[g(\mathcal M\bm{u}) - g(\mathcal M\bm{v})\r]} &\leq \abs{\E\l[g(\mathcal M\bm{u}) - g(\mathcal M\bm{u} + \bm Z)\r]} + \abs{\E\l[g(\mathcal M\bm{v}) - g(\mathcal M\bm{v} + \bm Z)\r]}\label{eq:firsttwoguys}\\
&\ \ \ \ \ \ \ \ \ + \abs{\E\l[g(\mathcal M\bm{u} + \bm Z) - g(\mathcal M\bm{v} + \bm Z)\r]}.\label{eq:annoyingdude}
\end{align}
The two quantities in \eqref{eq:firsttwoguys} are easily bounded, as 
\begin{align}
    \abs{\E\l[g(\mathcal M\bm{u}) - g(\mathcal M\bm{u} + \bm Z)\r]} &\leq \E\abs{g(\mathcal M\bm{u}) - g(\mathcal M\bm{u} + \bm Z)}\nonumber\\
    &\leq \|g\|_{\text{Lip}}\cdot\E\|\bm Z\|_2\nonumber\\
    &\leq \|g\|_{\text{Lip}}\cdot\sigma\sqrt{2k\ell},\label{eq:sigma2kl} 
\end{align}
and similarly for $\abs{\E\l[g(\mathcal M\bm{v}) - g(\mathcal M\bm{v} + \bm Z)\r]}$. 
To control the quantity in \eqref{eq:annoyingdude}, for $R > 0$ we consider the functions 
\begin{align*}
    g_R(y) 
    \;\coloneqq&\; 
    (
        g(y) 
        -
        \bar g(y)
    )
    \, 
    \ind_{\{ \| y \|_\infty \leq R \}}\;,
    \\
    \bar g(y)
    \;\coloneqq&\;
    \msum_{r \leq 2k, r' \leq l}
        \mfrac{g(y^{(r,r')}(R)) -g(y^{(r,r')}(-R))}{2R} 
        y_{r,r'} \be_{r,r'}
    \;,
\end{align*}
where $y^{(r,r')}(R) \in \R^{2kl}$ is a copy of $y$ with the $(r,r')$-th coordinate replaced by $R$, $y_{r,r'}$ is the $(r,r')$-th coordinate of $y$ and $\be_{r,r'} \in \R^{2kl}$ is the standard Euclidean basis vector with $1$ at the position $(r,r')$. The support of the function $g_R$ is contained within the hyperrectangle $\{ y \in \R^{2kl} \,|\, \| y \|_\infty \leq R\}$. Moreover,  $g_R$ is continuous inside the hyperrectangle and by the choice of $\bar g$, $g_R$ agrees on the boundary of the hyperrectangle. This allows us to extend $g_R$ to a continuous and $\| g \|_{\text{Lip}}$-Lipschitz function $\tilde g_R$ that is $2R$-periodic in each coordinate. By construction 
\begin{align*}
    &\;
    \big| 
        \mean[ g(\cM\bu + \bZ) ] 
        - 
        \mean[\tilde g_R(\cM\bu + \bZ) ]
    \big|
    \\
    &\;\leq\;
    \big| 
        \mean[g(\cM \bu + \bZ) - g(\cM \bu + \bZ)  \ind_{\{ \| \cM \bu + \bZ \|_\infty \leq R \}} ] 
    \big|
    +
     \big| 
        \mean[ 
        \bar g(\cM \bu + \bZ) \, \ind_{\{ \| \cM \bu + \bZ \|_\infty \leq R \}} 
        ]
    \big| 
    \\
    &\qquad 
    +
    \big| 
        \mean[ \tilde g_R(\cM\bu + \bZ)  \, \ind_{\{ \| \cM \bu + \bZ \|_\infty > R \}} ]
    \big| 
    \\
    &\;\leq\;
    \| g \|_\infty \, \P( \| \cM \bu + \bZ \|_\infty > R )
    + 
    \mfrac{2kl \, \| g \|_\infty}{R} \mean [ \|\cM \bu + \bZ \|_\infty ]
    +
    (1 + 2kl)  \, \| g \|_\infty  \, \P( \| \cM \bu + \bZ \|_\infty > R )
    \\
    &\;\leq\;
    \mfrac{( 2 + 4kl) \| g \|_\infty \, \mean[ \| \cM \bu + \bZ \|_\infty]  }{R}
    \;,
    \tagaligneq \label{eq:hyperrectangle:bound}
\end{align*}
and the same bound holds with $\bu$ replaced by $\bv$. We can now approximate $\tilde g_R$ coordinate-wise by a truncated Fourier series.  Since $\tilde g_R$ is continuous and $\| g \|_{\text{Lip}}$-Lipschitz, applying a well-established result on uniform convergence of Fourier series to each coordinate (see e.g.~\cite{jackson1930theory,alimov1992multiple}) says that there exists some absolute constant $C' > 0$ such that, for every $N \in \N$,
\begin{align*}
    \| \tilde g_R -  \tilde g_R^{(N)}
    \|_\infty \;\leq\;  C' \| g \|_{\text{Lip}} \, \mfrac{\log N}{N}\;,
\end{align*}
where the truncated Fourier series is given by 
\begin{align*}
    \tilde g_R^{(N)}( y )
    \;\coloneqq\;
    \,
    \msum_{w \in \{-N, \ldots, N\}^{2kl} } 
    \, e^{i \frac{\pi w^\intercal y}{R}  } \, 
    \Big(\mfrac{1}{(2R)^{2kl}} \mint_{[0,2R]^{2kl}} \tilde g_R (t) e^{- i \frac{\pi w^\intercal t}{R}} dt  \Big)
    \;.
\end{align*}
This implies that
\begin{align*}
    &\;\big| 
        \mean\big[ 
            g( \cM \bu + \bZ )
            -
            g(  \cM \bv + \bZ  )
        \big] 
    \big|
    \\
    &\;\leq\;
    \big| 
        \mean\big[ 
            g( \cM \bu + \bZ )
            -
            \tilde g_R(  \cM \bu + \bZ  )
        \big] 
    \big|
    +
    \big| 
        \mean\big[ 
            g( \cM \bv + \bZ )
            -
            \tilde g_R(  \cM \bv + \bZ  )
        \big] 
    \big|
    \\
    &\qquad 
    +
    \big| 
        \mean\big[ 
            \tilde g( \cM \bu + \bZ )
            -
            \tilde g_R^{(N)}(  \cM \bu + \bZ  )
        \big] 
    \big|
    +
    \big| 
        \mean\big[ 
            \tilde g( \cM \bv + \bZ )
            -
            \tilde g_R^{(N)}(  \cM \bv + \bZ  )
        \big] 
    \big|
    \\
    &\qquad 
    +(\star)
    \\
    &\;\overset{\eqref{eq:hyperrectangle:bound}}{\leq}\;
    \mfrac{( 2 + 4kl) \, \| g \|_\infty \, \mean[ \| \cM \bu + \bZ \|_\infty + \| \cM \bv + \bZ \|_\infty]  }{R}
    +
    \mfrac{2 C' \| g \|_{\text{Lip}} \, \log N}{N}
    +
    (\star)
    \;,  \tagaligneq\label{eq:smoothed:Fourier:halfway}
\end{align*}
where 
\begin{align*}
    (\star) 
    \;\coloneqq&\;
    \big| 
        \mean\big[ 
            \tilde g_R^{(N)}( \cM \bu + \bZ  )
            -
            \tilde g_R^{(N)}( \cM \bv + \bZ  )
        \big] 
    \big|
    \\
    \;=&\;
    \bigg| 
        \sum_{w \in \{-N, \ldots, N\}^{2kl} } 
        \Big( 
                \mean[
                     e^{i \frac{\pi w^\intercal ( \cM \bu + \bZ )}{R}} 
                     -
                     e^{i\frac{\pi w^\intercal ( \cM \bv + \bZ )}{R}} 
                ]  
        \Big)
        \Big(\mfrac{1}{(2R)^{2kl}} \mint_{[0,2R]^{2kl}} \tilde g_R (t) e^{- i \frac{\pi w^\intercal t}{R}} dt \Big)
    \bigg|
    \\
    \;=&\;
    \bigg| 
        \sum_{w \in \{-N, \ldots, N\}^{2kl} } 
        \Big( 
                \mean[
                     e^{i \frac{\pi w^\intercal( \cM \bu)}{R} } 
                     -
                     e^{i \frac{\pi w^\intercal (\cM \bv)}{R} } 
                ]  
        \Big)
        e^{ - \frac{\pi^2 \sigma^2 \| w \|^2}{2R^2}}
        \Big(\mfrac{1}{(2R)^{2kl}} \mint_{[0,2R]^{2kl}} \tilde g_R (t) e^{- i \frac{\pi w^\intercal t}{R}} dt \Big)
    \bigg|
    \\
    \;\leq&\;
    \| g \|_\infty 
    \,
    \msum_{w \in \{-N, \ldots, N\}^{2kl} } 
        \big|
                \psi_{\cM u}\big( \mfrac{\pi w}{R} \big)
                -
                \psi_{\cM v}\big( \mfrac{\pi w}{R} \big)
        \big|
    \,
    e^{ - \frac{\pi^2 \sigma^2 \| w \|^2}{2 R^2}}
    \\
    \;=&\;
    \| g \|_\infty 
    \,
    \msum_{w \in \{- \frac{N}{R}, \ldots, \frac{N}{R}\}^{2kl} } 
        \big|
                \psi_{\cM u}(\pi w)
                -
                \psi_{\cM v}(\pi w)
        \big|
    \,
    e^{ - \frac{\pi^2 \sigma^2 \| w \|^2}{2}}
    \;.
    \tagaligneq\label{eq:trucated:Fourier:diff}
\end{align*}
To control the difference of characteristic functions, let $t \in \R^{2k\ell}$, and decompose it as $t = (s, \tilde s)$ for $s, \tilde s \in \R^{k\ell}$. Then we have that 
\begin{align}
    \abs{\psi_{\mathcal M\bm{u}}(t) - \psi_{\mathcal M\bm{v}}(t)}^2 &\stackrel{(i)}{=} \l(\E\l(e^{i\textsf{vec}(X_{\mathcal B_i}B)^\intercal s}\r)\E\l(e^{i\textsf{vec}({\tilde G}_{\mathcal B_i}B)^\intercal \tilde s}\r) - \E\l(e^{i\textsf{vec}(G_{\mathcal B_i}B)^\intercal s}\r)\E\l(e^{i\textsf{vec}({\tilde G}_{\mathcal B_i}B)^\intercal \tilde s}\r)\r)^2\nonumber\\
    &\siil 2\abs{\E\l(e^{i\textsf{vec}(X_{\mathcal B_i}B)^\intercal s}\r) - \E\l(e^{i\textsf{vec}(G_{\mathcal B_i}B)^\intercal s}\r)},\label{eq:spiritfarer}
\end{align}
where $(i)$ is because the characteristic function factors due to $\tilde{\bG} \indep (\bX, \bG)$, and $(ii)$ is because the characteristic function always has modulus in $[0,1]$, and $(x-y)^2 \leq 2|x-y|$ for $x, y \in [0,1]$. From here, let us now decompose our vector $s\in \R^{k\ell}$ into $k$ subvectors by defining 
\begin{align*}
     s_r := s_{r(\ell-1)+1:r\ell} = \l(s_{r(\ell-1)+1}, \hdots, s_{r\ell}\r) \in \R^\ell
\end{align*}
for each $r = 1, \hdots, k$. Then, we expand 
\begin{align}
    \textsf{vec}(X_{\mathcal B_i}B)^\intercal s &= \sum_{r=1}^{k}(X_{r+i-1}^\intercal\beta_1, \hdots, X_{r+i-1}^\intercal\beta_\ell)^\intercal s_r\nonumber\\
    &= \sum_{r=1}^{k}X_{r+i-1}^\intercal(B s_r)\nonumber\\
    &= \sum_{r=1}^{k}X_{r+i-1}^\intercal\sum_{t=1}^\ell s_{rt}\beta_t\nonumber\\
    &= \sum_{r=1}^{k}\|s_r\|_1X_{r+i-1}^\intercal\sum_{t=1}^\ell \frac{|s_{rt}|}{\|s_r\|_1}\text{sgn}(s_{rt})\beta_t\nonumber\\
    &=: \sum_{r=1}^{k}\|s_r\|_1X_{r+i-1}^\intercal \nu_r,\label{eq:pingyou}
\end{align}
where each $\nu_r :=  \sum_{t=1}^\ell \frac{|s_{rt}|}{\|s_r\|_1}\text{sgn}(s_{rt})\beta_t$ is in $\mathcal S_p$, since if we define 
\begin{align*}
    (\beta_{\rm{mix}})_{rt} := \text{sgn}(s_{rt})\beta_t \in \mathcal S_p,\ \ \ \ \lambda_{rt} := \frac{|s_{rt}|}{\|s_r\|_1} \in [0,1],
\end{align*}
then we know that since $\mathcal S_p$ is symmetric and convex and $\sum_t\lambda_{rt} = 1$, and since $\mathcal{S}_p$ contains the convex closure of $\tilde S$,  it must be that
\begin{align*}
    \nu_r = \sum_{t=1}^\ell \frac{|s_{rt}|}{\|s_r\|_1}\text{sgn}(s_{rt})\beta_t = \lambda_{r1}\tilde\beta_{r1} + \hdots + \lambda_{r\ell}\tilde\beta_{r\ell} \in \mathcal S_p.
\end{align*}
Thus we conclude that 
\begin{align*}
\sup_{B}\abs{\psi_{\mathcal M\bm{u}}(t) - \psi_{\mathcal M\bm{v}}(t)}^2 &\sil 2\sup_{B}\abs{\E\l(e^{i\sum_r\|s_r\|_1X_{r+i-1}^\intercal\beta_r}\r) - \E\l(e^{i\sum_r\|s_r\|_1G_{r+i-1}^\intercal\beta_r}\r)}\\
&\leq 2\sup_{B}\abs{\E\l(e^{{i}{\sqrt{\sum_{r}\|s_r\|_1^2}}\sum_r\frac{\|s_r\|_1X_{r+i-1}^\intercal\beta_r}{{\sqrt{\sum_{r}\|s_r\|_1^2}}}}\r) - \E\l(e^{{i}{\sqrt{\sum_{r}\|s_r\|_1^2}}\sum_r\frac{\|s_r\|_1G_{r+i-1}^\intercal\beta_r}{{\sqrt{\sum_{r}\|s_r\|_1^2}}}}\r)}\\
&\leq 2\sup_{\theta\in \mathcal{S}^{k-1}}\sup_{B}\abs{\E\l(e^{{i}{\sqrt{\sum_{r}\|s_r\|_1^2}}\sum_r \theta_rX_{r+i-1}^\intercal\beta_r}\r) - \E\l(e^{{i}{\sqrt{\sum_{r}\|s_r\|_1^2}}\sum_r \theta_rG_{r+i-1}^\intercal\beta_r}\r)}\\
&\siil 2\sqrt{\sum_{r}\|s_r\|_1^2}\underbrace{\sup_{\substack{f\in \mathcal{F}\\\theta\in \mathcal{S}^{k-1}}}\sup_{\substack{(\beta_1, \hdots, \beta_k) \\ \in\mathcal S_p^k}}\abs{\E f\l(\sum_{j=1}^k\theta_jX_{j+k-1}^\intercal\beta_j\r) - \E f\l(\sum_{j=1}^k\theta_jG_{j+k-1}^\intercal\beta_j\r) }}_{\textsf = K}.
\end{align*}
Here, $(i)$ is via \eqref{eq:spiritfarer} and \eqref{eq:pingyou}, and $(ii)$ is because the map $x\mapsto c\inv e^{icx} $ is 1-Lipschitz and bounded by $c\inv$.
Now let $w_{r,r'}$ be the $r(l-1)+r'$-th coordinate of $w$, where $1 \leq r \leq 2k$ and $1 \leq r' \leq l$. 
Plugging the above bound into \eqref{eq:trucated:Fourier:diff}, we obtain 
\begin{align*}
    (\star) 
    \;\leq&\;
    \| g \|_\infty  
    \, 
    \sqrt{2 K}
    \,
    \msum_{w \in \{-\frac{N}{R}, \ldots, \frac{N}{R}\}^{2kl} }  
    \,
    \Big( \msum_{r=1}^k \Big( \msum_{r'=1}^l  |\pi w_{r,r'} |  \Big)^2 \Big)^{1/4}
    \,e^{ - \frac{\pi^2 \sigma^2 \| w \|^2}{2}}
    \\
    \;\leq&\;
    \| g \|_\infty  
    \, 
    \sqrt{2 K}
    \,
    (k l^2)^{1/4} 
    \,
    \mfrac{\pi N}{R}
    \,
    \msum_{w \in \{-\frac{N}{R}, \ldots, \frac{N}{R}\}^{2kl} }  
    \,
    e^{ - \frac{\pi^2 \sigma^2 \| w \|^2}{2}}
    \\
    \;=&\;
    \| g \|_\infty  
    \, 
    \sqrt{2 K}
    \,
    (k l^2)^{1/4} 
    \,
    \mfrac{\pi N}{R}
    \,
    \Big(
        1
        +
        2 R
        \msum_{w' \in \{1, \frac{1}{R}, \ldots, \frac{N}{R}\} }  
        \,
        \mfrac{1}{R} 
        \,
        e^{ - \frac{\pi^2 \sigma^2 (w')^2}{2}}
    \Big)^{2kl}
    \\
    \;\leq&\;
    \| g \|_\infty  
    \, 
    \sqrt{2 K}
    \,
    (k l^2)^{1/4} 
    \,
    \mfrac{\pi N}{R}
    \,
    \Big(
        1
        +
        2 R
        \mint_0^\infty
        \,
        e^{ - \frac{\pi^2 \sigma^2 (w')^2}{2}}
        \,
        dw'
    \Big)^{2kl}
    \\
    \;=&\;
    \| g \|_\infty  
    \, 
    \sqrt{2 K}
    \,
    (k l^2)^{1/4} 
    \,
    \mfrac{\pi N}{R}
    \,
    \Big(
        1
        +
        \mfrac{R \sqrt{2}}{{\sigma}}
    \Big)^{2kl}
    \;.
\end{align*}
Combining this bound with \eqref{eq:sigma2kl} and \eqref{eq:smoothed:Fourier:halfway} by the triangle inequality, we get 
\begin{align*}
   &\;\big| 
        \mean\big[ 
            g( \cM \bu  )
            -
            g(  \cM \bv  )
        \big] 
    \big|
    \\
    &\;\leq\;
    \| g \|_{\text{Lip}} \sigma \sqrt{2 kl}
    +
    \mfrac{( 2 + 4kl) \, \| g \|_\infty \, \mean[ \| \cM \bu + \bZ \|_\infty + \| \cM \bv + \bZ \|_\infty]  }{R}
    +
    \mfrac{2 \| g \|_{\text{Lip}} C' \log N}{N}
    \\
    &\qquad 
    \;+
    \| g \|_\infty  
    \, 
    \sqrt{2 K}
    \,
    (k l^2)^{1/4} 
    \,
    \mfrac{\pi N}{R}
    \,
    \Big(
        1
        +
        \mfrac{R \sqrt{2}}{{\sigma}}
    \Big)^{2kl}
    \\
    &\;\leq\;
    \| g \|_{\text{Lip}} \sigma \sqrt{2 kl}
    +
    \mfrac{( 2 + 4kl) \, \| g \|_\infty \, \big( \mean[ \| \cM \bu  \|_\infty + \| \cM \bv \|_\infty]  + 2kl \sigma \big) }{R}
    +
    \mfrac{2 \| g \|_{\text{Lip}} C' \log N}{N}
    \\
    &\qquad 
    \;+
    \| g \|_\infty  
    \, 
    \sqrt{2 K}
    \,
    (k l^2)^{1/4} 
    \,
    \mfrac{\pi N}{R}
    \,
    \Big(
        1
        +
        \mfrac{R \sqrt{2}}{{\sigma}}
    \Big)^{2kl}
    \;,
\end{align*}
where we have used the Markov inequality and a union bound. Now choose
\begin{align*}
    \sigma \;=\; K^{\frac{1}{8 kl - 4} }
    \;,
    \qquad 
    N \;=\; R \;=\; \lfloor K^{-\frac{1}{8 kl } } \rfloor
    \;.
\end{align*}
We get that for some absolute constant $C > 0$,
\begin{align*}
    &\;\big| 
    \mean\big[ 
        g( \cM \bu  )
        -
        g(  \cM \bv  )
    \big] 
    \big|
    \\
    &\;\leq\;
    C
    \Big(
    \| g \|_{\rm Lip} \, \big(\sqrt{2 kl} \,  + 1\big)
    + 
    \| g \|_\infty
    (1+2kl)
    \big(  \mean[ \| \cM \bu \|_\infty + \| \cM \bv \|_\infty ]
    +
    kl \big)
    +
    \| g \|_\infty
    \, 
    3^{2kl}
    \,
    (k l^2)^{1/4}   \Big)\, K^{\frac{1}{8kl-4}} 
    \\ 
    &\;\leq\; 
    \textnormal{\textsf C}_{k,l} \,
    \Big(
    \| g \|_{\rm Lip} 
    + 
    \| g \|_\infty
    \,
    \big( 
    \mean[ \| X_{\mathcal B_i}B \|_\infty + \| G_{\mathcal B_i}B \|_\infty + 2 \| \tilde G_{\mathcal B_i}B \|_\infty  ]
    + 
    1 \big) 
    \Big)\, K^{\frac{1}{8kl-4}} 
\end{align*}
for some $\textnormal{\textsf C}_{k,l} > 0$. We also remark that by sub-Gaussianity, the term $\mean[ \| X_{\mathcal B_i}B \|_\infty + \| G_{\mathcal B_i}B \|_\infty + 2 \| \tilde G_{\mathcal B_i}B \|_\infty  ]$ is $O(kl)$.

To prove the second statement of the lemma, we will sketch the outline and refer to Lemma 30 of \cite{montanari2022universality} for the specific details of a similar approach. We fix $B > 0$, and consider $g_B$, which forces $g$ to be bounded Lipschitz like so:
\begin{align*}
    g_B(\bm x) := g\l(\bm x\r)\mathbb I\l(\|\bm x\| \leq B\r) + g\l(B\bm x / \|\bm x\|\r)\mathbb I\l(\|\bm x\| > B\r).
\end{align*}
We may then apply the first statement of the lemma to $g_B$, which under \Cref{gauss_approx} converges to zero. To bound the leftover differences of the form
\begin{align*}
    \sup_B\abs{\E\l[g(\bm M\bm u) - g_B(\bm M\bm u)\r]},
\end{align*}
we use square-integrability of $g$ and the fact that
\begin{align*}
    g(\bm M\bm u) - g_B(\bm M\bm u) \neq 0 \implies \|\bm M\bm u\| > B,
\end{align*}
which by sub-Gaussianity of all $2k\ell$ components of $\bm M\bm u$ occurs with probability bounded by \\$\textsf C_1k\ell\exp\l(-\textsf c_2B^2/k\ell\r)$. Sending $B\to\infty$ thus concludes the result.
\end{proof}

 \section{Comparing training loss with deleted blocks to the original training loss}\label{equivalence_sec}


Let $M ,\tilde m\in \Z^+$ be fixed. Define new matrices $\bX^M, \bG^M \in \R^{n' \times p}$ as
\begin{align*}
    \bX^M &:= (X_1, \hdots, X_M, X_{M+\tilde m+1}, \hdots, X_{2M+\tilde m}, X_{2M+2\tilde m+1}, \hdots)^\intercal\\
    \bG^M &:= (G_1, \hdots, G_M, G_{M+\tilde m+1}, \hdots, G_{2M+\tilde m}, G_{2M+2\tilde m+1}, \hdots)^\intercal.
\end{align*} We show that the training loss
$\min_\beta \hat R_n(\beta; \bX)$ and $\min_\beta \hat R_{n'}(\beta; \bX^M)$ are close if $M$ is large.
\begin{theorem}[Equivalence of losses]\label{equivalence} Let $\l(X_i, y_i(X_i)\r)_{i=1}^n$ and $\l(G_i, y_i(G_i)\r)_{i=1}^n$ be generated under Assumptions \ref{logmodel}-\ref{signalsizinghehe}, where each $G_i \sim \mathcal N\l(\bm0, \Var(X_i)\r)$. Then if \cref{general}(i) is respected then there exists a constant $\textsf C_d$ such that 
\begin{align}
    d_{\mathcal H}\l(\min_\beta \hat R_n(\beta; \bX), \min_\beta \hat R_{n'}(\beta; \bX^M)\r)\le 
    &
\textsf{C}_d\frac{\max(m,\tilde m)}{M}\sqrt{M+\tilde m }\\
    d_{\mathcal H}\l(\min_\beta \hat R_n(\beta; \bG), \min_\beta \hat R_{n'}(\beta; \bG^M)\r)\le 
    &\textsf{C}_d\frac{\max(m,\tilde m)}{M}\sqrt{M+\tilde m }
\end{align} If instead \cref{general}(ii) is respected then there exists a constant $\textsf C'_d$ such that \begin{align}
    d_{\mathcal H}\l(\min_\beta \hat R_n(\beta; \bX), \min_\beta \hat R_{n'}(\beta; \bX^M)\r)\le 
    &
\textsf{C}'_d\frac{\max(\mathcal{S},\tilde m)}{M}\sqrt{M+\tilde m }\\
    d_{\mathcal H}\l(\min_\beta \hat R_n(\beta; \bG), \min_\beta \hat R_{n'}(\beta; \bG^M)\r)\le 
    &\textsf{C}'_d\frac{\max(\mathcal{S},\tilde m)}{M}\sqrt{M+\tilde m }
\end{align}
\end{theorem}

\begin{proof}[Proof of \Cref{equivalence} ]We present the proof in the $m$-dependent case but the proof is identical in the other case. Similarly we show the proof only for $\bX$ as the proof is equivalent for $\bG$. 
Let $M \in \Z^+$ be fixed. Define new matrices $\bX^M \in \R^{n' \times p}$ as
\begin{align*}
    \bX^M &:= (X_1, \hdots, X_M, X_{M+\tilde m+1}, \hdots, X_{2M+\tilde m}, X_{2M+2m+1}, \hdots)^\intercal\\
\end{align*}
noting that 
\begin{align}
n' \in [n-(r+1)m+1,\ n - rm] \subset \l[n\frac{M}{M+\tilde m} - m,\ n\frac{M}{M+\tilde m} + m\r] = [nq -\tilde m, nq +\tilde m], \label{eq:nprimesecond}
\end{align}
where $r := \lfloor\frac{n}{M+\tilde m} \rfloor$ and $q := \frac{M}{M+\tilde m}$.  We may also define
\begin{align*}
    \bX^{\tilde m} &:=(X_{M+1},\dots, X_{M+\tilde m},X_{2M+\tilde m+1},\dots,X_{2M+2\tilde m},\dots)\\
\end{align*}
so that every vector $X_i$ is either in $\bX^M$ or $\bX^m$. For simplicity we can also write these indexing sets as 
\begin{align*}
    B_M &:= \{1, \hdots, M, M+\tilde m+1, \hdots, 2M + m, \hdots, \}\\
    B_m &:= [n] \setminus B_M.
\end{align*}
By a triangle inequality argument we note that
\begin{align}
    d_{\mathcal H}\l(\min_\beta \hat R_n(\beta; \bX), \min_\beta \hat R_{n'}(\beta; \bX^M)\r) &= \sup_{h\in\mathcal H}\abs{\E\l[h\l(\min_\beta \hat R_n(\beta; \bX)\r) - h\l(\min_\beta \hat R_{n'}(\beta; \bX^M)\r)\r]} \nonumber\\
    &\leq \sup_{h\in\mathcal H}\E\abs{h\l(\min_\beta \hat R_n(\beta; \bX)\r) - h\l(\min_\beta \hat R_{n'}(\beta; \bX^M)\r)}\nonumber\\
    &\stackrel{(i)}{\leq} \E\abs{\min_\beta \hat R_n(\beta; \bX) - \min_\beta \hat R_{n'}(\beta; \bX^M)}\nonumber\\&\stackrel{(ii)}{=:} \E\abs{\hat R_n(\hat\beta; \bX) - \hat R_{n'}(\tilde\beta; \bX^M)}
   \label{eq:imdying},
\end{align}
where $(i)$ is via the Lipschitzness of $h$, and in $(ii)$ we defined the minimizers $\hat\beta:=\text{argmin}_\beta\hat R_{n}(\beta; \bX)$ and  $\tilde\beta:=\text{argmin}_\beta\hat R_{n'}(\beta; \bX^M)$. We first control the difference inside the absolute value of \eqref{eq:imdying} by noting that, by definition of being minimizers, 
\begin{align*}
    \hat R_n(\hat\beta; \bX) - \hat R_{n'}(\tilde\beta; \bX^M) &\leq \l(1 - \frac{n}{n'}\r)\hat R_n(\hat\beta) + \frac{1}{n'}\sum_{i\in B_m}\ell\l(\eta_i, X_i^\intercal\tilde\beta\r) \leq \frac{1}{n'}\sum_{i\in B_m}\ell\l(\eta_i, X_i^\intercal\tilde\beta\r)\\
    \hat R_{n'}(\tilde\beta; \bX^M) - \hat R_n(\hat\beta; \bX) &\leq \l(1 - \frac{n'}{n}\r)\hat R_{n'}(\tilde\beta) - \frac{1}{n}\sum_{i\in B_m}\ell\l(\eta_i, X_i^\intercal\hat\beta\r) \leq \l(1 - \frac{n'}{n}\r)\hat R_{n'}(\tilde\beta),
\end{align*}
where the terms removed are because of the fact that the risk/loss is always positive and $n > n'$. Now, we will use the fact that for any $x \in \R$ and $a, b > 0$, we have  
\begin{align*}
    x \leq a,\ -x \leq b \implies |x| \leq a \vee b \leq a + b
\end{align*}
to say that 
\begin{align}
    \abs{\hat R_n(\hat\beta; \bX) - \hat R_{n'}(\tilde\beta; \bX^M)} \leq \l(1 - \frac{n'}{n}\r)\hat R_{n'}(\tilde\beta) + \frac{1}{n'}\sum_{i\in B_m}\ell\l(\eta_i, X_i^\intercal\tilde\beta\r).\label{eq:friday}
\end{align}
For the first term on the right in \eqref{eq:friday}, we note by definition of being a minimizer that 
\begin{align}
    \l(1 - \frac{n'}{n}\r)\hat R_{n'}(\tilde\beta) \leq \l(1 - \frac{n'}{n}\r)\hat R_{n'}(0) = \l(1 - \frac{n'}{n}\r)\log(2).\label{eq:yomomma}
\end{align}
For the second term in \eqref{eq:friday}, we have that 
\begin{align}
    \frac{1}{n'}\sum_{i\in B_m}\ell\l(\eta_i, X_i^\intercal\tilde\beta\r) &\sil \frac{1}{n'}\sum_{i\in B_m}1 + |X_i^\intercal\tilde\beta|\nonumber\\
    &\leq \frac{n-n'}{n'} + \frac{1}{n'}\sup_{\beta}\sum_{i\in B_m} |X_i^\intercal\beta|\nonumber\\
    &= \frac{n-n'}{n'} + \frac{1}{n'}\sup_{\beta}\sum_{i=1}^n |X_i^\intercal\beta|\mathbb I(i \in B_m)\nonumber\\
    &\siil \frac{n-n'}{n'} + \frac{1}{n'}\sup_{\beta}\sqrt{\sum_{i=1}^n |X_i^\intercal\beta|^2}\sqrt{\sum_{i=1}^n\mathbb I(i \in B_m)}\nonumber\\
    &= \frac{n-n'}{n'} + \frac{\sqrt{n-n'}}{n'}\|\bX\|_{\mathcal S_p},\label{eq:novernprime}
\end{align}
where $(i)$ is because
\begin{align*}
    \log(1 + e^{-x}) \leq |x| + 1 \text{ for all } x \in \R,
\end{align*}
and $(ii)$ is via Cauchy-Schwarz. Combining \eqref{eq:yomomma} and \eqref{eq:novernprime}, we obtain 
\begin{align}
    \E\abs{\hat R_n(\hat\beta; \bX) - \hat R_{n'}(\tilde\beta; \bX^M)} &\leq \l(1 - \frac{n'}{n}\r)\log(2) + \frac{n-n'}{n'} + \frac{\sqrt{n-n'}}{n'}\E\l[\|\bX\|_{\mathcal S_p}\r]\\
    &\sil  \l(1 - \frac{n'}{n}\r)\log(2) + \frac{n-n'}{n'} + \frac{\sqrt{n-n'}}{n'}\sqrt{\textsf C_Lmp},\label{eq:almostdone}
\end{align}
where $(i)$ is via Jensen's Inequality with \corollaryref{dep_norm_bound}. From here we can use \eqref{eq:nprime} to note that, for these quantities involving both $n$ and $n'$, we have 
\begin{align}
    n' \in nq \pm \tilde m \implies \frac{n'}{n} \in q \pm \frac{\tilde m }{n} \implies \lim_{n\to\infty}\frac{n'}{n} = q.
\end{align}
Substituting this into \eqref{eq:almostdone} and using that $q = \frac{M}{M+\tilde m }$, after some simplification we can conclude that
\begin{align}
    \limsup_{n\to\infty}\E\abs{\hat R_n(\hat\beta; \bX) - \hat R_{n'}(\tilde\beta; \bX^M)} &\leq (1-q)\log(2) + \l(\frac{1}{q} - 1\r) + \sqrt{\frac{\textsf{C}_Lm\kappa}{q}}\sqrt{\frac{1}{q} - 1}\\
    &\leq \textsf{C}_2\frac{\max(m,\tilde m)}{M}\sqrt{M+\tilde m }.
\end{align}
\end{proof}


\section{Verifying Assumption~\ref{gauss_approx} for Different Data Augmentation Schemes}\label{check:augmentation}
\subsection{Random Cropping}

The data augmentation procedure that we are considering in this subsection is the random cropping method where a portion of the data is randomly set to $0$. For a vector $e:=(e_i)\in \{0,1\}^p$ and $x\in \mathbb{R}^p$
we write $e\cdot x:=(e_ix_i)$. Let $(E_i)$ be an i.i.d sequence of random vectors in $\{0,1\}^p$ and define the random transformations $\phi_i(x)=E_i\cdot x$. We will prove that \Cref{gauss_approx} holds for this type of data augmentation procedure under general conditions. Let $(Z_i)$ a sequence of i.i.d vectors satisfying the following condition $$(H_{\text{cropping}}(\kappa))\qquad\mathbb{E}(Z_1)=0,~\sup_i\|Z_{1,i}\|_4<\kappa/\sqrt{n.}$$  Define $$X_i:=\phi_i(Z_{\lceil i/k\rceil}).$$

\begin{lemma}
    Suppose that the assumption $H_{\text{cropping}}(\kappa)$ holds and that the entries $(Z_{1,i})$ are locally dependent and write $N_i$ the dependency neighborhood of $Z_{1,i}$.  Suppose that $(E_{1,i})$ is locally dependent and write $\tilde N_i$ the local dependency neighborhood of $E_{1,i}$. Then if $|N_i|\times|\tilde N_i|=o(n^{r/2})$ \Cref{gauss_approx} holds for a sequence of random Gaussian vector $(G_i)$ with covariance: $$\text{cov}(G_{i,j},G_{m,l})=\begin{cases}   p_{l}(\mathbb{I}(j=l)+\mathbb{I}(j\ne l)p_j) \Var(Z_{1,j},Z_{1,l})~\text{if}~  i= m\\
     p_{j} p_l \text{cov}(Z_{1,j},Z_{1,l})~\text{if}~  |i-m|\le k,~i\ne m\\0~\text{otherwise}\end{cases}$$ where $p_{j}=P(E_{1,j}=1).$
\end{lemma}
\begin{proof}
Note that as the blocks $(X_{mk+1},\dots,X_{(m+1)k})$ are identically distributed it is enough to prove \Cref{gauss_approx} for $m=0$. Denote $B_{i,j}=\cup_{j\in N_i}\tilde N_j\times[|1,k|] $. Then we remark that the sequence $(X_i)$ is locally dependent and that the dependency neighborhood of $X_{i,j}$ is $B_{i,j}.$ The desired result follows from \lemmaref{dalia} with $q=1$.
\end{proof}
\subsection{Noise Injection}
The data augmentation procedure that we are considering in this subsection is the noise injection method where random Gaussian noise is injected to the entries. For vectors $g:=(g_i),x:=(x_i)\in \mathbb{R}^p$
we write $g\cdot x:=(g_i+x_i)$. Let $(n_i)$ be an i.i.d sequence of random vectors in $\mathbb{R}^p$ such that $n_1\sim \mathcal{N}(0,\sigma^2/n Id)$ and define the random transformations $\phi_i(x)=n_i\cdot x$. We will prove that \Cref{gauss_approx} holds for this type of data augmentation procedure under general conditions.  Let $(Z_i)$ be a sequence of i.i.d vectors satisfying the following condition $$(H_{\text{noise}}(\kappa))~\mathbb{E}(Z_1)=0,\qquad\sup_i\|Z_{1,i}\|_4<\kappa/\sqrt{n.}$$ Define $X_i:=\phi_i( Z_{\lceil i/k\rceil}).$

\begin{lemma}
    Suppose that the assumption $H_{\text{noise}}(\kappa)$ holds for an absolute constant $\kappa<\infty$ and that the entries $(Z_{1,i})$ are locally dependent. Write $N_i$ the dependency neighborhood of $Z_{1,i}.$ Suppose that $(n_{1,i})\sim N(0,\frac{\sigma^2}{n}Id)$. Then if $|N_i|=o(n^{r/2})$, \Cref{gauss_approx} holds for a sequence of random Gaussian vector $(G_i)$ with covariance: $$\text{cov}(G_{i,j},G_{m,l})=\begin{cases}\text{cov}(Z_{1,j},Z_{1,l})+\delta_{j,l}\delta_{i,m}\sigma^2/n~\text{if}~  i=m\\
    \text{cov}(Z_{1,j},Z_{1,l})~\text{if}~  |i-m|\le k,i\ne m\\0~\text{otherwise}\end{cases}.$$ 
\end{lemma}
\begin{proof}
Note that as the blocks $(X_{mk+1},\dots,X_{(m+1)k})$ are identically distributed it is enough to prove \Cref{gauss_approx} for $m=0$. Denote $B_{i,j}=N_i\times[|1,k|] $. Then we remark that the sequence $(X_i)$ is locally dependent and that the dependency neighborhood of $X_{i,j}$ is $B_{i,j}.$ The desired result follows from  \lemmaref{dalia} with $q=1$.
\end{proof}
\subsection{Random Sign Flipping}
The data augmentation procedure that we are considering in this subsection is the random sign-flip method, where a portion of the data is has it sign randomly flipped. For a vector $e:=(e_i)\in \{-1,1\}^p$ and $x\in \mathbb{R}^p$
we write $e\cdot x:=(e_ix_i)$. Let $(E_i)$ be an i.i.d sequence of random vectors in $\{-1,1\}^p$ and define the random transformations $\phi_i(x)=E_i\cdot x$. We will prove that \Cref{gauss_approx} holds for this type of data augmentation procedure under general conditions. Let $(Z_i)$ a sequence of i.i.d vectors satisfying the following condition $$(H_{\text{flip}}(\kappa))\qquad\mathbb{E}(Z_1)=0,~\sup_i\|Z_{1,i}\|_4<\kappa/\sqrt{n.}$$  Define $$X_i:=\phi_i(Z_{\lceil i/k\rceil}).$$

\begin{lemma}
    Suppose that the assumption $H_{\text{flip}}(\kappa)$ holds and that the entries $(Z_{1,i})$ are locally dependent and write $N_i$ the dependency neighborhood of $Z_{1,i}$.  Suppose that $(E_{1,i})$ is locally dependent and write $\tilde N_i$ the local dependency neighborhood of $E_{1,i}$. Then if $|N_i|\times|\tilde N_i|=o(n^{r/2})$ \Cref{gauss_approx} holds for a sequence of random Gaussian vector $(G_i)$ with covariance: $$\text{cov}(G_{i,j},G_{m,l})=\begin{cases} p^*_{j,l} \E(Z_{1,j},Z_{1,l})-(1-p^*_{j,l}) \E(Z_{1,j},Z_{1,l})\quad \rm{if}~i=m\\
     p_{j,l} \E(Z_{1,j},Z_{1,l})-(1-p_{j,l}) \E(Z_{1,j},Z_{1,l})~\text{if}~  |i-m|\le k~i\ne m\\0~~\text{~otherwise}\end{cases}$$ where $p_{j,l}=P(E_{1,j}=1)^2+P(E_{1,j}=-1)^2$ and $p^*_{j,l}=P(E_{1,j}=1,E_{1,l}=1)+P(E_{1,j}=-1,E_{1,l}=-1).$
\end{lemma}
\begin{proof}
Note that as the blocks $(X_{mk+1},\dots,X_{(m+1)k})$ are identically distributed it is enough to prove \Cref{gauss_approx} for $m=0$. Denote $B_{i,j}=\cup_{j\in N_i}\tilde N_j\times[|1,k|] $. Then we remark that the sequence $(X_i)$ is locally dependent and that the dependency neighborhood of $X_{i,j}$ is $B_{i,j}.$ The desired result follows from \lemmaref{dalia} with $q=1$.
\end{proof}\begin{remark}
    Note that we do not assume that the probability of having a sign flipped at one position is the same than at any other positions
\end{remark}
\subsection{Random Small Permutations}In this section we will show that \Cref{gauss_approx} holds for a random permutation scheme. In this goal, let $(Z_i)$ be a sequence of centered i.i.d random vectors with independent (not necessarily identically distributed) entries. We assume that the vectors $(Z_i)$ have blocks of identically distributed entries. More precisely we suppose that there is a partition $(B_i)_{i\le M_n}$ of $[|p|]$ in $M_n$ subsets such that the entries $(Z_{1,i})_{i\in B_u}$ are i.i.d for all $u\le M_n$. We choose $(\pi_i)$ to be an i.i.d sequence of random permutations of $[|n|]$ that preserve the partition, meaning that for all $j,k\le n$ that do not belong to the same permutation element then $P(\pi_1(j)=k)=0$. Define $$X_i:= (Z_{\lceil i/k\rceil,\pi^{-1}_i(\ell)}).$$  We will show that this data augmentation scheme satisfies \Cref{gauss_approx}. In this goal we define the following condition $$(H_{\text{small~permutation}}(\kappa))\qquad\mathbb{E}(Z_1)=0,~\sup_i\|Z_{1,i}\|_4<\kappa/\sqrt{n.}$$ 

\begin{lemma}
    Suppose that the assumption $H_{\text{small~permutation}}(\kappa)$ holds. Suppose  that  $\max_i|B_i|=o(n^{r/2})$ \Cref{gauss_approx} holds for a sequence of random Gaussian vector $(G_i)$ with covariance: $$\text{cov}(G_{i,j},G_{m,l})=\begin{cases}
    \sum_{k\in B_{B^{-1}(j)}}p_k^2\rm{var}(Z_{1,k})~\text{if}~  |i-m|\le k\\0~\text{otherwise}\end{cases}.$$ where $B^{-1}(j)$ denotes the unique index such that $j\in B_{B^{-1}(j)}$ and where $p_k:=P(\pi(j)=k)$
\end{lemma}
\begin{proof}
Note that as the blocks $(X_{mk+1},\dots,X_{(m+1)k})$ are identically distributed it is enough to prove \Cref{gauss_approx} for $m=0$. Denote $B_{i,j}=B_\ell\times[|1,k|] $ if $i\in P_{\ell}$. Then we remark that the sequence $(X_i)$ is locally dependent and that the dependency neighborhood of $X_{i,j}$ is $B_{i,j}.$ The desired result follows from \lemmaref{dalia} with $q=1$.
\end{proof}
\subsection{Random Large Permutations}
In this section we will show that \Cref{gauss_approx} holds for a random permutation scheme. In this goal, let $(Z_i)$ be a sequence of centered i.i.d random vectors with independent (not necessarily identically distributed) entries. We assume that the vectors $(Z_i)$ have blocks of identically distributed entries. More precisely we suppose that there is a partition $(B_i)_{i\le M}$ of $[|p|]$ in $M$ subsets such that the entries $(Z_{1,i})_{i\in B_u}$ are i.i.d for all $u\le M$. Contrary to the previous section we will have $M<<n$. We choose $(\pi_i)$ to be an i.i.d sequence of random permutations of $[|n|]$ that preserve the partition, meaning that for all $j,k\le n$ that do not belong to the same permutation element then $P(\pi_1(j)=k)=0$. Moreover we assume that for all $j,k\in B_u$ in the same partition we have $P(\pi_1(j)=k)=1/|B_u|.$ Define $$X_i:= (Z_{\lceil i/k\rceil,\pi^{-1}_i(\ell)}).$$ 
\begin{lemma}
     Suppose that there exists an absolute constant $\kappa<\infty$ such that $\max_i\|Z_{1,i}\|_3\le \kappa/\sqrt{n}$ and $\max|B_u|/\min|B_u|<\infty$ \Cref{gauss_approx} holds for a sequence of random Gaussian vectors that are such that $$\text{cov}(G_{i,j},G_{m,l})=\sum_u \frac{\sigma_u^2\mathbb{I}(|i-m|\le k)}{|B_u|^2}$$ where $\sigma_u^2=\Var(Y_{1,l})$ where $l\le p$ is chosen to be in $l\in B_u.$
\end{lemma}
\begin{proof}
Note that as the blocks $(X_{mk+1},\dots,X_{(m+1)k})$ are identically distributed it is enough to prove \Cref{gauss_approx} for $m=0$.    For all $(\theta_i)\in \mathcal{S}^{k-1}$ and all $(\beta_i)\in \mathcal{S}_p^k$ we have \begin{align*}
        \sum_{i\le k}\theta_iX_i^T\beta_i=\sum_{i\le k} \theta_i \sum_l X_{i,l}\beta_{i,l}=\sum_l Z_{1,l}\sum_{i\le k} \theta_i \beta_{i,\pi_i(l)}.
    \end{align*} Now we notice that conditionally on $(\pi_i)$ the random variables $\big(Z_{1,l}\sum_{i\le k} \theta_i \beta_{i,\pi_i(l)}\big)$ are independent. Moreover we observe that \begin{align*}
        \mathbb{E}\big( Z_{1,l}\sum_{i\le k} \theta_i \beta_{i,\pi_i(l)}\big|(\pi_i)\big)&=0
    \end{align*} and for all $l,m\le p$ we have  \begin{align*}
        \text{cov}\big( Z_{1,l}\sum_{i\le k} \theta_i \beta_{i,\pi_i(l)}, Z_{1,m}\sum_{i\le k} \theta_i \beta_{i,\pi_i(l)}\big|\pi\big)&=\delta_{l,m}\big(\sum_{i\le k} \theta_i \beta_{i,\pi_i(l)}\big)^2\Var(Z_{1,l})\\ \mathbb{E}\big(\big| Z_{1,l}\sum_{i\le k} \theta_i \beta_{i,\pi_i(l)}\big|^3\big|\pi\big)&\le \sqrt{k}^3\max_i\|Z_{1,i}\|^3_3\sqrt{\sum_{i\le k} \beta_{i\pi_i(l)}^2}^3
        \\&\le k^4Lp^{3/2-3r/2}\max_i\|Z_{1,i}\|^3_3
    \end{align*} where for the last inequality we used \Cref{ts}. Define $$\hat\sigma^2:=\sum_{\ell\le p}\Var(Z_{1,\ell})\big(\sum_{i\le k}\theta_i \beta_{i,\pi_i(\ell)}\big)^2=\sum_u \sigma_u^2\sum_{\ell\in B_u}\big(\sum_{i\le k}\theta_i \beta_{i,\pi_i(\ell)}\big)^2.$$
According to \lemmaref{dalia2} we have that there exists an absolute constant $\kappa>0$ such that \begin{align}\sup_{f~\in \mathcal{F}}\Big|\mathbb{E}\Big(f\big(\sum_{i\le k}\theta_iX_i^T\beta_i\big)\big|(\pi_i)\Big)-\mathbb{E}_{Z\sim \mathcal{N}(0,\hat\sigma^2\big)}\Big(f(Z)|(\pi_i)\Big)\Big|\le \kappa^{1/3}p^{1/2-r/2}\max_i\|Z_{1,i}\|_3.
\end{align}  Now using the definition of $\hat\sigma^2$ we observe that \begin{align}
\mathbb{E}\big(\hat\sigma^2\big)&=\sum_u\sigma_u^2\frac{1}{|B_u|^k}\sum_{\ell_1,\dots,\ell_k\in B_u}\big(\sum_{i\le k}\theta_i \beta_{i,\ell_i}\big)^2
\end{align}
To show that $\hat\sigma^2$ converges to $\mathbb{E}(\hat\sigma^2)$ we will proceed by showing that for all $S\le k$ we have $\mathbb{E}(\hat\sigma^2|\pi_1,\dots,\pi_{S})$ concentrates around $\mathbb{E}(\hat\sigma^2|\pi_1,\dots,\pi_{S-1})$. The desired outcome then results from a telescopic sum argument.
In this goal we first notice that \begin{align*}&
    \mathbb{E}(\hat\sigma^2|\pi_1,\dots,\pi_{S})-\mathbb{E}(\hat\sigma^2|\pi_1,\dots,\pi_{S-1})\\&=\sum_u \sigma_u^2\sum_{\ell\in B_u}\sum_{i\ne S}\theta_i\theta_S \mathbb{E}\Big(\beta_{i,\pi_i(\ell)}\beta_{S,\pi_S(\ell)}|\pi_1,\dots,\pi_{S}\Big)-\mathbb{E}\Big(\beta_{i,\pi_i(\ell)}\beta_{S,\pi_S(\ell)}|\pi_1,\dots,\pi_{S-1}\Big)
\end{align*}
To bound this set $I\sim\text{unif}([|n|])$ and $J\sim \text{unif}(B_{B^{-1}(I)})$ where we denote $B^{-1}(i)$ the index $u$ such that $i\in B_u$. Define $\pi'_S=\pi_S \circ (I,J)$ and $\pi'_i=\pi_i$ for all $i\ne S$. Define $$(\hat\sigma')^2:=\sum_{\ell\le p}\Var(Y_{1,l})\big(\sum_{i\le k}\theta_i \beta_{i,\pi'_i(\ell)})^2$$ then we notice that $(\hat\sigma,\hat\sigma')$ forms an exchangeable pair and that \begin{align*}&\mathbb{E}\big(\hat\sigma^2|\pi_{1:S})-\mathbb{E}\big((\hat\sigma')^2|\pi_{1:S})\\&=\sum_u \sigma_u^4\sum_{\ell\in B_u}\sum_{i\ne S}\theta_i\theta_S \mathbb{E}\Big(\beta_{i,\pi_i(\ell)}\beta_{S,\pi_S(\ell)}|\pi_1,\dots,\pi_{S}\Big)-\mathbb{E}\Big(\beta_{i,\pi_i(\ell)}\beta_{S,\pi_S(\ell)}|\pi_1,\dots,\pi_{S-1}\Big).\end{align*} Moreover we observe that 
\begin{align*}
    \hat\sigma^2-(\hat\sigma')^2&\le \sigma^2_{B^{-1}(I)}\sum_{i\ne S}\theta_i\theta_S(\beta_{i,\pi_i(I)}\beta_{S,\pi_S(I)}-\beta_{i,\pi_i(I)}\beta_{S,\pi_S(J)})\\&\qquad+\sigma^2_{B^{-1}(I)}\sum_{i\ne S}\theta_i\theta_S(\beta_{i,\pi_i(J)}\beta_{S,\pi_S(J)}-\beta_{i,\pi_i(J)}\beta_{S,\pi_S(I)})
\end{align*}Hence using 
\cite{chatterjee2005concentration} we obtain that\begin{align*}    \Var\Big(\mathbb{E}\big(\hat\sigma^2\big|\pi_1,\dots,\pi_{S}\big)\Big|\pi_1,\dots,\pi_{S-1}\Big)&\le \mathbb{E}\Big([\hat\sigma^2-(\hat\sigma')^2]^2\Big)\\&\le4\frac{\max_u\sigma^2_u}{p\min_u|B_u|}\sum_{\ell\le p}\sum_{k\in B_{B^{-1}
    (\ell)}}\big(\sum_{i\le k}\theta_i\theta_S\beta_{i,\pi_i(\ell)}^2(\beta_{S,\pi_S(\ell)}-\beta_{S,\pi_S(k)})\big)^2\\&\le8k\frac{\max_u\sigma^4_u}{p\min_u|B_u|}\sum_\ell\sum_{k\in B_{B^{-1}
    (\ell)}}\theta_S^2\sum_{i\le k}\theta_i^2\beta_{i,\pi_i(\ell)}^2(\beta_{S,\pi_S(\ell)}^2+\beta_{S,\pi_S(k)}^2)
    \\&\overset{(a)}{\le} 8k\frac{\max_u\sigma^4_uL}{p^{r}\min_u|B_u|}\sum_\ell\sum_{k\in B_{B^{-1}
    (\ell)}}(\beta_{S,\pi_S(\ell)}^2+\beta_{S,\pi_S(k)}^2)    \\&\overset{}{\le} 16k\frac{\max_u\sigma^4_uL\max_u|B_u|}{p^{r}\min_u|B_u|}\sum_\ell\beta_{S,\ell}^2\\&\overset{}{\le} 16k\frac{\max_u\sigma^4_uS^2\max_u|B_u|p^{1-r}}{\min_u|B_u|}
\end{align*} where (a) is a result of \Cref{ts}. Hence as every function $f\in \mathcal{F}$ is Lipschitz we obtain that \begin{align*}
    \sup_{f~\in \mathcal{F}}\Big|\mathbb{E}\Big(f\big(\sum_{i\le k}\theta_iX_i^T\beta_i\big)\Big)-\mathbb{E}_{Z\sim \mathcal{N}(0,\sigma^2\big)}\Big(f(Z)\Big)\Big|\le \kappa p^{1/2-r/2}\max_i\|Z_{1,i}\|_3+4\sqrt{k\frac{\max_u\sigma^4_uL^2\max_u|B_u|p^{1-r}}{\min_u|B_u|}}
\end{align*}

Finally we note that $\sum_{i\le k}\theta_i G_i^T\beta_i\sim N(0,\sigma^2)$ and the required result is hence deduced.\end{proof}
 
\section{Auxiliary Lemmas} \label{appendix:auxiliary:lemma}
\subsection{m-dependent Bernstein}
The first lemma aims to extend the classic Bernstein's Inequality from the usual independent setting to $m$-dependence, which of course also then includes block dependence.

\begin{lemma}[$m$-dependent Bernstein]\label{bern_dep} Let $Z_1, \hdots, Z_n$ be centered, sub-exponential, $m$-dependent random variables, and $\textnormal{\textsf K} := \sup_i\|Z_i\|_{\psi_1}$. Then there exists a universal constant $c>0$ such that for $t \geq 0$, 
\begin{align*}
    \P\l(\abs{\frac{1}{n}\sumi{n}Z_i} \geq t\r) \leq 4\cdot\exp\l[-\frac{cn}{m}\l(\frac{t}{\textnormal{\textsf K}} \wedge \frac{t^2}{\textnormal{\textsf{K}}^2}\r)\r].
\end{align*}
\end{lemma}

\begin{proof}
Define $r := \lfloor n/m\rfloor$, and assume without loss of generality that $r$ is even. For each $j = 1, \hdots, r$, let 
\begin{align*}
    Y_j &:= \hspace{-8pt}\sum_{\ell=(j-1)m+1}^{jm}\hspace{-8pt}Z_\ell.
\end{align*}
Also, define $Y_{r+1}:=\sum_{\ell=rm+1}^{n}\hspace{-1pt}Z_\ell$. By construction, the set $(Y_1, Y_3, \hdots, Y_{r+1})$ is independent, as is the set $(Y_2, Y_4, \hdots, Y_r)$. We can then see that 
\begin{align}
    \P\l(\abs{\frac{1}{n}\sumi{n}Z_i} \geq t\r) &= \P\l(\abs{\frac{1}{n}\sum_{j=1}^{r+1}Y_j} \geq t\r)\nonumber\\
    &\leq \P\l(\abs{\frac1n\sum_{j=0}^{r/2}Y_{2j+1}} \geq \frac t2\r) + \P\l(\abs{\frac1n\sum_{j=1}^{r/2}Y_{2j}} \geq \frac t2\r)\nonumber\\
    &\leq \P\l(\abs{\frac{1}{\frac{r}{2}+1}\sum_{j=0}^{r/2}Y_{2j+1}} \geq \frac{nt}{r+2}\r) + \P\l(\abs{\frac{1}{\frac r2}\sum_{j=1}^{r/2}Y_{2j}} \geq  \frac{nt}{r}\r)\label{eq:ihateequations}
\end{align} 
By the Triangle Inequality, we know that each $Y_j$ is still still sub-exponential with norm 
\begin{align*}
    \|Y_j\|_{\psi_1} \leq \sum_{\ell=(j-1)m+1}^{jm}\hspace{-8pt}\|Z_\ell\|_{\psi_1}  \leq m\textsf{K}.
\end{align*}
To bound the first term in \eqref{eq:ihateequations}, we use Corollary 2.8.3 of \cite{vershyninhighdimprob} to say that
\begin{align*}
    \P\l(\abs{\frac{1}{\frac{r}{2}+1}\sum_{j=0}^{r/2}Y_{2j+1}} \geq \frac{nt}{r+2}\r) &\leq 2\cdot\exp\l[-c\l(\frac r2 + 1\r)\l(\frac{nt}{m(r+2)\textsf K} \wedge \frac{n^2t^2}{m^2(r+2)^2\textsf{K}^2}\r)\r] \\
    &\leq 2\cdot\exp\l[-\frac{cn}{m}\l(\frac{t}{2\textsf K} \wedge \frac{t^2}{4\textsf{K}^2}\r)\r],
\end{align*}
where above we used the fact that $r + 2 \geq \frac{n}{m}$. We conclude by noting that this inequality also holds for the second term in \eqref{eq:ihateequations} by the same reasoning. 
\end{proof}
\subsection{Lemmas for mixing processes} In this subsection, we present some useful lemmas to deal with $\beta$-mixing processes. The first one allows one to relate the expectation of a bounded function of mixing data to the one of independent blocks.\begin{lemma}[\cite{yu1994rates} (Corollary 2.7)]\label{Prim}
   Let $K \geq 1$ be an integer and  let $r_1\le s_1\le r_2\le\dots \le s_K$ be a sequence of integers. Let $\tilde X=(\tilde X_i)$ be a stochastic process taking value in $\mathbb{R}^p$. Suppose that $h:\prod_{\ell\le K}\mathbb{R}^{p|s_\ell-r_\ell|}\rightarrow\mathbb{R}$ is measurable function. Let $Q$ be the distribution of $(\tilde X_j)_{j\in\cup_\ell [r_\ell,s_\ell ]}$ and $Q_i$ be the marginal distribution of $(\tilde X_j)_{j\in [r_i,s_i ]}$. Let $\beta_{\max}=\sup _{1 \leq i \leq K-1} \beta\left(k_i\right)$, where $k_i=r_{i+1}-s_i$, and the $P$ the product measure $P=\prod_{i=1}^K Q_i$. Then if $\|h\|_\infty$ we have,
$$
|\mathbb{E}_{X\sim Q}[h(X)]-\mathbb{E}_{Y\sim P}[h(Y)]| \leq(K-1) B \beta_{\max}
$$

\end{lemma} We adapt this lemma to hold for functions with bounded moments.  \begin{lemma}\label{tur}
      Let $K \geq 1$ be an integer and  let $r_1\le s_1\le r_2\le\dots \le s_K$ be a sequence of integers. Let $\tilde X=(\tilde X_i)$ be a stochastic process taking value in $\mathbb{R}^p$. Suppose that $h:\prod_{\ell\le K}\mathbb{R}^{p|s_\ell-r_\ell|}\rightarrow\mathbb{R}$ is measurable function. Let $Q$ be the distribution of $(\tilde X_j)_{j\in\cup_\ell [r_\ell,s_\ell ]}$ and $Q_i$ be the marginal distribution of $(\tilde X_j)_{j\in [r_i,s_i ]}$. Let $\beta_{\max}=\sup _{1 \leq i \leq K-1} \beta\left(k_i\right)$, where $k_i=r_{i+1}-s_i$, and the $P$ the product measure $P=\prod_{i=1}^K Q_i$. Then, if there exists $L>1$ such that $|\mathbb{E}_{X\sim Q}[h(X)^L],\mathbb{E}_{Y\sim P}[h(Y)^L]<\infty$ we have 
$$
|\mathbb{E}_{X\sim Q}[{h}(X)]-\mathbb{E}_{Y\sim P}[{h}(Y)]| \leq2({(K-1) \beta_{\max}})^{1-1/L}\big(\mathbb{E}_{X\sim Q}(h(X)^L)\big)^{1/L}+\big(\mathbb{E}_{Y\sim P}(h(Y)^2)\big)^{1/L}.
$$
\end{lemma}
\begin{proof}
    Let $C>0$ be a positive constant then define $\bar{h}^C$ as the following function $\bar{h}^C(x)=h(x)\mathbb{I}(|h(x)|\le C).$ Now according to \lemmaref{Prim} we have that $$
|\mathbb{E}_{X\sim Q}[\bar{h}^C[X]]-\mathbb{E}_Y\sim {P}[\bar{h}^C[Y]]| \leq(K-1) C \beta_{\max}
$$
Now, moreover, note that 
\begin{align*}
|\mathbb{E}_{X\sim Q}\Big(h(X)\mathbb{I}(|h(X)|>C)\Big)-\mathbb{E}_{Y\sim P}\Big(h(Y)\mathbb{I}(|h(Y)|>C)\Big)|\le \frac{1}{C^{L-1}}\Big\{\mathbb{E}_{X\sim Q}(h(X)^L)+\mathbb{E}_{Y\sim P}(h(Y)^L)\Big\}
\end{align*}
Hence, by triangle inequality, we obtain that $$
|\mathbb{E}_{X\sim Q}[{h}(X)]-\mathbb{E}_{Y\sim P}[{h}(Y)]| \leq(K-1) C \beta_{\max}+ \frac{1}{C^{L-1}}\Big\{\mathbb{E}_{X\sim Q}(h(X)^L)+\mathbb{E}_{Y\sim P}(h(Y)^L)\Big\}.
$$ By choosing $C:=\Big({\frac{(L-1)\Big\{\mathbb{E}_{X\sim Q}(h(X)^L)+\mathbb{E}_{Y\sim P}(h(Y)^L)\Big\}}{(K-1)\beta_{\max}}}\Big)^{1/L}$ we get the desired result.
\end{proof}
The next lemma establishes concentration for some function of $\beta-$mixing random variables.

\begin{lemma}\label{lem:lemmethink}
    Let $(\mathbf{X}_i)$ be a sequence of random variables with mixing coefficients $(\beta(i))$. Assume that there exists a process $(\mathbf{Z}_i)$ of centered  independent random vectors such for every $i\in \mathbb{N}$ there exists constants $(c_{i,j})$ such that  we have $$\mathbf{X}_i=\sum_{j\in \mathbb{N}}c_{i,j}\mathbf{Z}_j.$$Suppose that $K:=\sup_{j}\sqrt{p}\|\mathbf{Z}_j\|_{\psi_2}<\infty$ and that $\rm{Var}(Z_j)=\Sigma/p$ is such that $\lambda_{\min}(\Sigma)\ge \underline{c}>0$. Moreover assume that there exists $\epsilon>0$ such that
 $\mathcal{S}:=\sum_{\ell<\infty}{\beta(\ell)}^{\frac{\epsilon}{2+\epsilon}}<\infty$. Then there exists constants $C,\tilde C,\tilde C'$ such that for all $\beta\in \mathcal{S}_p$$$P\big(\big|\sum_{i\le n}(\bX_i^T\beta)^2-\mathbb{E}((\bX_i^T\beta)^2)\big|\ge t\big)\le 2 e^{-\frac{Ct\underline{c}}{K\mathcal{S}\|\beta\|^2_2/p}\min\Big(\frac{t\underline{c}}{\tilde C{K}n\|\beta\|^2/p}, \frac{1}{\tilde C'}\Big)}.$$ Let $(\bG_i)$ be a sequence of centered Gaussian random variables such that for all $t > 0$, $\rm{Cov}(\bG)=\rm{Cov}(\mathbf{X})$ then we obtain that there exists universal constants $C_2,\tilde C_2,\tilde C_2'>\infty$ such that  for all $\beta\in\mathcal{S}_p$    $$P\big(\big|\sum_{i\le n}(\bG_i^T\beta)^2-\mathbb{E}((\bG_i^T\beta)^2)\big|\ge t\big)\le2e^{-\frac{C_2t}{S\|\beta\|_2^2/p}\min\Big(\frac{t\tilde C_2}{n\|\beta\|_2^2/p},\tilde C_2'\Big)} . $$%
\end{lemma}
\begin{proof}
Firstly we remark that we can assume without loss of generality that there exists an  $M\in \mathbb{N}$ such that $\bX_i=\sum_{j\le M}c_{i,j}Z_j$ for all $i\le n$. 
We define $F_m:=\sum_\ell \beta_\ell Z_{m,\ell}$ and remark that the process $(F_m)$ is still a sequence of independent random variables that are $K\|\beta\|^2/p$ subGaussian. Moreover we remark that for all $i\le n$ we have $$\beta^T\bX_i=\sum_{m\le M}c_{i,m}F_m.$$ Define $C_M:=(c_{i,j})_{i\le n,j\le M}$ then we immediately note that $$\sum_{i\le n}(\beta^T\bX_i)^2=F_m^T(C_M^TC_M)F_m.$$
Hence using Theorem 1.1 of \cite{rudelson2013hanson} we remark that there exists $C>0$ such that $$P\Big(\big|\sum_{i\le n}(\bX_i^T\beta)^2-\mathbb{E}((\bX_i^T\beta)^2)\big|\ge t\Big)\le 2 e^{-\frac{tC}{K\|\beta\|^2/p}\min\Big(\frac{t}{{K\|\beta\|^2}/{p}\|C_M^TC_M\|_{HS}^2}, \frac{1}{\|C_M^TC_M\|}\Big)}.$$We note that $$(C_M^TC_M)_{m_1,m_2}=\sum_{i\le n}c_{i,m_1}c_{j,m_2}.$$Hence we remark that \begin{align*}
    \|C_M^TC_M\|_{HS}^2&=\sum_{m_1,m_2\le M}\big(\sum_i c_{i,m_1}c_{i,m_2}\big)^2\\&=\sum_{i,j\le n}(\sum_{m}c_{i,m}c_{j,m})^2
\end{align*} We observe that if we define $\nu:=(1,\dots,1)^T$ then we have $\rm{cov}(\bX_i^T\nu,\bX_j^T\nu)=\sum_{m}c_{i,m}c_{j,m}\nu^T\Sigma\nu$. As we assumed $\lambda_{min}(\Sigma)>\underline{c}$ we obtain that  $\nu^T\Sigma\nu\ge \underline{c}.$ This implies that  $$\|C_M^TC_M\|_{HS}^2\le (\underline{c})^{-2}\sum_{i,j\le n}\rm{Cov}(\bX_i^T\nu,\bX_j^T\nu)^2.$$ 
To further bound this, we note that according to Lemma 26 of \cite{austern2022limit} we know that 
    $$|\rm{cov}(\nu^T\bX_i,\nu^T\bX_j)|\le4{ \beta(i-j)^{\frac{\epsilon}{2+\epsilon}}}\max_{i}\|\nu^T\bX_i\|_{2+\epsilon}^2,$$
which directly implies, coupled with the fact that  $\bX_i^T\nu$ is a $K_{\bX}$ subgaussian random variable, that there exists a constant $\tilde C>0$ such that \begin{align*}
    \|C_M^TC_M\|_{HS}^2&\le 16(\underline{c})^{-2}\sum_{i,j\le n}\rm{Cov}(\bX_i^T\nu,\bX_j^T\nu)^2{ \beta(i-j)^{\frac{2\epsilon}{2+\epsilon}}}\max_{i}\|\nu^T\bX_i\|_{2+\epsilon}^4
    \\&\le \tilde C n(\underline{c})^{-2}\mathcal{S}
\end{align*}Similarly we obtain that there exists a constant $\tilde C'$ such that \begin{align*}
    \|C_m^TC_M\|\le \sup_{m_1}\sum_{m_2\le M}c_{i,m_1},c_{i,m_2}\le \tilde C'(\underline{c})^{-1}\mathcal{S}.
\end{align*}
Combining this together we obtain that 
    $$P\Big(\big|\sum_{i\le n}(\bX_i^T\beta)^2-\mathbb{E}((\bX_i^T\beta)^2)\big|\ge t\Big)\le 2 e^{-\frac{Ct\underline{c}}{K\mathcal{S}\|\beta\|^2_2/p}\min\Big(\frac{t\underline{c}}{\tilde C{K}n\|\beta\|^2/p}, \frac{1}{\tilde C'}\Big)}.$$
To prove the second statement, we observe that  $(\mathbf{G}_i^T\beta)$ is still a Gaussian process. Hence if we let $\mathbf{N}:=(N_i)_{i\le n}$ be a vector of standard normal we obtain that  $\sum_{i\le n}(\bG_i^T\beta)^2\overset{d}{=}\mathbf{N}^T\Sigma_n\mathbf{N}$ where $\Sigma_n:=\rm{Cov}((\bG^T_i\beta)_{i=1}^n).$ Hence for every $t\ge 0$ we have that $$P\big(\big|\sum_{i\le n}(\bG_i^T\beta)^2-\mathbb{E}((\bG_i^T\beta)^2)\big|\ge t\big)\le P\big(\Big|\mathbf{N}^T\Sigma_n\mathbf{N}-\mathbb{E}(\mathbf{N}^T\Sigma_n\mathbf{N})\Big|\ge t\big)$$
    To further bound this the first step is to bound. In this goal we remark that $(\Sigma_n)_{i,j}=\rm{Cov}(\bG_i^T\beta,\bG^T_j\beta)=\rm{Cov}(\bX_i^T\beta,\bX^T_j\beta)$. Now once again using Lemma 26 of \cite{austern2022limit} we know that 
     $$|\rm{cov}(\beta^T\bX_i,\beta^T\bX_j)|\le4{ \beta(i-j)^{\frac{\epsilon}{2+\epsilon}}}\max_{i}\|\beta^T\bX_i\|_{2+\epsilon}^2\le 4\|\beta\|_2^2/p{ \beta(i-j)^{\frac{\epsilon}{2+\epsilon}}}\max_{i,\|\mu\|\le \sqrt{p}}\|\nu\bX_i\|_{2+\epsilon}^2.$$ This directly implies that there exists another constant $\tilde C_2$ such that $$\|\Sigma_n\|^2_F\le (\|\beta\|_2^2/p)^2\tilde C_2\sqrt{n}\mathcal{S}. $$We moreover similarly  notice that there exists another constant $\tilde C_2'$ such that $$\|\Sigma_n\|\le \max_i \sum_{j\le n}\rm{cov}(\bX_i^T\beta,\bX_j^T\beta)\le \|\beta\|_2^2/p \mathcal{S}\tilde C'_2.$$
    Hence using the Hanson-Wright Inequality (Theorem 1.1 of \cite{rudelson2013hanson}), we obtain that there exists a constant $C_2>0$ such that 
    $$P\big(\Big|\mathbf{N}^T\Sigma_n\mathbf{N}-\mathbb{E}(\mathbf{N}^T\Sigma_n\mathbf{N})\Big|\ge t\big)\le2e^{-\frac{C_2t}{S\|\beta\|_2^2/p}\min\Big(\frac{t\tilde C_2}{n\|\beta\|_2^2/p},\tilde C_2'\Big)} . $$
\end{proof}
\subsection{Smoothing lemma}\begin{lemma} \label{lem:lipschitz:approx} For every $\delta > 0$, there exists a family of continuous and $\frac{1}{\delta}$-Lipschitz $\R \rightarrow \R$ functions $(h_{\tau;\delta})_{\tau \in \R}$ such that, for any $x, y \in \R$,
\begin{enumerate}
    \item[(i)] $h_{\tau;\delta}(x) \leq \ind\{x < \tau \}  \leq h_{\tau+\delta;\delta}(x)$;
    \item[(ii)] $
        \big| 
        h_{\tau;\delta}(x)
        -
        \ind\{ x < \tau \}
        \big| 
        \;\leq\;
            \ind\{ x \in [\tau-\delta, \tau) \}
        $\;.
\end{enumerate}
\end{lemma}
    
\begin{proof}[Proof of \lemmaref{lem:lipschitz:approx}.] We construct $h_{\tau;\delta}$ as 
    $$
            h_{\tau;\delta}(x) 
            \;\coloneqq\; 
            \begin{cases}
                1  & \text{ if }  x < \tau - \delta\;, \\
                \frac{\tau - x}{\delta} & \text{ if }  x \in [\tau - \delta, \tau)\;, \\
                0  & \text{ if } x \geq \tau\;,
            \end{cases} 
    $$
    which satisfies (i) automatically. To prove (ii), we first note that 
    \begin{align*}
        \ind\{ x < \tau - \delta \}
        \;\leq\; 
        h_{\tau;\delta}(x) 
        \;\leq\;
        \ind\{ x < \tau \}
        \;,
    \end{align*}
    which implies the desired bound that 
    \begin{align*}
        \big| 
        \,
        h_{\tau;\delta}(x) 
        -
        \ind\{ x < \tau \}
        \, 
        \big| 
        \leq\,
        \ind\{ x \in [\tau-\delta, \tau) \}
        \;.
    \end{align*}
\end{proof}

\subsection{Additional Lindeberg Lemma}
 \begin{lemma}\label{dalia2}Suppose that $(X_i)$ are independent, centered random variables. We obtain that there is an absolute constant $\kappa>0$ such that 
\begin{align}&\sup_{f \in \mathcal F} \abs{\E f\l(\sum_{j\le n}X_j\r) - \E f\l(\sum_{j\le n}Z_j\r) }
 \\  & \le \Big( \frac{\kappa n}{\epsilon^2}\Big)^{1/3}\max_{j\le n}\|X_{j}\|_3,
\end{align} where $(Z_i)$ is an independent sequence of independent and centered Gaussian random variables. 
\end{lemma}

\begin{proof}
Let $f \in \mathcal F$, and let $\epsilon>0$. Define  $f_\epsilon(u)=\frac{1}{4\epsilon^2}\int_{u-\epsilon}^{u+\epsilon}\int_{t-\epsilon}^{t+\epsilon} f(y)dydt$. We remark that $f_{\epsilon}$ is three times differentiable and as $f$ is Lipschitz we have $\sup_x |f_\epsilon(x)-f(x)|\le 2\epsilon.$  
    
    Write $$X_{j}(t):=\sqrt{t}X_{j}+\sqrt{1-t}G_{j}$$ and $$X^{j,0}_m(t):=\begin{cases}0\text{~if~}m=j
        \\X_{m}(t)~~\text{otherwise}.
    \end{cases}$$
Using the fundamental theorem of calculus we obtain that 
    \begin{align*}&
        \Big|\mathbb{E}(f_\epsilon(\sum_jX_j)-\mathbb{E}(f_\epsilon(\sum_jZ_j))\Big|
        \\&\le \int_0^1 \Big|\partial_t\mathbb{E}(f_\epsilon(\sum_jX_j(t)))\Big|dt
        \\&\overset{(d_1)}{\le}\int_0^1 \Big|\mathbb{E}\Big(f'_\epsilon(\sum_j X_j(t))\sum_j\frac{X_j}{2\sqrt{t}}-\frac{Z_j}{2\sqrt{1-t}}\Big)\Big|dt
                \\&\overset{(d_2)}{\le}\int_0^1 \sum_{j\le n }\Big|\mathbb{E}\Big(f'_\epsilon(\sum X^{j,0}(t)) \big[\frac{X_j}{2\sqrt{t}}-\frac{Z_j}{2\sqrt{1-t}}\big]\Big)\Big|dt
                \\&+\int_0^1 \sum_{j\le n}\Big|\mathbb{E}\Big(f''_\epsilon(\sum X^{j,0}(t))\big[ \frac{X_j}{2\sqrt{t}}-\frac{Z_j}{2\sqrt{1-t}}\big]\big[X_j\sqrt{t}+Z_j\sqrt{1-t}\big]\Big)\Big|dt
               \\&+\int_0^1 \sum_{j\le n}\Big|\mathbb{E}\Big(f^{(3)}_\epsilon(\sum_j\tilde X_j(t))\big[ \frac{X_{j}}{2\sqrt{t}}-\frac{Z_{j}}{2\sqrt{1-t}}\big]\big[X_{j}\sqrt{t}+Z_{j}\sqrt{1-t}\big]^2\Big)\Big|dt
                \\&\le (a)+(b)+(c)
    \end{align*}
    where $\sum \tilde X_{j}(t)\in[\sum X_j(t), \sum X^{j,0}(t)]$ and where $d_1$ is a result of the product law and $d_2$ of Taylor's expansion. 
    Using the independence between $X^{j,0}(t)$ and $X_{j}$ and $Z_{j}$ and the fact that those latter are centered, we obtain that $(a)=0$. Similarly, we notice for all $j\le n$ that 
    \begin{align*}&
\mathbb{E}\Big(f''_\epsilon(\sum X^{j,0}(t))\big[ \frac{X_j}{2\sqrt{t}}-\frac{Z_j}{2\sqrt{1-t}}\big]\big[X_j\sqrt{t}+Z_j\sqrt{1-t}\big]\Big)\\&=\mathbb{E}\Big(f''_\epsilon(\sum X^{j,0}(t))\Big)\mathbb{E}\Big(\big[ \frac{X_j}{2\sqrt{t}}-\frac{Z_j}{2\sqrt{1-t}}\big]\big[X_j\sqrt{t}+Z_j\sqrt{1-t}\big]\Big)
\\&\le \mathbb{E}\Big(f''_\epsilon(\sum X^{j,0}(t))\Big)\mathbb{E}\Big(\big[ \frac{X_j^2}{2}-\frac{Z_j^2}{2}\big]\Big)=0.
    \end{align*}  Hence $(b)=0.$
    
    Finally we can note that $\|f_\epsilon^{(3)}\|\le \frac{4}{\epsilon^2}$. Hence, thanks to Jensen inequality we know that there exists absolute constants $C,C_2>0$ such that
    \begin{align*}&
   \Big| \mathbb{E}\Big(f^{(3)}_\epsilon(\sum_j\tilde X_j(t))\big[ \frac{X_{j}}{2\sqrt{t}}-\frac{Z_{j}}{2\sqrt{1-t}}\big]\big[X_{j}\sqrt{t}+Z_{j}\sqrt{1-t}\big]^2\Big)\Big|\\&\le \frac{4}{\epsilon^2}  \mathbb{E}\Big(\Big|  \frac{X_{j}}{2\sqrt{t}}-\frac{Z_{j}}{2\sqrt{1-t}}\Big| \big[X_{j}\sqrt{t}+Z_{j}\sqrt{1-t}\big]^2\Big)\Big|
   \\&\le \frac{C}{\epsilon^2}  \max(\|X_j\|_3^3,\|Z_j\|_3^3)\Big[\frac{1}{\sqrt{t}}+\frac{1}{\sqrt{1-t}}\Big]
   \\&\le \frac{C}{\epsilon^2}  \max(\|X_j\|_3^3,\sqrt{3}^3\|X_j\|_2^3)\Big[\frac{1}{\sqrt{t}}+\frac{1}{\sqrt{1-t}}\Big]
    \\&\le \frac{C_2}{\epsilon^2}  \|X_j\|_3^3\Big[\frac{1}{\sqrt{t}}+\frac{1}{\sqrt{1-t}}\Big]
    \end{align*} where we used the fact that if $Z\sim \mathcal{N}(0,1)$ then $\|Z\|_3\le \sqrt{3}$ coupled with the fact that $\|Z_j\|_2=\|X_j\|_2.$  Hence we obtain that there is an absolute constant $\kappa>0$ such that 
    \begin{align}
        (c)&\le \frac{\kappa n}{\epsilon^2}\max_{j\le n}\|X_{j}\|_3^3
    \end{align}This gives us the desired result by choosing $\epsilon:=\Big( \kappa n\Big)^{1/3}\max_{j\le n}\|X_{j}\|_3.$
\end{proof}
\subsection{Asymptotic normality under local dependency assumption}

\begin{lemma}\label{dalia} Suppose that $(X_i)$ are centered random vectors such that the array $(X_{i,\ell})_{i,\ell}$ is locally dependent. Write $B_{i,\ell}$ the dependency neighborhood of the entry $X_{i,\ell}$. For a fixed $q \in \N$, let $\cF_q$ be the class of $\R^q \rightarrow \R$ continuously differentiable functions with $\| f \|_\infty \leq 1$ and $\| \,\| \partial f \|\, \|_\infty \leq 1$. Then there is a constant $C_q>0$ that depends only on $q$ such that 
\begin{align}&
    \sup_{\substack{f\in \mathcal{F}_q \\\theta_1, \ldots, \theta_q \in \mathcal{S}^{k-1}}}\sup_{\substack{(\beta_{11}, \hdots, \beta_{kq}) \\ \in\mathcal S_p^{kq}}}
        \abs{
            \E\textsf{f}
                \begin{pmatrix} 
                    \sum_{i=1}^k \theta_{1i} X_i^\intercal \beta_{1i} 
                    \\ \vdots \\
                    \sum_{i=1}^k \theta_{qi} X_i^\intercal \beta_{qi} 
                \end{pmatrix}
         - 
             \E\textsf{f}
             \begin{pmatrix} 
                \sum_{i=1}^k \theta_{1i} G_i^\intercal \beta_{1i} 
                \\ \vdots \\
                \sum_{i=1}^k \theta_{qi} G_i^\intercal \beta_{qi} 
            \end{pmatrix}
        } \nonumber
    \\
     &\hspace{15em} \le 
     C_q \, \Big( k^2 \, p^{3/2 -r } \, \max_{i, \ell} |B_{i,\ell}|^{3/2} \,  \max_{i,\ell} \| X_{i, \ell} \|^3_{L_3} \Big)^{1/(2q+1)}
     \;,
\end{align} 
where $\cS^{k-1}$ denotes the unit sphere in $\R^k$ and $(G_i)$ is an independent sequence of mean-zero Gaussian vectors chosen such that $\text{cov}(G_{i,l},G_{j,m})=\text{cov}(X_{i,l},X_{j,m})$.
\end{lemma}

\begin{proof} Fix $f \in \cF_q$. Let $\epsilon>0$ and define a smoothed version of $f$,
\begin{align*}
    f_\epsilon(u) 
    \;\coloneqq\; 
    \mfrac{1}{(2\epsilon)^{2q}} 
    \mint_{u_1-\epsilon}^{u_1+\epsilon} 
    \cdots 
    \mint_{u_q-\epsilon}^{u_q+\epsilon} 
    \;
    \mint_{t_1-\epsilon}^{t_1+\epsilon} 
    \cdots 
    \mint_{t_q-\epsilon}^{t_q+\epsilon} 
    f(y) 
    \; dy_1 \ldots dy_q \; dt_1 \ldots dt_q \; 
    \;.
\end{align*}
Note that $f_{\epsilon}$ is thrice differentiable and, as $f$ is Lipschitz, we have $\sup_x |f_\epsilon(x)-f(x)|\le 6\epsilon \sqrt{q}$. Write 
\begin{align*}
    X_{j}(t)
    \coloneqq \sqrt{t}X_{j}+\sqrt{1-t}G_{j}
    \;,
    \;\;
    a_{jl} \coloneqq 
    \big(
        \theta_{1j} \beta_{1jl}\,,\, \ldots \,,\, \theta_{qj} \beta_{qjl} 
    \big)
    \in \R^{q}\;,
    \;\;
    A_j \coloneqq
    \begin{pmatrix}
        \leftarrow & a_{j1}^\intercal & \rightarrow \\ 
        & \vdots & \\ 
        \leftarrow & a_{jp}^\intercal & \rightarrow \\ 
    \end{pmatrix}
    \in \R^{p \times q}\;,
\end{align*}
and $$\big(X_{j,i,\ell,0}(t)\big)_m \coloneqq \begin{cases}0\text{~if~}(j,m)\in B_{i,\ell}
        \\X_{j,m}(t)~~\text{otherwise}.
    \end{cases}.$$
Using the fundamental theorem of calculus we obtain that 
    \begin{align*}
        &
        \abs{
            \E f_\epsilon
                \begin{pmatrix} 
                    \sum_{i=1}^k \theta_{1i} X_i^\intercal \beta_{1i} 
                    \\ \vdots \\
                    \sum_{i=1}^k \theta_{qi} X_i^\intercal \beta_{qi} 
                \end{pmatrix}
         - 
             \E f_\epsilon
             \begin{pmatrix} 
                \sum_{i=1}^k \theta_{1i} G_i^\intercal \beta_{1i} 
                \\ \vdots \\
                \sum_{i=1}^k \theta_{qi} G_i^\intercal \beta_{qi} 
            \end{pmatrix}
         }
         \;=\; 
         \Big| 
            \mean\Big[ f_\epsilon\Big( \msum_{j \leq n} X_j^\intercal A_j \Big) \Big]
            -
            \mean\Big[ f_\epsilon\Big( \msum_{j \leq n} G_j^\intercal A_j \Big) \Big]
         \Big|
        \\&\le \int_0^1 \Big|\partial_t \, \mathbb{E}\Big[ f_\epsilon\Big(\msum_j\ X_j^\intercal(t) \, A_j \Big) \Big] \Big|
        \, dt
        \\&\overset{(d_1)}{\le}\int_0^1 
        \Big|
        \mathbb{E}\Big[
            \partial f_\epsilon(
                \msum_j X_j^\intercal(t) A_j)
            \msum_{i\le k}
                A_i^\intercal
                \Big(
                \mfrac{ X_i}{2\sqrt{t}}
                -
                \mfrac{G_i}{2\sqrt{1-t}}
               \Big)
            \Big]
        \Big|
        dt
        \\&\overset{(d_2)}{\le}
        \int_0^1 
        \,
        \Big|
            \msum_{i \leq k}
            \msum_{\ell \leq p}
            \mathbb{E}\Big[ 
            \partial f_\epsilon \Big(
            \msum_j X_{j,i,\ell,0}^\intercal(t) A_j \Big)^\intercal
            a_{il}
            \, \Big(\mfrac{X_{i,\ell}}{2\sqrt{t}}-\mfrac{G_{i,\ell}}{2\sqrt{1-t}} \Big) 
        \Big]\Big|
        dt
        \\&+
        \mfrac{1}{2} \int_0^1
        \Big|
            \msum_{i\le k} 
            \msum_{\ell \leq p}
            \msum_{(\tilde i,\tilde \ell)\in B_{i,\ell}}
            \mathbb{E}
            \Big[
                \big(X_{\tilde i,\tilde \ell}\sqrt{t}+G_{\tilde i,\tilde \ell}\sqrt{1-t}\big)
                a_{\tilde i, \tilde l}^\intercal \,
                \partial^2 f_\epsilon\Big(
                    \msum_j X_{j,i,\ell,0}^T(t) A_j
                \Big)
                \, 
                a_{il}
        \\
        &\hspace{20em}
                \,
                \Big(\mfrac{X_{i,l}}{2\sqrt{t}}-\mfrac{G_{i,l}}{2\sqrt{1-t}} \Big)
            \Big]     
        \Big|dt
        \\
        &+
        \mfrac{1}{6 \epsilon^q}
        \int_0^1 
        \,
        \mathbb{E} \Big[ \, 
        \Big\|
        \sum_{i\le k} 
        \sum_{\ell \leq p}
        \sum_{\substack{(\tilde i,\tilde \ell), (\tilde i_2,\tilde \ell_2)\\ \in B_{i,\ell}}}
        \hspace{-6pt}\Big(\mfrac{X_{i,l}}{2\sqrt{t}}-\mfrac{G_{i,l}}{2\sqrt{1-t}} \Big)
        \,
        \big(   X_{\tilde i,\tilde \ell}\sqrt{t}+G_{\tilde i,\tilde \ell}\sqrt{1-t}\big)
        \\
        &\hspace{18em}
        \,
        \big( X_{\tilde i_2,\tilde \ell_2}\sqrt{t}+G_{\tilde i_2,\tilde \ell_2}\hspace{-4pt}\sqrt{1-t} \big)  
        \,( a_{i l} \otimes a_{\tilde i\tilde \ell} \otimes a_{\tilde i_2, \tilde \ell_2}  )
        \Big\| \, \Big]  \, dt
        \\&\le (a)+(b)+(c)\;,
    \end{align*}
    where $\tilde X_{j,i,l}(t)\in[X_j(t), X_{j,i,l,0}(t)]$. In $(d_1)$, we have used the product rule; in $(d_2)$, we have used a third-order Taylor expansion together with the bound that $\| \partial^3 f_\epsilon \|_\infty \leq \frac{1}{\epsilon^q}$. Using the independence between $X_{j,i,\ell,0}(t)$ and $X_{i,\ell}$ and $G_{i,\ell}$ and the fact that these variables are centered, we obtain that for all $i\le k$ and $\ell\le p$ we have 
    \begin{align*}
        \mathbb{E}\Big[ \partial f_\epsilon \Big(
            \msum_j X_{j,i,\ell,0}^\intercal(t) \Theta^\beta_j \Big)  (\Theta^\beta_i)^\intercal
            \, \Big(\mfrac{X_{i,\ell}}{2\sqrt{t}}-\mfrac{G_{i,\ell}}{2\sqrt{1-t}} \Big) 
        \Big]
        \;=\; 0
        \;.
    \end{align*}
    Hence we know that $(a)=0$. Similarly we notice that 
    \begin{align*}
        \mean \Big[
            \big(X_{\tilde i,\tilde \ell}\sqrt{t}+G_{\tilde i,\tilde \ell}\sqrt{1-t}\big)
            \big( \Theta^{\beta, (\tilde l)}_{\tilde i}\big)^\intercal \,
            \partial^2 f_\epsilon\Big(
                \msum_j X_{j,i,\ell,0}^T(t) \Theta^\beta_j
            \Big)
            \, 
            (\Theta^\beta_i)^\intercal
            \,
            \Big(\mfrac{X_i}{2\sqrt{t}}-\mfrac{G_i}{2\sqrt{1-t}} \Big)
        \Big]
        \;=\; 
        0\;,
    \end{align*} 
    where we use the independence between $(X_{i,j})$ and $(G_{i,j})$ to ignore cross terms and the fact that $\mathbb{E}(X_{i,\ell}X_{\tilde i,\tilde \ell})=\mathbb{E}(G_{i,\ell}G_{\tilde i,\tilde \ell})$. Hence $(b)=0.$ Finally to handle $(c)$, we see that 
    \begin{align*}&
        \,
        \mathbb{E} \Big[ \, 
        \Big\|
        \sum_{i\le k} 
        \sum_{\ell \leq p}
        \sum_{(\tilde i,\tilde \ell), (\tilde i_2,\tilde \ell_2)\in B_{i,\ell}}
        \Big(\mfrac{X_{i,l}}{2\sqrt{t}}-\mfrac{G_{i,l}}{2\sqrt{1-t}} \Big)
        \\
        &\hspace{10em}
        \,
        \times 
        \big(   X_{\tilde i,\tilde \ell}\sqrt{t}+G_{\tilde i,\tilde \ell}\sqrt{1-t}\big)
        \,
        \big( X_{\tilde i_2,\tilde \ell_2}\sqrt{t}+G_{\tilde i_2,\tilde \ell_2}\sqrt{1-t} \big)  
        \,( a_{i l} \otimes a_{\tilde i\tilde \ell} \otimes a_{\tilde i_2, \tilde \ell_2}  )
        \Big\| \, \Big] 
        \\
        &=
        \mean \Big[ 
            \Big(
                \;
                \msum_{s, \tilde s, \tilde s_2 \le q} 
                \Big(
                    \sum_{i\le k} 
                    \sum_{\ell \leq p}
                    \sum_{(\tilde i,\tilde \ell), (\tilde i_2,\tilde \ell_2)\in B_{i,\ell}}
                    \Big(\mfrac{ \theta_{si} \beta_{sil} X_{i,l}}{2\sqrt{t}}-\mfrac{ \theta_{si} \beta_{sil} G_{i,l}}{2\sqrt{1-t}} \Big)
                    \,
                    \theta_{\tilde s\tilde i} \beta_{\tilde s\tilde i\tilde l}  \big(   X_{\tilde i,\tilde \ell}\sqrt{t}+G_{\tilde i,\tilde \ell}\sqrt{1-t}\big)
        \\
        &\hspace{15em} \,
                    \theta_{\tilde s_2\tilde i_2} \beta_{\tilde s_2\tilde i_2\tilde l_2}  \big( X_{\tilde i_2,\tilde \ell_2}\sqrt{t}+G_{\tilde i_2,\tilde \ell_2}\sqrt{1-t} \big)  
                \Big)^2 
                \;
            \Big)^{1/2}
        \Big]
        \\
        &\overset{(d_3)}{\le }
        \msum_{s, \tilde s, \tilde s_2 \le q} 
        \msum_{i\le k} 
        \, 
        \mean\Big[
                    \sum_{\ell \leq p}
                    \sum_{(\tilde i,\tilde \ell), (\tilde i_2,\tilde \ell_2)\in B_{i,\ell}}
                    \,
                    |\theta_{si} \beta_{sil}|
                    \,
                    \Big| 
                    \mfrac{  X_{i,l}}{2\sqrt{t}}-\mfrac{ G_{i,l}}{2\sqrt{1-t}} \Big|
                    \,
                    | \theta_{\tilde s\tilde i} \beta_{\tilde s\tilde i\tilde l}|  \, \big|   X_{\tilde i,\tilde \ell}\sqrt{t}+G_{\tilde i,\tilde \ell}\sqrt{1-t}\big|
        \\
        &\hspace{23em} \,
                    | \theta_{\tilde s_2\tilde i_2} \beta_{\tilde s_2\tilde i_2\tilde l_2} | \, \big| X_{\tilde i_2,\tilde \ell_2}\sqrt{t}+G_{\tilde i_2,\tilde \ell_2}\sqrt{1-t} \big|
        \Big]
        \\
        &\overset{(d_4)}{\le }
        \msum_{s, \tilde s, \tilde s_2 \le q} 
        \msum_{i\le k} 
                    \msum_{\ell \leq p}
                    \msum_{(\tilde i,\tilde \ell), (\tilde i_2,\tilde \ell_2)\in B_{i,\ell}}
                    \,
                    | \beta_{sil}|
                    \,
                    | \beta_{\tilde s\tilde i\tilde l}|   
                    \,
                    |  \beta_{\tilde s_2\tilde i_2\tilde l_2} |
        \\& \qquad \qquad \times
        \mean\Big[
                    \Big| 
                    \mfrac{  X_{i,l}}{2\sqrt{t}}-\mfrac{ G_{i,l}}{2\sqrt{1-t}} \Big|
                    \, \big|   X_{\tilde i,\tilde \ell}\sqrt{t}+G_{\tilde i,\tilde \ell}\sqrt{1-t}\big|
                     \big| X_{\tilde i_2,\tilde \ell_2}\sqrt{t}+G_{\tilde i_2,\tilde \ell_2}\sqrt{1-t} \big|
        \Big]
        \\
        &\overset{(d_5)}{\le }
        \msum_{s, \tilde s, \tilde s_2 \le q} 
        \msum_{i\le k} 
                    \msum_{\ell \leq p}
                    \msum_{(\tilde i,\tilde \ell), (\tilde i_2,\tilde \ell_2)\in B_{i,\ell}}
                    \,
                    | \beta_{sil}|
                    \,
                    | \beta_{\tilde s\tilde i\tilde l}|   
                    \,
                    |  \beta_{\tilde s_2\tilde i_2\tilde l_2} |
                    \,       
                    \times 
                    \,
                    \Big( \mfrac{1}{2\sqrt{t}} + \mfrac{1}{2\sqrt{1-t}}\Big)
                    (1+\sqrt{3})^2 
                    \max_{i,\ell} \| X_{i, \ell} \|^3_{L_3}
        \;.
    \end{align*} 
    In $(d_3)$, we have moved the summations outside a squareroot and an absolute value; in $(d_4)$, we have noted that $\theta \in \cS^{k-1}$; in $(d_5)$, we have used that for $Z\sim \mathcal{N}(0,1)$, $\|Z\|_{L_3} \le \sqrt{3}$. Now let $(\beta_{\rm{mix}})_s \coloneqq (|(\beta_{\rm{mix}})_{s11}|, \ldots, |(\beta_{\rm{mix}})_{skp}|) \in \R^{kp}$ and 
    $M^{(i, \ell)} \in \R^{kp \times kp}$ be a matrix with entries $M^{(i, \ell)}_{(i',\ell'), (i'',\ell'')} = \ind\{ (i',\ell') \in B_{i,\ell} \} \, \ind\{ (i'',\ell'') \in B_{i,\ell} \} $. Also recall that by the definition of $\cS_p$ in \eqref{eq:sp}, there are fixed constants $L, r >0 $ such that $\| \beta_{si} \|_\infty \leq L p^{1/2 -r}$ and $\| \beta_{si} \|_2 \leq L p^{1/2}$ for all $s \in q, i \leq k$. Then
    \begin{align*}
        \msum_{\ell \leq p}
        \msum_{(\tilde i,\tilde \ell), (\tilde i_2,\tilde \ell_2)\in B_{i,\ell}}
        | \beta_{sil}|
        \,
        | \beta_{\tilde s\tilde i\tilde l}|   
        \,
        |  \beta_{\tilde s_2\tilde i_2\tilde l_2} |
        \;\leq&\;
        L p^{1/2 -r} \,
        \msum_{\ell \leq p}
        \;
        \msum_{(\tilde i,\tilde \ell), (\tilde i_2,\tilde \ell_2)\in B_{i,\ell}}
        \,
        | \beta_{\tilde s\tilde i\tilde l}|   
        \,
        |  \beta_{\tilde s_2\tilde i_2\tilde l_2} |
        \\
        \;=&\;
        L p^{1/2 -r} \,
        \msum_{\ell \leq p} 
        (\beta_{\rm{mix}})_{\tilde s}^\intercal 
        M^{(i,\ell)}
        (\beta_{\rm{mix}})_{\tilde s_2}
        \\
        \;\leq&\;
        L p^{1/2 -r} \, \| (\beta_{\rm{mix}})_{\tilde s} \| \, \| (\beta_{\rm{mix}})_{\tilde s_2} \| \, \big\|  \msum_{\ell \leq p}   M^{(i,\ell)} \big\|_{op}
        \\
        \;\leq&\;
        L^3 k p^{3/2 -r } \, \big\|  \msum_{\ell \leq p}   M^{(i,\ell)} \big\|_{op}
    \end{align*}
    In the last line, we have noted that $ \| (\beta_{\rm{mix}})_{\tilde s} \|^2 = \sum_{\tilde i \leq k} \| (\beta_{\rm{mix}})_{\tilde s \tilde i} \|^2 \leq L^2 k p$. Now observe that the $(i', \ell')$-th column of the matrix $\sum_{\ell \leq p} M^{(i,\ell)}$ is given by
    \begin{align*}
        \Big(
            \msum_{\ell \leq p} 
            \ind\{ (i',\ell') \in B_{i,\ell} \}  \,
            \ind\{ (i'',\ell'') \in B_{i,\ell} \} 
        \Big)_{i'' \leq k, \ell'' \leq p}
        \;.
    \end{align*}
    Since $|B_{i,\ell}| \leq \max_{i, \ell} |B_{i,\ell}|$, the column has at most $\max_{i, \ell} |B_{i,\ell}|$ non-zero entries. For each $(i'', \ell'')$, since the dependency neighborhood induces an equivalence relation and $|B_{i'',\ell''}| \leq \max_{i, \ell} |B_{i,\ell}|$, the $(i'', \ell'')$-th entry cannot exceed $\max_{i, \ell} |B_{i,\ell}|$. In other words, the $l_2$-norm of each column vector of $\sum_{\ell \leq p} M^{(i,\ell)}$ cannot exceed $\sqrt{ \max_{i, \ell} |B_{i,\ell}| \times \max_{i, \ell} |B_{i,\ell}|^2 } = \max_{i, \ell} |B_{i,\ell}|^{3/2}$, which implies 
    \begin{align*}
        \msum_{\ell \leq p}
        \msum_{(\tilde i,\tilde \ell), (\tilde i_2,\tilde \ell_2)\in B_{i,\ell}}
        | \beta_{sil}|
        \,
        | \beta_{\tilde s\tilde i\tilde l}|   
        \,
        |  \beta_{\tilde s_2\tilde i_2\tilde l_2} |
        \;\leq&\;
        L^3 k p^{3/2 -r } \, \max_{i, \ell} |B_{i,\ell}|^{3/2} \ ;.
    \end{align*}
    Combining the bounds, we obtain that for some constant $C_q>0$ depending only on $q$,
    \begin{align*}
        (c)&\le \mfrac{C_q}{\epsilon^{2q}} \, k^2 \, p^{3/2 -r } \, \max_{i, \ell} |B_{i,\ell}|^{3/2} \,  \max_{i,\ell} \| X_{i, \ell} \|^3_{L_3} 
        \;.
    \end{align*}
    This gives us the desired result by choosing $\epsilon:=(C_q \, k^2 \, p^{3/2 -r } \, \max_{i, \ell} |B_{i,\ell}|^{3/2} \,  \max_{i,\ell} \| X_{i, \ell} \|^3_{L_3} )^{1/(2q+1)}.$
\end{proof}

\subsection{Polynomial Approximation Properties}

In this section we discuss some of the properties of our polynomial approximation that are used in the proof of our main theorem and also \Cref{poly_approx}.

\begin{lemma}\label{lem:remaindersquared}
Let $\alpha, \delta, \gamma, \tau > 0$. Then there exists finite $\textnormal{\textsf D} = \textnormal{\textsf D}(k, \alpha, \tau)$ such that, if we define 
\begin{align*}
    Q_{\textnormal{\textsf D}}(x) := \sum_{\ell=0}^\textnormal{\textsf D}(1-x)^\ell,\quad\quad R_\textnormal{\textsf D}(x) := \frac{1}{x} - Q_{\textnormal{\textsf D}}(x),
\end{align*}
then 
\begin{align*}
    \E_{(i, k)}\l[R_\textnormal{\textsf D}\l(\lip e^{-\alpha\sum_{j\in\mathcal B_i}\omega_j\ell(\eta_j, U_j^\intercal\beta)}\rip_{i, k}\r)^2\r] < \tau.
\end{align*}
\end{lemma}

\begin{proof}
For $t > 0$, define the event 
\begin{align*}
    \mathcal A_t := \l\{\max_{j\in\mathcal B_i}\abs{U_j^\intercal\beta} \leq t\r\} = \bigcap_{j\in\mathcal B_i}\l\{\abs{U_j^\intercal\beta} \leq t\r\}. 
\end{align*}
Then we have that 
\begin{align}
    \P\l(\mathcal A_t^c\r) &\stackrel{(i)}{=} \P\l(\bigcup_{j\in\mathcal B_i}\l\{\abs{U_j^\intercal\beta} > t\r\}\r)\nonumber\\
    &\stackrel{(ii)}{\leq} \sum_{j\in\mathcal B_i}\P\l(\abs{U_j^\intercal\beta} > t\r)\nonumber\\
    &\stackrel{(iii)}{\leq} \sum_{j\in\mathcal B_i}2e^{-ct^2/\textsf C_1^2}\nonumber\\
    &\leq \textsf C_2ke^{-ct^2}.\label{eq:c2kect2}
\end{align}
where $(i)$ is via De Morgan's Law, $(ii)$ is via a union bound, and $(iii)$ is from \eqref{eq:lol_see_one} which bounds the sub-Gaussian norm of each $\abs{U_j^\intercal\beta}$ along with Proposition 2.5.2 of \cite{vershyninhighdimprob}. We can then say that 
\begin{align}
    \E_{(i, k)}\l[R_\textnormal{\textsf D}\l(\lip e^{-\alpha\sum_{j\in\mathcal B_i}\omega_j\ell(\eta_j, U_j^\intercal\beta)}\rip_{i, k}\r)^2\r] &\stackrel{(i)}{\leq} \E_{(i, k)}\l\lip R_\textnormal{\textsf D}\l(e^{-\alpha\sum_{j\in\mathcal B_i}\omega_j\ell(\eta_j, U_j^\intercal\beta)}\r)^2\r\rip_{i, k}\nonumber\\
    &\stackrel{(ii)}{=} \l\lip \E_{(i, k)}\l[R_\textnormal{\textsf D}\l(e^{-\alpha\sum_{j\in\mathcal B_i}\omega_j\ell(\eta_j, U_j^\intercal\beta)}\r)^2\r]\r\rip_{i, k}\nonumber\\
    &\stackrel{(iii)}{=} \l\lip \E_{(i, k)}\l[R_\textnormal{\textsf D}\l(e^{-\alpha\sum_{j\in\mathcal B_i}\omega_j\ell(\eta_j, U_j^\intercal\beta)}\r)^2\mathbb I_{\mathcal A_t}\r]\r\rip_{i, k}\label{eq:r_dee_first} \\
    &\hspace{30pt}+ \l\lip \E_{(i, k)}\l[R_\textnormal{\textsf D}\l(e^{-\alpha\sum_{j\in\mathcal B_i}\omega_j\ell(\eta_j, U_j^\intercal\beta)}\r)^2\mathbb I_{\mathcal A_t^c}\r]\r\rip_{i, k}\label{eq:r_dee_second}
\end{align}
where $(i)$ is because $R_\textsf{D}^2$ is convex, as it is the square of a positive, convex function, $(ii)$ is because $\E_{(i, k)}[\ ]$ and $\lip\ \rip_{i, k}$ commute, and $(iii)$ is the Law of Total Probability. We first bound the term inside the expectation of \eqref{eq:r_dee_second} as 
\begin{align}
    \E_{(i, k)}\l[R_\textnormal{\textsf D}\l(e^{-\alpha\sum_{j\in\mathcal B_i}\omega_j\ell(\eta_j, U_j^\intercal\beta)}\r)^2\mathbb I_{\mathcal A_t^c}\r] &\stackrel{(i)}{\leq} \E\l[R_\textnormal{\textsf D}\l(e^{-\alpha\sum_{j\in\mathcal B_i}\omega_j\ell(\eta_j, U_j^\intercal\beta)}\r)^4\r]^{1/2}\P\l(\mathcal A_t^c\r)^{1/2}\nonumber\\
    &\siil \E\l[e^{4\alpha\sum_{j\in\mathcal B_i}\omega_j\ell(\eta_j, U_j^\intercal\beta)}\r]^{1/2} \textsf C_3\sqrt{k}e^{-ct^2}\nonumber\\
    &\siiil e^{\textsf C_4k^2\alpha^2}\textsf C_3\sqrt{k}e^{-ct^2}\nonumber\\
    &= \textsf C(k, \alpha)e^{-ct^2},\label{eq:addupmylove}
\end{align}
where $(i)$ is via Cauchy-Schwarz and block dependence, $(ii)$ is via \eqref{eq:c2kect2} and the fact that 
\begin{align*}
    Q_{\textsf D}(x) > 0 \implies R_{\textsf D}(x) = \frac{1}{x} - Q_{\textsf D}(x) < \frac{1}{x}
\end{align*}
for $x \in (0, 1)$, and $(iii)$ is via $\eqref{eq:see_six}$. Thus, if we choose $t$ sufficiently large, namely
\begin{align*}
    t > \sqrt{\frac{1}{c}\log\l(\frac{2\textsf C(k, \alpha)}{\tau}\r)},
\end{align*}
then \eqref{eq:addupmylove} yields that
\begin{align}
    \E_{(i, k)}\l[R_\textnormal{\textsf D}\l(e^{-\alpha\sum_{j\in\mathcal B_i}\omega_j\ell(\eta_j, U_j^\intercal\beta)}\r)^2\mathbb I_{\mathcal A_t^c}\r] \leq \frac\tau2.\label{eq:kokomo}
\end{align}
For \eqref{eq:r_dee_first}, we know that since the event $\mathcal A_t$ occurs in this case, we have 
\begin{align*}
    \mathcal A_t \implies \sum_{j\in\mathcal B_i}\omega_j\ell_j(\beta) \leq \sum_{j\in\mathcal B_i}\abs{U_j^\intercal\beta} + 1 \leq k(t+1),
\end{align*}
and so this forces that the argument of $R_{\textsf D}$ satisfies
\begin{align*}
    \exp\l(-\alpha \sum_{j\in\mathcal B_i}\omega_j\ell_j(\beta)\r) \in [e^{-\alpha k(t+1)}, 1]. 
\end{align*}
By definition of $Q_{\textsf D}(x)$ being the power series of $\frac1x$ with radius of convergence equal to 1, there must exist $\textsf D(k, \alpha, \tau)$ such that 
\begin{align*}
    \sup_{x\in [e^{-\alpha k(t+1)}, 1]}\abs{R_{\textsf D}(x)} \leq \sqrt{\frac{\tau}{2}}. 
\end{align*}
This means that the term inside the expectation of \eqref{eq:r_dee_first} may be bounded as 
\begin{align}
    \E_{(i, k)}\l[R_\textnormal{\textsf D}\l(e^{-\alpha\sum_{j\in\mathcal B_i}\omega_j\ell(\eta_j, U_j^\intercal\beta)}\r)^2\mathbb I_{\mathcal A_t}\r] \leq \frac{\tau}{2}\label{eq:onlyoneiknow},
\end{align}
and so the result follows from this choice of $\textsf D$ by combining \eqref{eq:kokomo} and \eqref{eq:onlyoneiknow}.

\end{proof}

\subsection{Properties of sub-Gaussian Vectors}
\begin{lemma}\label{cabris} Let $Y$ be a sub-Gaussian vector in $\mathbb{R}^d$ with sub-Gaussian constant $\textnormal{\textsf K}$. Write $\Sigma_Y:=\Var(Y)$. Then there exists $\textnormal{\textsf C}>0$ such that $$\|\Sigma_Y\|_{\textnormal{op}}\le \textnormal{\textsf C} \textnormal{\textsf K}^2. $$
\end{lemma}

\begin{proof}
Define $Z := Y - \E[Y]$, which by Lemma 2.6.8. of \cite{vershyninhighdimprob} is still sub-Gaussian with 
\begin{align*}
    \|Z\|_{\psi_2} \leq \textsf C_1\textsf{K}
\end{align*}
for some fixed $\textsf C_1 > 0$. Now, let $v\in \mathbb{R}^d$. We first know by \definitionref{ui} that since $Z$ is sub-Gaussian with constant $\textsf C_1\textsf{K}$, $Z^\intercal v$ must also be sub-Gaussian with constant at most $\textsf C_1\textsf{K}\|v\|$. We observe that 
\begin{align*}
    v^\intercal \Sigma_Yv \stackrel{(i)}{=} \Var(Z^\intercal v) \stackrel{(ii)}{\leq} \textsf C_2\textsf{K}^2\|v\|^2,
\end{align*}
where $(i)$ is because $Y$ and $Z$ share the same covariance matrix, and the inequality in $(ii)$ is via Proposition 2.5.2 of \cite{vershyninhighdimprob}. This lets us conclude that
\begin{align*}
    \frac{v^\intercal \Sigma_Yv }{\|v\|^2} \leq \textsf C_2\textsf{K}^2,
\end{align*}
and since this holds for all $v \in \R^d$, it holds for the supremum, which exactly defines the operator norm as $\Sigma_Y$ is necessarily positive semi-definite.
\end{proof}

\section{Proofs for the dependent CGMT} \label{appendix:proof:CGMT}

In this section, we prove \theoremref{thm:CGMT:full}, which recovers \theoremref{thm:CGMT_2} directly, and \corollaryref{cor:CGMT}. The proof recipe is similar to that of a standard CGMT: We start by proving a Gaussian min-max theorem (GMT) on discrete sets in \lemmaref{lem:GMT:discrete}, proceed to extend it to compact sets in \lemmaref{lem:GMT:compact}, and then prove the results in \theoremref{thm:CGMT:full}. \corollaryref{cor:CGMT} then follows directly from \theoremref{thm:CGMT:full}(ii).

\vspace{.5em}

As with the standard CGMT, the Gaussian min-max theorem (GMT) on discrete sets is proved for a surrogate optimization problem. Let $(\xi_l)_{l \leq M}$ be a collection of univariate standard Gaussians independent of $\bH$, and define 
\begin{align*}
    \Psi^\xi_{\cS_w, \cS_u}
    \;\coloneqq&\; 
    \min_{w \in \cS_w} \, \max_{u \in \cS_u} \, L^\xi_\Psi(w,u)
    \;,
    \qquad 
    \text{ where }
    \quad 
    L^\xi_\Psi(w,u) 
    \;\coloneqq\; 
    w^\intercal \bH u + \msum_{l=1}^M  \xi_l \, \| w \|_{\Sigma^{(l)}}  \| u \|_{\tilde \Sigma^{(l)}}
    + f(w,u)
    \;.
\end{align*}
We also recall the risk  $\psi_{\cI_p, \cI_n}$ of the auxiliary optimization defined in \theoremref{thm:CGMT:full}.

\begin{lemma}[GMT on discrete sets] \label{lem:GMT:discrete} Let $\cI_p \subseteq \R^p$, $\cI_n \subseteq \R^n$ be discrete sets, and $f$ be finite on $\cI_p \times \cI_n$. Then for all $c \in \R$,
    \begin{align*}
        \P \big( \Psi^\xi_{\cI_p, \cI_n} \geq c  \big) 
        \;\geq\; 
        \P \big( \psi_{\cI_p, \cI_n}  \geq c \big) 
        \;.
    \end{align*}    
\end{lemma}

\begin{proof}[Proof of \lemmaref{lem:GMT:discrete}] Similar to the proof for the standard GMT (see e.g.~proof of Lemma A.1.1~of \citet{thrampoulidis2016recovering}), the proof relies on an application of Gordon's Gaussian comparison inequality (see e.g.~Corollary 3.13~of \citet{ledoux1991probability}) applied to two suitably defined Gaussian processes. Consider the two centred Gaussian processes indexed on the set $\cI_p \times \cI_n$:
\begin{align*}
    Y_{w, u} \;\coloneqq&\; w^\intercal \bH u + \msum_{l=1}^M  \xi_l \, \| w \|_{\Sigma^{(l)}}  \| u \|_{\tilde \Sigma^{(l)}}\;,
    \\
    X_{w, u} \;\coloneqq&\;  \msum_{l=1}^M \big( 
        \| w \|_{\Sigma^{(l)}} \bh_l^\intercal \big( \tilde \Sigma^{(l)} \big)^{1/2} u + w^\intercal \big( \Sigma^{(l)} \big)^{1/2} \bg_l \| u \|_{\tilde \Sigma^{(l)}}
    \big) 
    \;.
\end{align*}
To compare their second moments, we use the independence of $\bH$ and $\{\xi_l\}_{l \leq M}$ as well as the independence of $(\bh_l, \bg_l)_{l \leq M}$: For $w, w' \in \cI_p$ and  $u, u' \in \cI_n$, we have
\begin{align*}
    \mean[ Y_{w, u} Y_{w', u'} ]
    -
    \mean[ X_{w, u} X_{w', u'} ]
    &\;\overset{(a)}{=}\;
    \mean[ w^\intercal \bH u (w')^\intercal \bH  u' ]
    +
    \msum_{l=1}^M  \| w \|_{\Sigma^{(l)}} \, \| w' \|_{\Sigma^{(l)}} \,  \| u \|_{\tilde \Sigma^{(l)}} \,  \| u' \|_{\tilde \Sigma^{(l)}}
    \\
    &\qquad 
    -
    \msum_{l=1}^M \big( 
        \| w \|_{\Sigma^{(l)}} \, \| w' \|_{\Sigma^{(l)}} \, u^\intercal  \tilde \Sigma^{(l)} u'  
        +
        w^\intercal \Sigma^{(l)} w' 
        \, \| u \|_{\tilde \Sigma^{(l)}} \, \| u' \|_{\tilde \Sigma^{(l)}}
    \big)
    \\
    &\;\overset{(b)}{=}\;
    \msum_{l=1}^M 
    \Big( 
        w^\intercal \Sigma^{(l)} w' \, u^\intercal  \tilde \Sigma^{(l)} u' 
        +
        \| w \|_{\Sigma^{(l)}} \, \| w' \|_{\Sigma^{(l)}} \,  \| u \|_{\tilde \Sigma^{(l)}} \,  \| u' \|_{\tilde \Sigma^{(l)}}
    \\
    &\qquad \qquad \qquad
        \;-\,
        \| w \|_{\Sigma^{(l)}} \, \| w' \|_{\Sigma^{(l)}} \, u^\intercal  \tilde \Sigma^{(l)} u'  
        \,-\,
        w^\intercal \Sigma^{(l)} w' 
        \| u \|_{\tilde \Sigma^{(l)}} \| u' \|_{\tilde \Sigma^{(l)}}
    \Big)
    \\
    &\;=\;
    \msum_{l=1}^M 
    \big(   \| w \|_{\Sigma^{(l)}} \, \| w' \|_{\Sigma^{(l)}} -  w^\intercal \Sigma^{(l)} w' \big)
    \big(   \| u \|_{\tilde \Sigma^{(l)}} \, \| u' \|_{\tilde \Sigma^{(l)}} -  u^\intercal \tilde \Sigma^{(l)} u' \big)
    \;.
    \tagaligneq \label{eq:GMT:tmp}
\end{align*} 
In $(a)$, we have used that $\xi_l$'s, $\bh_l$'s and $\bg_l$'s are all standard Gaussians; in $(b)$, we have used 
\begin{align*}
    \mean[ w^\intercal \bH u (w')^\intercal \bH  u' ]
    \;=&\;
    \msum_{i,i'=1}^n \msum_{j,j'=1}^p \, w_i w'_{i'} u_j u'_{j'} \, \mean[ H_{ij} H_{i'j'} ]
    \\
    \;=&\;
    \msum_{l=1}^M
    \msum_{i,i'=1}^n \msum_{j,j'=1}^p \, w_i \Sigma^{(l)}_{ii'} w'_{i'} u_j \tilde \Sigma^{(l)}_{jj'} u'_{j'} 
    \\
    \;=&\;
    \msum_{l=1}^M w^\intercal \Sigma^{(l)} w' \, u^\intercal  \tilde \Sigma^{(l)} u'  \;.
\end{align*}
By the positive semi-definiteness of $\Sigma^{(l)}$ and $\tilde \Sigma^{(l)}$, \eqref{eq:GMT:tmp} is non-negative, and equals to zero when $w = w'$. This shows that the Gaussian processes $(Y_{w, u})_{w \in \cI_p,u \in \cI_n}$ and $(X_{w, u})_{w \in \cI_p,u \in \cI_n}$ verify the conditions of the Gaussian comparison inequality (Corollary 3.13~of \citet{ledoux1991probability}) and therefore for any real sequence $(\lambda_{w,u})_{w \in \cI_p,u \in \cI_n}$,
\begin{align*}
    \P\big( \cap_{w \in \cI_p} \cup_{v \in \cI_n}  \, \{ Y_{w,u} \geq \lambda_{w,u} \} \big)
    \;\geq\; 
    \P\big( \cap_{w \in \cI_p} \cup_{v \in \cI_n}  \, \{ X_{w,u} \geq \lambda_{w,u} \} \big)
    \;.
\end{align*}
Choosing $\lambda_{w,u} = - f(w,u) + c$ yields that 
\begin{align*}
    \P\Big( \min_{w \in \cI_p} \max_{v \in \cI_n}  \, ( Y_{w,u} + f(w,u) ) \geq c  \Big)
    \;\geq\; 
    \P\Big( \min_{w \in \cI_p} \max_{v \in \cI_n}  \, ( X_{w,u} + f(w,u) ) \geq c  \Big)
    \;.
\end{align*}
Noting that the two min-max quantities correspond to $\Psi^\xi_{\cI_p, \cI_n}$ and  $\psi_{\cI_p, \cI_n}$  concludes the proof.
\end{proof}

\vspace{.5em}

The next result extends \lemmaref{lem:GMT:discrete} to compact sets.

\begin{lemma}[GMT for compact sets] \label{lem:GMT:compact} Suppose $\cS_w \subset \R^{p}$ and $\cS_u \subset \R^{n}$ are compact and $f$ is continuous on $\cS_w \times \cS_u$. Then for all $c \in \R$, 
    \begin{align*}
        \P( \Psi^\xi_{\cS_w,\cS_u} \geq c ) \;\geq\; \P( \psi_{S_p,S_n} \geq c ) \;.
    \end{align*}
\end{lemma}

\begin{proof}[Proof of \lemmaref{lem:GMT:compact}] The proof is almost identical to the proof of standard GMT results for compact sets, now that we have established \lemmaref{lem:GMT:discrete}: We show by a compactness argument that both losses only change a little when replacing $\cS_w$ and $\cS_u$ by their $\delta$-nets $\cS^\delta_p$ and $\cS^\delta_n$, induced by the Euclidean norms on $\R^n$ and $\R^d$ respectively. The only difference from their proof is that we use a slightly different concentration inequality. Therefore we only set up the essential notation, highlight the differences and refer interested readers to the proof of Theorem 3.2.1 of~\citet{thrampoulidis2016recovering}, found in Pg 185-187.

\vspace{.5em}
    
First fix some $\epsilon > 0$. Since $f$ is continuous and thereby uniformly continuous on the compact set $\cS^\delta_p \times \cS^\delta_n$, there exists some $\delta = \delta(\epsilon) > 0$ such that for all $(w,u), (w',u') \in \cS_w \times \cS_u$ with $\| (w,u) - (w', u' )\| \leq \delta$, we have $\| f(w,u)- f(w', u') \| \leq \epsilon$. Use this $\delta$ to form the $\delta$-nets $\cS^\delta_p$ and $\cS^\delta_n$. We also write $\| \argdot \|_{op}$ as the operator norm of a matrix, and write
\begin{align*}
    S \;\coloneqq&\; \mmax_{1 \leq l \leq M} \max \{ \| \Sigma^{(l)} \|_{\rm op}  \,,\, \| \tilde \Sigma^{(l)} \|_{\rm op}  \} 
    &\text{ and }&&
    K \;\coloneqq&\; \max \big\{ \msup_{w \in \cS_w} \| w \|
    \,,\, 
    \msup_{u \in \cS_u} \| u \|
    \big\}
    \;.
\end{align*}
$K$ is bounded since $\cS_w$ and $\cS_u$ are compact, and for $w \in \cS_w$, $u \in \cS_u$ and $l \leq M$, we have 
\begin{align*}
    \| w \| \,\leq\, K\,,
    \qquad 
    \| w  \|_{\Sigma^{(l)}}  \,\leq\, S K\,, 
    \qquad 
    \| u \| \,\leq\, K\,,
    \qquad 
    \| u  \|_{\tilde \Sigma^{(l)}}  \,\leq\, S K\,.
\end{align*}
Then by the same argument as the proof of Theorem 3.2.1 of~\citet{thrampoulidis2016recovering}, there exists $w_1 \in \cS_w$, $w'_1 \in \cS^\delta_p$ with $\| w_1 - w'_1 \| \leq \delta$ and $u_1 \in \cS^\delta_n$ such that 
\begin{align*}
    \Delta^\xi_\Psi 
    \;\coloneqq&\; 
    \mmin_{w \in \cS^\delta_p} \, \mmax_{u \in \cS^\delta_n} \, L^\xi_\Psi(w,u)
    \,-\,
    \mmin_{w \in \cS_w} \, \mmax_{u \in \cS_u} \, L^\xi_\Psi(w,u)
    \\
    \;\leq&\;
    L^\xi_\Psi(w'_1,u_1)
    -
    L^\xi_\Psi(w_1,u_1)
    \;.
\end{align*}
Computing the difference gives 
\begin{align*}
    \Delta^\xi_\Psi 
    \;\leq&\;
    (w'_1 - w_1)^\intercal \bH u_1 
    + 
    \msum_{l=1}^M  \xi_l \, ( \| w'_1 \|_{\Sigma^{(l)}}  - \| w_1 \|_{\Sigma^{(l)}}  )  \| u_1 \|_{\tilde \Sigma^{(l)}}
    + 
    ( f(w'_1,u_1) - f(w_1, u_1) )
    \\
    \;\leq&\;
    \delta \| \bH \| K
    \, + \,
    S K \, \msum_{l=1}^M |\xi_l|  \, \| w'_1 - w_1 \|_{\Sigma^{(l)}}
    \,+\,
    | f(w'_1,u_1) - f(w_1, u_1) | 
    \\
    \;\leq&\;
    \delta K \| \bH\|
    +
    \delta S^2 K \msum_{l=1}^M |\xi_l| 
    +
    \epsilon
    \;.
\end{align*}
We seek to control $\| \bH\|$ and $\msum_{1=1}^M |\xi_l| $ via concentration inequalities. Let $\tvec(\bH)$ denote the $\R^{pn}$-valued vector formed from the entries of $\bH$, and $\Sigma_\bH \coloneqq  \Var[ \tvec(\bH)  ]$. Then we can express, for some $\R^{pn}$-valued standard Gaussian vector $\eta$,
\begin{align*}
    \| \bH \|^2 \;=\; \| \tvec(\bH) \|^2 \;=\; \eta^\intercal \, \Sigma_\bH \, \eta\;.
\end{align*}
Then by a Chernoff bound, we have that for any $t > 0$,
\begin{align*}
    \P( \| \bH \| \geq t ) 
    \;\leq\;
    \minf_{a > 0} e^{-at^2} \mean\big[ e^{a \| \bH\|^2} \big]
    \;=&\;
    \minf_{a > 0} e^{-at^2} \mean\big[ e^{a  \, \eta^\intercal  \Sigma_\bH  \eta } \big]
\end{align*}
Applying the formula of the moment-generating function of a Gaussian quadratic form (see e.g.~\citet{rencher2008linear}) followed by setting $a= \frac{1}{4 \| \Sigma_\bH \|_{op}}$, we obtain
\begin{align*}
    \P( \| \bH \| \geq t ) 
    \;\leq&\;    
    \minf_{a > 0} 
    \, \mfrac{e^{-at^2}  }{\sqrt{\textrm{det}( I_{pn} - 2 a \Sigma_\bH )}}
    \;\leq\;
    \mfrac{e^{-t^2 / (4 \| \Sigma_\bH \|_{op} ) }  }{\sqrt{\textrm{det}( I_{pn} - \frac{1}{2 \|\Sigma_\bH \|_{op} } \Sigma_\bH )}}
    \;\leq\;
    2^{pn/2} \, e^{-t^2 / (4 \| \Sigma_\bH \|_{op} ) }\;.
    \tagaligneq \label{eq:GMT:gaussian:norm:conc}
\end{align*}
On the other hand, a standard concentration result on univariate Gaussians yields 
\begin{align*}
    \P( | \xi_l | > t ) \;\leq\; 2 e^{-t^2/2} \;.
\end{align*}
Taking a union bound, we obtain that for any $t > 0$,
\begin{align*}
    \P\big( \,  \Delta^\xi_\Psi \,\leq\, \delta K t + \delta S^2 K M t + \epsilon \, \big)
    \;\geq\;
    1 - 2^{pn/2} \, e^{-t^2 / (4 \| \Sigma_\bH \|_{op} ) } - 2 M e^{-t^2/2}
    \;,
\end{align*}
and therefore for any $c \in \R$ and $t > 0$,
\begin{align*}
    &\;
    \P\big( \,  \mmin_{w \in \cS_w} \, \mmax_{u \in \cS_u} \, L^\xi_\Psi(w,u) \,\geq\, c - \delta K t - \delta S^2 K M t - \epsilon \, \big)
    \\
    &\qquad 
    \;\geq\;
    \P\big( \,  \mmin_{w \in \cS^\delta_p} \, \mmax_{u \in \cS^\delta_n} \, L^\xi_\Psi(w,u) \,\geq\, c \big)
     - 2^{pn/2} \, e^{-t^2 / (4 \| \Sigma_\bH \|_{op} ) } - 2 e^{-t^2/2}
    \;.
    \tagaligneq \label{eq:GMT:PO:conc}
\end{align*}
A similar argument as in the proof of Theorem 3.2.1 of~\citet{thrampoulidis2016recovering} shows that, there exists $w_2 \in \cS^\delta_p$, $u_2 \in \cS_w$ and $u_2' \in \cS^\delta_n$ with $\| u_2 - u'_2 \| \leq \delta$ such that 
\begin{align*}
    \min_{w \in \cS^\delta_p} \, \max_{u \in \cS^\delta_n} \, L_\psi(w,u)
    &\,-\,
    \min_{w \in \cS_w} \, \max_{u \in \cS_u} \, L_\psi(w,u)
    \;\geq\;
    L_\psi(w_2,u'_2)
    -
    L_\psi(w_2,u_2)
    \\
    &\;=\;
    \msum_{l=1}^M 
    \Big( 
        \| w_2 \|_{\Sigma^{(l)}} 
        \bh_l^\intercal 
        \big( \tilde \Sigma^{(l)} \big)^{1/2} 
        (u'_2 - u_2) 
        +
        w_2^\intercal 
        \big( \Sigma^{(l)} \big)^{1/2} 
        \bg_l 
        ( 
          \| u'_2 \|_{\tilde \Sigma^{(l)}}
          - 
          \| u_2 \|_{\tilde \Sigma^{(l)}}
        )
    \Big)
    \\
    &\qquad
    +
    ( f(w_2, u'_2) - f(w_2, u_2))
    \\
    &\;\geq\;
    - 
    \delta S^2 K \msum_{l=1}^M ( \| \bh_l \| + \| \bg_l \| )
    -
    \epsilon
    \;.
\end{align*}
Applying \eqref{eq:GMT:gaussian:norm:conc} to each $\| \bh_l \|$ and $\| \bg_l \|$ yields that, for any $t > 0$ and $1 \leq l \leq M$,
\begin{align*}
    \P( \| \bh_l \| \geq t ) \;\leq&\; 2^{n/2} e^{-t^2/4}
    &\text{ and }&&
    \P( \| \bg_l \| \geq t ) \;\leq&\; 2^{p/2} e^{-t^2/4}\;.
\end{align*}
Taking another union bound, we get that for any $t > 0$,
\begin{align*}
    &\; 
    \P(  \, \mmin_{w \in \cS_w} \, \mmax_{u \in \cS_u}  \, L_\psi(w,u) \,\geq\, c +     2 \delta S^2 K M t
    +
    \epsilon   \, ) 
    \\
    &\;\leq\;
    \P(  \, \mmin_{w \in \cS_w^\delta} \, \mmax_{u \in \cS_u^\delta}  \, L_\psi(w,u) \,\geq\, c \, )
    \,+\, 2^{n/2} M e^{-t^2/4}
    \,+\, 2^{p/2} M e^{-t^2/4}
        \;.
    \tagaligneq \label{eq:GMT:AO:conc}
\end{align*}
Now by \lemmaref{lem:GMT:discrete}, we have 
\begin{align*}
    \P(  \, \mmin_{w \in \cS_u^\delta} \, \mmax_{u \in \cS_d^\delta}  \, L_\psi(w,u) \,\geq\, c \, )
    \;\leq\; 
    \P\big( \,  \mmin_{w \in \cS^\delta_p} \, \mmax_{u \in \cS^\delta_n} \, L^\xi_\Psi(w,u) \,\geq\, c \big)\;.
\end{align*} 
Combining this with \eqref{eq:GMT:PO:conc} and \eqref{eq:GMT:AO:conc} yields 
\begin{align*}
    &\; 
    \P(  \, \mmin_{w \in \cS_u} \, \mmax_{u \in \cS_d}  \, L_\psi(w,u) \,\geq\, c +     2 \delta S^2 K M t
    +
    \epsilon   \, ) 
    \\
    &\;\leq\;
    \P\big( \,  \mmin_{w \in \cS_u} \, \mmax_{u \in \cS_d} \, L^\xi_\Psi(w,u) \,\geq\, c - \delta K t - \delta S^2 K M t - \epsilon \, \big)
    \\
    &\qquad 
    + 
    2^{n/2} M e^{-t^2/4}
    + 
    2^{p/2} M e^{-t^2/4}
    +
    2^{np/2} \, e^{-t^2 / (4 \| \Sigma_\bH \|_{op} ) } 
    + 
    2 e^{-t^2/2}
    \;.
\end{align*}
The above holds for all $\epsilon > 0$ and $t > 0$. Set $t = \delta^{-1/2}$, take $\epsilon \rightarrow 0$ and choosing a sequence $\delta(\epsilon) \rightarrow 0$, we obtain that 
\begin{align*}
    \P(  \, \mmin_{w \in \cS_w} \, \mmax_{u \in \cS_u}  \, L_\psi(w,u) \,\geq\, c   \, ) 
    \;\leq\;
    \P\big( \,  \mmin_{w \in \cS_w} \, \mmax_{u \in \cS_u} \, L^\xi_\Psi(w,u) \,\geq\, c \, \big)
    \;,
\end{align*}
i.e.~$\P( \Psi^\xi_{\cS_w,\cS_u} \geq c ) \;\geq\; \P( \psi_{S_p,S_n} \geq c )$.
\end{proof}

We are now ready to prove \theoremref{thm:CGMT:full} and \corollaryref{cor:CGMT}.

\vspace{.5em}

\begin{proof}[Proof of \theoremref{thm:CGMT:full}] The proof is almost identical to the proof of Theorem 3.3.1 of \citet{thrampoulidis2016recovering} given the GMT result from \lemmaref{lem:GMT:compact}, and we focus on highlighting the differences. To prove the first bound in (i), we first apply \lemmaref{lem:GMT:compact} to obtain that for all $c \in \R$,
\begin{align*}
    \P\big(  \min_{w \in \cS_n} \, \max_{u \in \cS_d} \, L_\Psi(w,u) + \msum_{l=1}^M  \xi_l \, \| w \|_{\Sigma^{(l)}}  \| u \|_{\tilde \Sigma^{(l)}} \leq c \big) \;\leq\; \P( \psi_{S_n,S_d} \leq c ) \;,
\end{align*}
where $(\xi_l)_{l \leq M}$ is a collection of univariate standard Gaussians independent of $\bH$. First notice that, by conditioning on the event $\cap_{l \leq M} \{\xi_l \geq 0 \}$, we have that
\begin{align*}
    \P(  \Psi_{\cS_p, \cS_n} \leq c )  \;=&\;
    \P\big(  \min_{w \in \cS_n} \, \max_{u \in \cS_d} \, L_\Psi(w,u)  \leq c  \big)
    \\
    \;\leq&\;
    \P\big(  \min_{w \in \cS_n} \, \max_{u \in \cS_d} \, L_\Psi(w,u) + \msum_{l=1}^M  \xi_l \, \| w \|_{\Sigma^{(l)}}  \| u \|_{\tilde \Sigma^{(l)}} \leq c  \,\big|\, 
    \xi_1, \ldots, \xi_M \leq 0 \big) 
\end{align*}
which holds almost surely. Since $\xi_l$'s are all independent and symmetric about zero, and there are $2^M$ possibilities for the signs of $(\xi_1, \ldots, \xi_M)$, we obtain that
\begin{align*}
    \mfrac{1}{2^M}
    \P(  \Psi_{\cS_p, \cS_n} \leq c )  
    \;\leq&\; 
    \mfrac{1}{2^M}
    \P\big(  \min_{w \in \cS_n} \, \max_{u \in \cS_d} \, L_\Psi(w,u) + \msum_{l=1}^M  \xi_l \, \| w \|_{\Sigma^{(l)}}  \| u \|_{\tilde \Sigma^{(l)}} \leq c  \,\big|\, 
    \xi_1, \ldots, \xi_M \leq 0 \big) 
    \\
    \;\leq&\;
    \P\big(  \min_{w \in \cS_n} \, \max_{u \in \cS_d} \, L_\Psi(w,u) + \msum_{l=1}^M  \xi_l \, \| w \|_{\Sigma^{(l)}}  \| u \|_{\tilde \Sigma^{(l)}} \leq c \big)
    \\
    \;\leq&\; \P( \psi_{S_n,S_d} \leq c ) \;,
\end{align*}
which gives the desired statement.

\vspace{.5em}
    
The proof of the bound in (ii) is exactly the same as the proof of Theorem 3.3.1(ii) of \citet{thrampoulidis2016recovering}: It relies on the ability to apply a min-max theorem or a min-max inequality for swapping minimum and maximum under the stated convex-concave assumptions, as well as the invariance of the random term of the loss under a sign change. Both hold for our losses $L_\Psi$ and $L_\psi$, since $\bH$ in our $L_\Psi$ is still zero-mean Gaussian, $L_\psi$ is a linear sum of independent mean-zero Gaussian terms and all additional matrices $\Sigma^{(l)}$ and $\tilde \Sigma^{(l)}$ are positive semi-definite. We refer readers to the proof of Theorem 3.3.1(ii) of \citet{thrampoulidis2016recovering} for a detailed derivation, and note that the only difference in our result is in that the coefficient from the first bound in (i) is now $2^M$ instead of $2$.

\vspace{.5em}

The proof of (iii) is also exactly the same as the proof of Theorem 3.3.1(iii) of \citet{thrampoulidis2016recovering}, which only relies on the three assumptions, the statements (i) and (ii) proved above and a union bound. We again refer readers to the proof of Theorem 3.3.1(iii) of \citet{thrampoulidis2016recovering} for a detailed derivation.
\end{proof}

\begin{proof}[Proof of \corollaryref{cor:CGMT}] The result follows directly from \theoremref{thm:CGMT:full}(ii); see Corollary 3.3.2~of \citet{thrampoulidis2016recovering}. 
    
\end{proof}

\section{Intermediate results for applying CGMT to data augmentation}  \label{appendix:DA:cgmt:results}

For clarity, throughout \cref{appendix:DA:cgmt:results,appendix:proof:DA}, we will index all augmentations as $\phi_{ij}$ where $i \leq m$, the number of original data, and $j \leq k$, the number of augmentations. Recall that $n=mk$. We also write the label of $\phi_{ij}(Z_i)$ as $y_i(Z_i)$ to emphasize the dependence on the original data $Z_i$.

\subsection{Equivalence of different optimization problems} \label{appendix:DA:optimizations}

To prove \theoremref{thm:DA:risk}, we seek to obtain a set of deterministic equations whose solutions characterize the high-dimensional behavior of logistic regression estimate. This involves establishing the equivalence of a series of optimization problems, which are defined in this section. We also formally state all lemmas used to establish the equivalence. 

\vspace{.5em}

\textbf{Original optimization (OO). } The loss on the augmented data computed on $\beta \in \R^p$ is given as
\begin{align}
    L_\beta(\bX, \bX^\Phi ) \;\coloneqq\;
    \mfrac{1}{mk} \msum_{i=1}^m \msum_{j=1}^k\, 
    \Big( 
        \log \big( 1 + e^{( \phi_{ij} (Z_i))^\intercal \beta} \big) 
        -
        y_i(Z_i) \times ( \phi_{ij}(Z_i))^\intercal \beta
    \Big)
    +    
    \mfrac{\lambda}{2n} \| \beta \|^2_2 
    \;.
\end{align}
Here, the loss is computed on the two dependent $\R^{m \times p}$ and $\R^{mk \times p}$-valued data matrices 
\begin{align*}
    \bX \;\coloneqq\; \begin{psmallmatrix}
        \leftarrow Z_1^\intercal  \rightarrow \\
        \vdots \\
        \leftarrow Z_m^\intercal \rightarrow \\
   \end{psmallmatrix}
   \quad\text{ and }\quad
    \bX^\Phi \;\coloneqq\; \begin{psmallmatrix}
         \bX_1^\Phi \\
         \vdots \\
         \bX_m^\Phi \\
    \end{psmallmatrix}
    \qquad 
    \text{ where }
    \qquad 
    \bX^\Phi_i \;\coloneqq\; 
    \begin{psmallmatrix}
        \leftarrow (\phi_{i1}(Z_i) )^\intercal \rightarrow \\
        \vdots \\
        \leftarrow (\phi_{ik}(Z_i) )^\intercal \rightarrow \\
   \end{psmallmatrix}
    \;.
\end{align*}
Let $S$ be any convex and compact subset of $\R^p$. We denote the minimized risk over $S$ and the corresponding minimizer respectively as 
\begin{align}
    \hat R_S(\bX, \bX^\Phi)
    \;\coloneqq&\; 
    \min_{\beta \in S} L_\beta(\bX,\bX^\Phi)
    &\text{ and }&&
    \hat \beta_S(\bX,\bX^\Phi)
    \;\coloneqq&\;
    \argmin_{\beta \in S} L_\beta(\bX,\bX^\Phi)
    \;.
    \tag{OO}\label{OO}
\end{align}
We label \eqref{OO} as the \textbf{original optimization}. By our universality result, we may replace the dependent data matrices in \eqref{OO} by Gaussian matrices.

\vspace{.5em}

\textbf{Gaussian optimization (GO).}  Recall that $\Sigma_o = \Var[Z_1]$ and $\Sigma = \Var[\phi_{11}(Z_1)]$. We denote the corresponding minimized risk under Gaussian data as
\begin{align}
    \hat R_S(\bG \Sigma_o^{1/2}, \bG^\Phi \Sigma^{1/2})
    \;\coloneqq&\; 
    \mmin_{\beta \in S} L_\beta(\bG \Sigma_o^{1/2}, \bG^\Phi \Sigma^{1/2})\;,
    \notag
    \\
    \text{ and }\qquad
    \hat \beta_S(\bG \Sigma^{1/2}, \bG^\Phi \Sigma_o^{1/2})
    \;\coloneqq&\;
    \argmin_{\beta \in S} L_\beta(\bG \Sigma^{1/2}, \bG^\Phi \Sigma_o^{1/2})
    \;.
    \tag{GO} \label{GO}
\end{align}
The risk is computed on the two correlated Gaussian matrices 
\begin{align*}
    \bG \;\coloneqq\; \begin{psmallmatrix}
        \leftarrow G_1^\intercal  \rightarrow \\
        \vdots \\
        \leftarrow G_m^\intercal \rightarrow \\
   \end{psmallmatrix}
   \quad\text{ and }\quad
    \bG^\Phi \;\coloneqq\; \begin{psmallmatrix}
         \bG_1^\Phi \\
         \vdots \\
         \bG_m^\Phi \\
    \end{psmallmatrix}
    \qquad 
    \text{ where }
    \qquad 
    \bG^\Phi_i \;\coloneqq\; 
    \begin{psmallmatrix}
        \leftarrow (G^\Phi_{i1})^\intercal \rightarrow \\
        \vdots \\
        \leftarrow (G^\Phi_{ijk})^\intercal \rightarrow \\
   \end{psmallmatrix}
    \;,
\end{align*}
where $\Sigma_o^{1/2} G_i$ corresponds to $Z_i$, $\bG^\Phi_i \Sigma^{1/2}$ corresponds to $\bX^\Phi_i$ and $\Sigma^{1/2} G^\Phi_{ij}$ corresponds to $\phi_{ij}(Z_i)$, and
\begin{align*}
    \mean[(\bG, \bG^\Phi)] \,=&\, \mean[(\bX, \bX^\Phi) ] \,=\, 0
    &\text{ and }&&
    \Var[(\bG \Sigma_o^{1/2}, \bG^\Phi \Sigma^{1/2})] \,=&\, \mean[(\bX, \bX^\Phi) ]
    \;.
\end{align*}

\vspace{.5em}

\textbf{Primary optimization (PO).} Since \eqref{GO} only depends on Gaussian data, we may adapt the CGMT technique to analyze its limiting behaviour. This requires a reformulation of \eqref{GO} in a similar way to the reformulation of the primary optimization in \citet{salehi2019impact}. To make this reformulation precise, we introduce some more notations. Given an $\R^{mk}$-valued vector $\bv$, we denote
\begin{align*}
    \rho(\bv) \;\coloneqq&\; 
    \big( 
        \log(1 + e^{v_{11}}) 
        \,,\, \ldots \,,\, 
        \log(1 + e^{v_{mk}}) 
    \big)^\intercal 
    \;\in\; \R^{mk}
    \;.
\end{align*}
Also write the $\R^{mk}$-valued vector of labels for (the Gaussian surrogates for) the augmented data as
\begin{align*}
    \by(\bG \Sigma_o^{1/2} \beta^*) \;\coloneqq\; \big( \,  \underbrace{y_1( \Sigma_o^{1/2} G_1), \ldots, y_1( \Sigma_o^{1/2} G_1)}_{\textrm{repeated $k$ times}}, \; \ldots, \; \underbrace{y_m( \Sigma_o^{1/2} G_m), \ldots, y_m(\Sigma_o^{1/2} G_m)}_{\textrm{repeated $k$ times}} \, \big)^\intercal \;,
\end{align*}
where we highlight that $\by$ depends on $\bG$ only through the $\R^n$ vector $\bG \Sigma_o^{1/2} \beta^*$. For $d \in \N$, we also write $\bone_d$ as the all-one vector in $\R^d$.This allows us to rewrite the loss in \eqref{GO} as 
\begin{align*}
    L_\beta(\bG \Sigma^{1/2}_o, \bG^\Phi \Sigma^{1/2})
    \;=&\;
    \mfrac{1}{mk} \bone_{mk}^\intercal \, \rho( \bG^\Phi \Sigma^{1/2} \beta )
    -
    \mfrac{1}{mk} \by(\bG \Sigma_o^{1/2} \beta^*)^\intercal \bG^\Phi \Sigma^{1/2} \beta
    +    
    \mfrac{\lambda}{2n} \| \beta \|^2_2 
    \;.
\end{align*}
Introducing a new variable $u \in \R^{mk}$ and a corresponding Lagrange multiplier $v \in \R^{mk}$, we can consider an alternative loss 
\begin{align*}
    L^{\rm PO}_{\beta, u, v}(\bG \Sigma^{1/2}_o \beta^*, \bG^\Phi \Sigma^{1/2}) 
    \;\coloneqq&\;
    \mfrac{1}{mk} \bone_{mk}^\intercal \, \rho( u )
    -
    \mfrac{1}{mk} \by(\bG \Sigma^{1/2}_o \beta^*)^\intercal u
    +    
    \mfrac{\lambda}{2n} \| \beta \|^2_2 
    +
    \mfrac{1}{mk} v^\intercal ( u - \bG^\Phi \Sigma^{1/2} \beta )
    \;.
\end{align*}
For subsets $S \subseteq \R^p$ and $S_u, S_v \subseteq \R^{mk}$, we denote the minimized loss and the minimizer as 
\begin{align}
    R^{\rm PO}_{S, S_u, S_v}(\bG \Sigma^{1/2}_o \beta^*, \bG^\Phi \Sigma^{1/2}) 
    \;\coloneqq&\; 
    \min_{\beta \in S, u \in S_u} \, \max_{v \in S_v} L^{\rm PO}_{\beta, u, v}(\bG \Sigma^{1/2}_o \beta^*, \bG^\Phi \Sigma^{1/2}) 
    \notag
    \\
    \text{ and }\qquad
    \beta^{\rm PO}_{S, S_u, S_v}(\bG \Sigma^{1/2}_o \beta^*, \bG^\Phi \Sigma^{1/2}) 
    \;\coloneqq&\;
    \argmin_{\beta \in S}\, \min_{u \in S_u} \, \max_{v \in S_v}  L^{\rm PO}_{\beta, u, v}(\bG \Sigma^{1/2}_o \beta^*, \bG^\Phi \Sigma^{1/2}) 
    \;.
    \tag{PO} \label{PO}
\end{align}

\begin{lemma}[Equivalence of \eqref{GO} and \eqref{PO}] \label{lem:GO:PO}  
\begin{align*}
    \hat R_S(\bG \Sigma^{1/2}_o,\bG^\Phi \Sigma^{1/2}) \;=&\; R^{\rm PO}_{S, \R^{mk}, \R^{mk}}(\bG \Sigma^{1/2}_o \beta^*, \bG^\Phi \Sigma^{1/2})\;,
    \\
    \hat \beta_S(\bG \Sigma^{1/2}_o,\bG^\Phi \Sigma^{1/2})  \;=&\; \beta^{\rm PO}_{S,\R^{mk}, \R^{mk}}(\bG \Sigma^{1/2}_o \beta^*, \bG^\Phi \Sigma^{1/2})
    \;.
\end{align*}
\end{lemma}

\begin{proof}[Proof of \lemmaref{lem:GO:PO}] The proof is exactly the same to the reformulation of the primary optimization in \citet{salehi2019impact} by the Lagrange multiplier method, and we refer readers to their (37) -- (40) in Appendix C for the proof. 
\end{proof}

\vspace{.5em}

\textbf{Auxiliary optimization (AO).}  Before we present the auxiliary optimization, we notice that two key issues make our problem more complicated from the setup in \citet{salehi2019impact}:
\begin{itemize}
    \item In \citet{salehi2019impact}, they have the same data matrices for $\bG$ and $\bG^\Phi $ with i.i.d.~standard normal entries and $\Sigma_o = \Sigma = I_d$. This allows them to project $\bG^\Phi$ onto the subspace orthogonal to $\beta^*$, which is independent of $\bG \beta^*$, and apply CGMT. In our case, $\bG$ and $\bG^\Phi$ are different and have non-trivial dependence. We instead make use of a projection $P^\perp_*$  adapted to the variance-covariance structure in \Cref{assumption:CGMT:var}, defined through 
    \begin{align*}
        P_* \;\coloneqq&\; 
        \begin{cases}
            \mfrac{(\Sigma_* \Sigma_o^{1/2} \beta^* )  (\Sigma_* \Sigma_o^{1/2} \beta^* )^\intercal }{\|  \Sigma_*  \Sigma_o^{1/2} \beta^*  \|^2}
            & \text{ if }    \Sigma_*  \Sigma_o^{1/2} \beta^* \neq 0 \\
            0 & \text{ otherwise }\;,
        \end{cases}
        &\text{ and }&&
        P^\perp_* 
        \;\coloneqq&\; 
        I_p - P_* \;.
    \end{align*}
    In other words, $P^\perp_*$ is a projection onto the subspace orthogonal to $\Sigma_* \Sigma_o^{1/2} \beta^*$. This is explicitly addressed in \Cref{appendix:DA:PO:AO};
    \item As discussed in \Cref{sec:cgmt}, the Gaussian matrix handled by existing work on CGMT is either one with i.i.d.~coordinates, or one formed by multiplying a coordinate-wise i.i.d.~matrix by an $\R^{m \times m}$ matrix and an $\R^{p \times p}$ matrix from both sides. Our augmented matrix, $\bG^\Phi$, cannot be expressed in either form due to the simultaneous presence of two forms of variances: Each row of $\bG^\Phi$ admits a variance of $I_p$, whereas the rows corresponding to different augmentations of the same data admit a variance of $\Sigma_*$. We resolve this issue by applying our dependent CGMT (\theoremref{thm:CGMT_2}) with $M=2$.
\end{itemize}
Having addressed these two issues, we are able to borrow most of the algebraic calculations from \citet{salehi2019impact} for analyzing the auxiliary optimization, except that the limiting terms we obtain are different due to augmentations.

\vspace{.5em}

To state the auxiliary optimization, let $\bg_1, \bg_2, \bh_1, \bh_2$ be independent standard Gaussians such that $\bg_l$'s are $\R^p$-valued and $\bh_l$'s are $\R^{mk}$-valued. We also denote the collections $\bg = (\bg_1, \bg_2)$ and $\bh = (\bh_1, \bh_2)$ for short, and define the matrices
\begin{align*}
    &
    \;
    \Sigma_1 
    \;\coloneqq\;
    \Sigma^{1/2} P_*^\perp (I_p - \Sigma_*) P_*^\perp \Sigma^{1/2}
    \;,
    \quad 
    \Sigma_2 
    \;\coloneqq\;
    \Sigma^{1/2} P_*^\perp \Sigma_*  P_*^\perp \Sigma^{1/2}\;,
    \quad
    J_{mk}
    \;\coloneqq\;
    \begin{psmallmatrix}
        \bone_{k \times k} & & \\
        & \ddots & \\
        & & \bone_{k \times k} \\
    \end{psmallmatrix}
    \in \R^{mk \times mk}
    \;.
\end{align*}
The loss of \eqref{AO},  parameterized by $\beta \in \R^p$ and $u, v \in \R^{mk}$, is given as
\begin{align*}
    L^{\rm AO}_{\beta,u,v}(\by, \bG^\Phi  P_*, \bg, \bh)
    \;\coloneqq&\;
    \mfrac{1}{mk} \bone_{mk}^\intercal \, \rho( u )
    -
    \mfrac{1}{mk} \by^\intercal u
    +    
    \mfrac{\lambda}{2n} \| \beta \|^2_2 
    +
    \mfrac{1}{mk} v^\intercal ( u - \bG^\Phi P_* \Sigma^{1/2} \beta )
    -
    \mfrac{1}{mk} 
    v^\intercal \bh_1 \| \beta \|_{\Sigma_1}
    \\
    &\;
    -
    \mfrac{1}{mk} 
    \| v \| \bg_1^\intercal \Sigma_1^{1/2} \beta
    -
    \mfrac{1}{mk^{3/2}} 
    v^\intercal J_{mk} \bh_2 \| \beta \|_{\Sigma_2}
    -
    \mfrac{1}{mk} 
    \| v \|_{J_{mk}} \bg_2^\intercal \Sigma_2^{1/2} \beta
    \;.
\end{align*}
Note that we have abbreviated $\by = \by(\bG \Sigma_o^{1/2} \beta^*)$. We also denote the minimized loss with respect to the subset $(S, S_u, S_v) \subseteq \R^p \times \R^{mk} \times \R^{mk}$ as 
\begin{align}
    R^{\rm AO}_{S, S_u, S_v}(\by, \bG^\Phi  P_*, \bg, \bh)
    \;\coloneqq&\; 
    \min_{\beta \in S, u \in S_u} \, \max_{v \in S_v} L^{\rm AO}_{\beta,u,v}(\by, \bG^\Phi  P_*, \bg, \bh)
    \;.
    \tag{AO} \label{AO}
\end{align}
The next result applies \theoremref{thm:CGMT_2} to convert \eqref{PO} into \eqref{AO}.

\begin{lemma}[Equivalence of \eqref{PO} and \eqref{AO}] \label{lem:PO:AO} Suppose \Cref{assumption:CGMT:var} holds. Let $S \subset \R^p$ and $S_u, S_v \in \R^{mk}$ be compact, convex and non-empty. Then all conclusions of \theoremref{thm:CGMT_2} hold with $\Psi_{\cS_p, \cS_n}$ replaced by $R^{\rm PO}_{S, S_u, S_v}(\bG \Sigma^{1/2}_o \beta^*, \bG^\Phi \Sigma^{1/2})$ and $\psi_{\cS_p, \cS_m}$ replaced by $ R^{\rm AO}_{S, S_u, S_v}(\by, \bG^\Phi  P_*, \bg, \bh) $.
\end{lemma}

\vspace{.5em}

\textbf{Scalar optimization (SO).} 
The next step is to convert \eqref{AO} into a scalar formulation. For convenience we write $\| \argdot \| = \| \argdot \|_2$ as the Euclidean norm throughout this section, unless otherwise specified. While the form of the optimization is complicated, we note that the terms are largely similar to the AO in \citet{salehi2019impact}, except for additional parameters $(\sigma_1, \nu_1, r_1, \tau_1)$ introduced to handle the additional covariance across different augmented versions of the same data.  To define the scalar formulation, given the convex compact and non-empty subsets $S \subset \R^p$, $S_u, S_v \in \R^{mk}$, we define the following compact domains of optimization:
\begin{align*}
    &
    S_{r_1}
    \;\coloneqq\; 
    \Big\{ \mfrac{1}{\sqrt{mk}} \| P^\perp_{mk} v \| \,\Big|\, v \in S_v \Big\}
    \;,
    \qquad 
    S_{r_2}
    \;\coloneqq\; 
    \Big\{ \mfrac{1}{\sqrt{mk}} \| P_{mk} v \| \,\Big|\, v \in S_v \Big\}\;,
\end{align*}
where we have defined the projection matrices $P_{mk} \coloneqq \frac{1}{k} J_{mk}$ and write $P^\perp_{mk} = I_{mk} - P_{mk}$. Also define 
\begin{align*}
    S^\alpha \;\coloneqq\; 
    \Big\{ \mfrac{v(\beta^*)^\intercal \Sigma^{1/2} \beta}{\sqrt{p} \, \kappa^2_*} \,\Big| \, \beta \in S \Big\} 
    \;,
\end{align*}
where $v(\beta^*) \coloneqq \sqrt{p} \, \Sigma_* \Sigma_o^{1/2} \beta^*$, $\kappa_* = \| \Sigma_* \Sigma_o^{1/2} \beta^* \|$. Define
\begin{align*}
    S_{\sigma_1} 
    \;\coloneqq\; 
    \Big\{ \| (I_p - \Sigma_*) P_*^\perp \Sigma^{1/2} \beta \| \, \big|  \, \beta \in S  \Big\}
    \;,
    \qquad 
    S_{\sigma_2}
    \;\coloneqq\; 
    \Big\{ \|\Sigma_* P_*^\perp \Sigma^{1/2} \beta \| \, \big|  \, \beta \in S  \Big\}
    \;.
\end{align*}
The optimization will be performed over the two $\R^5$-valued vectors 
\begin{align*}
    \tilde \alpha 
    \;\coloneqq&\; 
    (\alpha, \sigma_1, \sigma_2, \nu_1, \nu_2) \;\in\; S^\alpha \times S_{\sigma_1} \times S_{\sigma_2} \times (\R^+_0)^2
    \;\coloneqq\; 
    S_1
    \;,
    \\
    \tilde \theta 
    \;\coloneqq&\; 
    (r_1, r_2, \tau_1, \tau_2, \theta) \;\in\; S_{r_1} \times S_{r_2} \times (\R^+_0)^2 \times \R
    \;\coloneqq\;
    S_2
    \;.
\end{align*}
We also define $P_\Sigma = (\Sigma^\dagger)^{1/2} \Sigma^{1/2}$, the projection onto the positive eigenspace of $\Sigma$, and the matrix 
\begin{align*}
    \tilde \Sigma_{\sigma, \tau}
    \;\coloneqq&\; 
    \mfrac{1}{2 \sigma_1 \tau_1} (P_\Sigma  - \Sigma_* ) 
    + 
    \mfrac{1}{2 \sigma_2 \tau_2}  \Sigma_*
    \;.
\end{align*}
Also define the Gaussian random vectors 
\begin{align*}
    &\;
    \bq 
    \;\coloneqq\;
    \mfrac{1}{\kappa_* \, \sqrt{p}} \bG^\Phi \, v(\beta^*)
    \;=\;
    \bG^\Phi  \mfrac{\Sigma_* \Sigma_o^{1/2} \beta^*}{\| \Sigma_* \Sigma_o^{1/2} \beta^* \|}
    \;,
    \qquad
    \tilde \bh_{\alpha, \sigma} 
    \;\coloneqq\;
    \kappa_*\alpha \bq - \sigma_1 \bh_1 - \mfrac{\sigma_2}{\sqrt{k}} J_{mk} \bh_2 \;,
    \\
    &\;
    \tilde \bg 
    \;\coloneqq\; 
    -
    \mfrac{r_1 + r_2}{\sqrt{mk}} \, (P_\Sigma  - \Sigma_* ) \bg_1 
    -
    \mfrac{r_2}{\sqrt{m}} \, \Sigma_* \bg_2 
    \;.
\end{align*}
For a function $f: \cS' \rightarrow \R$ and some $\cS' \subseteq \R^{mk}$, we define the Moreau envelope
\begin{align*}
    \cM_S(f;v, t) 
    \;\coloneqq\;
    \min_{x \in \cS} f(x) + \mfrac{1}{2 t} \| x - v \|^2_2 
    \;.
\end{align*}
Now we are ready to define the loss 
\begin{align*}
    L^{\rm SO}_{\tilde \alpha, \tilde \theta}(\by, \bq, \bg, \bh) 
    \;\coloneqq&\;
   -
   \mfrac{\sigma_1}{2 \tau_1} 
   -
   \mfrac{\sigma_2}{2 \tau_2} 
   +
   \mfrac{r_1}{2 \nu_1} 
   +
   \mfrac{r_2}{2 \nu _2}
   +
   \alpha \theta \kappa_*^2 
   -
   \mfrac{\alpha^2 \kappa_*^2}{2 \sigma_2 \tau_2}
   +
   M_{\bg,\sigma, \tau, \theta}
   -
   \mfrac{1}{4} \big\|  (\tilde \Sigma_{\sigma, \tau}^\dagger)^{1/2}
   \big( \tilde \bg
           + 
           \mfrac{\theta}{\sqrt{p}}    v(\beta^*)
   \big) 
   \big\|^2 
   \\
   &\,
   +
   \mfrac{1}{mk}  M_{\by, \tilde \bh_{\alpha, \sigma}, r, \nu}  
   -
   \mfrac{1}{2 r_2 \nu_2 mk }  
   \| \by \|^2
   -
   \mfrac{1}{mk}  \by^\intercal \tilde \bh_{\alpha, \sigma}  
    \;,
\end{align*}
where we have defined the nested Moreau envelope $M_{\by, \tilde \bh_{\alpha, \sigma}, r, \nu}$ via
\begin{align*}
    M^\perp_{\tilde \bh_{\alpha, \sigma},r,\nu}(\tilde u)
    \;\coloneqq&\;
    \cM_{P_{mk}^\perp(S_u) }
    \big( \bone_{mk}^\intercal \, \rho(\tilde u + \argdot ) \,;\, P^\perp_{mk} \tilde \bh_{\alpha, \sigma} , \mfrac{1}{r_1 \nu_1}  \big)
    \;,
    \\
    M_{\by, \tilde \bh_{\alpha, \sigma}, r, \nu} 
    \;\coloneqq&\; 
    \cM_{P_{mk}(S_u)} 
    \big( 
        M^\perp_{\tilde \bh_{\alpha, \sigma},r,\nu} \,;\, \mfrac{1}{r_2 \nu_2}  \by -  P_{mk} \tilde \bh_{\alpha, \sigma} \,,\, r_2 \nu_2
    \big)
    \;,
\end{align*}
as well as another Moreau envelope like term 
\begin{align*}
    M_{\bg,\sigma, \tau, \theta}
    \;\coloneqq\;
    \min_{ \mu \in S} 
    \,
    &\;
    \mfrac{\lambda}{2n} \| P_\Sigma  \mu \|^2_2  
    + 
    \big\| 
        \tilde \Sigma_{\sigma, \tau}^{1/2} (\Sigma^{1/2} \mu)
         -
         \mfrac{1}{2} 
         (\tilde \Sigma_{\sigma, \tau}^\dagger)^{1/2}
        \big( \tilde \bg
                + 
                \mfrac{\theta}{\sqrt{p}}    v(\beta^*)
        \big) 
    \big\|^2 
    +
    \mfrac{r_2}{\sqrt{n}} \bg_2^\intercal P_*  \Sigma^{1/2} \mu
    \;.
\end{align*}
The minimized risk is denoted as 
\begin{align*}
    R^{\rm SO}_{S, S_u, S_v}(\by,\bq,\bg,\bh)
    \;\coloneqq\; 
    \min_{\tilde \alpha \in S_1}  \, \max_{\tilde \theta \in S_2} \, \min_{\tilde \chi \in S_3}
     L^{\rm SO}_{\tilde \alpha \tilde \theta}(\by,\bq,\bg, \bh)
    \;.
    \tag{SO} \label{SO}
\end{align*}

The next lemma shows that \eqref{AO} can be replaced by  \eqref{SO} in that it satisfies similar inequalities as \eqref{AO} in terms of their relationships to \eqref{PO}. The inequalities in the result are to be compared with those in \theoremref{thm:CGMT_2}.

\vspace{.5em}

\begin{lemma}[Equivalence of \eqref{PO} and \eqref{SO}] \label{lem:PO:SO} Let $S \in \R^p$ and $S_u, S_v \in \R^{mk}$ be compact, convex and non-empty. Also assume that the linear span $\textrm{span}(S) = \R^p$. Then for any $c \in \R$,
\begin{align*}
    \P( R^{\rm PO}_{S, S_u, S_v} \leq c ) 
    \;\leq&\; 
    4 \, \P( R^{\rm SO}_{S, S_u, S_v} \,\leq\, c )
    &\text{ and }&&
    \P( R^{\rm PO}_{S, S_u, S_v} \geq c ) 
    \;\leq&\; 
    4 \, \P( R^{\rm SO}_{S, S_u, S_v} \,\geq\, c )
    \;.
\end{align*}
If instead of $S$, the set of values of $\beta$ we consider is the non-convex set
\begin{align*}
    S_{c,\epsilon} \;\coloneqq\; S \setminus \big\{ \beta \in S \;\big|\; | (P_\Sigma \beta)^\intercal \Sigma_o (P_\Sigma \beta) - c |  \leq \epsilon \; \big\} 
\end{align*}
for some $c \in \R$ and some sufficiently small $\epsilon > 0$ such that $S_{c,\epsilon}$ is non-empty. Then we have 
\begin{align*}
    \P( R^{\rm PO}_{S_{c,\epsilon}, S_u, S_v} \leq c ) 
    \;\leq&\; 
    4 \, \P( R^{\rm SO}_{S_{c,\epsilon}, S_u, S_v} \,\leq\, c )
    \;.
\end{align*}
\end{lemma}

\vspace{.5em}

\textbf{Deterministic optimization (DO). } The next step is to compute the asymptotics of \eqref{SO} as $m, p \rightarrow \infty$  and $p/m \rightarrow \kappa / k$ for special cases of $S \subset \R^p$. The limit is given by a deterministic optimization 
\begin{align}
    R^{DO}_S 
    \;\coloneqq\; 
    \min_{\substack{\alpha \in S^\alpha \\ (\sigma_1, \sigma_2) \in S_{\sigma_1}  \times S_{\sigma_2} \\ \nu_1, \nu_2 \geq 0 }}
        \,
        \max_{\substack{ (r_1, r_2) \in S_{r_1 } \times S_{r_2} \\ \tau_1, \tau_2 \geq 0 \\ \theta \in \R }}
        \,
    & \,
    -
    \mfrac{\sigma_1}{2 \tau_1} 
    -
    \mfrac{\sigma_2}{2 \tau_2} 
    +
    \mfrac{r_1}{2 \nu_1} 
    +
    \mfrac{r_2}{2 \nu _2}
    +
    \alpha \theta \bar \kappa_*^2
    -
    \mfrac{\alpha^2 \bar \kappa_*^2}{2 \sigma_2 \tau_2} \notag
    - \bar \chi^{r, \theta, \sigma, \tau}_1 
    + 
    \epsilon_S^2 \,
     \mfrac{\bar \chi^{r, \theta,  \sigma, \tau}_3}{\bar \chi^{r, \theta, \sigma, \tau}_2}
    \notag 
    \\ 
    &\, 
     - 
     \mfrac{1}{4 r_2 \nu_2}
     -
     \alpha \,  \mean\big[ \sigma( \bar \kappa_o \bar Z_0 + \bar \kappa_* \bar Z_1 ) \bar \kappa_* \bar Z_1 \big] \notag
    +
    \bar M_\rho^{r, \nu, \alpha, \sigma}
    \;,
    \tag{DO} \label{DO}
\end{align}
where we have defined the limits 
\begin{align*}
    \bar \kappa_*
    \;\coloneqq&\; 
    \lim_{p \rightarrow \infty} \kappa_* 
    \;=\;
    \lim_{p \rightarrow \infty} \| \Sigma_* \Sigma_o^{1/2} \beta^* \|
    \;,
    \qquad 
    \bar \kappa_o 
    \;\coloneqq\;  
    \lim_{p \rightarrow \infty} \| (I_p - \Sigma_*) \Sigma_o^{1/2} \beta^* \| 
    \;,
    \\
    \bar \chi^{r, \theta, \sigma, \tau}_1
    \;\coloneqq&\; 
    \mfrac{(r_1 + r_2)^2 \sigma_1 \tau_1}{2k} 
    \,
    \bar \chi^{\sigma, \tau}_{11}
    +
    \mfrac{r_2^2 \sigma_2 \tau_2}{2} \, \bar \chi^{\sigma, \tau}_{12}
    + 
    \mfrac{\theta^2  \bar \kappa_*^2 \sigma_2 \tau_2}{2}
    \, \bar \chi^{\sigma, \tau}_{13}
    \;,
    \\
    \bar \chi^{r, \theta, \sigma, \tau}_2
    \;\coloneqq&\;
    \mfrac{(r_1+r_2)^2\sigma_1^2\tau_1^2}{k} 
    \bar \chi^{\sigma,\tau}_{21}
    +
     r_2^2 \sigma_2^2\tau_2^2 \,
    \bar \chi^{\sigma,\tau}_{22}
    +
     \theta^2  \bar \kappa_*^2 \sigma_2^2 \tau_2^2
    \,
    \bar \chi^{\sigma,\tau}_{23}
    \;,
    \\
    \bar \chi^{r, \theta, \sigma, \tau}_3
    \;\coloneqq&\;
    \mfrac{(r_1+r_2)^2\sigma_1^2\tau_1^2}{k} 
    \bar \chi^{\sigma,\tau}_{31}
    +
    r_2^2 \sigma_2^2\tau_2^2 \,
    \bar \chi^{\sigma,\tau}_{32} 
    +
    \theta^2  \bar \kappa_*^2 \sigma_2^2 \tau_2^2
    \,
    \bar \chi^{\sigma,\tau}_{33} 
    \;,
\end{align*}
with 
\begin{align*}
    \bar \chi^{\sigma,\tau}_{11}
    \;\coloneqq&\;
    \lim
    \mfrac{\Tr\big(
    \big( \frac{\sigma_1 \tau_1\lambda}{m}  \Sigma^\dagger  
        + 
        I_p  \big)^\dagger
        (P_\Sigma - \Sigma_*)
    \big)
    }{m}\;,
    \\
    \bar \chi^{\sigma, \tau}_{12}
    \;\coloneqq&\;
    \lim
    \mfrac{\Tr\big(
        \big( 
            \frac{\sigma_2 \tau_2\lambda}{m}  \Sigma^\dagger  
            + 
            I_p 
        \big)^\dagger
        \Sigma_*
    \big)
    }{m}\;,
    \\
    \bar \chi^{\sigma,\tau}_{13}
    \;\coloneqq&\;
    \lim
    \Tr\Big(
    \Big( 
        \mfrac{\sigma_2 \tau_2\lambda}{m}  \Sigma^\dagger  
        + 
        I_p   
    \Big)^\dagger
    P_*
    \Big)
    \;,
    \\
    \bar \chi^{\sigma,\tau}_{21} 
    \;\coloneqq&\;
    \lim 
    \mfrac{\big\|
            \Sigma_{\rm new}^{1/2} (\Sigma^\dagger)^{1/2}
            \big( \frac{\sigma_1 \tau_1 \lambda}{m} \Sigma^\dagger + I_p \big)^\dagger (P_\Sigma - \Sigma_*) 
        \big\|^2 
    }
    {m}
    \;,
    \\
    \bar \chi^{\sigma,\tau}_{22}
    \;\coloneqq&\;
    \lim 
    \mfrac{\big\|
            \Sigma_{\rm new}^{1/2} (\Sigma^\dagger)^{1/2}
            \Big( \mfrac{\sigma_2 \tau_2 \lambda}{m} \Sigma^\dagger + I_p \Big)^\dagger \Sigma_* 
        \big\|^2 
    }
    {m}
    \;,
    \\
    \bar \chi^{\sigma,\tau}_{23}
    \;\coloneqq&\;
    \lim \,
    \Big\|
        \Sigma_{\rm new}^{1/2} (\Sigma^\dagger)^{1/2}
        \Big( 
            \mfrac{\sigma_2 \tau_2 \lambda}{m} \Sigma^\dagger  
            + 
            I_p   
        \Big)^\dagger
        P_*
    \Big\|^2
    \;,
    \\
    \bar \chi^{\sigma,\tau}_{31} 
    \;\coloneqq&\;
    \lim 
    \mfrac{\big\|
                \big( \frac{\lambda}{2m} \Sigma^\dagger +  \tilde \Sigma_{\sigma, \tau} 
                \big)^{1/2} 
            \Sigma^{1/2} P_{\Sigma_{\rm new}}   
            (\Sigma^\dagger)^{1/2}
            \big( \frac{\sigma_1 \tau_1 \lambda}{m} \Sigma^\dagger + I_p \big)^\dagger (P_\Sigma - \Sigma_*) 
        \big\|^2 
    }
    {m}
    \;,
    \\
    \bar \chi^{\sigma,\tau}_{32} 
    \;\coloneqq&\;
    \lim 
    \mfrac{\big\|
            \big( \frac{\lambda}{2m} \Sigma^\dagger +  \tilde \Sigma_{\sigma, \tau} 
                    \big)^{1/2} 
            \Sigma^{1/2} P_{\Sigma_{\rm new}}        
            (\Sigma^\dagger)^{1/2}
            \Big( \mfrac{\sigma_2 \tau_2 \lambda}{m} \Sigma^\dagger + I_p \Big)^\dagger \Sigma_* 
        \big\|^2 
    }
    {m}
    \;,
    \\
    \bar \chi^{\sigma,\tau}_{33} 
    \;\coloneqq&\;
    \lim \,
    \Big\|
        \Big( \mfrac{\lambda}{2m} \Sigma^\dagger +  \tilde \Sigma_{\sigma, \tau} 
        \Big)^{1/2} 
        \Sigma^{1/2} P_{\Sigma_{\rm new}}
        (\Sigma^\dagger)^{1/2}
        \Big( \mfrac{\sigma_2 \tau_2 \lambda}{m} \Sigma^\dagger + I_p \Big)^\dagger
        P_*
    \Big\|^2 
    \;.
\end{align*}
We have also defined an expected Moreau-envelope-like term 
\begin{align*}
    \bar M_\rho^{r, \nu, \alpha, \sigma}
    \;\coloneqq\;
    &\; \mean\bigg[
       \min_{\tilde u \in \R^{k}}
       \mfrac{1}{k} \bone_k^\intercal \rho(\tilde u)
    + 
    \mfrac{r_1 \nu_1}{2 k} \, 
    \Big\| 
        \big(I_k - \mfrac{1}{k} \bone_{k \times k}\big)
        ( \tilde u  + \sigma_1 \eta )
    \Big\|^2
    \\
    &\qquad
    +
    \mfrac{r_2 \nu_2}{2 k}
    \Big\| 
        \mfrac{1}{k} \bone_{k \times k}
        \Big(
            \tilde u 
            -
            \mfrac{1}{r_2 \nu_2}  \ind_{\geq 0}\{ \bar \kappa_o \bar Z_0 + \bar \kappa_* \bar Z_1 - \varepsilon_1 \} 
            \bone_k 
            -
            \alpha \bar \kappa_* \bar Z_1 
            \bone_k
            +
            \sigma_1 \eta 
            +
            \sigma_2 \bar Z_2 \bone_k 
        \Big)
    \Big\|^2
    \bigg]
    \;,
\end{align*}
where $\bar Z_0,\bar Z_1, \bar Z_2, \eta_1, \ldots, \eta_k$ are i.i.d.~univariate Gaussians and $\eta=(\eta_1, \ldots, \eta_k)$, and $\varepsilon_1$ is an independent $\textrm{Logistic}(0,1)$ variable. The two cases of $S$ we consider are 
\begin{align*}
    S \;=&\; \cS_p 
    &\text{ and }&&
    S 
    \;=&\; 
    \cS_\epsilon^c 
    \;\coloneqq\; 
    \big\{ \beta \in \cS_p \,\big|\, \big| \sqrt{\beta^\intercal \Sigma_{\rm new} \beta} - (\bar \chi^{\bar r, \bar \theta, \bar \sigma, \bar\tau}_2)^{1/2}  \big| > \epsilon \big\}
    \;,
\end{align*}
where $\bar r = (\bar r_1, \bar r_2)$, $\bar \sigma = (\bar \sigma_1, \bar \sigma_2)$, $\bar \theta$ and $\bar \tau =(\bar \tau_1, \bar \tau_2)$ are the optimal solutions to \eqref{DO}. We also set $\epsilon_S = 0$ for $S=\cS_p$ and $\epsilon_S = \epsilon$ for $S=\cS_\epsilon^c$. 

\vspace{.5em}

\begin{lemma}[Equivalence between \eqref{SO} and \eqref{DO}] \label{lem:SO:DO} Assume that the set $S_u \subset \R^{mk}$ is closed under permutation of the $m$ blocks of $k$ coordinates. Also suppose that as $m, p \rightarrow \infty$, 
$\sup_{u \in S_u} \frac{\|  u \|_2^2}{mk} \rightarrow \infty$ and $\sup_{u \in S_v} \frac{\| v \|_2^2}{mk} \rightarrow \infty$. Also assume that the limits $\bar \kappa_*$, $\bar \chi^{r, \theta, \sigma, \tau}_1$, $\bar \chi^{r, \theta, \sigma, \tau}_2$ and $\bar \chi^{r, \theta, \sigma, \tau}_3$ exist for every $r_1, r_2, \theta, \sigma_1, \sigma_2, \tau_1, \tau_2 $. Then  for $S = \cS_p$ and $S = \cS_\epsilon^c$,
\begin{align*}
    \big| \, R^{\rm SO}_{S, S_u, S_v}(\by, \bq, \bg, \bh)  -  R^{\rm DO}_{S}  \, \big| 
    \;\xrightarrow{\P}\; 0
    \;.
\end{align*}
\end{lemma}

As with \citet{salehi2019impact}, it remains to prove that the first order condition of \eqref{DO} for $S = \cS_p$ is equivalent to the system of $10$ equations \eqref{EQs} in $(\alpha, \sigma_1, \sigma_2, \tau_1, \tau_2, \nu_1, \nu_2, r_1, r_2, \theta)$. This involves computing the derivative of the Moreau-envelope-like term $\bar M^{r,\nu,\alpha,\sigma}_\rho$ using the envelope theorem.

\vspace{.5em}

\begin{lemma} \label{lem:EQs} Assume that the minimizer-maximizers of \eqref{DO} are within the interior of the domain of optimization and that $S = \cS_p$. Then these minimizer-maximizers are solutions to \eqref{EQs}.
\end{lemma}

\subsection{Verifying conditions for different augmnetations} \label{appendix:DA:special:cases}

\textbf{Isotropic data with no augmentation. } \citet{salehi2019impact} derives a set of equations that governs the behavior of high-dimensional logistic regression with ridge regularization, isotropic data and no data augmentation. Here, we verify that our formula recover their formula exactly as a special case, and that $(r_2, \nu_2, \sigma_2, \tau_2, \alpha, \theta)$ play the role of the parameters in the original unaugmented optimization.

\begin{lemma} \label{lem:EQs:iso} Suppose that $X_{\rm new} \overset{d}{=} Z_1$ with $\Var[Z_1] = \frac{1}{p}I_p$, that $k=1$ and $\phi_1(Z_i) = Z_i$ almost surely for all $i \leq m=n$. Also write $\gamma = \frac{1}{r_2 \nu_2}$, $\rho(\argdot) = \log(1+\exp(\argdot))$ and denote the proximal operator ${\rm Prox}_{t \rho(\argdot)}(v) \coloneqq \argmin_{x \in \R} \frac{1}{2t} (v-x)^2 + \rho(x)$. Then \eqref{EQs} is equivalent to the following system of equations:
\begin{align*}
    \begin{cases}
        \theta \;=\; \mfrac{\alpha \kappa}{\gamma}\;,
        \\
        \tau_2 \;=\; \mfrac{ \kappa^{-1} \gamma}{\sigma_2(1 -  \gamma \lambda)}\;,
        \\
        r_2 \;=\; \mfrac{\sigma_2 \sqrt{ \kappa}}{\gamma}\;,
        \\
        \mfrac{\sigma^2  \kappa}{2}
        \;=\; 
        \mean\big[ \partial \rho( - \bar \kappa_* \bar Z_1 ) 
        \big( 
            \alpha \bar \kappa_* \bar Z_1 + \sigma_2 \bar Z_2 
            - 
            {\rm Prox}_{\gamma \rho(\argdot )}( \alpha \bar \kappa_*  \bar Z_1 + \sigma_2 \bar Z_2  ) \big)^2 
        \big] 
        \;,
        \\
        1 
        - 
        \kappa
        +
        \gamma \lambda \kappa
        \;=\; 
        \mean\Big[ \mfrac{ 2 \rho'(- \bar \kappa_* Z_1)}{1 + \gamma \rho''({\rm Prox}_{\gamma \rho(\argdot)} ( \bar \kappa_* \alpha \bar Z_1 + \sigma_2 \bar Z_2 ) )} \Big] 
        \;,
        \\
        - \mfrac{\alpha  \kappa }{2} 
        \;=\;  
        \mean[ \partial^2 \rho(- \bar \kappa_* \bar Z_1) {\rm Prox}_{\gamma \rho(\argdot)}( \bar \kappa_* \alpha \bar Z_1 + \sigma_2 \bar Z_2 ) ]\;,
    \end{cases}
\end{align*}
 with $r_1 = \sigma_1 = 0$, $\nu_1, \tau_1 \rightarrow \infty$,  $\bar \kappa_* = \lim_{p \rightarrow \infty} \frac{\| \beta^* \|}{\sqrt{p}}$ and $\kappa = \lim p/n$.
\end{lemma}

\begin{remark} \label{remark:EQs:iso} (i) \lemmaref{lem:EQs} and \theoremref{thm:DA:risk} apply even though the values of $r_1$, $\sigma_1$, $\nu_1$ and $\tau_1$ are not in the interior of the domain of optimization, as these variables can be removed much earlier in the proof of \lemmaref{lem:PO:SO} and allow us to handle only a smaller system of equations. Moreover, the only quantity $\bar \chi^{r, \theta, \sigma, \tau}_2$ that enters the test risk is independent of these variables.
(ii) To identify the equations in \lemmaref{lem:EQs:iso} with those in (14) and (16) from Theorem 2 of \cite{salehi2019impact}, we note several notational differences: We have used $\kappa = \lim p/n$, whereas they use $\delta = \lim n/p$; our $\bar \kappa_*$, $\bar Z_1$ and $\bar Z_2$ should be identified with their $\kappa$, $Z_1$ and $Z_2$; our $r_2$, $\sigma_2$ and $\tau_2$ should be identified with their $r$, $\sigma$ and $\tau$; our regularization is defined as $\frac{\lambda}{2n} \| \argdot \|^2$ whereas theirs is defined as $\frac{\lambda}{2p} \| \argdot \|^2$, so to see the equivalence, one needs to make the replacement $\lambda \mapsto \lambda \kappa^{-1}$ above.
\end{remark}

\textbf{Random permutations and sign flipping. }  Recall the setup for random permutations and random sign flipping in \Cref{sec:DA}. We first verify that the equations \eqref{EQs} do apply to these two augmentations in special cases. In view of \Cref{thm:DA:risk}, the key condition to verify is \Cref{assumption:CGMT:var}.

\begin{lemma} \label{lem:CGMT:permute} Suppose the coordinates of each $Z_1^{(t)}$ are i.i.d.~within the group. Then \Cref{assumption:CGMT:var} holds for random permutations.
\end{lemma}

\begin{lemma} \label{lem:CGMT:sign:flip}  Suppose $\Var[Z_1] = \frac{1}{p} I_p$. Then \Cref{assumption:CGMT:var} holds for random sign flipping.
\end{lemma}

\textbf{Random cropping. } Recall the random cropping scheme defined in \Cref{sec:DA}. While random cropping does not satisfy  \Cref{assumption:CGMT:var}, it does satisfy a slightly relaxed notion of  \Cref{assumption:CGMT:var}:

\begin{lemma} \label{lem:CGMT:crop} Suppose $\Var[Z_1] = \frac{1}{p} I_p$. For random cropping, there exist some $a_1, a_2 > 0$ such that
\begin{align*}
    (i) \; \Sigma_* \;=&\; a_1
    ( \Sigma^\dagger )^{1/2}\, \Cov[ \phi_1(Z_1)\,,\, Z_1 ]  ( \Sigma_o^\dagger )^{1/2}
    &\text{ and }&&
    (ii) \; \Sigma_*^2 \;=&\; a_2 \Sigma_*\;.
    \tagaligneq \label{eq:relaxed:assumption:CGMT:var}
\end{align*} 
\end{lemma}

The core CGMT statement --- the equivalence of \eqref{PO} and \eqref{AO} --- does hold for random cropping. To see this, notice that \Cref{assumption:CGMT:var} is equivalent to having $a_1=a_2 =1$ in \eqref{eq:relaxed:assumption:CGMT:var}. We observe that to prove the equivalence of \eqref{PO} and \eqref{AO} in \lemmaref{lem:PO:AO}, \Cref{assumption:CGMT:var} is only critical for showing the independence of the differently projected data matrices, which hold even under the rescaling $a_1$ and $a_2$ in \eqref{eq:relaxed:assumption:CGMT:var}; see the proof of \lemmaref{lem:CGMT:proj:cov} below. As such, 

Meanwhile, a tedious extension of \eqref{EQs} also holds for random cropping. Notice that \Cref{assumption:CGMT:var} is used again only in the calculations from \eqref{eq:AO:2} onwards in the proof of \lemmaref{lem:PO:SO}, which relates \eqref{AO} to \eqref{SO}. There, we only use the idempotency of $\Sigma_*$ such that $\Sigma_*$ and $I_p-\Sigma_*$ are projections onto orthogonal subspaces, which simplify many subsequent calculations. If instead \eqref{eq:relaxed:assumption:CGMT:var}(ii) holds, a similar calculation still works by writing $\Sigma_* = \Sigma_1 + \Sigma_2$ and $I_p - \Sigma_* = \Sigma'_2 + \Sigma_3$, such that $\Sigma_1$, $\Sigma_2$ and $\Sigma_3$  have mutually orthogonal positive eigenspaces, and $\Sigma_2$ and $\Sigma'_2$ share the same positive eigenspace. This would lead to a system of equations involving $(\sigma_1, \sigma_2, \sigma_3, \tau_1, \tau_2, \tau_3)$ instead of just $(\sigma_1, \sigma_2, \tau_1, \tau_2)$ in \eqref{EQs}, and we omit the calculations for simplicity.

\section{Proofs for \Cref{appendix:DA:cgmt:results}}   \label{appendix:proof:DA}

\subsection{Proofs for the equivalence of \eqref{PO} and \eqref{AO}} \label{appendix:DA:PO:AO}

The next lemma confirms that the projection $P_*$ decouples the different random quantities.

\begin{lemma} \label{lem:CGMT:proj:cov} Under \Cref{assumption:CGMT:var}, $\bG^\Phi P^\perp_*$ is independent of $(\bG \Sigma^{1/2}_o \beta^*,  \bG^\Phi P_*)$.
\end{lemma}

\begin{proof}[Proof of \lemmaref{lem:CGMT:proj:cov}] By Gaussianity, to prove independence, it suffices to check that the covariance between the random quantities are zero. We first verify that the covariance between $\bG^\Phi P^\perp_*$ and $\bG \Sigma^{1/2}_o \beta^*$ is zero, for which it suffices to compute 
\begin{align*}
    \Cov[ P^\perp_* G^\Phi_{11} \,,\, G_1^\perp \Sigma^{1/2}_o \beta^* ]
    \;=&\;
     P^\perp_*  \,\Cov[ (\Sigma^\dagger)^{1/2} \phi_{11}(X_1) , (\Sigma_o^\dagger)^{1/2}  X_1 ] \,\Sigma^{1/2}_o \beta^*
     \\
     \;=&\;
     P^\perp_* (\Sigma^\dagger)^{1/2} \, \Cov[  \phi_{11}(X_1),  X_1 ] \,  (\Sigma_o^\dagger)^{1/2}  \Sigma^{1/2}_o \beta^*
     \\
     \;=&\;
     P^\perp_* \Sigma_* \Sigma^{1/2}_o \beta^*
     \;=\;
     0\;.
\end{align*}
In the last line, we used \Cref{assumption:CGMT:var}(i), 
and concluded that the covariance evaluates to zero by the definition of $P_*$. This proves that $\bG^\Phi P^\perp_*$ is independent of $\bG \Sigma^{1/2}_o \beta^*$. 

\vspace{.5em}

To check the independence between $\bG^\Phi P^\perp_*$ and $ \bG^\Phi P_*$, we first note that since $\Sigma = \Var[\phi_{11}(X_1)]$, we have 
\begin{align*}
    \Cov[ P^\perp_* G^\Phi_{11} \,,\, P_* G^\Phi_{11} ]
    \;=&\;
    P^\perp_* (\Sigma^\dagger)^{1/2} \, \Var[ \phi_{11}(X_1)] \, (\Sigma^\dagger)^{1/2} P_* 
    \;=\; P^\perp_* P_*
    \;=\; 0\;.
\end{align*}
We also need to compute 
\begin{align*}
    \Cov[ P^\perp_* G^\Phi_{11} \,,\, P_* G^\Phi_{12} ]
    \;=&\;
    P^\perp_* (\Sigma^\dagger)^{1/2} \, \Cov[ \phi_{11}(X_1) \,,\, \phi_{12}(X_1)] \, (\Sigma^\dagger)^{1/2} P_* 
    \\
    \;=&\;
    P^\perp_* \, \Sigma_* \, P_* 
    \;.
\end{align*}
Now note that if $\Sigma_* \Sigma^{1/2}_o \beta^* = 0$, the above evaluates to zero automatically. Otherwise, we have
\begin{align*}
    \Sigma_* \, P_* 
    \;=&\;
    \Sigma_* \, 
    \mfrac{(\Sigma_* \Sigma_o^{1/2} \beta^* )  (\Sigma_* \Sigma_o^{1/2} \beta^* )^\intercal }{\|  \Sigma_*  \Sigma_o^{1/2} \beta^*  \|^2}
    \;=\;
    P_*\;,
\end{align*}
where we have used $\Sigma_*^2 = \Sigma_*$ by \Cref{assumption:CGMT:var}(ii).  This implies 
\begin{align*}
    \Cov[ P^\perp_* G^\Phi_{11} \,,\, P_* G^\Phi_{12} ] \;=\; 
    P^\perp_* P_* \;=\; 0\;,
\end{align*}
which proves that $\bG^\Phi P^\perp_*$  is independent of $\bG^\Phi P_*$. 
\end{proof}

\lemmaref{lem:CGMT:proj:cov} suggests that we can apply \theoremref{thm:CGMT_2} to $\bG^\Phi P_*^\perp \Sigma^{1/2}$ conditionally on $(\bG \Sigma^{1/2}_o \beta^*,  \bG^\Phi P_*)$. To facilitate this, the next lemma computes the covariance structure of $\bG^\Phi P_*^\perp \Sigma^{1/2}$.

\begin{lemma} \label{lem:cov:GPstar} For $i, i' \leq m$, $j, j' \leq k$ and $l, l' \leq p$, 
\begin{align*}
    \Cov[ (\Sigma^{1/2} P_*^\perp G^\Phi_{ij})_l \,,\, (\Sigma^{1/2}  P_*^\perp G^\Phi_{i'j'})_{l'} ] 
    \;=\;
    (I_{mk})_{ij,i'j'} \, (\Sigma_1)_{l,l'}
    +
    (J_{mk})_{ij,i'j'} \, (\Sigma_2)_{l,l'}
    \;.
\end{align*}
Moreover, $\Sigma_*$, $\Sigma_1$ and $\Sigma_2$ are all positive semi-definite.
\end{lemma}

\begin{proof}[Proof of \lemmaref{lem:cov:GPstar}] For $i, i' \leq m$, $j, j' \leq k$ and $l, l' \leq p$, we have 
\begin{align*}
    &\;\Cov[ ( \Sigma^{1/2}  P_*^\perp G^\Phi_{ij})_l \,,\, (\Sigma^{1/2}  P_*^\perp G^\Phi_{i'j'})_{l'} ] 
    \\
    &\;=\;
    \Cov[ (\Sigma^{1/2} P_*^\perp (\Sigma^\dagger)^{1/2} \phi_{ij}(X_i))_l \,,\,  ( \Sigma^{1/2}  P_*^\perp  (\Sigma^\dagger)^{1/2}  \phi_{i'j'}(X_{i'}))_{l'}  ] 
    \\
    &\;=\;
    \ind\{ i = i'\} 
    \ind\{ j = j'\} 
        \,\big(\, \Sigma^{1/2}  P^\perp_*   (\Sigma^\dagger)^{1/2} \,  \Var[ \phi_{11}(X_1) ] \,  (\Sigma^\dagger)^{1/2}  P^\perp_* \Sigma^{1/2}  \,\big)_{l,l'}
    \\
    &\qquad 
        +
        \ind\{ i = i'\} 
        \ind\{ j \neq j'\} 
        \,\big(\, \Sigma^{1/2}  P^\perp_*  (\Sigma^\dagger)^{1/2}   \, \Cov[ \phi_{11}(X_1) \,,\, \phi_{12}(X_1) ] \,  (\Sigma^\dagger)^{1/2}   P^\perp_* \Sigma^{1/2} \, \big)_{l,l'}
    \\
    &\;\overset{(a)}{=}\;
    (I_{mk})_{ij,i'j'} \, \big( \Sigma^{1/2} P^\perp_*  \Sigma^{1/2} - \Sigma^{1/2} P^\perp_* \Sigma_* P^\perp_* \Sigma^{1/2}   \big)_{l,l'}
    +
    (J_{mk})_{ij,i'j'} \, \big( \Sigma^{1/2} P^\perp_* \Sigma_* P^\perp_* \Sigma^{1/2}  \big)_{l,l'}
    \\
    &\;=\;
    (I_{mk})_{ij,i'j'} \, (\Sigma_1)_{l,l'}
    +
    (J_{mk})_{ij,i'j'} \, (\Sigma_2)_{l,l'}
    \;.
\end{align*}
In $(a)$, we have used that $(\Sigma^\dagger)^{1/2} \,  \Var[ \phi_{11}(X_1) ] \,  (\Sigma^\dagger)^{1/2} = I_p$, $(P^\perp_* )^2 = P^\perp_*$ and the definition of $\Sigma_*$. This gives the desired formula. Now by the total law of covariance (see e.g.~Lemma 41(i) of \citet{huang2022data}),
\begin{align*}
    \Sigma_*
    \;=&\; 
    (\Sigma^\dagger)^{1/2}   \, \Cov[ \phi_{11}(X_1) \,,\, P^\perp_* \phi_{12}(X_1) ] \,  (\Sigma^\dagger)^{1/2} 
    \\
    \;=&\;
    (\Sigma^\dagger)^{1/2}   \, \Var \, \mean[ \phi_{11}(X_1) \,|\, X_1 ] \,  (\Sigma^\dagger)^{1/2} 
\end{align*}
which is positive semi-definite. This implies that $\Sigma_2$ is also positive semi-definite. Moreover, by another total law of variance, we get that 
\begin{align*}
    \Sigma_*
    \;\preceq\;
    (\Sigma^\dagger)^{1/2}   \, \Var[ \phi_{11}(X_1) ] \,  (\Sigma^\dagger)^{1/2}  
    \;=\; I_p\;,
\end{align*}
where $\preceq$ denotes the Loewner partial order on positive semi-definite matrices. This implies that $I_p - \Sigma_*$ is positive semi-definite and so is $\Sigma_1$. 
\end{proof}

We are now ready to prove the equivalence of \eqref{PO} and \eqref{AO}.

\vspace{.5em}

\begin{proof}[Proof of \lemmaref{lem:PO:AO}] We recall that \eqref{PO} can be expressed as 
    \begin{align*}
        \min_{\beta \in S, u \in S_u} \, \max_{v \in S_v}
        \;&\;
        \mfrac{1}{mk} \bone_{nk}^\intercal \, \rho( u )
        -
        \mfrac{1}{mk} \by(\bG \Sigma^{1/2}_o \beta^*)^\intercal u
        +    
        \mfrac{\lambda}{2m} \| \beta \|^2_2 
        +
        \mfrac{1}{mk} v^\intercal u 
        \\
        &\;
        -
        \mfrac{1}{mk} v^\intercal \bG^\Phi P_* \Sigma^{1/2} \beta 
        - 
        \mfrac{1}{mk} v^\intercal \bG^\Phi P^\perp_* \Sigma^{1/2} \beta 
        \;.
    \end{align*}
    By \lemmaref{lem:CGMT:proj:cov}, $\bG^\Phi P^\perp_*$ is independent of $(\bG \Sigma^{1/2}_o \beta^*, \bG^\Phi P_* )$. This allows us to condition on the random variables $(\bG \Sigma^{1/2}_o \beta^*, \bG^\Phi P_* )$, apply the CGMT result to  $\bG^\Phi P^\perp_* \Sigma^{1/2}$, and then marginalize out $(\bG \Sigma^{1/2}_o \beta^*, \bG^\Phi P_* )$. Notice that the loss is convex-concave in $(\beta, v)$, the sets of optimization are compact convex, and the variance-covariance structure of  $\bG^\Phi P^\perp_* \Sigma^{1/2}$ is given by \lemmaref{lem:cov:GPstar}, which satisfies the condition of \theoremref{thm:CGMT_2} with $M=2$. The conclusions of \theoremref{thm:CGMT_2} therefore hold for \eqref{PO} and \eqref{AO}.
\end{proof}

\subsection{Proof of \lemmaref{lem:PO:SO}: Equivalence between \eqref{PO}, \eqref{AO} and \eqref{SO}}

The calculations are mostly similar to that of \citet{salehi2019impact}, so we focus on highlighting the differences in the proof.

\vspace{.5em}
    
\textbf{Analyzing the auxiliary optimization. } We first notice that, other than the regularization term $\| \beta \|^2_2$, $\beta$ appears in the loss only through $\Sigma^{1/2} \beta$, $\Sigma_1^{1/2} \beta$ and $\Sigma_2^{1/2} \beta$, where 
\begin{align*}
    \Sigma_1 
    \;=&\;
    \Sigma^{1/2} P_*^\perp (I_p - \Sigma_*) P_*^\perp \Sigma^{1/2}
    &\text{ and }&&
    \Sigma_2 
    \;=&\;
    \Sigma^{1/2} P_*^\perp \Sigma^{1/2}\;.
\end{align*}
Therefore it suffices to restrict the set of minimization, $\beta \in S$, to the intersection of $S $ and the positive eigenspace of $\Sigma$. Define the projection to the positive eigenspace of $\Sigma$ as $P_{\Sigma} \coloneqq \Sigma^\dagger \Sigma$, which allows us to rewrite the auxiliary optimization as 
\begin{align*}
    \min_{\beta \in S, u \in S_u} \, \max_{v \in S_v} 
    & \,
    \mfrac{1}{mk} \bone_{mk}^\intercal \, \rho( u )
    -
    \mfrac{1}{mk} \by^\intercal u
    +    
    \mfrac{\lambda}{2m} \| P_\Sigma \beta \|^2_2 
    +
    \mfrac{1}{mk} v^\intercal ( u - \bG^\Phi P_* \Sigma^{1/2} \beta )
    \\
    &\;
    -
    \mfrac{1}{mk} 
    v^\intercal \bh_1 \| \beta \|_{\Sigma_1}
    +
    \mfrac{1}{mk} 
    \| v \| \bg_1^\intercal \Sigma_1^{1/2} \beta
    \\
    &\;
    -
    \mfrac{1}{mk^{3/2}} 
    v^\intercal J_{mk} \bh_2 \| \beta \|_{\Sigma_2}
    +
    \mfrac{1}{mk} 
    \| v \|_{J_{mk}} \bg_2^\intercal \Sigma_2^{1/2} \beta
    \;.
    \tagaligneq \label{eq:AO:1}
\end{align*}
For simplicity, we have abbreviated $\by \equiv \by(\bG \Sigma_o^{1/2} \beta^*)$. 

\vspace{.5em}

\citet{salehi2019impact} showed that, under their CGMT result (analogous to our \theoremref{thm:CGMT_2}(i) and \theoremref{thm:CGMT_2}(ii)), the minimum and maximum can be exchanged in the auxiliary optimization in an asymptotic sense since they can be exchanged in the primary optimization. Throughout the analysis of AO, we will highlight explicitly where such flipping is done, and in the case where the min-max theorem is not applicable, we defer a rigorous justification to the end of the proof. 

\vspace{.5em}

For simplicity of notation, given a matrix $A \in \R^{d' \times d}$ and a subset $S \in \R^d$, we also write $A(S) = \{Av \,|\, v \in S \}$ for short.

\vspace{.5em}

\textbf{Maximizing over $v \in S_v \subset \R^{mk}$.} Consider the projection matrix $P_{mk} \coloneqq \frac{1}{k} J_{mk}$ and write $P^\perp_{mk} = I_{mk} - P_{mk}$. Notice also that $\| \argdot \|_{J_{mk}} = \sqrt{k}\, \| P_{mk} (\argdot) \|$. Then the maximization over $v$ can be re-expressed as
\begin{align*}
    \max_{ P^\perp_{mk} v \in P^\perp_{mk}(S_v)} \max_{P_{mk} v \in P_{mk} (S_v)}\,
    \,
    &\, 
    \mfrac{1}{mk} v^\intercal P_{mk} ( u - \bG^\Phi P_* \Sigma^{1/2} \beta - \bh_1 \| \beta \|_{\Sigma_1} - \mfrac{1}{\sqrt{k}} J_{mk} \bh_2 \| \beta \|_{\Sigma_2} )
    \\
    &\;
    +
    \mfrac{1}{mk} v^\intercal P^\perp_{mk} ( u - \bG^\Phi P_* \Sigma^{1/2} \beta - \bh_1 \| \beta \|_{\Sigma_1} - \mfrac{1}{\sqrt{k}} J_{mk} \bh_2 \| \beta \|_{\Sigma_2} )
    \\
    &\;
    +
    \mfrac{1}{mk} 
    \| P_{mk} v \| \bg_1 \Sigma_1^{1/2} \beta
    +
    \mfrac{1}{mk} 
    \| P^\perp_{mk} v \| \bg_1^\intercal \Sigma_1^{1/2} \beta
    +
    \mfrac{1}{m \sqrt{k}} 
    \| P_{mk} v \| \bg_2^\intercal \Sigma_2^{1/2} \beta
    \;.
\end{align*}
Maximizing the above over $P_{mk} v$ and $P^\perp_{mk} v$ separately, choosing each vector to be what it multiplies and writing $r_1 = \| P^\perp_{mk} v \| / \sqrt{mk}$ and $r_2 = \| P_{mk} v \| / \sqrt{mk}$ (analogous to (44) -- (45) in \citet{salehi2019impact}), the above can be rewritten as 
\begin{align*}
    \max_{(r_1,r_2) \in S_{r_1} \times S_{r_2}}
    \,&\;
    r_1
    \Big( 
        \mfrac{1}{\sqrt{mk}} \bg_1^\intercal \Sigma_1^{1/2} \beta
        +
        \mfrac{1}{\sqrt{mk}}  \big\| P^\perp_{mk} ( u - \bG^\Phi P_* \Sigma^{1/2} \beta - \bh_1 \| \beta \|_{\Sigma_1} - \mfrac{1}{\sqrt{k}} J_{mk} \bh_2 \| \beta \|_{\Sigma_2} ) \big\|
    \Big)
    \\
    &\;
    +
    r_2
    \Big( 
        \mfrac{1}{\sqrt{mk}} \bg_1^\intercal \Sigma_1^{1/2} \beta
        +
        \mfrac{1}{\sqrt{m}} \bg_2^\intercal \Sigma_2^{1/2} \beta
    \\
    &\hspace{3em}
        +
        \mfrac{1}{\sqrt{mk}}  \big\| P_{mk} ( u - \bG^\Phi P_* \Sigma^{1/2} \beta - \bh_1 \| \beta \|_{\Sigma_1} - \mfrac{1}{\sqrt{k}}J_{mk} \bh_2 \| \beta \|_{\Sigma_2} )\big\|
    \Big)
    \;,
\end{align*}
where we have denoted 
\begin{align*}
    S_{r_1}
    \;\coloneqq&\; 
    \Big\{ \mfrac{1}{\sqrt{mk}} \| P^\perp_{mk} v \| \,\Big|\, v \in S_v \Big\}
    &\text{ and }&&
    S_{r_2}
    \;\coloneqq&\; 
    \Big\{ \mfrac{1}{\sqrt{mk}} \| P_{mk} v \| \,\Big|\, v \in S_v \Big\}
    \;.
\end{align*}
Substituting this into \eqref{eq:AO:1} yields  
\begin{align*}
    \min_{\substack{\beta \in S  \\ u \in S_u}} \, \max_{(r_1, r_2) \in S_{r_1} \times S_{r_2}}    
    &\;
    \mfrac{1}{mk} \bone_{mk}^\intercal \, \rho( u )
    -
    \mfrac{1}{mk} \by^\intercal u
    +    
    \mfrac{\lambda}{2m} \| P_\Sigma \beta \|^2_2 
    +
    \mfrac{r_1+r_2}{\sqrt{mk}} \bg_1^\intercal \Sigma_1^{1/2} \beta
    +
    \mfrac{r_2}{\sqrt{n}} \bg_2^\intercal \Sigma_2^{1/2} \beta
    \\
    &\;
        +
        \mfrac{r_1}{\sqrt{mk}}  \big\| P^\perp_{mk} ( u - \bG^\Phi P_* \Sigma^{1/2} \beta - \bh_1 \| \beta \|_{\Sigma_1} -\mfrac{1}{\sqrt{k}} J_{mk} \bh_2 \| \beta \|_{\Sigma_2} ) \big\|
    \\
    &\;
        +
        \mfrac{r_2}{\sqrt{mk}}  \big\| P_{mk} ( u - \bG^\Phi P_* \Sigma^{1/2} \beta - \bh_1 \| \beta \|_{\Sigma_1} -\mfrac{1}{\sqrt{k}} J_{mk} \bh_2 \| \beta \|_{\Sigma_2} )\big\|
    \;.
\end{align*}

\vspace{1em}

\textbf{Minimizing over $\beta \in S $.} As with (47) of \citet{salehi2019impact}, we introduce new variables $\mu, w \in \R^p$ to replace $ \beta $ in the regularization term via the Lagrange multiplier method applied to the constraint $P_\Sigma \mu = P_\Sigma \beta$:
\begin{align*}
    \min_{\substack{\beta \in S \\ u \in S_u \\ \mu \in S}}
    \;\;
    \max_{\substack{w \in \R^p \\ (r_1, r_2) \in S_{r_1} \times S_{r_2}}}
    \cL_1(\beta, u, \mu, w, r_1, r_2) 
    \;,
\end{align*}
where 
\begin{align*}
    \cL_1(\beta, u, \mu, w, r_1, r_2) 
    \;\coloneqq&\;
    \mfrac{1}{mk} \bone_{mk}^\intercal \, \rho( u )
    -
    \mfrac{1}{mk} \by^\intercal u
    +    
    \mfrac{\lambda}{2m} \| P_\Sigma \mu \|^2_2 
    +
    \mfrac{1}{p} w^\intercal P_\Sigma (\mu - \beta)
    \\
    &\;
    +
    \mfrac{r_1+r_2}{\sqrt{mk}} \bg_1^\intercal \Sigma_1^{1/2} \beta
    +
    \mfrac{r_2}{\sqrt{m}} \bg_2^\intercal \Sigma_2^{1/2} \beta
    \\
    &\;
        +
        \mfrac{r_1}{\sqrt{mk}}  \big\| P^\perp_{mk} ( u - \bG^\Phi P_* \Sigma^{1/2} \beta - \bh_1 \| \beta \|_{\Sigma_1} -\mfrac{1}{\sqrt{k}} J_{mk} \bh_2 \| \beta \|_{\Sigma_2} ) \big\|
    \\
    &\;
        +
        \mfrac{r_2}{\sqrt{mk}}  \big\| P_{mk} ( u - \bG^\Phi P_* \Sigma^{1/2} \beta - \bh_1 \| \beta \|_{\Sigma_1} -\mfrac{1}{\sqrt{k}} J_{mk} \bh_2 \| \beta \|_{\Sigma_2} )\big\|
    \;.
    \tagaligneq \label{eq:AO:2}
\end{align*}
To minimize over $\beta \in S$, we first swap the order of $\min_{\beta \in S }$ and $\max_{w \in \R^p, \; (r_1,r_2) \in S_{r_1} \times S_{r_2}}$. Notice that the $\beta$-dependence in the loss comes from $P_{\Sigma} \beta $, $\Sigma^{1/2}_1 \beta$, $\Sigma^{1/2}_2 \beta$ and $\bG^\Phi P_* \Sigma^{1/2} \beta$. Writing $\tilde \beta = \sqrt{p}\, P_* \Sigma^{1/2}
\beta$, $\tilde \beta^\perp = \sqrt{p}\, P_*^\perp \Sigma^{1/2}
\beta$, $v(\beta^*) \coloneqq \sqrt{p} \, \Sigma_* \Sigma_o^{1/2} \beta^*$ and 
$\kappa_* \coloneqq  \| \Sigma_* \Sigma_o^{1/2} \beta^* \|$, we can express
\begin{align*}
    \Sigma^{1/2}_1 \beta
    \;=&\;
    \big( \Sigma^{1/2} P_*^\perp (I_p - \Sigma_*) P_*^\perp \Sigma^{1/2} \big)^{1/2} \beta 
    \;\overset{(a)}{=}\;
    \mfrac{1}{\sqrt{p}}
    (I_p - \Sigma_*) \tilde \beta^\perp
    \;,
    \\
    \Sigma^{1/2}_2 \beta 
    \;=&\;
    \big( \Sigma^{1/2} P_*^\perp \Sigma_*  P_*^\perp \Sigma^{1/2} \big)^{1/2} \beta
    \;\overset{(b)}{=}\;
    \mfrac{1}{\sqrt{p}}
    \Sigma_*
    \tilde \beta^\perp
    \;,
    \\
    \bG^\Phi P_* \Sigma^{1/2} \beta
    \;=&\;
    \underbrace{ \mfrac{1}{\sqrt{p}} \bG^\Phi v(\beta^*) }_{\eqqcolon \kappa_* \bq}
    \, \times \, 
    \underbrace{ \mfrac{v(\beta^*)^\intercal \Sigma^{1/2} \beta}{\sqrt{p}  \, \kappa_*^2} }_{\eqqcolon \alpha(\tilde \beta)}
    \;,
    \\
    P_\Sigma \beta 
    \;=&\;
    \mfrac{1}{\sqrt{p}}
    (\Sigma^\dagger)^{1/2} \tilde \beta + \mfrac{1}{\sqrt{p}} (\Sigma^\dagger)^{1/2} \tilde \beta^\perp
    \;=\; 
    \mfrac{\alpha(\tilde \beta)}{\sqrt{p}} \, (\Sigma^\dagger)^{1/2} v(\beta^*)
    +
    \mfrac{1}{\sqrt{p}}
    (\Sigma^\dagger)^{1/2} \tilde \beta^\perp
    \;.
\end{align*}
In $(a)$ and $(b)$ above, we have used \Cref{assumption:CGMT:var}(ii) to note that $ I_p - \Sigma_*$ and $\Sigma_*$ are both idempotent. This allows us to express all $\beta$-dependent terms in terms of $\alpha(\tilde \beta)$ and $\tilde \beta^\perp$, where $\tilde \beta$ and $\tilde \beta^\perp$ are orthogonal and can be optimized separately. Therefore, the optimization can be rewritten as 
\begin{align*}
    \min_{\substack{u \in S_u \\ \mu \in S \\ \alpha \in S^\alpha }} \, \max_{\substack{w \in \R^p \\ (r_1, r_2) \in S_{r_1} \times S_{r_2}}} \, \min_{\tilde \beta^\perp \in \tilde S^\perp}
    & \,
    \mfrac{1}{mk} \bone_{mk}^\intercal \, \rho( u )
    -
    \mfrac{1}{mk} \by^\intercal u
    +    
    \mfrac{\lambda}{2m} \| P_\Sigma \mu \|^2_2 
    + 
    \mfrac{1}{p} w^\intercal P_\Sigma \mu 
    \\
    &\, \hspace{-3em}
    -
    \mfrac{\alpha}{p \sqrt{p}} 
    \,
    w^\intercal (\Sigma^\dagger)^{1/2} v(\beta^*)
    -
    \mfrac{1}{p \sqrt{p}} w^\intercal (\Sigma^\dagger)^{1/2} \tilde \beta^\perp 
    \\
    & \,\hspace{-3em}
    +
    \mfrac{r_1 + r_2}{\sqrt{mkp}} \bg_1^\intercal (I_p - \Sigma_*)\tilde \beta^\perp
    +
    \mfrac{r_2}{\sqrt{mp}} \bg_2^\intercal \Sigma_* \tilde \beta^\perp 
    \\
    &\;\hspace{-3em}
    +
    \mfrac{r_1}{\sqrt{mk}}  \Big\| P^\perp_{mk} \Big( u - \kappa_* \alpha \bq - \mfrac{1}{\sqrt{p}} \bh_1 \| (I_p - \Sigma_*) \tilde \beta^\perp \| - \mfrac{1}{\sqrt{pk}} J_{mk} \bh_2 \|  \Sigma_* \tilde \beta^\perp \| \Big) \Big\|
    \\
    &\; \hspace{-3em}
        +
        \mfrac{r_2}{\sqrt{mk}}  \Big\| P_{mk} \Big( u - \kappa_* \alpha \bq  - \mfrac{1}{\sqrt{p}} \bh_1 \| (I_p - \Sigma_*) \tilde \beta^\perp \| -  \mfrac{1}{\sqrt{pk}} J_{mk} \bh_2 \|   \Sigma_* \tilde \beta^\perp \| \Big) \Big\|
    \;,
    \tagaligneq 
\end{align*}
where we have defined the sets  
\begin{align*}
    S^\alpha \;\coloneqq&\; 
    \big\{ \mfrac{v(\beta^*)^\intercal \Sigma^{1/2} \beta}{\sqrt{p} \, \kappa^2_*} \,\big| \, \beta \in S \big\} 
    &\text{ and }&&
    \tilde S^\perp 
    \;\coloneqq&\; 
    \{ \sqrt{p} \, P^\perp_* \Sigma^{1/2} \beta \,|\, \beta \in S \}
    \;.
\end{align*}
Note that we have moved the minimization over $\alpha$ to the outmost part of the loss. The steps so far are analogous to (46) -- (47) of \citet{salehi2019impact}. Before proceeding, we notice that since $I_p - \Sigma_*$ and $\Sigma_*$ are symmetric and idempotent by \Cref{assumption:CGMT:var}(ii), they are in fact projection matrices onto two orthogonal subspaces. Therefore to optimize the above over $\tilde \beta^\perp$, it suffices to do so over $(I_p - \Sigma_*) \tilde \beta^\perp$ and $\Sigma_* \tilde \beta^\perp$ individually. Moreover, when optimizing over each of the projected $\tilde \beta^\perp$'s, the optimization takes exactly the same form as (47) of \citet{salehi2019impact}. Similar to them, we introduce
\begin{align*}
    \sigma_1 \;\coloneqq&\; \mfrac{1}{\sqrt{p}} \| (I_p - \Sigma_*) \tilde \beta^\perp \| \;\in\; S_{\sigma_1}
    &\text{ and }&&
    \sigma_2 \;\coloneqq&\; \mfrac{1}{\sqrt{p}}  \| \Sigma_* \tilde \beta^\perp \| \;\in\;  S_{\sigma_2} \;,
\end{align*}
where $ S_{\sigma_1} \coloneqq \{ \| (I_p - \Sigma_*) P_*^\perp \Sigma^{1/2} \beta \| \, |  \, \beta \in S  \}$ and  $ S_{\sigma_2} \coloneqq \{ \|\Sigma_* P_*^\perp \Sigma^{1/2} \beta \| \, |  \, \beta \in S \}$ are both subsets of non-negative real numbers, as well as the auxiliary variables $\nu_1, \nu_2, \tau_1, \tau_2 \geq 0$. We also take note of the fact that 
\begin{align*}
    (I_p - \Sigma_*) \tilde \beta^\perp 
    \;\in&\;
    S_\Sigma^\perp 
    \;\coloneqq\; \{ \sqrt{p}\, (I_p - \Sigma_*) P_*^\perp \Sigma^{1/2} \beta \,|\, \beta \in S  \} \;,
    \\ 
    \Sigma_* \tilde \beta^\perp 
    \;\in&\;
    S_\Sigma 
    \;\coloneqq\;
    \{  \sqrt{p}\,  \Sigma_* P_*^\perp P_\Sigma \Sigma^{1/2} \beta \,|\, \beta \in S  \} \;,
\end{align*}
and denote the projection onto $\textrm{span}(S_\Sigma^\perp )$ as $P_{S_\Sigma^\perp}$ and the projection onto $\textrm{span}(S_\Sigma )$ as $P_{S_\Sigma}$. Then by the same algebra from (47) -- (49) of \citet{salehi2019impact}, we obtain 
\begin{align*}
    \min_{\substack{u \in S_u \\ \mu \in S \\ \alpha \in S^\alpha \\ (\sigma_1, \sigma_2) \in S_{\sigma_1}  \times S_{\sigma_2}  \\ \nu_1, \nu_2 \geq 0}} \, \max_{\substack{w \in \R^p \\ (r_1, r_2) \in S_{r_1} \times S_{r_2} \\ \tau_1, \tau_2 \geq 0 }}
    & \,
    \mfrac{1}{mk} \bone_{mk}^\intercal \, \rho( u )
    -
    \mfrac{1}{mk} \by^\intercal u
    +    
    \mfrac{\lambda}{2m} \| P_\Sigma \mu \|^2_2 
    + 
    \mfrac{1}{p} w^\intercal P_\Sigma \mu 
    \\
    &\;
    -
    \mfrac{\alpha}{p \sqrt{p}}
    \,
    w^\intercal (\Sigma^\dagger)^{1/2} v(\beta^*)
    \\
    &\;
    -
    \mfrac{\sigma_1}{2 \tau_1} 
    -
    \mfrac{\sigma_1 \tau_1}{2}
    \, 
    \Big\| P_{S_\Sigma^\perp}
    \Big( 
    \mfrac{r_1 + r_2}{\sqrt{mk}} \bg_1
    -
    \mfrac{1}{p}
    (I_p - \Sigma_*) (\Sigma^\dagger)^{1/2} w
    \Big) \Big\|^2
    \\
    &\;
    -
    \mfrac{\sigma_2}{2 \tau_2} 
    -
    \mfrac{\sigma_2 \tau_2}{2}
    \, \Big\| P_{S_\Sigma}
    \Big( 
        \mfrac{r_2}{\sqrt{n}} \bg_2 
        -
        \mfrac{1}{p}
        \Sigma_* (\Sigma^\dagger)^{1/2} w
    \Big) \Big\|^2
    \\
    &\;
    +
    \mfrac{r_1}{2 \nu_1} 
    +
    \mfrac{r_1 \nu_1}{2 mk}  
    \Big\| P^\perp_{mk} \Big( u - \kappa_* \alpha \bq - \sigma_1 \bh_1 - \mfrac{\sigma_2}{\sqrt{k}} J_{mk} \bh_2 \Big) \Big\|^2
    \\
    &\;
    +
    \mfrac{r_2}{2 \nu _2}
    +
    \mfrac{r_2 \nu_2}{ 2 mk}  \Big\| P_{mk} \Big( u - \kappa_* \alpha \bq - \sigma_1 \bh_1  - \mfrac{\sigma_2}{\sqrt{k}} J_{mk} \bh_2  \Big) \Big\|^2
    \;.
    \tagaligneq \label{eq:AO:beta:done}
\end{align*}
Note that we have moved the maximization over $w, r_1, r_2$ to be inside the minimization over $\nu_1$, $\nu_2$, $\sigma_1$ and $\sigma_2$. Note also that $P^\perp_{mk} \sigma_2 J_{mk} \bh_2$ evaluates to zero, but we keep this term for the ease of computation later. We also remark that $\nu_1$ can be restricted to be in a compact set $\big\{ \big\| P^\perp_{mk} \big( u - \kappa_* \alpha \bq - \sigma_1 \bh_1 - \frac{\sigma_2}{\sqrt{k}} J_{mk} \bh_2 \big) \big\| \,\big| \, u \in S_u \big\}$ for the purpose of flipping minimization and maximization, and so are $\nu_2, \tau_1, \tau_2$, but we do not do so for notational simplicity.

\vspace{.5em}

\textbf{Maximization over $w \in \R^p$.} We first derive some useful relationships between the different projection matrices introduced so far: By the definition of $\Sigma_*$, we have 
\begin{align*}
    \Sigma_* P_\Sigma \;=\; (\Sigma^\dagger)^{1/2} \Cov[\phi_{11}(X_1), \phi_{11}(X_2) ] (\Sigma^\dagger)^{1/2}  P_\Sigma  
    \;=\; \Sigma_*
    \;.
    \tagaligneq \label{eq:Sigma:star:Sigma}
\end{align*}
Also by the definition of $P_*$ and the idempotency of $\Sigma_*$,
\begin{align*}
    P_* \Sigma_* 
    \;=\;
    \Sigma_* P_* 
    \;=\; 
    \begin{cases}
        \Sigma_* \, \mfrac{(\Sigma_* \Sigma_o^{1/2} \beta^* )  (\Sigma_* \Sigma_o^{1/2} \beta^* )^\intercal }{\|  \Sigma_*  \Sigma_o^{1/2} \beta^*  \|^2} \;=\; P_*
        & \text{ if }    \Sigma_*  \Sigma_o^{1/2} \beta^* \neq 0 \\
        \Sigma_* \times 0 \;=\; P_* & \text{ otherwise }\;.
    \end{cases}
    \tagaligneq \label{eq:Sigma:star:P:star}  
\end{align*}
This implies that 
\begin{align*}
    S_\Sigma^\perp 
    \;=&\;
    \{ \sqrt{p}\, (I_p - \Sigma_*) (I_p - P_*) \Sigma^{1/2} \beta \,|\, \beta \in S  \} 
    \;=\;
    \{ \sqrt{p}\, (I_p - \Sigma_* ) \Sigma^{1/2} \beta \,|\, \beta \in S  \} 
    \;,
    \\
    S_\Sigma 
    \;=&\;
    \{ \Sigma_* (I_p - P_*) \Sigma^{1/2} \beta \,|\, \beta \in S  \} 
    \;=\;
    \{ (\Sigma_* - P_* ) \Sigma^{1/2} \beta \,|\, \beta \in S  \} 
    \;,
\end{align*}
and combining these with the assumption that $\textrm{span}(S) = \R^p$, we can express 
\begin{align*}
    P_{S_\Sigma^\perp} 
    \;=&\; 
    I_p - \Sigma_*
    &\text{ and }&&
    P_{S_\Sigma} 
    \;=&\; 
    \Sigma_* - P_*\;.
    \tagaligneq \label{eq:defn:proj:S:sig}
\end{align*}
This in turn implies that 
\begin{align*}
    P_{S_\Sigma^\perp} (I_p - \Sigma_*) = P_{S_\Sigma^\perp}\;,
    \qquad 
    P_{S_\Sigma} \Sigma_* = P_{S_\Sigma^\perp}\;,
    \qquad
    P_{S_\Sigma^\perp} P_{S_\Sigma} = P_* P_{S_\Sigma^\perp} = P_* P_{S_\Sigma} = 0\;,
    \tagaligneq \label{eq:defn:proj:S:ortho} 
\end{align*}
and that 
\begin{align*}
    P_\Sigma w 
    \;=&\;
    \Sigma^{1/2} (\Sigma^\dagger)^{1/2} w 
    \\
    \;=&\;
    \Sigma^{1/2} P_* (\Sigma^\dagger)^{1/2} w 
    +
    \Sigma^{1/2} P_{S_\Sigma^\perp} (\Sigma^\dagger)^{1/2} w 
    +
    \Sigma^{1/2} P_{S_\Sigma} (\Sigma^\dagger)^{1/2} w 
    \\
    \;=&\;
    \Sigma^{1/2} 
    \mfrac{v(\beta^*) v(\beta^*)^\intercal}{p \kappa_*^2} (\Sigma^\dagger)^{1/2} w
    +
    \Sigma^{1/2} P_{S_\Sigma^\perp} (\Sigma^\dagger)^{1/2} w 
    +
    \Sigma^{1/2} P_{S_\Sigma} (\Sigma^\dagger)^{1/2} w 
    \;.
\end{align*}
Substituting these into \eqref{eq:AO:beta:done}, we obtain 
\begin{align*}
    \min_{\substack{u \in S_u \\ \mu \in S \\ \alpha \in S^\alpha \\ (\sigma_1, \sigma_2) \in S_{\sigma_1}  \times S_{\sigma_2} \\ \nu_1, \nu_2 \geq 0}} \, \max_{\substack{w \in \R^p \\ (r_1, r_2) \in S_{r_1 } \times S_{r_2} \\ \tau_1, \tau_2 \geq 0 }}
    & \,
    \mfrac{1}{mk} \bone_{mk}^\intercal \, \rho( u )
    -
    \mfrac{1}{mk} \by^\intercal u
    +    
    \mfrac{\lambda}{2n} \| P_\Sigma \mu \|^2_2 
    \\
    &\;
    + 
    \Big(
    \mfrac{1}{p^2 \kappa_*^2} 
    \mu^\intercal \Sigma^{1/2}  v(\beta^*)
    -
    \mfrac{\alpha}{p \sqrt{p}}
    \,
    \Big) 
    \,
    v(\beta^*)^\intercal  
    P_*
    (\Sigma^\dagger)^{1/2} w
    \\
    &\;
    +
    \mfrac{1}{p} (\Sigma^{1/2} \mu)^\intercal  P_{S_\Sigma^\perp} (\Sigma^\dagger)^{1/2} w 
    +
    \mfrac{1}{p} (\Sigma^{1/2} \mu)^\intercal   P_{S_\Sigma} (\Sigma^\dagger)^{1/2} w 
    \\
    &\;
    -
    \mfrac{\sigma_1}{2 \tau_1} 
    -
    \mfrac{\sigma_1 \tau_1}{2}
    \, 
    \Big\| P_{S_\Sigma^\perp}
    \Big( 
    \mfrac{r_1 + r_2}{\sqrt{mk}} \bg_1
    -
    \mfrac{1}{p}
    (\Sigma^\dagger)^{1/2} w
    \Big) \Big\|^2
    \\
    &\;
    -
    \mfrac{\sigma_2}{2 \tau_2} 
    -
    \mfrac{\sigma_2 \tau_2}{2}
    \, \Big\| P_{S_\Sigma}
    \Big( 
        \mfrac{r_2}{\sqrt{n}} \bg_2 
        -
        \mfrac{1}{p}
        (\Sigma^\dagger)^{1/2} w
    \Big) \Big\|^2
    \\
    &\;
    +
    \mfrac{r_1}{2 \nu_1} 
    +
    \mfrac{r_1 \nu_1}{2 mk}
    \big\| P^\perp_{mk} ( u - \kappa_* \alpha \bq - \sigma_1 \bh_1 - \mfrac{\sigma_2}{\sqrt{k}} J_{mk} \bh_2 ) \big\|^2
    \\
    &\;
    +
    \mfrac{r_2}{2 \nu _2}
    +
    \mfrac{r_2 \nu_2}{ 2 mk}  \big\| P_{mk} ( u - \kappa_* \alpha \bq - \sigma_1 \bh_1  - \mfrac{\sigma_2}{\sqrt{k}} J_{mk} \bh_2  )\big\|^2
    \;.
    \tagaligneq \label{eq:AO:w:projections}
\end{align*}
To optimize the above over $w$, it again suffices to optimize over three mutually orthogonal vectors $P_* (\Sigma^\dagger)^{1/2} w$, $P_{S_\Sigma^\perp} (\Sigma^\dagger)^{1/2} w$ and $P_{S_\Sigma} (\Sigma^\dagger)^{1/2} w$. The optimization over $P_* (\Sigma^\dagger)^{1/2} w$ is exactly analogous to the optimization over $\bP \bw$ in (49) of \citet{salehi2019impact}, whereas the optimization over the other two vectors are exactly analogous to that over $\bP^\perp \bw$ in (49) of \citet{salehi2019impact}. Therefore by the exact same completion-of-squares argument as in (49) -- (51) in \citet{salehi2019impact} but without taking the asymptotic approximation, the optimization becomes 
\begin{align*}
    \min_{\substack{u \in S_u \\ \mu \in S \\ \alpha \in S^\alpha \\ (\sigma_1, \sigma_2) \in S_{\sigma_1}  \times S_{\sigma_2} \\ \nu_1, \nu_2 \geq 0 \\ \frac{1}{\sqrt{p}} 
    \mu^\intercal 
    \Sigma^{1/2}  v(\beta^*)
    =
    \alpha \kappa_*^2  }}
     \,
     \max_{\substack{ (r_1, r_2) \in S_{r_1 } \times S_{r_2} \\ \tau_1, \tau_2 \geq 0 }}
    & \,
    \mfrac{1}{mk} \bone_{mk}^\intercal \, \rho( u )
    -
    \mfrac{1}{mk} \by^\intercal u
    +    
    \mfrac{\lambda}{2m} \| P_\Sigma \mu \|^2_2 
    -
    \mfrac{\sigma_1}{2 \tau_1} 
    -
    \mfrac{\sigma_2}{2 \tau_2} 
    +
    \mfrac{r_1}{2 \nu_1} 
    +
    \mfrac{r_2}{2 \nu _2}
    \\
    &\; 
    +
    \mfrac{1}{2\sigma_1 \tau_1} 
    \| 
    P_{S_\Sigma^\perp} \Sigma^{1/2} \mu 
    \|^2 
    +
    \mfrac{r_1 + r_2}{\sqrt{mk}} \bg_1^\intercal P_{S_\Sigma^\perp} \Sigma^{1/2} \mu
    \\
    &\;
    + 
    \mfrac{1}{2\sigma_2 \tau_2}
    \| 
    P_{S_\Sigma} \Sigma^{1/2} \mu 
    \|^2
    +
    \mfrac{r_2}{\sqrt{m}}  \bg_2^\intercal P_{S_\Sigma} \Sigma^{1/2} \mu  
    \\
    &\;
    +
    \mfrac{r_1 \nu_1}{2 mk}
    \big\| P^\perp_{mk} ( u -  \kappa_* \alpha \bq - \sigma_1 \bh_1 - \mfrac{\sigma_2}{\sqrt{k}} J_{mk} \bh_2 ) \big\|^2
    \\
    &\;
    +
    \mfrac{r_2 \nu_2}{ 2 mk}  \big\| P_{mk} ( u - \kappa_* \alpha \bq - \sigma_1 \bh_1  - \mfrac{\sigma_2}{\sqrt{k}} J_{mk} \bh_2  )\big\|^2
    \;,
    \tagaligneq \label{eq:AO:w:done}
\end{align*}

\vspace{.5em}

\textbf{Rewriting the minimization over $\mu \in S$. } 
We now flip the order of optimization such that we can perform the minimization over $\mu$ first. This involves computing 
\begin{align*}
    &\;
    \min_{\mu \in S} \,
    \mfrac{\lambda}{2n} \| P_\Sigma \mu \|^2_2  
    +
    \mfrac{1}{2\sigma_1 \tau_1} 
    \| 
    P_{S_\Sigma^\perp} \Sigma^{1/2} \mu 
    \|^2 
    +
    \mfrac{r_1 + r_2}{\sqrt{mk}} \bg_1^\intercal P_{S_\Sigma^\perp} \Sigma^{1/2} \mu
    \\
    &\;\qquad
    +
    \mfrac{1}{2\sigma_2 \tau_2}
    \| 
    P_{S_\Sigma} \Sigma^{1/2} \mu 
    \|^2
    +
    \mfrac{r_2}{\sqrt{n}}  \bg_2^\intercal P_{S_\Sigma} \Sigma^{1/2} \mu 
    \hspace{3em}
    \text{ s.t. }
    \mfrac{1}{\sqrt{p}} 
    \mu^\intercal 
    \Sigma^{1/2}  v(\beta^*)
    =
    \alpha \kappa_*^2  
    \;.
    \tagaligneq \label{eq:AO:minimized:mu}
\end{align*}
Recall from \eqref{eq:defn:proj:S:sig} that $P_{S_\Sigma^\perp} =I_p - \Sigma_*$ and $P_{S_\Sigma}  = \Sigma_* - P_*$. Denote 
\begin{align*}
    \tilde \Sigma_{\sigma, \tau}^c 
    \;\coloneqq&\;
        \mfrac{1}{2 \sigma_1 \tau_1} (P_\Sigma - \Sigma_*) 
        + 
        \mfrac{1}{2 \sigma_2 \tau_2}  (\Sigma_* - P_*)
    \;,
    \\
    \tilde \bg^c 
    \;\coloneqq&\; 
    -
    \mfrac{r_1 + r_2}{\sqrt{mk}} \, (P_\Sigma - \Sigma_*)  \bg_1 
    -
    \mfrac{r_2}{\sqrt{m}} \, (\Sigma_* - P_*) \bg_2 
    \;.
\end{align*}
Then the problem comes 
\begin{align*}
    &\;
    \min_{\mu \in S} \,
    \mfrac{\lambda}{2n} \| P_\Sigma \mu \|^2_2  
    +
    (\Sigma^{1/2} \mu)^\intercal \tilde \Sigma_{\sigma, \tau}^c ( \Sigma^{1/2} \mu)
    - 
    (\tilde \bg^c)^\intercal \Sigma^{1/2} \mu
    \hspace{3em}
    \text{ s.t. }
    \mfrac{1}{\sqrt{p}} 
    \mu^\intercal 
    \Sigma^{1/2}  v(\beta^*)
    =
    \alpha \kappa_*^2  
    \;.
\end{align*}
By \eqref{eq:Sigma:star:P:star} and \eqref{eq:Sigma:star:Sigma}, $P_\Sigma P_* = P_\Sigma \Sigma_* P_* = \Sigma_* P_* = P_*$ and by \eqref{eq:defn:proj:S:sig}, $P_{S_\Sigma^\perp} =I_p - \Sigma_*$ and $P_{S_\Sigma}  = \Sigma_* - P_*$. Then by a similar argument as \eqref{eq:AO:w:projections}, we may express 
\begin{align*}
    P_\Sigma \;=\; P_\Sigma ( P_* + P_{S_\Sigma^\perp} +  P_{S_\Sigma} ) \;=\; P_* + (P_\Sigma - \Sigma_*) + (\Sigma_* - P_*)
    \;,
    \tagaligneq \label{eq:AO:Rp:projections}
\end{align*}
where $P_*$, $P_\Sigma - \Sigma_*$  and $\Sigma_* - P_*$ are projections onto mutually orthogonal subspaces. Meanwhile, recalling the definition of $\tilde \Sigma_{\sigma, \tau}$ and $\tilde \bg$, we can express
\begin{align*}
    \tilde \Sigma_{\sigma, \tau}
    \;=&\; 
    \mfrac{1}{2 \sigma_1 \tau_1} (P_\Sigma  - \Sigma_* ) 
    + 
    \mfrac{1}{2 \sigma_2 \tau_2}  \Sigma_*
    \;=\;
    \tilde \Sigma_{\sigma, \tau}^c 
    +
    \mfrac{1}{2 \sigma_2 \tau_2} P_* 
    \;,
    \\
    \tilde \bg 
    \;=&\; 
    -
    \mfrac{r_1 + r_2}{\sqrt{mk}} \, (P_\Sigma - \Sigma_*)  \bg_1 
    -
    \mfrac{r_2}{\sqrt{m}} \, \Sigma_* \bg_2 
    \;=\; 
    \tilde \bg^c 
    -
    \mfrac{r_2}{\sqrt{m}} P_* \bg_2 
    \;.
\end{align*}
Recalling also that $P_* = v(\beta^*) v(\beta^*)^\intercal / ( p\kappa_*^2 )$, we can write 
\begin{align*}
    &\;
    (\Sigma^{1/2} \mu)^\intercal \tilde \Sigma_{\sigma, \tau}^c ( \Sigma^{1/2} \mu)
    - 
    (\tilde \bg^c)^\intercal \Sigma^{1/2} \mu
    \\
    &\;=\;
    (\Sigma^{1/2} \mu)^\intercal \tilde \Sigma_{\sigma, \tau} (\Sigma^{1/2} \mu)
    -
    \tilde \bg^\intercal \Sigma^{1/2} \mu
    -
    \mfrac{1}{2 \sigma_2 \tau_2} 
    (\Sigma^{1/2} \mu)^\intercal  P_* ( \Sigma^{1/2} \mu)
    +
    \mfrac{r_2}{\sqrt{n}} \big( P_* \bg_2 \big)^\intercal  \Sigma^{1/2} \mu
    \\
    &\;=\;
    (\Sigma^{1/2} \mu)^\intercal \tilde \Sigma_{\sigma, \tau} (\Sigma^{1/2} \mu)
    -
    \tilde \bg^\intercal \Sigma^{1/2} \mu
    -
    \mfrac{\alpha^2 \kappa_*^2}{2 \sigma_2 \tau_2}
    +
    \mfrac{r_2}{\sqrt{m}}  \bg_2^\intercal P_*  \Sigma^{1/2} \mu
    \;.
\end{align*}
Now using a Lagrange multiplier $\theta$ to remove the constraint, the optimization becomes
\begin{align*}
    \min_{ \mu \in S} \, \max_{\theta \in \R} 
    \,
    &\;
    \mfrac{\lambda}{2n} \| P_\Sigma  \mu \|^2_2  
    + 
    (\Sigma^{1/2} \mu)^\intercal \tilde \Sigma_{\sigma, \tau} (\Sigma^{1/2} \mu)
    -
    \Big( \tilde \bg
    + 
    \mfrac{\theta}{\sqrt{p}}    v(\beta^*)
    \Big)^\intercal \Sigma^{1/2} \mu
    \\
    &\;
    -
    \mfrac{\alpha^2 \kappa_*^2}{2 \sigma_2 \tau_2}
    +
    \mfrac{r_2}{\sqrt{m}} \bg_2^\intercal  P_* \Sigma^{1/2} \mu
    +
    \alpha \theta \kappa_*^2 
    \;.
\end{align*}
Since the problem is convex-concave, we can apply the min-max theorem of \citet{rockafellar1997convex} to flip the order of minimum and maximum. Doing this together with a completion of squares, we obtain 
\begin{align*}
    \max_{\theta \in \R} 
    \,
    M_{\bg,\sigma, \tau, \theta}
    -
    \mfrac{1}{4} \big\|  (\tilde \Sigma_{\sigma, \tau}^\dagger)^{1/2}
    \big( \tilde \bg
            + 
            \mfrac{\theta}{\sqrt{p}}    v(\beta^*)
    \big) 
    \big\|^2 
    -
    \mfrac{\alpha^2 \kappa_*^2}{2 \sigma_2 \tau_2}
    +
    \alpha \theta \kappa_*^2 
    \;,
    \tagaligneq \label{eq:AO:minimized:mu:done} 
\end{align*}
where we have denoted the Moreau envelope like term 
\begin{align*}
    M_{\bg,\sigma, \tau, \theta}
    \;\coloneqq\;
    \min_{ \mu \in S} 
    \,
    &\;
    \mfrac{\lambda}{2n} \| P_\Sigma  \mu \|^2_2  
    + 
    \big\| 
        \tilde \Sigma_{\sigma, \tau}^{1/2} (\Sigma^{1/2} \mu)
         -
         \mfrac{1}{2} 
         (\tilde \Sigma_{\sigma, \tau}^\dagger)^{1/2}
        \big( \tilde \bg
                + 
                \mfrac{\theta}{\sqrt{p}}    v(\beta^*)
        \big) 
    \big\|^2 
    \\
    &\;
    +
    \mfrac{r_2}{\sqrt{m}} \bg_2^\intercal P_*  \Sigma^{1/2} \mu
    \;.
\end{align*}

\textbf{Rewriting the minimization over $u \in S_u$.} Meanwhile, the minimization over $u \in S_u$ involves 
\begin{align*}
    \min_{u \in S_u }
    \,
    \mfrac{1}{mk} \bone_{mk}^\intercal \, \rho( u )
    -
    \mfrac{1}{mk} \by^\intercal u
    +
    \mfrac{r_1 \nu_1}{2 mk}  
    \big\| P^\perp_{mk} ( u -  \tilde \bh_{\alpha, \sigma} ) \big\|^2
    +
    \mfrac{r_2 \nu_2}{ 2mk}  \big\| P_{mk} ( u - \tilde \bh_{\alpha, \sigma} )\big\|^2
    \;,
    \tagaligneq \label{eq:AO:minimized:u}
\end{align*}
where we have used the shorthand $\tilde \bh_{\alpha, \sigma} = \kappa_*\alpha \bq - \sigma_1 \bh_1 - \frac{\sigma_2}{\sqrt{k}} J_{mk} \bh_2$. Recall that by definition, $\by = P_{mk} \by$ since $\by$ is a length-$mk$ vector formed by $k$-fold repetitions of $m$ entries. Then we can re-express the loss above as
\begin{align*}
    &
    \;\min_{u \in S_u }
    \,
    \mfrac{1}{mk} \bone_{mk}^\intercal \, \rho( u )
    -
    \mfrac{1}{mk} \by^\intercal P_{mk} u
    +
    \mfrac{r_1 \nu_1}{2 mk}  
    \big\| P^\perp_{mk} ( u -  \tilde \bh_{\alpha, \sigma} ) \big\|^2
    +
    \mfrac{r_2 \nu_2}{ 2 mk}  \big\| P_{mk} ( u - \tilde \bh_{\alpha, \sigma} )\big\|^2
    \\
    \;=&\;
    \min_{u \in S_u }
    \,
    \mfrac{1}{mk} \bone_{mk}^\intercal \, \rho( P_{mk} u + P_{mk}^\perp u )
    +
    \mfrac{r_2 \nu_2}{2 mk}  
    \Big\| P_{mk}  \big( u  -  \mfrac{1}{r_2 \nu_2}  \by - \tilde \bh_{\alpha, \sigma}  \big)\Big\|^2 
    +
    \mfrac{r_1 \nu_1}{ 2 mk}  \big\| P^\perp_{mk} ( u - \tilde \bh_{\alpha, \sigma} )\big\|^2
    \\
    &\qquad \quad
    -
    \mfrac{1}{2 r_2 \nu_2 mk }  
    \big\| P_{mk}  \by \big\|^2
    -
    \mfrac{1}{mk}  \by^\intercal P_{mk} \tilde \bh_{\alpha, \sigma} 
    \;.
\end{align*}
The loss can therefore be minimized separately in $P_{mk} u$ and $P_{mk}^\perp u$. Recall that for a function $f: \cS \rightarrow \R$ and $\cS \subseteq \R^{mk}$, we defined the Moreau envelope
\begin{align*}
    \cM_S(f;v, t) 
    \;\coloneqq\;
    \min_{x \in \cS} f(x) + \mfrac{1}{2 t} \| x - v \|^2_2 
    \;,
\end{align*}
Also recall the definition
\begin{align*}
    M^\perp_{\tilde \bh_{\alpha, \sigma},r,\nu}(\tilde u)
    \;\coloneqq&\;
    \cM_{P_{mk}^\perp(S_u) }
    \big( \bone_{mk}^\intercal \, \rho(\tilde u + \argdot ) \,;\, P^\perp_{mk} \tilde \bh_{\alpha, \sigma} , \mfrac{1}{r_1 \nu_1}  \big)
    \;,
    \\
    M_{\by, \tilde \bh_{\alpha, \sigma}, r, \nu} 
    \;\coloneqq&\; 
    \cM_{P_{mk}(S_u)} 
    \big( 
        M^\perp_{\tilde \bh_{\alpha, \sigma},r,\nu} \,;\, \mfrac{1}{r_2 \nu_2}  \by -  P_{mk} \tilde \bh_{\alpha, \sigma} \,,\, r_2 \nu_2
    \big)
    \;.
\end{align*}
Then \eqref{eq:AO:minimized:u} can be expressed as 
\begin{align*}
    &\; 
    \min_{\tilde u \in  P_{mk}(S_u) }
    \,
    \mfrac{1}{mk}
    M^\perp_{\tilde \bh_{\alpha, \sigma},r,\nu}(\tilde u)
    +
    \mfrac{r_2 \nu_2}{2 mk}  
    \Big\| P_{mk}  \big( \tilde u  -  \mfrac{1}{r_1 \nu_1}  \by - \tilde \bh_{\alpha, \sigma}  \big)\Big\|^2 
    -
    \mfrac{1}{2 r_2 \nu_2 mk }  
    \big\| P_{mk}  \by \big\|^2
    -
    \mfrac{1}{mk}  \by^\intercal P_{mk} \tilde \bh_{\alpha, \sigma} 
    \\
    &\;=\;
    \mfrac{1}{mk}  M_{\by, \tilde \bh_{\alpha, \sigma}, r, \nu}  
    -
    \mfrac{1}{2 r_2 \nu_2 mk }  
    \| \by \|^2
    -
    \mfrac{1}{mk}  \by^\intercal \tilde \bh_{\alpha, \sigma} 
    \;, \tagaligneq \label{eq:AO:minimized:u:done}
\end{align*}
where we have used $P_{mk} \by = \by$ again in the last line. Substituting both \eqref{eq:AO:minimized:mu:done} and \eqref{eq:AO:minimized:u:done} into \eqref{eq:AO:w:done} yields 
\begin{align*}
    \min_{\substack{\alpha \in S^\alpha \\ (\sigma_1, \sigma_2) \in S_{\sigma_1}  \times S_{\sigma_2} \\ \nu_1, \nu_2 \geq 0 }}
     \,
     \max_{\substack{ (r_1, r_2) \in S_{r_1 } \times S_{r_2} \\ \tau_1, \tau_2 \geq 0 \\ \theta \in \R }}
     \,
    & \,
    -
    \mfrac{\sigma_1}{2 \tau_1} 
    -
    \mfrac{\sigma_2}{2 \tau_2} 
    +
    \mfrac{r_1}{2 \nu_1} 
    +
    \mfrac{r_2}{2 \nu _2}
    +
    \alpha \theta \kappa_*^2 
    -
    \mfrac{\alpha^2 \kappa_*^2}{2 \sigma_2 \tau_2}
    \\
    &\, 
    +
    M_{\bg,\sigma, \tau, \theta}
    -
    \mfrac{1}{4} \big\|  (\tilde \Sigma_{\sigma, \tau}^\dagger)^{1/2}
    \big( \tilde \bg
            + 
            \mfrac{\theta}{\sqrt{p}}    v(\beta^*)
    \big) 
    \big\|^2 
    \\
    &\,
    +
    \mfrac{1}{mk}  M_{\by, \tilde \bh_{\alpha, \sigma}, r, \nu}  
    -
    \mfrac{1}{2 r_2 \nu_2 mk }  
    \| \by \|^2
    -
    \mfrac{1}{mk}  \by^\intercal \tilde \bh_{\alpha, \sigma} 
    \;,
\end{align*}
which equals $R^{\rm SO}_{S, S_u, S_v}$. 

\vspace{.5em}

\textbf{Justifying the flipping of the minima and maxima.} To conclude, we need to justify the flipping of min-max in the analysis of the auxiliary optimization above. The same argument has been performed for the logistic loss in the isotropic, unaugmented case in \citet{salehi2019impact} and in more details for general losses in \citet{thrampoulidis2018precise}. For completeness, we repeat the arguments of the proof of Lemma A.3~of \citet{thrampoulidis2018precise} in our context to illustrate how the non-asymptotic inequalities arise in our result for one particular flipping, and refer readers to Appendix A of \citet{thrampoulidis2018precise} for more details in the general setup. 

\vspace{.5em}

We now consider the flipping of $\min_{\beta \in S}$ and $\max_{(r_1,r_2) \in S_{r_1} \times S_{r_2}}$ in \eqref{eq:AO:2}. First define the loss function $\cL_1(\beta, u, \mu, w, r_1, r_2)$ as in  \eqref{eq:AO:2} and denote the risk at \eqref{eq:AO:2} by
\begin{align*}
    \cR_1
    \;\coloneqq&\;
    \min_{\substack{\beta \in S \\ u \in \cS_u \\ \mu \in \cS}}
    \;\;
    \max_{\substack{w \in \R^p \\ (r_1, r_2) \in S_{r_1} \times S_{r_2}}}
    \cL_1(\beta, u, \mu, w, r_1, r_2) 
    \;.
\end{align*}
For convenience, we also abbreviate the dependence on random variables in
\begin{align*}
    R^{\rm PO}_{S, S_u, S_v} 
    \;=&\; 
    R^{\rm PO}_{S, S_u, S_v}(\bG \Sigma^{1/2}_o \beta^*, \bG^\Phi \Sigma^{1/2})\;,
    &\;&&
    L^{\rm PO}_{\beta, u, v} 
    \;=&\; 
    L^{\rm PO}_{\beta, u, v}(\bG \Sigma^{1/2}_o \beta^*, \bG^\Phi \Sigma^{1/2})
    \;,
    \\
    L^{\rm AO}_{\beta,u,v}
    \;=&\; 
    L^{\rm AO}_{\beta,u,v}(\by, \bG^\Phi  P_*, \bg, \bh)
    \;.
\end{align*} 
By the computation of the auxiliary formulation up to \eqref{eq:AO:2}, and by \lemmaref{lem:PO:AO}, we can apply \theoremref{thm:CGMT_2}(i) and (ii) to obtain that
\begin{align*}
    \P\big( 
        R^{\rm PO}_{S, S_u, S_v}
        \;\leq\; c
    \big)
    \;\leq&\; 
    4 \,
    \P\big( 
       \cR_1
        \;\leq\; c
    \big)
    &\text{ and }&&
    \P\big( 
        R^{\rm PO}_{S, S_u, S_v}
        \;\geq\; c
    \big)
    \;\leq&\; 
    4 \,
    \P\big( 
       \cR_1
        \;\geq\; c
    \big)
    \tagaligneq \label{eq:AO:SO:cgmt}
\end{align*}
for all $c \in \R$. Now define 
\begin{align*}
    \cR'_1
    \;\coloneqq\;
    \max_{(r_1, r_2) \in S_{r_1} \times S_{r_2}}
    \;
    \min_{\substack{\beta \in S \\ u \in \cS_u \\ \mu \in \cS}}
    \;\;
    \max_{w \in \R^p}
    \cL_1(\beta, u, \mu, w, r_1, r_2) 
    \;.
\end{align*}
By the min-max inequality (\citet{rockafellar1997convex}, Lemma 36.1), we have 
$\cR'_1 \leq \cR_1$ and therefore 
\begin{align*}
    \P\big( 
        R^{\rm PO}_{S, S_u, S_v}
        \;\leq\; c
    \big)
    \;\leq&\; 
    4 \,
    \P\big( 
       \cR_1
        \;\leq\; c
    \big)
    \;\leq\;
    4 \,
    \P\big( 
       \cR'_1
        \;\leq\; c
    \big)
    \;. 
    \tagaligneq
    \label{eq:AO:SO:minmax:one:side}
\end{align*}
To relate $\{ R^{\rm PO}_{S, S_u, S_v} \geq c\}$ to $\{\cR'_1 \geq c\}$, we apply the min-max theorem  (\citet{rockafellar1997convex}, Corollary 37.3.2) to obtain that 
\begin{align*}
    R^{\rm PO}_{S, S_u, S_v}
    \;=\;
    \min_{\substack{\beta \in S \\ u \in S_u}} \, \max_{v \in S_v} 
    \,
    L^{\rm PO}_{\beta, u, v} 
    \;=\;
    \max_{v \in S_v}  \,  \min_{\substack{\beta \in S \\ u \in S_u}} \,  L^{\rm PO}_{\beta, u, v} 
    \;,
\end{align*}
and applying \theoremref{thm:CGMT_2} gives 
\begin{align*}
    \P( R^{\rm PO}_{S, S_u, S_v} \geq c ) \;\leq\; 4 \, \P\bigg(   \max_{v \in S_v}  \,  \min_{\substack{\beta \in S \\ u \in S_u}}  L^{\rm AO}_{\beta,u,v}  \,\geq\, c \bigg)\;.
    \tagaligneq \label{eq:AO:SO:minmax:intermediate}
\end{align*}
Now recall that 
\begin{align*}
    S_{r_1}
    \;\coloneqq&\; 
    \Big\{ \mfrac{1}{\sqrt{mk}} \| P_{mk} v \| \,\Big|\, v \in S_v \Big\}
    &\text{ and }&&
    S_{r_2}
    \;\coloneqq&\; 
    \Big\{ \mfrac{1}{\sqrt{mk}} \| P^\perp_{mk} v \| \,\Big|\, v \in S_v \Big\}
    \;.
\end{align*}
Defining $\tilde S_v(r_1,r_2) \coloneqq \{ v \in S_v  \,\big|\,  \frac{1}{\sqrt{mk}} \| P_{mk} v \| = r_1 \,,\, \frac{1}{\sqrt{mk}} \| P^\perp_{mk} v \| = r_2 \}$,  we can rewrite 
\begin{align*}
    \max_{v \in S_v}  \,  \min_{\substack{\beta \in S \\ u \in S_u}}  L^{\rm AO}_{\beta,u,v} 
    \;=&\;
    \max_{(r_1, r_2) \in S_{r_1} \times S_{r_2}} \, 
    \max_{\tilde v_1 \in \tilde S_v(r_1,r_2)} \,
    \min_{\substack{\beta \in S \\ u \in S_u}} \,
    \,
    L^{\rm AO}_{\beta,u,\tilde v} 
    \\
    \;\overset{(a)}{\leq}&\;
    \max_{(r_1, r_2) \in S_{r_1} \times S_{r_2}} \, 
    \min_{\substack{\beta \in S \\ u \in S_u}} \,
        \max_{\tilde v \in \tilde S_v(r_1,r_2)}
        \,
        L^{\rm AO}_{\beta,u,\tilde v} 
    \;\overset{(b)}{=}\;
    \cR'_1
    \;,
\end{align*}
where we have applied the min-max inequality  (\citet{rockafellar1997convex}, Lemma 36.1) in $(a)$ followed by the same computation up to \eqref{eq:AO:2} to maximize the loss over $\tilde v$. Combining this with \eqref{eq:AO:SO:minmax:intermediate} gives 
\begin{align*}
    \P( R^{\rm PO}_{S, S_u, S_v} \geq c ) 
    \;\leq\; 
    4 \, \P\bigg(   \max_{v \in S_v}  \,  \min_{\substack{\beta \in S \\ u \in S_u}}  L^{\rm AO}_{\beta,u,v}  \,\geq\, c \bigg)
    \;\leq\; 
    4 \, \P( \cR'_1 \,\geq\, c )
    \;.
\end{align*}
Together with \eqref{eq:AO:SO:minmax:one:side}, this shows that $\cR'_1$ is equivalent to $\cR_1$ in the sense that the CGMT inequalities of \eqref{eq:AO:SO:cgmt} hold also with $\cR_1$ replaced by $\cR'_1$, therefore justifying the flipping of the minimum and the maximum. The remaining flipping of minimum and maximum over compact sets hold for the same reason, and any flipping that involves the Lagrange multiplier $w \in \R^p$ in \eqref{eq:AO:2} can be done in a similar manner by introducing the Lagrange multiplier directly to the \eqref{PO}. This proves the first statement that for all $c \in \R$,
\begin{align*}
    \P( R^{\rm PO}_{S, S_u, S_v} \leq c ) 
    \;\leq&\; 
    4 \, \P( R^{\rm SO}_{S, S_u, S_v} \,\leq\, c )
    &\text{ and }&&
    \P( R^{\rm PO}_{S, S_u, S_v} \geq c ) 
    \;\leq&\; 
    4 \, \P( R^{\rm SO}_{S, S_u, S_v} \,\geq\, c )
    \;.
\end{align*}

\vspace{.5em}

\textbf{Partial statement when $S$ is replaced by $S_{c, \epsilon}$.} When the optimization is over $S_{c, \epsilon}$, which is no longer compact, we cannot apply the min-max theorem for flipping min and max that involve $S_{c,\epsilon}$.  However, notice that this change only affects optimizations over $\beta$, $\alpha$, $\sigma_1$ and $\sigma_2$. For the optimization over $\beta$, we have shown in \eqref{eq:AO:SO:minmax:one:side} that for the desired partial bound, the flipping of min and max does not require the min-max theorem. For the optimizations over $\alpha$, $\sigma_1$ and $\sigma_2$, notice that the new domains of optimization for each of these variables are 
\begin{align*}
    \Big\{ \mfrac{v(\beta^*)^\intercal \Sigma^{1/2} \beta}{\sqrt{p} \, \kappa^2_*} \,\Big| \, \beta \in S \Big\} 
    \;,
    \quad
    \big\{ \| (I_p - \Sigma_*) P_*^\perp \Sigma^{1/2} \beta \| \, \big|  \, \beta \in S  \big\}
    \;,
    \quad
    \big\{ \|\Sigma_* P_*^\perp \Sigma^{1/2} \beta \| \, \big|  \, \beta \in S  \big\}
    \;,
\end{align*}
which are the 1d images of continuous functions on $\R^p$. Since $S_{c,\epsilon}$ is connected, the above sets are connected and therefore convex since they are one-dimensional. Therefore the replacement of $S$ by $S_{c,\epsilon}$ does not affect the application of min-max theorem that concerns $\alpha$, $\sigma_1$ and $\sigma_2$. This proves the partial upper bound analogous to \eqref{eq:AO:SO:minmax:one:side}: For all $c \in \R$,
\begin{align*}
    \P( R^{\rm PO}_{S_{c,\epsilon}, S_u, S_v} \leq c ) 
    \;\leq&\; 
    4 \, \P( R^{\rm SO}_{S_{c,\epsilon}, S_u, S_v}  \,\leq\, c )
    \;.    
\end{align*}
\jmlrQED

\subsection{Proof of \lemmaref{lem:SO:DO}: Equivalence between \eqref{SO} and \eqref{DO}}

It is convenient to restate the optimization \eqref{SO}:
\begin{align*}
    \min_{\substack{\alpha \in S^\alpha \\ (\sigma_1, \sigma_2) \in S_{\sigma_1}  \times S_{\sigma_2} \\ \nu_1, \nu_2 \geq 0 }}
        \,
        \max_{\substack{ (r_1, r_2) \in S_{r_1 } \times S_{r_2} \\ \tau_1, \tau_2 \geq 0 \\ \theta \in \R }}
        \,
    & \,
    -
    \mfrac{\sigma_1}{2 \tau_1} 
    -
    \mfrac{\sigma_2}{2 \tau_2} 
    +
    \mfrac{r_1}{2 \nu_1} 
    +
    \mfrac{r_2}{2 \nu _2}
    +
    \alpha \theta \kappa_*^2 
    -
    \mfrac{\alpha^2 \kappa_*^2}{2 \sigma_2 \tau_2}
    \\
    &\, 
    +
    M_{\bg,\sigma, \tau, \theta}
    -
    \mfrac{1}{4} \big\|  (\tilde \Sigma_{\sigma, \tau}^\dagger)^{1/2}
    \big( \tilde \bg
            + 
            \mfrac{\theta}{\sqrt{p}}    v(\beta^*)
    \big) 
    \big\|^2 
    \\
    &\,
    +
    \mfrac{1}{mk}  M_{\by, \tilde \bh_{\alpha, \sigma}, r, \nu}  
    -
    \mfrac{1}{2 r_2 \nu_2 mk }  
    \| \by \|^2
    -
    \mfrac{1}{mk}  \by^\intercal \tilde \bh_{\alpha, \sigma} 
    \;,
\end{align*}
As with \cite{salehi2019impact}, we exploit the fact that the optimization is over finitely many one-dimensional variables. It therefore suffices to analyze the asymptotics of the loss function directly, as one may first approximate the minimization and maximization over $(\tilde \alpha, \tilde \theta)$ by a smooth function and then take the approximation error to zero as $m, p \rightarrow \infty$. 

\vspace{.5em}

\textbf{Compute terms involving $\bg_1$ and $\bg_2$. } Recall that $\bg_1$ and $\bg_2$ are independent standard Gaussian $\R^p$ vectors, and that 
\begin{align*}
    \tilde \bg 
    \;=&\; 
    -
    \mfrac{r_1 + r_2}{\sqrt{mk}} \, (P_\Sigma  - \Sigma_* ) \bg_1 
    -
    \mfrac{r_2}{\sqrt{m}} \, \Sigma_* \bg_2 
    \;,
    \\
    \tilde \Sigma_{\sigma, \tau}
    \;=&\; 
    \mfrac{1}{2 \sigma_1 \tau_1} (P_\Sigma  - \Sigma_* ) 
    + 
    \mfrac{1}{2 \sigma_2 \tau_2}  \Sigma_*
    \;,
    \\
    M_{\bg,\sigma, \tau, \theta}
    \;=&\;
    \min_{ \mu \in S} 
    \;
    \mfrac{\lambda}{2n} \| P_\Sigma  \mu \|^2_2  
    + 
    \big\| 
        \tilde \Sigma_{\sigma, \tau}^{1/2} (\Sigma^{1/2} \mu)
         -
         \mfrac{1}{2} 
         (\tilde \Sigma_{\sigma, \tau}^\dagger)^{1/2}
        \big( \tilde \bg
                + 
                \mfrac{\theta}{\sqrt{p}}    v(\beta^*)
        \big) 
    \big\|^2 
    \\
    &\;\qquad
    +
    \mfrac{r_2}{\sqrt{m}} \bg_2^\intercal  P_* \Sigma^{1/2} \mu
    \;.
\end{align*}
We focus on handling $M_{\bg,\sigma, \tau, \theta}$. First note that $\frac{r_2}{\sqrt{n}} \bg_2^\intercal P_*  \Sigma^{1/2} \mu$ depends on $\mu$ through the scalar $v(\beta^*) \Sigma^{1/2} \mu$, so by a similar reasoning as above, we can apply the law of large numbers directly to this term and obtain that it converges to zero in probability. Using $o_\P(1)$ to denote terms that converge in probability to zero, we then have
\begin{align*}
    M_{\bg,\sigma, \tau, \theta}
    \;=\;
    &\;
    o_\P(1)
    +
    \min_{ \mu \in S} 
    \mfrac{\lambda}{2m} \| P_\Sigma  \mu \|^2_2  
    + 
    \big\| 
        \tilde \Sigma_{\sigma, \tau}^{1/2} (\Sigma^{1/2} \mu)
         -
         \mfrac{1}{2} 
         (\tilde \Sigma_{\sigma, \tau}^\dagger)^{1/2}
        \big( \tilde \bg
                + 
                \mfrac{\theta}{\sqrt{p}}    v(\beta^*)
        \big) 
    \big\|^2 
    \\
    \;=&\;
    \tilde M_{\bg, \sigma,\tau,\theta}
    + 
    \mfrac{1}{4} \Big\| (\tilde \Sigma_{\sigma, \tau}^\dagger)^{1/2}
    \big( \tilde \bg
            + 
            \mfrac{\theta}{\sqrt{p}}    v(\beta^*)
    \big)  \Big\|^2
    + o_\P(1)
    \;,
    \tagaligneq \label{eq:DO:mu:Moreau:first}
\end{align*}
where 
\begin{align*}
    \tilde M_{\bg, \sigma,\tau,\theta}
    \;\coloneqq\;
    \min_{ \mu \in S} 
    (\Sigma^{1/2} \mu)^\intercal \Big( \mfrac{\lambda}{2n} \Sigma^\dagger +  \tilde \Sigma_{\sigma, \tau} 
    \Big) 
    (\Sigma^{1/2} \mu)
    -
    (\Sigma^{1/2} \mu)^\intercal 
    \big( \tilde \bg
        + 
        \mfrac{\theta}{\sqrt{p}}    v(\beta^*)
    \big)
    \;.
\end{align*}
The second term of \eqref{eq:DO:mu:Moreau:first} cancels with the other $(\bg_1, \bg_2)$-dependent term in the overall loss, so the only remaining $(\bg_1, \bg_2)$-dependent term is $\tilde M_{\bg, \sigma,\tau,\theta}$. By a completion of squares, we obtain 
\begin{align*}
    \tilde M_{\bg, \sigma,\tau,\theta}
    \;=\; 
    \min_{ \mu \in S} 
    &\; \Big\| 
        \Big( \mfrac{\lambda}{2m} \Sigma^\dagger +  \tilde \Sigma_{\sigma, \tau} 
        \Big)^{1/2} 
        (\Sigma^{1/2} \mu)
        -
        \mfrac{1}{2}
        \Big(\Big( \mfrac{\lambda}{2m} \Sigma^\dagger +  \tilde \Sigma_{\sigma, \tau} 
        \Big)^\dagger\Big)^{1/2}
        \big( \tilde \bg
            + 
            \mfrac{\theta}{\sqrt{p}}    v(\beta^*)
        \big)
    \Big\|^2
    \\
    &\;
    - 
    \mfrac{1}{4}
    \big( \tilde \bg
        + 
        \mfrac{\theta}{\sqrt{p}}    v(\beta^*)
    \big)^\intercal 
    \Big( \mfrac{\lambda}{2m} \Sigma^\dagger +  \tilde \Sigma_{\sigma, \tau} \Big)^\dagger
    \big( \tilde \bg
        + 
        \mfrac{\theta}{\sqrt{p}}    v(\beta^*)
    \big)
    \;.
\end{align*}
The second term does not involve $\mu$, so we first seek to take a limit of this term. Recall from \eqref{eq:AO:Rp:projections} that $P_*$, $P_\Sigma - \Sigma_*$ and $\Sigma_* - P_*$ are projections onto mutually orthogonal subspaces and that $P_\Sigma \Sigma_* = \Sigma_*$, $P_\Sigma P_* = P_*$. We can then express 
\begin{align*}
    \Big( \mfrac{\lambda}{2m} \Sigma^\dagger +  \tilde \Sigma_{\sigma, \tau} \Big)^\dagger
    \;=&\;
    \Big( 
        \mfrac{\lambda}{2m} \Sigma^\dagger 
        +
        \mfrac{1}{2 \sigma_1 \tau_1} (P_\Sigma  - \Sigma_* )
        + 
        \mfrac{1}{2 \sigma_2 \tau_2} \Sigma_*
    \Big)^\dagger
    \;,
\end{align*}
which implies that 
\begin{align*}
    \Big( \mfrac{\lambda}{2m} \Sigma^\dagger +  \tilde \Sigma_{\sigma, \tau} \Big)^\dagger
    \big( \tilde \bg
        + 
        \mfrac{\theta}{\sqrt{p}}    v(\beta^*)
    \big)
    \;=&\;
    - \mfrac{2(r_1+r_2)\sigma_1\tau_1}{\sqrt{mk}} 
        \Big( \mfrac{2\sigma_1 \tau_1 \lambda}{2m} \Sigma^\dagger + I_p \Big)^\dagger (P_\Sigma - \Sigma_*) 
    \bg_1
    \\
    &\;
    -
    \mfrac{2 r_2 \sigma_2\tau_2}{\sqrt{m}} 
        \Big( \mfrac{2\sigma_2 \tau_2 \lambda}{2m} \Sigma^\dagger + I_p \Big)^\dagger \Sigma_* 
    \bg_2
    \\
    &\;
    +
    \mfrac{2 \sigma_2 \tau_2 \theta}{\sqrt{p}} 
    \Big( 
        \mfrac{2 \sigma_2 \tau_2 \lambda}{2m}
        \Sigma^\dagger  
        + 
        I_p   
    \Big)^\dagger
    v(\beta^*)
    \\
    &\;
    + O\big( \| m^{-1/2} P_* \bg_1 \| +  \| m^{-1/2} P_* \bg_2 \|  \big)
    \;.
    \tagaligneq \label{eq:DO:decompose:tilde:Sig}
\end{align*}
Also notice that any term linear in $\bg_1$ or $\bg_2$ has expectation zero, which vanishes. Recall also that  $v(\beta^*) = \sqrt{p} \Sigma_* \Sigma_o^{1/2} \beta^*$ and  $\kappa_* = \| \Sigma_* \Sigma_o^{1/2} \beta^* \|$. 
Computing the inverse along each orthogonal subspace explicitly and taking a limit with $m, p \rightarrow \infty$ and $p/n = p/(mk) \rightarrow \kappa$, we obtain 
\begin{align*}
    -
    \mfrac{1}{4}
    \big( \tilde \bg
        + &
        \mfrac{\theta}{\sqrt{p}}    v(\beta^*)
    \big)^\intercal 
    \Big( \mfrac{\lambda}{2m} \Sigma^\dagger +  \tilde \Sigma_{\sigma, \tau} \Big)^\dagger
    \big( \tilde \bg
        + 
        \mfrac{\theta}{\sqrt{p}}    v(\beta^*)
    \big)
    \\
    \;\overset{\P}{\rightarrow}&\;
    - 
    \mfrac{(r_1 + r_2)^2 \sigma_1 \tau_1}{2k} 
    \,
    \bar \chi^{\sigma, \tau}_{11}
    -
    \mfrac{r_2^2 \sigma_2 \tau_2}{2} \, \bar \chi^{\sigma, \tau}_{12}
    - 
    \mfrac{\theta^2  \bar \kappa_*^2 \sigma_2 \tau_2}{2} 
    \, \bar \chi^{\sigma, \tau}_{13}
    \;=\; 
    - \bar \chi^{r, \theta, \sigma, \tau}_1
    \;.
\end{align*}
where we have recalled that $P_* = v(\beta^*) v(\beta^*)^\intercal / (p \kappa_*^2)$, $ \bar \kappa_* = \lim \kappa_*$ and
\begin{align*}
    \bar \chi^{\sigma,\tau}_{11}
    \;\coloneqq&\;
    \lim
    \mfrac{\Tr\big(
    \big( \frac{2 \sigma_1 \tau_1\lambda}{2m}  \Sigma^\dagger  
        + 
        I_p  \big)^\dagger
        (P_\Sigma - \Sigma_*)
    \big)
    }{m}\;,
    \\
    \bar \chi^{\sigma, \tau}_{12}
    \;\coloneqq&\;
    \lim
    \mfrac{\Tr\big(
        \big( 
            \frac{2 \sigma_2 \tau_2\lambda}{2m}  \Sigma^\dagger  
            + 
            I_p 
        \big)^\dagger
        \Sigma_*
    \big)
    }{m}\;,
    \\
    \bar \chi^{\sigma,\tau}_{13}
    \;\coloneqq&\;
    \lim
    \Tr\Big(
    \Big( 
        \mfrac{2 \sigma_2 \tau_2\lambda}{2m}  \Sigma^\dagger  
        + 
        I_p   
    \Big)^\dagger
    P_*
    \Big)
    \;.
\end{align*}
To address the minimization over $\mu \in S$, notice that the only difference between the two choices of $S$ are via the restriction on $\mu^\intercal \Sigma_{\rm new} \mu$. Recall that the two different choices of $S$ differs only in $\beta^\intercal \Sigma_{\rm new} \beta$. Let $P_{\Sigma_{\rm new}}$ be the projection onto the positive eigenspace of $P_{\Sigma_{\rm new}}$ and $P_{\Sigma_{\rm new}}^\perp = I_p - P_{\Sigma_{\rm new}}$. Then we can rewrite the minimization as 
\begin{align*}
    \min_{ \substack{\mu \in P_{\Sigma_{\rm new}}(S) \\ \mu' \in P_{\Sigma_{\rm new}}^\perp(S)  }} 
    &\; \Big\| 
        \Big( \mfrac{\lambda}{2m} \Sigma^\dagger +  \tilde \Sigma_{\sigma, \tau} 
        \Big)^{1/2} 
        \Sigma^{1/2} 
        \Big( 
            \mu + \mu' 
        -
        \mfrac{1}{2}
        (\Sigma^\dagger)^{1/2}
        \Big( \mfrac{\lambda}{2m} \Sigma^\dagger +  \tilde \Sigma_{\sigma, \tau} 
        \Big)^\dagger
        \big( \tilde \bg
            + 
            \mfrac{\theta}{\sqrt{p}}    v(\beta^*)
        \big)
        \Big)
    \Big\|^2
    \;.
\end{align*}
With either choice of $S$, $\mu'$ can be chosen freely within $P^\perp_{\Sigma_{\rm new}}(\R^p)$ so long as $\| \mu' \|_2 = O(\sqrt{p})$. Minimizing over $\mu'$ first and noting that $P_{\Sigma_{\rm new}} = (\Sigma_{\rm new}^\dagger)^{1/2} \Sigma_{\rm new}^{1/2}$, we obtain
\begin{align*}
    \min_{ \mu \in P_{\Sigma_{\rm new}}(S)} 
    &\; 
    \Big\| 
        \Big( \mfrac{\lambda}{2m} \Sigma^\dagger +  \tilde \Sigma_{\sigma, \tau} 
        \Big)^{1/2} 
        \Sigma^{1/2} 
        (\Sigma_{\rm new}^\dagger)^{1/2}
        \Big( 
            \Sigma_{\rm new}^{1/2} \mu 
        -
        \mfrac{1}{2}
        \Sigma_{\rm new}^{1/2} (\Sigma^\dagger)^{1/2}
        \Big( \mfrac{\lambda}{2m} \Sigma^\dagger +  \tilde \Sigma_{\sigma, \tau} 
        \Big)^\dagger
        \big( \tilde \bg
            + 
            \mfrac{\theta}{\sqrt{p}}    v(\beta^*)
        \big)
        \Big)
    \Big\|^2
    \;.
\end{align*}
Setting $\Sigma_{\rm new}^{1/2} \mu$ in the direction of minimization, we have that for some $c(\mu) \in \R$,
\begin{align*}
    \Sigma_{\rm new}^{1/2} \mu 
    \;=&\;
    c(\mu) 
    \bg'\;,
    \qquad
    \bg' \;\coloneqq\; 
    \mfrac{1}{2}
    \Sigma_{\rm new}^{1/2} (\Sigma^\dagger)^{1/2}
        \Big( \mfrac{\lambda}{2m} \Sigma^\dagger +  \tilde \Sigma_{\sigma, \tau} 
        \Big)^\dagger
        \big( \tilde \bg
            + 
            \mfrac{\theta}{\sqrt{p}}    v(\beta^*)
        \big)
    \;,
\end{align*}
which allows us to rewrite 
\begin{align*}
     \mfrac{\Big\| 
        \Big( \mfrac{\lambda}{2m} \Sigma^\dagger +  \tilde \Sigma_{\sigma, \tau} 
        \Big)^{1/2} 
        \Sigma^{1/2} (\Sigma_{\rm new}^\dagger)^{1/2} \bg' 
    \Big\|^2}{\| \bg' \|^2}
    \;\times\; 
    \min_{ \mu \in P_{\Sigma_{\rm new}}(S)} 
    \;
    (c(\mu) \| \bg' \|  - \|\bg'\|)^2
    \;.
\end{align*}
This is now an optimization over a scalar, so we can again take the limit inside the minimization. We proceed to compute the limits of the two norms involving $\bg'$. Recycling the computation in \eqref{eq:DO:decompose:tilde:Sig}, we have  
\begin{align*}
    \| \bg'\|^2
    \;=&\; 
    \mfrac{(r_1+r_2)^2\sigma_1^2\tau_1^2}{mk} 
    \,
    \Big\| 
        \Sigma_{\rm new}^{1/2} (\Sigma^\dagger)^{1/2}
        \Big( \mfrac{2\sigma_1 \tau_1 \lambda}{2m} \Sigma^\dagger + I_p \Big)^\dagger (P_\Sigma - \Sigma_*) 
    \bg_1
    \Big\|^2
    \\
    &\;
    +
    \mfrac{r_2^2 \sigma_2^2\tau_2^2}{n}
    \,
    \Big\| 
        \Sigma_{\rm new}^{1/2} (\Sigma^\dagger)^{1/2}
        \Big( \mfrac{2\sigma_2 \tau_2 \lambda}{2m} \Sigma^\dagger + I_p \Big)^\dagger P_*
    \bg_2
    \Big\|^2
    \\
    &\;
    +
    \mfrac{\theta^2 \sigma_2^2\tau_2^2}{p}
    \Big\|
    \Sigma_{\rm new}^{1/2} (\Sigma^\dagger)^{1/2}
    \Big( 
        \mfrac{2\sigma_2 \tau_2 \lambda}{2m} \Sigma^\dagger
        + 
        I_p   
    \Big)^\dagger
    v(\beta^*)
    \Big\|^2
    +
    o_\P(1)
    \\
    \;\overset{\P}{\rightarrow}&\;
    \mfrac{(r_1+r_2)^2\sigma_1^2\tau_1^2}{k} 
    \bar \chi^{\sigma,\tau}_{21}
    +
    r_2^2 \sigma_2^2\tau_2^2 \,
    \bar \chi^{\sigma,\tau}_{22}
    +
    \theta^2  \bar \kappa_*^2 \sigma_2^2 \tau_2^2
    \,
    \bar \chi^{\sigma,\tau}_{23}
    \;=\;
    \bar \chi^{r, \theta, \sigma, \tau}_2
    \;,
\end{align*}
where we have denoted 
\begin{align*}
    \bar \chi^{\sigma,\tau}_{21} 
    \;\coloneqq&\;
    \lim 
    \mfrac{\big\|
            \Sigma_{\rm new}^{1/2} (\Sigma^\dagger)^{1/2}
            \big( \frac{2\sigma_1 \tau_1 \lambda}{2m} \Sigma^\dagger + I_p \big)^\dagger (P_\Sigma - \Sigma_*) 
        \big\|^2 
    }
    {m}
    \;,
    \\
    \bar \chi^{\sigma,\tau}_{22}
    \;\coloneqq&\;
    \lim 
    \mfrac{\big\|
            \Sigma_{\rm new}^{1/2} (\Sigma^\dagger)^{1/2}
            \Big( \mfrac{2\sigma_2 \tau_2 \lambda}{2m} \Sigma^\dagger + I_p \Big)^\dagger \Sigma_* 
        \big\|^2 
    }
    {m}
    \;,
    \\
    \bar \chi^{\sigma,\tau}_{23}
    \;\coloneqq&\;
    \lim \,
    \Big\|
        \Sigma_{\rm new}^{1/2} (\Sigma^\dagger)^{1/2}
        \Big( 
            \mfrac{2\sigma_2 \tau_2 \lambda}{2m} \Sigma^\dagger  
            + 
            I_p   
        \Big)^\dagger
        P_*
    \Big\|^2 
    \;.
\end{align*}
Similarly, we have 
\begin{align*}
    &\;\Big\| 
        \Big( \mfrac{\lambda}{2m} \Sigma^\dagger +  \tilde \Sigma_{\sigma, \tau} 
        \Big)^{1/2} 
        \Sigma^{1/2} (\Sigma_{\rm new}^\dagger)^{1/2} \bg' 
    \Big\|^2
    \\
    \;=&\;
    \Big\| 
        \Big( \mfrac{\lambda}{2m} \Sigma^\dagger +  \tilde \Sigma_{\sigma, \tau} 
            \Big)^{1/2} 
            \Sigma^{1/2} P_{\Sigma_{\rm new}} (\Sigma^\dagger)^{1/2}
        \Big( \mfrac{\lambda}{2m} \Sigma^\dagger +  \tilde \Sigma_{\sigma, \tau} 
        \Big)^\dagger
        \big( \tilde \bg
            + 
            \mfrac{\theta}{\sqrt{p}}    v(\beta^*)
        \big)
    \Big\|^2 
    \\
    \;\overset{\P}{\rightarrow}&\;
    \mfrac{(r_1+r_2)^2\sigma_1^2\tau_1^2}{k} 
    \bar \chi^{\sigma,\tau}_{31}
    +
     r_2^2 \sigma_2^2\tau_2^2 \,
    \bar \chi^{\sigma,\tau}_{32} 
    +
     \theta^2  \bar \kappa_*^2 \sigma_2^2 \tau_2^2
    \,
    \bar \chi^{\sigma,\tau}_{33} 
    \;=\;
    \bar \chi^{r, \theta, \sigma, \tau}_3
    \;,
\end{align*}
where we used
\begin{align*}
    \bar \chi^{\sigma,\tau}_{31} 
    \;\coloneqq&\;
    \lim 
    \mfrac{\big\|
                \big( \frac{\lambda}{2m} \Sigma^\dagger +  \tilde \Sigma_{\sigma, \tau} 
                \big)^{1/2} 
            \Sigma^{1/2} P_{\Sigma_{\rm new}}   
            (\Sigma^\dagger)^{1/2}
            \big( \frac{2\sigma_1 \tau_1 \lambda}{2m} \Sigma^\dagger + I_p \big)^\dagger (P_\Sigma - \Sigma_*) 
        \big\|^2 
    }
    {m}
    \;,
    \\
    \bar \chi^{\sigma,\tau}_{32} 
    \;\coloneqq&\;
    \lim 
    \mfrac{\big\|
            \big( \frac{\lambda}{2m} \Sigma^\dagger +  \tilde \Sigma_{\sigma, \tau} 
                    \big)^{1/2} 
            \Sigma^{1/2} P_{\Sigma_{\rm new}}        
            (\Sigma^\dagger)^{1/2}
            \Big( \mfrac{2\sigma_2 \tau_2 \lambda}{2m} \Sigma^\dagger + I_p \Big)^\dagger \Sigma_* 
        \big\|^2 
    }
    {m}
    \;,
    \\
    \bar \chi^{\sigma,\tau}_{33} 
    \;\coloneqq&\;
    \lim \,
    \Big\|
        \Big( \mfrac{\lambda}{2m} \Sigma^\dagger +  \tilde \Sigma_{\sigma, \tau} 
        \Big)^{1/2} 
        \Sigma^{1/2} P_{\Sigma_{\rm new}}
        (\Sigma^\dagger)^{1/2}
        \Big( \mfrac{2\sigma_2 \tau_2 \lambda}{2m} \Sigma^\dagger + I_p \Big)^\dagger 
        P_*
    \Big\|^2 
    \;.
\end{align*}
Combining the calculations above, we obtain that 
\begin{align*}
    \tilde M_{\bg, \sigma,\tau,\theta}
    \;\overset{\P}{\rightarrow}&\;
    - \bar \chi^{r, \theta, \sigma, \tau}_1 + \mfrac{\bar \chi^{r, \theta,\sigma, \tau}_3}{\bar \chi^{r, \theta,\sigma, \tau}_2} 
    \, \times \, 
    \min_{\mu \in P_{\Sigma_{\rm new}}(S)} 
    \Big( 
        c(\mu) \sqrt{\bar \chi^{r,\theta,\sigma, \tau}_2 }
        - 
        \sqrt{\bar \chi^{r, \theta,\sigma, \tau}_2}
    \Big)^2
    \;.
\end{align*}
Notice that $\| \Sigma^{1/2}_{\rm new} \mu \| = c(\mu) \| \bg' \| \xrightarrow{\P} c(\mu)  \sqrt{\bar \chi^{r, \theta, \sigma, \tau}_2 } $, and recall that our two choices of $S$ only differs through $\| \Sigma^{1/2}_{\rm new} \mu \| - (\bar \chi^{\bar r, \bar \theta, \bar \sigma, \bar\tau}_2)^{1/2}$, where $\bar r = (\bar r_1, \bar r_2)$, $\bar \sigma = (\bar \sigma_1, \bar \sigma_2)$, $\bar \theta$ and $\bar \tau =(\bar \tau_1, \bar \tau_2)$ are the optimal solutions to \eqref{DO}. This implies that 
\begin{align*}
    \tilde M_{\bg, \sigma,\tau,\theta}
    \;\overset{\P}{\rightarrow}&\;
    - \bar \chi^{r,  \theta, \sigma, \tau}_1 
    + 
    \epsilon_S^2 \,
     \mfrac{\bar \chi^{r,\theta,  \sigma, \tau}_3}{\bar \chi^{r, \theta, \sigma, \tau}_2}
     \;,
\end{align*}
where $\epsilon_S = 0$ for $S=\cS_p$ and $\epsilon_S = \epsilon$ for $S=\cS_\epsilon^c$. Substituting this back into the overall optimization, we can approximate \eqref{SO} in distribution by
\begin{align*}
    \min_{\substack{\alpha \in S^\alpha \\ (\sigma_1, \sigma_2) \in S_{\sigma_1}  \times S_{\sigma_2} \\ \nu_1, \nu_2 \geq 0 }}
        \,
        \max_{\substack{ (r_1, r_2) \in S_{r_1 } \times S_{r_2} \\ \tau_1, \tau_2 \geq 0 \\ \theta \in \R }}
        \,
    & \,
    -
    \mfrac{\sigma_1}{2 \tau_1} 
    -
    \mfrac{\sigma_2}{2 \tau_2} 
    +
    \mfrac{r_1}{2 \nu_1} 
    +
    \mfrac{r_2}{2 \nu _2}
    +
    \alpha \theta  \bar \kappa_*^2
    -
    \mfrac{\alpha^2 \bar \kappa_*^2}{2 \sigma_2 \tau_2}
    - \bar \chi^{r, \theta, \sigma, \tau}_1 
    + 
    \epsilon_S^2 \,
     \mfrac{\bar \chi^{r, \theta, \sigma, \tau}_3}{\bar \chi^{r, \theta, \sigma, \tau}_2}
    \\
    &\,
    +
    \mfrac{1}{mk}  M_{\by, \tilde \bh_{\alpha, \sigma}, r, \nu}  
    -
    \mfrac{1}{2 r_2 \nu_2 mk }  
    \| \by \|^2
    -
    \mfrac{1}{mk}  \by^\intercal \tilde \bh_{\alpha, \sigma} 
    \;.
\end{align*}

\vspace{.5em}

\textbf{Compute terms involving $\by$ and $\tilde \bh_{\alpha, \sigma}$. } Recall that $\tilde \bh_{\alpha, \sigma} = \kappa_*\alpha \bq - \sigma_1 \bh_1 - \frac{\sigma_2}{\sqrt{k}} J_{mk} \bh_2$, where $\bq = \bq(\bG^\Phi) = \frac{1}{\kappa_* \sqrt{p}} \bG^\Phi v(\beta^*)$, and $\bh_1$ and $\bh_2$ are i.i.d.~standard $\R^{mk}$ Gaussians independent of $\bq(\bG^\Phi)$ and $\by = \by(\bG)$. Also recall that $\by = \by P_{mk} = \frac{1}{k} \by J_{mk}$. We can then express the last two terms of the loss above as 
\begin{align*}
    -
    \mfrac{1}{2 r_2 \nu_2 mk }  
    \| \by \|^2
    -
    \mfrac{1}{mk}  \by^\intercal \tilde \bh_{\alpha, \sigma} 
    \;=&\;
    -
    \mfrac{1}{2 r_2 \nu_2 mk }  
    \| \by \|^2
    -
    \mfrac{\kappa_* \alpha }{mk}  \by^\intercal \bq
    +
    \mfrac{\sigma_1}{mk}  \by^\intercal \bh_1 
    +
    \mfrac{\sigma_2}{m\sqrt{k}}  \by^\intercal P_{mk} \bh_2
    \;.
    \tagaligneq \label{eq:SO:lim:y:tilde:h}
\end{align*}
Since $\bh_1$ and $\bh_2$ are zero-mean and $\by$ is coordinate-wise bounded by one, by the weak law of large numbers, 
\begin{align*}
    \mfrac{1}{mk} \, \by^\intercal \bh_1
    \;\xrightarrow{\P}&\; 
    0
    &\text{ and }&&
    \mfrac{1}{m\sqrt{k}} \, \by^\intercal P_{mk} \bh_2
    \;\xrightarrow{\P}&\; 
    0
    \;.
    \tagaligneq \label{eq:SO:lim:y:h}
\end{align*}
To handle the first two terms, recall that 
\begin{align*}
    \by \;=\;  \big( \,  \underbrace{y_1( \Sigma_o^{1/2} G_1), \ldots, y_1( \Sigma_o^{1/2} G_1)}_{\textrm{repeated $k$ times}}, \; \ldots, \; \underbrace{y_n( \Sigma_o^{1/2} G_n), \ldots, y_n(\Sigma_o^{1/2} G_n)}_{\textrm{repeated $k$ times}} \, \big)^\intercal \;,
\end{align*}
where $y_i( \Sigma_o^{1/2} G_i)$'s are i.i.d.~by definition. Therefore by the weak law of large numbers, 
\begin{align*}
    \mfrac{1}{mk}
    \| \by \|^2
    \;=\;
    \mfrac{1}{m} \msum_{i=1}^m y_i( \Sigma_o^{1/2} G_i)
    \;\xrightarrow{\P}&\;
    \mean[y_1( \Sigma_o^{1/2} G_1) ]
    \;=\; 
    \mfrac{1}{2}
    \;.
    \tagaligneq \label{eq:SO:lim:y:square}
\end{align*}
In the last equality, we recall that 
\begin{align*}
    \P\big( y_1( \Sigma_o^{1/2} G_1) = 1\,\big|\, G_1 \big) 
    \;=\; 
    \sigma( G_1^\intercal \Sigma_o^{1/2} \beta^* )
    \;.
\end{align*}
$y_i \in \{0,1\}$ is a logistic variable evaluated at a random input $(\Sigma_o^{1/2} G_1)^\intercal \beta^*$ that is symmetric about zero. On the other hand, recalling that $v(\beta^*) = \sqrt{p} \Sigma_* \Sigma_o^{1/2} \beta^*$,
\begin{align*}
    \mfrac{1}{mk}  \by^\intercal \bq
    \;=&\;
    \mfrac{1}{mk}  \msum_{i \leq m} \msum_{j \leq k} 
    y_i(\Sigma_o^{1/2} G_i) \, 
    \mfrac{1}{\kappa_* \sqrt{p}} \big( G^\Phi_{ij} \big)^\intercal  v(\beta^*)
    \\
    \;=&\;
    \mfrac{1}{\kappa_*} 
    \, 
    \mfrac{1}{mk}
    \msum_{i \leq m} \msum_{j \leq k} 
    y_i(\Sigma_o^{1/2} G_i) \, 
    \big( G^\Phi_{ij} \big)^\intercal \Sigma_* \Sigma_o^{1/2} \beta^*\;.
    \tagaligneq \label{eq:limit:y:q:useful}
\end{align*}
Notice that each $ y_i(\Sigma_o^{1/2} G_i)$ depends on $G_i$ only through $G_i^\intercal \Sigma_o^{1/2} \beta^*$, so $G_i$ and $G^\Phi_{ij}$ appear in each summand only via the Gaussian vector 
\begin{align*}
    \begin{psmallmatrix}
        G_i^\intercal \Sigma_o^{1/2} \beta^* 
        \\
        \big( G^\Phi_{i1} \big)^\intercal \Sigma_* \Sigma_o^{1/2} \beta^*
        \\
        \vdots 
        \\
        \big( G^\Phi_{ik} \big)^\intercal \Sigma_* \Sigma_o^{1/2} \beta^*
    \end{psmallmatrix}
    \;\sim&\; 
    \cN\bigg( 0 \,,\, 
    \begin{psmallmatrix}
        (\beta^*)^\intercal \Sigma_o \beta^*  
        & 
        (\beta^*)^\intercal \Sigma_o^{1/2} \Sigma_* \Sigma_o^{1/2} \beta^*  
        & \ldots & 
        \\
        (\beta^*)^\intercal \Sigma_o^{1/2} \Sigma_*  \Sigma_o^{1/2} \beta^*  
        & (\beta^*)^\intercal  \Sigma_o^{1/2} \Sigma_* \Sigma_o^{1/2} \beta^*
        &  & \\ 
        \vdots & & \ddots & \\
    \end{psmallmatrix}
    \bigg)
    \overset{d}{\rightarrow}
    \begin{psmallmatrix}
        \bar \kappa_o \bar Z_0 + \bar \kappa_* \bar Z_1
        \\
        \bar \kappa_* \bar Z_1
        \\
        \vdots 
        \\
        \bar \kappa_* \bar Z_1
    \end{psmallmatrix}
    \;,
\end{align*}
where we recall that by \Cref{assumption:CGMT:var},
\begin{align*}
    \Sigma_o^{1/2} \, \Cov[ G_i, G^\Phi_{i1}]  \, \Sigma_* \Sigma_o^{1/2} 
    \;=&\; 
    \Sigma_o^{1/2} (\Sigma_*)^2 \Sigma_o^{1/2} 
    \;=\; 
    \Sigma_o^{1/2} \Sigma_* \Sigma_o^{1/2}
    \;,
    \\
    \Sigma_o^{1/2} \, \Sigma_* \Cov[ G^\Phi_{i1}, G^\Phi_{i2}]  \, \Sigma_* \Sigma_o^{1/2} 
    \;=&\; 
    \Sigma_o^{1/2} (\Sigma_*)^2 \Sigma_o^{1/2} 
    \;=\;
    \Sigma_o^{1/2} \Sigma_* \Sigma_o^{1/2}\;,
\end{align*}
 $\bar Z_0$ and $\bar Z_1$ are two i.i.d.~standard normals and 
\begin{align*}
    \bar \kappa_*
    \;\coloneqq\; 
    \lim_{p \rightarrow \infty} \kappa_* 
    \;=\;
    \lim_{p \rightarrow \infty} \| \Sigma_* \Sigma_o^{1/2} \beta^* \|
    \;,
    \qquad 
    \bar \kappa_o 
    \;\coloneqq\;  
    \lim_{p \rightarrow \infty} \| (I_p - \Sigma_*) \Sigma_o^{1/2} \beta^* \| 
    \;.
\end{align*}
Then by the law of large numbers, we have
\begin{align*}
    \mfrac{\kappa_*}{mk} \, \by^\intercal \bq
    \;\xrightarrow{\P}\;
    \, 
    \mean\big[ \sigma( \bar \kappa_o  \bar Z_0 + \bar \kappa_* \bar Z_1 ) \bar \kappa_* \bar Z_1 \big]
    \;.
    \tagaligneq \label{eq:SO:lim:y:q} 
\end{align*}
Combining \eqref{eq:SO:lim:y:tilde:h}, \eqref{eq:SO:lim:y:h}, \eqref{eq:SO:lim:y:square} and \eqref{eq:SO:lim:y:q} gives 
\begin{align*}
    -
    \mfrac{1}{2 r_2 \nu_2 mk }  
    \| \by \|^2
    -
    \mfrac{1}{mk}  \by^\intercal \tilde \bh_{\alpha, \sigma} 
    \;\xrightarrow{\P}\;
    - 
    \mfrac{1}{4 r_2 \nu_2}
    -
    \alpha \,  \mean\big[ \sigma( \bar \kappa_o \bar Z_0 + \bar \kappa_*\bar Z_1 )\bar \kappa_* \bar Z_1 \big]\;,
\end{align*}
so the optimization can be approximated by 
\begin{align*}
    \min_{\substack{\alpha \in S^\alpha \\ (\sigma_1, \sigma_2) \in S_{\sigma_1}  \times S_{\sigma_2} \\ \nu_1, \nu_2 \geq 0 }}
        \,
        \max_{\substack{ (r_1, r_2) \in S_{r_1 } \times S_{r_2} \\ \tau_1, \tau_2 \geq 0 \\ \theta \in \R }}
        \,
    & \,
    -
    \mfrac{\sigma_1}{2 \tau_1} 
    -
    \mfrac{\sigma_2}{2 \tau_2} 
    +
    \mfrac{r_1}{2 \nu_1} 
    +
    \mfrac{r_2}{2 \nu _2}
    +
    \alpha \theta \bar \kappa_*^2
    -
    \mfrac{\alpha^2 \bar \kappa_*^2}{2 \sigma_2 \tau_2}
    - \bar \chi^{r, \theta, \sigma, \tau}_1 
    + 
    \epsilon_S^2 \,
     \mfrac{\bar \chi^{r, \theta,  \sigma, \tau}_3}{\bar \chi^{r, \theta, \sigma, \tau}_2}
     - 
     \mfrac{1}{4 r_2 \nu_2} 
    \\
    &\,
    -
    \alpha \,  \mean\big[ \sigma(\bar \kappa_o \bar Z_0 + \bar \kappa_* \bar Z_1 ) \bar \kappa_* \bar Z_1 \big]
    +
    \mfrac{1}{mk}  M_{\by, \tilde \bh_{\alpha, \sigma}, r, \nu}  
    \;.
    \tagaligneq \label{eq:DO:before:nested:M}
\end{align*}

\vspace{.5em}

\textbf{Computing the nested Moreau envelope.} We are left with
\begin{align*}
    \mfrac{1}{mk}
    M_{\by, \tilde \bh_{\alpha, \sigma}, r, \nu} 
    \;=&\;
    \min_{u_2 \in  P_{mk}(S_u) }
    \,
    \mfrac{1}{mk}
    M^\perp_{\tilde \bh_{\alpha, \sigma},r,\nu}(u_2)
    +
    \mfrac{r_2 \nu_2}{2 mk}  
    \Big\| P_{mk}  \big( u_2  -  \mfrac{1}{r_2 \nu_2}  \by - \tilde \bh_{\alpha, \sigma}  \big)\Big\|^2 
    \\
    \;=&\;
    \min_{u_2 \in  P_{mk}(S_u) } 
    \,
    \min_{u_1 \in  P_{mk}^\perp(S_u) }
    \,
    \mfrac{1}{mk} \bone_{mk}^\intercal \, \rho( u_1 + u_2 )
    +
    \mfrac{r_2 \nu_2}{2 mk}  
    \Big\| P_{mk}  \big( u_2  -  \mfrac{1}{r_2 \nu_2}  \by - \tilde \bh_{\alpha, \sigma}  \big)\Big\|^2 
    \\
    &\;\hspace{9em}
    +
    \mfrac{r_1 \nu_1}{ 2 mk}  \big\| P^\perp_{mk} ( u_1 - \tilde \bh_{\alpha, \sigma} )\big\|^2
    \;.
\end{align*}
Write $u_{1ij}$ as the $(i,j)$-th coordinate of $u_1 \in \R^{mk}$ and similarly write $u_{2ij}$ for that of $u_2$, $q_{ij}$ for $\bq$, $h_{1ij}$ for $\bh_1$ and $h_{2ij}$ for $\bh_2$.  Recalling the definition of $\rho$, $P_{mk} = \frac{1}{k} J_{mk}$ and $P_{mk}^\perp = I_{mk} - P_{mk}$, we can re-express the loss above as 
\begin{align*}
    \mfrac{1}{mk} \msum_{i,j=1}^k L_{ij}(u_1, u_2)\;,
\end{align*}
where 
\begin{align*}
    L_{ij}(u_1, u_2)
    \,&\;\coloneqq\;
        \log( 1 + e^{u_{1ij} + u_{2ij}} )
    \\
    &\; 
        +
        \mfrac{r_2 \nu_2}{2} 
        \Big(
        \mfrac{1}{k}
        \msum_{j'=1}^k 
        \big(
            u_{2ij'} 
            - 
            \mfrac{1}{r_2 \nu_2}  y_i(\Sigma_o^{1/2} G_i) 
            -
            \kappa_* \alpha q_{ij'}
            +
            \sigma_1 h_{1ij'}  
            +
            \mfrac{\sigma_2}{\sqrt{k}} \big( \msum_{j'' \leq k}  h_{2ij''} \big)
        \big)
        \Big)^2
    \\
    &\;
        +
        \mfrac{r_1 \nu_1}{ 2}  
        ( u_{1ij} - \kappa_* \alpha q_{ij}
        +
        \sigma_1 h_{1ij}  
        )^2
        -
        \mfrac{r_1 \nu_1}{ 2}
        \Big( 
        \mfrac{1}{k}
        \msum_{j'=1}^k   
        ( u_{1ij'} - \kappa_* \alpha q_{ij'}
        +
        \sigma_1 h_{1ij'}
        )
        \Big)^2
    \;.
\end{align*}
Consider the $\R^k$-valued vectors $u_{1i} = (u_{1i1}, \ldots, u_{1ik})$ and $u_{2i} = (u_{2i1}, \ldots, u_{2ik})$ for $1 \leq i \leq m$. Notice that the loss $ L_{ij}(u_1, u_2) = \tilde L_{ij}(u_{1i}, u_{2i}) $ only depends on $u_1$ and $u_2$ through $u_{1i}, u_{2i}$. This allows us to rewrite
\begin{align*}
    \mfrac{1}{mk}
    M_{\by, \tilde \bh_{\alpha, \sigma}, r, \nu} 
    \;=\; 
    \mfrac{1}{m} \msum_{i=1}^m
    \min_{u_{2i} \in  (P_{mk}(S_u))_i } 
    \,
    \min_{u_{1i} \in  (P_{mk}^\perp(S_u))_i }
    \,
    \mfrac{1}{k} \msum_{j=1}^k
    \tilde L_{ij}(u_{1i}, u_{2i}) 
    \;,
\end{align*}
where $(P_{mk}(S_u))_i$ and $(P_{mk}^\perp(S_u))_i $ are the corresponding subspaces in which $u_{2i}$ and $u_{1i}$ take values. Since $S_u$ is closed under permutation of its $m$ blocks of $k$ coordinates, the $m$ summands above are i.i.d., which allows us to apply a weak law of large numbers to the above average. Also note that the minima are over $\R^k$-valued vectors, which allows us again to take a limit with $p \rightarrow \infty$ inside the loss function. Using the computation of $y_{ij}q_{ij}$ via $\bar Z_0$ and $\bar Z_1$ in \eqref{eq:limit:y:q:useful}, we obtain that $M_{\by, \tilde \bh_{\alpha, \sigma}, r, \nu} $ can be approximated by
\begin{align*}
    \mean\bigg[
    &
    \min_{\substack{u' \in (P_{mk}(S_u))_1 \\ u'' \in  (P_{mk}^\perp(S_u))_1 }}
     \;
    \mfrac{1}{k} \msum_{j=1}^k \log(1+e^{u'_j+u''_j})
    \\
    &\;\qquad
    +
        \mfrac{r_2 \nu_2}{2}
        \Big(
        \mfrac{1}{k}
        \msum_{j=1}^k 
        \big(
            u'_j
            - 
            \mfrac{1}{r_2 \nu_2}  \ind_{\geq 0}\{ \bar \kappa_o \bar Z_0 + \bar \kappa_* \bar Z_1 - \varepsilon_1 \}
            -
            \alpha \bar \kappa_* \bar Z_1
            +
            \sigma_1 \eta_j  
            +
             \sigma_2 \bar Z_2
        \big)
        \Big)^2
    \\
    &\;\qquad 
        + 
        \mfrac{r_1 \nu_1}{2} \, 
        \Big(
        \mfrac{1}{k} \msum_{j=1}^k
        (u''_j - \alpha \bar \kappa_* \bar Z_1 + \sigma_1 \eta_j )^2
        -
        \Big(
        \mfrac{1}{k}
        \msum_{j=1}^k 
        \big(
            u''_j
            -
            \alpha \bar \kappa_*\bar Z_1
            +
            \sigma_1 \eta_j 
        \big)
        \Big)^2
        \Big)
    \bigg]
    \;,
\end{align*}
where $\eta_1, \ldots, \eta_k$ and $\bar Z_2$ are i.i.d.~standard normals and $\varepsilon_1$ is an independent $\textrm{Logistic}(0,1)$ variable.. Notice that $u' \in (P_{mk}(S_u))_1$ has equal entries, say $u_0$, and $u'' \in (P^\perp_{mk}(S_u))_1$ satisfies $\sum_{j=1}^k u''_j = 0$. Also recall the assumption that $\sup_{u \in S_u} \frac{\| u \|^2_2}{mk} \rightarrow \infty$. Setting $\tilde u = (u''_1 + u_0, \ldots, u''_k + u_0)$, the above can be further approximated by 
\begin{align*}
    \mean\bigg[
    & \min_{\tilde u \in \R^{k}}
    \;
    \mfrac{1}{k} \msum_{j=1}^k \log(1+e^{\tilde u_j})
    \\
    &\;
    +
        \mfrac{r_2 \nu_2}{2}
        \Big(
            \,
            \mfrac{1}{k} \msum_{j \leq k} 
            \Big( 
            \tilde u_j
            - 
            \mfrac{1}{r_2 \nu_2}  \ind_{\geq 0}\{\bar \kappa_o \bar Z_0 +\bar \kappa_* \bar Z_1 - \varepsilon_1 \}
            -
            \alpha \bar \kappa_*\bar Z_1
            +
            \sigma_1 \eta_j 
            +
            \sigma_2  \bar Z_2
            \Big) 
            \,
        \Big)^2
    \\
    &\;
        + 
        \mfrac{r_1 \nu_1}{2} \, 
        \Big(
        \mfrac{1}{k} \msum_{j=1}^k
        ( \tilde u_j - \alpha \bar \kappa_*\bar Z_1 + \sigma_1 \eta_j )^2
        -
        \Big(
        \mfrac{1}{k}
        \msum_{j=1}^k 
        \big(
            \tilde u_j
            -
            \alpha \bar \kappa_*\bar Z_1
            +
            \sigma_1 \eta_j
        \big)
        \Big)^2
        \Big)
    \bigg]
    \\
    \;=\;
    \mean\bigg[
    & \min_{\tilde u \in \R^{k}}
    \;
    \mfrac{1}{k} \bone_k^\intercal \rho(\tilde u)
    + 
    \mfrac{r_1 \nu_1}{2 k} \, 
    \Big\| 
        \big(I_k - \mfrac{1}{k} \bone_{k \times k}\big)
        ( \tilde u - \alpha \bar \kappa_* \bar Z_1 \bone_k + \sigma_1 \eta )
    \Big\|^2
    \\
    &\;
    +
    \mfrac{r_2 \nu_2}{2 k}
    \Big\| 
        \mfrac{1}{k} \bone_{k \times k}
        \Big(
            \tilde u 
            -
            \mfrac{1}{r_2 \nu_2}  \ind_{\geq 0}\{ \bar \kappa_o \bar Z_0 + \bar \kappa_* \bar Z_1 - \varepsilon_1 \} 
            \bone_k 
            -
            \alpha \bar \kappa_* \bar Z_1 
            \bone_k
            +
            \sigma_1 \eta 
            +
            \sigma_2 \bar Z_2 \bone_k 
        \Big)
    \Big\|^2
    \bigg]
    \;,
\end{align*}
which equals $\bar M_\rho^{r, \nu, \alpha, \sigma}$. Substituting this into \eqref{eq:DO:before:nested:M} while also applying the assumption that $\sup_{v \in S_v} \frac{\| v \|^2_2}{mk} \rightarrow \infty$, we obtain
\begin{align*}
    \min_{\substack{\alpha \in S^\alpha \\ (\sigma_1, \sigma_2) \in S_{\sigma_1}  \times S_{\sigma_2} \\ \nu_1, \nu_2 \geq 0 }}
        \,
        \max_{\substack{ (r_1, r_2) \in S_{r_1 } \times S_{r_2} \\ \tau_1, \tau_2 \geq 0 \\ \theta \in \R }}
        \,
    & \,
    -
    \mfrac{\sigma_1}{2 \tau_1} 
    -
    \mfrac{\sigma_2}{2 \tau_2} 
    +
    \mfrac{r_1}{2 \nu_1} 
    +
    \mfrac{r_2}{2 \nu _2}
    +
    \alpha \theta \bar \kappa_*^2
    -
    \mfrac{\alpha^2 \bar \kappa_*^2}{2 \sigma_2 \tau_2}
    \\
    &\, 
    - \bar \chi^{r, \theta, \sigma, \tau}_1 
    + 
    \epsilon_S^2 \,
     \mfrac{\bar \chi^{r, \theta,  \sigma, \tau}_3}{\bar \chi^{r, \theta, \sigma, \tau}_2}
     - 
     \mfrac{1}{4 r_2 \nu_2}
     -
     \alpha \,  \mean\big[ \sigma( \bar \kappa_o \bar Z_0 + \bar \kappa_* \bar Z_1 ) \bar \kappa_* \bar Z_1 \big]
    +
    \bar M_\rho^{r, \nu, \alpha, \sigma}
    \;.
\end{align*}
\jmlrQED

\subsection{Proof of \lemmaref{lem:EQs}: \eqref{DO} to \eqref{EQs}}

Since $S = \cS_p$, we can ignore terms involving $\bar \chi^{r, \theta, \sigma, \tau}_2$ and $\bar \chi^{r, \theta, \sigma, \tau}_3$. Setting the first derivative of \eqref{DO} to zero with respect to each variable, we obtain 
\begin{align}
    \begin{cases}
        0
        \;=\;
        \theta \bar \kappa_*^2 
        - 
        \frac{ \alpha \bar \kappa_*^2}{\sigma_2 \tau_2}
        -
        \mean\big[ \sigma( \bar \kappa_o \bar Z_0 + \bar \kappa_* \bar Z_1 ) \bar \kappa_* \bar Z_1 \big]
        +
        \partial_\alpha  \bar M_\rho^{r, \nu, \alpha, \sigma}
        \;,
        \\
        0
        \;=\;
        -
        \frac{1}{2 \tau_1} 
        -
        \partial_{\sigma_1} \bar \chi_1^{r, \theta, \sigma, \tau} 
        +
        \partial_{\sigma_1} \bar M_\rho^{r, \nu, \alpha, \sigma}
        \;,
        \\
        0
        \;=\; 
        - 
        \frac{1}{2 \tau_2}
        +
        \frac{\alpha^2 \bar \kappa_*^2}{2 \sigma_2^2 \tau_2}
        -
        \partial_{\sigma_2} \bar \chi_1^{r, \theta, \sigma, \tau} 
        +
        \partial_{\sigma_2} \bar M_\rho^{r, \nu, \alpha, \sigma}
        \;,
        \\
        0 
        \;=\;
        \frac{\sigma_1}{2 \tau_1^2}
        -
        \partial_{\tau_1} \bar \chi_1^{r, \theta, \sigma, \tau} 
        \;,
        \\
        0 
        \;=\;
        \frac{\sigma_2}{2 \tau_2^2}
        +
        \frac{\alpha^2 \bar \kappa_*^2}{2 \sigma_2 \tau_2^2}
        -
        \partial_{\tau_2} \bar \chi_1^{r, \theta, \sigma, \tau} 
        \;,
        \\
        0
        \;=\;
        - 
        \frac{r_1}{2 \nu_1^2}
        +
        \partial_{\nu_1} \bar M_\rho^{r, \nu, \alpha, \sigma}
        \;,
        \\
        0
        \;=\;
        -
        \frac{r_2}{2 \nu_2^2}
        +
        \frac{1}{4 r_2 \nu_2^2}
        +
        \partial_{\nu_2} \bar M_\rho^{r, \nu, \alpha, \sigma}
        \;,
        \\
        0
        \;=\;
        \frac{1}{2\nu_1} 
        -
        \partial_{r_1} \bar \chi_1^{r, \theta, \sigma, \tau} 
        +
        \partial_{r_1} \bar M_\rho^{r, \nu, \alpha, \sigma}
        \;,
        \\
        0
        \;=\;
        \frac{1}{2\nu_2} 
        +
        \frac{1}{4 r_2^2 \nu_2}
        -
        \partial_{r_2} \bar \chi_1^{r, \theta, \sigma, \tau} 
        +
        \partial_{r_2} \bar M_\rho^{r, \nu, \alpha, \sigma}
        \;,
        \\
        0 
        \;=\;
        \alpha \bar \kappa_*^2 
        -
        \partial_\theta \bar \chi_1^{r, \theta, \sigma, \tau} 
        \;.
    \end{cases}
\end{align}
The next step is to compute the derivatives of 
\begin{align*}
    \bar M_\rho^{r, \nu, \alpha, \sigma}
    \;\coloneqq\;
    &\; \mean\bigg[
       \min_{\tilde u \in \R^{k}}
       \mfrac{1}{k} \bone_k^\intercal \rho(\tilde u)
    + 
    \mfrac{r_1 \nu_1}{2 k} \, 
    \big\| 
        \big(I_k - \mfrac{1}{k} \bone_{k \times k}\big)
        ( \tilde u  + \sigma_1 \eta )
    \big\|^2
    \\
    &\qquad
    +
    \mfrac{r_2 \nu_2}{2 k}
    \Big\| 
        \mfrac{1}{k} \bone_{k \times k}
        \Big(
            \tilde u 
            -
            \mfrac{1}{r_2 \nu_2}  \ind_{\geq 0}\{ \bar \kappa_o \bar Z_0 + \bar \kappa_* \bar Z_1 - \varepsilon_1 \} 
            \bone_k 
            -
            \alpha \bar \kappa_* \bar Z_1 
            \bone_k
            +
            \sigma_1 \eta 
            +
            \sigma_2 \bar Z_2 \bone_k 
        \Big)
    \Big\|^2
    \bigg]
    \;.
\end{align*}
Recall that we denote $u_{\bar Z, \varepsilon_1, \eta}$ as the minimizer of the minimization inside the expectation. By the envelope theorem and noting that $\mean[\ind_{\geq 0}\{ \bar \kappa_o \bar Z_0 + \bar \kappa_* \bar Z_1 - \varepsilon_1 \}  ] = \mean[ \sigma( \bar \kappa_o \bar Z_0 + \bar \kappa_* \bar Z_1  ) ]$, we have 
\begin{align*}
    \partial_\alpha \bar M_\rho^{r, \nu, \alpha, \sigma}
    \;=&\;
    -
    \mfrac{r_2 \nu_2 \bar \kappa_* }{k}
    \mean \,
    \Big[
        \bar Z_1  
        \Big(
            \bone_k^\intercal
            u_{\bar Z, \varepsilon_1, \eta}
            -
            \mfrac{k}{r_2 \nu_2} \sigma( \bar \kappa_o \bar Z_0 + \bar \kappa_* \bar Z_1 )
        \Big) 
    \Big] 
    +
    r_2 \nu_2 \alpha \bar \kappa_*^2 
    \;,
    \\
    \partial_{\sigma_1} \bar M_\rho^{r, \nu, \alpha, \sigma}
    \;=&\;
    \mfrac{r_1 \nu_1}{ k} \, 
    \mean\Big[
        \eta^\intercal \, \big(I_k - \mfrac{1}{k} \bone_{k \times k}\big)
         \, u_{\bar Z, \varepsilon_1, \eta} 
    \Big] 
    +
    \mfrac{r_1 \nu_1 \sigma_1 (k-1)}{ k} 
    +
    \mfrac{r_2 \nu_2 }{ k}
    \mean\Big[
        \eta^\intercal
        \mfrac{1}{k} \bone_{k \times k}
        u_{\bar Z, \varepsilon_1, \eta} 
    \bigg]
    +
    \mfrac{r_2 \nu_2\sigma_1 }{ k}
    \;,
    \\
    \partial_{\sigma_2} \bar M_\rho^{r, \nu, \alpha, \sigma}
    \;=&\;
    \mfrac{r_2 \nu_2 }{ k}
    \mean\Big[
        \bar Z_2
        \bone_k^\intercal
            u_{\bar Z, \varepsilon_1, \eta} 
    \Big]
    +
    r_2 \nu_2 \sigma_2 
    \;,
    \\
    \partial_{\nu_1} \bar M_\rho^{r, \nu, \alpha, \sigma}
    \;=&\;
    \mfrac{r_1}{2 k} \,  \mean\Big[ \big\| 
    \big(I_k - \mfrac{1}{k} \bone_{k \times k}\big)
    ( u_{\bar Z, \varepsilon_1, \eta} + \sigma_1 \eta )
    \big\|^2 \Big]
    \;,
    \\
    \partial_{\nu_2} \bar M_\rho^{r, \nu, \alpha, \sigma}
    \;=&\;
    \mfrac{r_2}{2 k}
    \mean\bigg[
    \Big\| 
        \mfrac{1}{k} \bone_{k \times k}
        \Big(
            u_{\bar Z, \varepsilon_1, \eta} 
            -
            \mfrac{1}{r_2 \nu_2}  \ind_{\geq 0}\{ \bar \kappa_o \bar Z_0 + \bar \kappa_* \bar Z_1 - \varepsilon_1 \} 
            \bone_k 
            -
            \alpha \bar \kappa_* \bar Z_1 
            \bone_k
            +
            \sigma_1 \eta 
            +
            \sigma_2 \bar Z_2 \bone_k 
        \Big)
    \Big\|^2
    \bigg]
    \\
    &\;
    +
    \mfrac{1}{\nu_2 k}
    \mean\bigg[
    \ind_{\geq 0}\{ \bar \kappa_o \bar Z_0 + \bar \kappa_* \bar Z_1 - \varepsilon_1 \} 
        \Big(
            \bone_k^\intercal
            u_{\bar Z, \varepsilon_1, \eta} 
            -
            \mfrac{k}{r_2 \nu_2}  \sigma( \bar \kappa_o \bar Z_0 + \bar \kappa_* \bar Z_1 )
            -
            k \alpha \bar \kappa_* \bar Z_1 
        \Big)
    \bigg]
    \;,
    \\
    \partial_{r_1} \bar M_\rho^{r, \nu, \alpha, \sigma}
    \;=&\;
    \mfrac{\nu_1}{2 k} \,  \mean\Big[ \big\| 
    \big(I_k - \mfrac{1}{k} \bone_{k \times k}\big)
    ( u_{\bar Z, \varepsilon_1, \eta} + \sigma_1 \eta )
    \big\|^2 \Big]
    \;,
    \\
    \partial_{r_2} \bar M_\rho^{r, \nu, \alpha, \sigma}
    \;=&\;
    \mfrac{\nu_2}{2 k}
    \mean\bigg[
    \Big\| 
        \mfrac{1}{k} \bone_{k \times k}
        \Big(
            u_{\bar Z, \varepsilon_1, \eta}
            -
            \mfrac{1}{r_2 \nu_2}  \ind_{\geq 0}\{ \bar \kappa_o \bar Z_0 + \bar \kappa_* \bar Z_1 - \varepsilon_1 \} 
            \bone_k 
            -
            \alpha \bar \kappa_* \bar Z_1 
            \bone_k
            +
            \sigma_1 \eta 
            +
            \sigma_2 \bar Z_2 \bone_k 
        \Big)
    \Big\|^2
    \bigg]
    \\
    &\;
    +
    \mfrac{1}{r_2 k}
    \mean\bigg[
    \ind_{\geq 0}\{ \bar \kappa_o \bar Z_0 + \bar \kappa_* \bar Z_1 - \varepsilon_1 \} 
        \Big(
            \bone_k^\intercal
            u_{\bar Z, \varepsilon_1, \eta} 
            -
            \mfrac{k}{r_2 \nu_2}  \sigma( \bar \kappa_o \bar Z_0 + \bar \kappa_* \bar Z_1 )
            -
            k \alpha \bar \kappa_* \bar Z_1 
        \Big)
    \bigg]
    \;.
\end{align*}
Writing $\bar Y =  \ind_{\geq 0}\{ \bar \kappa_o \bar Z_0 + \bar \kappa_* \bar Z_1 - \varepsilon_1 \} $ and substituting the bounds above into the system of equations recovers \eqref{EQs}.

\subsection{Proofs for \Cref{appendix:DA:special:cases}}

\subsubsection{Proof of \lemmaref{lem:EQs:iso}: isotropic, no augmentation}

Under the stated setup, the covariance matrices in the formula evaluate to $\Sigma_o = \Sigma = \frac{1}{p} I_p$ and $\Sigma_* = I_p$. In this case, as $m=n, p \rightarrow \infty$ and $p/n \rightarrow \kappa$, we can compute the limit terms defined in \eqref{DO}:
\begin{align*}
    &\;
    \bar \kappa_* \;=\; \lim_{p \rightarrow \infty} \mfrac{\| \beta^* \|}{\sqrt{p}}
    \;,
    \qquad
    \bar \kappa_o 
    \;=\;
    \bar \chi^{\sigma, \tau}_{11}
    \;=\;  
    0\;,
    \qquad 
    \bar \chi^{\sigma, \tau}_{12}
    \;=\; 
    \mfrac{\kappa}{\sigma_2 \tau_2 \lambda \kappa + 1}
    \;,
    \qquad 
    \bar \chi^{\sigma, \tau}_{13}
    \;=\;
    \mfrac{1}{\sigma_2 \tau_2 \lambda \kappa + 1}
    \;,
    \\
    &\;
    \bar \chi^{\sigma, \tau}_{21}
    \;=\;
    0 \;,
    \qquad 
    \bar \chi^{\sigma, \tau}_{22}
    \;=\;
    \mfrac{\kappa}{(\sigma_2 \tau_2 \lambda \kappa + 1)^2}
    \;,
    \qquad 
    \bar \chi^{\sigma, \tau}_{23}
    \;=\;
    \mfrac{1}{(\sigma_2 \tau_2 \lambda \kappa+ 1)^2}
    \;.
\end{align*}
This implies 
\begin{align*}
    \bar \chi_1^{r, \theta, \sigma, \tau}
    \;=\;
    \mfrac{r_2^2  \kappa + \theta^2 \bar \kappa_*^2 }{2 ( \lambda \kappa + \sigma_2^{-1} \tau_2^{-1} )}
    \;,
    \qquad
    \bar \chi_2^{r, \theta, \sigma, \tau}
    \;=\;
    \mfrac{r_2^2 \kappa
    +
    \theta^2 \bar \kappa_*^2 
    }
    {(\lambda \kappa + \sigma_2^{-1} \tau_2^{-1}  )^2}
    \;,
\end{align*}
which are in particular independent of $\sigma_1$, $\tau_1$ and $r_1$. Now recall that \eqref{EQs} read 
\begin{align*}
    \begin{cases}
        0
        \;=\;
        \theta \bar \kappa_*^2 
        - 
        \frac{ \alpha \bar \kappa_*^2}{\sigma_2 \tau_2}
        -
        \frac{r_2 \nu_2 \bar \kappa_* }{k}
        \mean \,
        \Big[
            \bar Z_1  
            \bone_k^\intercal
            u_{\bar Z, \varepsilon_1, \eta}
        \Big] 
        +
        r_2 \nu_2 \alpha \bar \kappa_*^2 
        \;,
        \\
        0
        \;=\;
        -
        \frac{1}{2 \tau_1} 
        -
        \partial_{\sigma_1} \bar \chi_1^{r, \theta, \sigma, \tau} 
        +
        \frac{r_1 \nu_1}{ k} \, 
        \mean\Big[
            \eta^\intercal \, \big(I_k - \frac{1}{k} \bone_{k \times k}\big)
             \, u_{\bar Z, \varepsilon_1, \eta} 
        \Big] 
        +
        \frac{r_1 \nu_1 \sigma_1 (k-1)}{ k} 
        +
        \frac{r_2 \nu_2 }{ k}
        \mean\Big[
            \eta^\intercal
            \frac{1}{k} \bone_{k \times k}
            u_{\bar Z, \varepsilon_1, \eta} 
        \bigg]
        \\
        \;\qquad
        +
        \frac{r_2 \nu_2\sigma_1 }{ k}
        \;,
        \\
        0
        \;=\; 
        - 
        \frac{1}{2 \tau_2}
        +
        \frac{\alpha^2 \bar \kappa_*^2}{2 \sigma_2^2 \tau_2}
        -
        \partial_{\sigma_2} \bar \chi_1^{r, \theta, \sigma, \tau} 
        +
        \frac{r_2 \nu_2 }{ k}
        \mean\Big[
            \bar Z_2
            \bone_k^\intercal
                u_{\bar Z, \varepsilon_1, \eta} 
        \Big]
        +
        r_2 \nu_2 \sigma_2 
        \;,
        \\
        0 
        \;=\;
        \frac{\sigma_1}{2 \tau_1^2}
        -
        \partial_{\tau_1} \bar \chi_1^{r, \theta, \sigma, \tau} 
        \;,
        \\
        0 
        \;=\;
        \frac{\sigma_2}{2 \tau_2^2}
        +
        \frac{\alpha^2 \bar \kappa_*^2}{2 \sigma_2 \tau_2^2}
        -
        \partial_{\tau_2} \bar \chi_1^{r, \theta, \sigma, \tau} 
        \;,
        \\
        0
        \;=\;
        - 
        \frac{r_1}{2 \nu_1^2}
        +
        \frac{r_1}{2 k} \,  \mean\Big[ \big\| 
            \big(I_k - \frac{1}{k} \bone_{k \times k}\big)
            ( u_{\bar Z, \varepsilon_1, \eta} + \sigma_1 \eta )
            \big\|^2 \Big]
        \;,
        \\
        0
        \;=\;
        -
        \frac{r_2}{2 \nu_2^2}
        +
        \frac{1}{4 r_2 \nu_2^2}
        +
        \mfrac{r_2}{2 k}
        \mean\bigg[
        \Big\| 
            \mfrac{1}{k} \bone_{k \times k}
            \Big(
                u_{\bar Z, \varepsilon_1, \eta} 
                -
                \mfrac{1}{r_2 \nu_2} \bar Y
                \bone_k 
                -
                \alpha \bar \kappa_* \bar Z_1 
                \bone_k
                +
                \sigma_1 \eta 
                +
                \sigma_2 \bar Z_2 \bone_k 
            \Big)
        \Big\|^2
        \bigg]
        \\
        \;\qquad
        +
        \mfrac{1}{\nu_2 k}
        \mean\bigg[
            \bar Y
            \Big(
                \bone_k^\intercal
                u_{\bar Z, \varepsilon_1, \eta} 
                -
                \mfrac{k}{r_2 \nu_2} \bar Y
                -
                k \alpha \bar \kappa_* \bar Z_1 
            \Big)
        \bigg]
        \;,
        \\
        0
        \;=\;
        \frac{1}{2\nu_1} 
        -
        \partial_{r_1} \bar \chi_1^{r, \theta, \sigma, \tau} 
        +
        \frac{\nu_1}{2 k} \,  \mean\Big[ \big\| 
        \big(I_k - \frac{1}{k} \bone_{k \times k}\big)
        ( u_{\bar Z, \varepsilon_1, \eta} + \sigma_1 \eta )
        \big\|^2 \Big]
        \;,
        \\
        0
        \;=\;
        \frac{1}{2\nu_2} 
        +
        \frac{1}{4 r_2^2 \nu_2}
        -
        \partial_{r_2} \bar \chi_1^{r, \theta, \sigma, \tau} 
        +
        \mfrac{\nu_2}{2 k}
        \mean\bigg[
        \Big\| 
            \mfrac{1}{k} \bone_{k \times k}
            \Big(
                u_{\bar Z, \varepsilon_1, \eta} 
                -
                \mfrac{1}{r_2 \nu_2} \bar Y
                \bone_k 
                -
                \alpha \bar \kappa_* \bar Z_1 
                \bone_k
                +
                \sigma_1 \eta 
                +
                \sigma_2 \bar Z_2 \bone_k 
            \Big)
        \Big\|^2
        \bigg]
        \\
        \;\qquad
        +
        \mfrac{1}{r_2 k}
        \mean\bigg[
            \bar Y
            \Big(
                \bone_k^\intercal
                u_{\bar Z, \varepsilon_1, \eta} 
                -
                \mfrac{k}{r_2 \nu_2} \bar Y
                -
                k \alpha \bar \kappa_* \bar Z_1 
            \Big)
        \bigg]
        \;,
        \\
        0 
        \;=\;
        \alpha \bar \kappa_*^2 
        -
        \partial_\theta \bar \chi_1^{r, \theta, \sigma, \tau} 
        \;.
    \end{cases}
\end{align*}
By the 4th equation, $\sigma_1=0$.  In this case, the defining optimization of $u_{\bar Z, \varepsilon_1, \eta}$ is symmetric under permutation of $\tilde u \in \R^k$ and in particular $\frac{1}{k} \bone_{k \times k} u_{\bar Z, \varepsilon_1, \eta} = u_{\bar Z, \varepsilon_1, \eta}$. This implies that $u_{\bar Z, \varepsilon_1, \eta} = u_{\bar Z, \varepsilon_1} \bone_k$ where $ u_{\bar Z, \varepsilon_1} $ is the minimizer of the 1-d random optimization problem 
\begin{align*}
    \min_{\tilde u \in \R}
    \,
    \rho(\tilde u)
    +
    \mfrac{r_2 \nu_2}{2}
    \big( 
            \tilde u
            -
            \mfrac{1}{r_2 \nu_2} \bar Y
            -
            \alpha \bar \kappa_* \bar Z_1 
            +
            \sigma_2 \bar Z_2 
    \big)^2
    \;.
    \tagaligneq \label{eq:iso:unaug:1d:Moreau}
\end{align*}
Recall that  ${\rm Prox}_{t \rho(\argdot)}(v) \coloneqq \argmin_{x \in \R} \frac{1}{2t} (v-x)^2 + \rho(x)$. This allows us to express
\begin{align*}
    u_{\bar Z, \varepsilon_1} 
    \;=&\; 
    \textrm{Prox}_{(r_2 \nu_2)^{-1} \rho(\argdot)} 
    \Big(  \mfrac{1}{r_2 \nu_2} \bar Y
    +
    \alpha \bar \kappa_* \bar Z_1 
    -
    \sigma_2 \bar Z_2  \Big)
    \;.
\end{align*}
Meanwhile, substituting $(I_k - \frac{1}{k} \bone_{k \times k}) u_{\bar Z, \varepsilon_1, \eta} = 0$  into the 6th and 8th equations above yields $r_1 = 0$ and $\nu_1 \rightarrow \infty$. We can WLOG take $\nu_1 \rightarrow \infty$ such that $r_1 \nu_1 \rightarrow 0$.  By the 2nd equation we then obtain 
\begin{align*}
    \tau_1
    \;=\;
    \mfrac{1}{2} \Big(
    \frac{r_2 \nu_2 }{ k}
    \mean\big[
        \eta^\intercal
        \bone_k
        u_{\bar Z, \varepsilon} 
    \big]
    \Big)^{-1}
    \;\rightarrow\; \infty
    \;,
\end{align*}
by noting that $\eta$ is zero-mean and independent of $u_{\bar Z, \varepsilon}$. This removes $(\sigma_1, r_1, \nu_1, \tau_1)$ from the equations. Substituting $u_{\bar Z, \varepsilon_1, \eta} = \bar u_{\bar Z, \varepsilon_1} \bone_k$ and the derivatives of $ \bar \chi_1^{r, \theta, \sigma, \tau} $, we obtain
\begin{align}
    \begin{cases}
        0
        \;=\;
        \theta \bar \kappa_*^2 
        - 
        \frac{ \alpha \bar \kappa_*^2}{\sigma_2 \tau_2}
        -
        r_2 \nu_2 \bar \kappa_* 
        \mean \,
        \big[
            \bar Z_1  
            u_{\bar Z, \varepsilon_1}
        \big] 
        +
        r_2 \nu_2 \alpha \bar \kappa_*^2 
        \;,
        \\
        0
        \;=\; 
        - 
        \frac{1}{2 \tau_2}
        +
        \frac{\alpha^2 \bar \kappa_*^2}{2 \sigma_2^2 \tau_2}
        -
        \frac{1}{\sigma_2^2 \tau_2} \,
        \frac{r_2^2  \kappa + \theta^2 \bar \kappa_*^2 }{2 ( \lambda \kappa + \sigma_2^{-1} \tau_2^{-1} )^2}
        +
        r_2 \nu_2 
        \mean\big[
            \bar Z_2
            u_{\bar Z, \varepsilon_1}
        \big]
        +
        r_2 \nu_2 \sigma_2 
        \;,
        \\
        0 
        \;=\;
        \frac{\sigma_2}{2 \tau_2^2}
        +
        \frac{\alpha^2 \bar \kappa_*^2}{2 \sigma_2 \tau_2^2}
        -
        \frac{1}{\sigma_2 \tau_2^2} \,
        \frac{r_2^2  \kappa + \theta^2 \bar \kappa_*^2 }{2 ( \lambda \kappa + \sigma_2^{-1} \tau_2^{-1} )^2}
        \;,
        \\
        0
        \;=\;
        -
        \frac{r_2}{2 \nu_2^2}
        +
        \frac{1}{4 r_2 \nu_2^2}
        +
        \frac{r_2}{2}
        \mean\big[
        \big( 
                u_{\bar Z, \varepsilon_1} 
                -
                \frac{1}{r_2 \nu_2} \bar Y
                -
                \alpha \bar \kappa_* \bar Z_1 
                +
                \sigma_2 \bar Z_2 
        \big)^2
        \big]
        +
        \mfrac{1}{\nu_2}
        \mean\Big[
            \bar Y
            \Big(
                u_{\bar Z, \varepsilon_1} 
                -
                \frac{1}{r_2 \nu_2} \bar Y
                -
                \alpha \bar \kappa_* \bar Z_1 
            \Big)
        \Big]
        \;,
        \\
        0
        \;=\;
        \frac{1}{2\nu_2} 
        +
        \frac{1}{4 r_2^2 \nu_2}
        -
        \frac{r_2 \kappa }{\lambda \kappa + \sigma_2^{-1} \tau_2^{-1}}
        +
        \mfrac{\nu_2}{2}
        \mean\Big[
        \big( 
                u_{\bar Z, \varepsilon_1, \eta} 
                -
                \frac{1}{r_2 \nu_2} \bar Y
                -
                \alpha \bar \kappa_* \bar Z_1 
                +
                \sigma_2 \bar Z_2 
        \big)^2
        \Big]
        \\
        \;\qquad
        +
        \mfrac{1}{r_2}
        \mean\Big[
            \bar Y
            \Big(
                u_{\bar Z, \varepsilon_1, \eta} 
                -
                \frac{1}{r_2 \nu_2} \bar Y
                -
                \alpha \bar \kappa_* \bar Z_1 
            \Big)
        \Big]
        \;,
        \\
        0 
        \;=\;
        \alpha \bar \kappa_*^2 
        -
        \frac{\theta \bar \kappa_*^2 }{ \lambda \kappa + \sigma_2^{-1} \tau_2^{-1} }
        \;.
    \end{cases}
    \label{eq:iso:tmp}
\end{align}
Now let $\gamma = \frac{1}{r_2 \nu_2}$. Notice that the 4th and 5th equations above both involve
\begin{align*}
    (\star)
    \;\coloneqq&\;
    \mfrac{1}{2}
    \mean\big[
    \big( 
            u_{\bar Z, \varepsilon_1} 
            -
            \gamma \bar Y
            -
            \alpha \bar \kappa_* \bar Z_1 
            +
            \sigma_2 \bar Z_2 
    \big)^2
    \big]
    +
    \gamma
    \mean\Big[
        \bar Y
        \Big(
            u_{\bar Z, \varepsilon_1} 
            -
            \gamma \bar Y
            -
            \alpha \bar \kappa_* \bar Z_1 
        \Big)
    \Big]
    \\
    \;\overset{(a)}{=}&\;
    \mean\bigg[
    \mfrac{1 - \bar Y}{2}
    \big( 
            \textrm{Prox}_{\gamma \rho(\argdot)} 
            \big(
            \alpha \bar \kappa_* \bar Z_1 
            -
            \sigma_2 \bar Z_2  
            \big)
            -
            \alpha \bar \kappa_* \bar Z_1 
            +
            \sigma_2 \bar Z_2 
    \big)^2
    \bigg]
    \\
    &\;
    +
    \mean\bigg[
    \mfrac{\bar Y}{2}
    \big( 
            \textrm{Prox}_{\gamma \rho(\argdot)} 
            \big(  \gamma
            +
            \alpha \bar \kappa_* \bar Z_1 
            -
            \sigma_2 \bar Z_2  
            \big)
            -
            \gamma 
            -
            \alpha \bar \kappa_* \bar Z_1 
            +
            \sigma_2 \bar Z_2 
    \big)^2
    \bigg]
    \\
    &\;
    +
    \gamma
    \mean\Big[
        \bar Y
        \Big(
            \textrm{Prox}_{\gamma \rho(\argdot)} 
            \big(  \gamma \bar Y
            +
            \alpha \bar \kappa_* \bar Z_1 
            -
            \sigma_2 \bar Z_2  
            \big)
            -
            \gamma 
            -
            \alpha \bar \kappa_* \bar Z_1 
        \Big)
    \Big]
    \\
    \;\overset{(b)}{=}&\;
    \mean\bigg[
    \mfrac{\partial \rho( - \bar \kappa_* \bar Z_1 )}{2}
    \big( 
            \textrm{Prox}_{\gamma \rho(\argdot)} 
            \big( 
            \alpha \bar \kappa_* \bar Z_1 
            -
            \sigma_2 \bar Z_2  
            \big)
            -
            \alpha \bar \kappa_* \bar Z_1 
            +
            \sigma_2 \bar Z_2 
    \big)^2
    \bigg]
    \\
    &\;
    +
    \mean\bigg[
    \mfrac{\partial \rho(\bar \kappa_* \bar Z_1)}{2}
    \big( 
            \textrm{Prox}_{\gamma \rho(\argdot)} 
            \big(  \gamma
            +
            \alpha \bar \kappa_* \bar Z_1 
            -
            \sigma_2 \bar Z_2  
            \big)
            -
            \gamma 
            -
            \alpha \bar \kappa_* \bar Z_1 
            +
            \sigma_2 \bar Z_2 
    \big)^2
    \bigg]
    \\
    &\;
    -
    \gamma
    \mean\Big[
        \partial \rho(\bar \kappa_* \bar Z_1)
        \Big(
            \alpha \bar \kappa_* \bar Z_1 
            +
            \gamma 
            -
            \textrm{Prox}_{\gamma \rho(\argdot)} 
            \big(  \gamma 
            +
            \alpha \bar \kappa_* \bar Z_1 
            -
            \sigma_2 \bar Z_2  
            \big)
        \Big)
    \Big]
    \\
    \;\overset{(c)}{=}&\;
    \mean\bigg[
    \mfrac{\partial \rho( - \bar \kappa_* \bar Z_1 )}{2}
    \big( 
            \textrm{Prox}_{\gamma \rho(\argdot)} 
            \big(  
            \alpha \bar \kappa_* \bar Z_1 
            -
            \sigma_2 \bar Z_2  
            \big)
            -
            \alpha \bar \kappa_* \bar Z_1 
            +
            \sigma_2 \bar Z_2 
    \big)^2
    \bigg]
    \\
    &\;
    +
    \mean\Big[ \mfrac{\partial \rho(\bar \kappa_* \bar Z_1)}{2} \big(
        \textrm{Prox}_{\gamma \rho(\argdot)} 
            \big(  \gamma 
            +
            \alpha \bar \kappa_* \bar Z_1 
            -
            \sigma_2 \bar Z_2  
            \big)
        -
        \alpha \bar \kappa_* \bar Z_1 
        +
        \sigma_2 \bar Z_2 
    \big)^2 \Big]
    -
    \mfrac{\gamma^2}{2}
    +
    \mfrac{\gamma^2}{4} 
    \\
    \;\overset{(d)}{=}&\;
    \mean\Big[ 
        \, \partial \rho( - \bar \kappa_* \bar Z_1) \, 
        \big(
        \alpha \bar \kappa_* \bar Z_1 
        +
        \sigma_2 \bar Z_2 
        -
        \textrm{Prox}_{\gamma \rho(\argdot)} 
            \big(  
            \alpha \bar \kappa_* \bar Z_1 
            +
            \sigma_2 \bar Z_2  
            \big)
    \big)^2 \Big]
    -
    \mfrac{\gamma^2}{4}
    \;.
\end{align*}
In $(a)$ above, we have recalled that $\bar Y =  \ind_{\geq 0}\{ \bar \kappa_o \bar Z_0 + \bar \kappa_* \bar Z_1 - \varepsilon_1 \} = \ind_{\geq 0}\{\bar \kappa_* \bar Z_1 - \varepsilon_1 \} $ is an indicator function; in $(b)$ we have noted the equality of the conditional distributions $1 - \ind_{\geq 0}\{ \bar \kappa_* \bar Z_1 - \varepsilon_1\} \,|\, \bar Z_1 \overset{d}{=} \ind_{\geq 0}\{ - \bar \kappa_* \bar Z_1 - \varepsilon_1\} \,|\, \bar Z_1 $ by the symmetry of $\varepsilon_1$ followed by $\sigma(\argdot) = \partial \rho(\argdot)$; in $(c)$ we have expanded the square in the second term and noted that $\mean[\partial \rho(\bar \kappa_* \bar Z_1) ] = \mean[(1+e^{-\bar \kappa_* \bar Z_1})^{-1} ] = \frac{1}{2}$ since $\bar Z_1$ is symmetric about zero; in $(d)$, we have used in the second expectation that $\bar Z_1 \overset{d}{=} - \bar Z_1$, $\bar Z_2 \overset{d}{=} - \bar Z_2$ and that 
\begin{align*}
    \textrm{Prox}_{\gamma \rho(\argdot)} 
    \Big( \gamma 
    +
    \alpha \bar \kappa_* \bar Z_1 
    -
    \sigma_2 \bar Z_2  \Big)
    \;=\; 
    -
    \textrm{Prox}_{\gamma \rho(\argdot)} 
    \Big( 
    -
    \alpha \bar \kappa_* \bar Z_1 
    +
    \sigma_2 \bar Z_2  \Big)
    \;,
\end{align*}
where we have used $ \textrm{Prox}_{\gamma \rho(\argdot)}( x + \gamma) = -  \textrm{Prox}_{\gamma \rho(\argdot)}(-x)$ (see e.g.~Lemma 3 of \cite{salehi2019impact}). Substituting this into the last three lines of \eqref{eq:iso:tmp} gives 
\begin{align}
    \begin{cases}
    \gamma^2
    \;=\;
    \mfrac{2}{r_2^2}
    \mean\Big[ 
        \, \partial \rho( - \bar \kappa_* \bar Z_1) \, 
        \big(
        \bar \kappa_* \alpha  \bar Z_1 
        +
        \sigma_2 \bar Z_2 
        -
        \textrm{Prox}_{\gamma \rho(\argdot)} 
            \big(  
            \alpha \bar \kappa_* \bar Z_1 
            +
            \sigma_2 \bar Z_2  
            \big)
    \big)^2 \Big]
    \;,
    \\
    \gamma
    \;=\;
    \frac{\kappa }{\lambda \kappa + \sigma_2^{-1} \tau_2^{-1}}
    \;,
    \\
    \alpha
    \;=\;
    \frac{\theta }{ \lambda \kappa + \sigma_2^{-1} \tau_2^{-1} }
    \;.
    \end{cases}
    \label{eq:iso:EQs:tmp:tmp}
\end{align}
Meanwhile, the third line of \eqref{eq:iso:tmp} implies
\begin{align*}
    \sigma_2^2
    + \alpha^2 \bar \kappa_*^2
    \;=&\;
    \mfrac{ r_2^2  \kappa + \theta^2 \bar \kappa_*^2 }{(\lambda \kappa + \sigma_2^{-1} \tau_2^{-1}   )^2} 
    \;=\; 
    \mfrac{ r_2^2  \kappa }{(\lambda \kappa + \sigma_2^{-1} \tau_2^{-1}   )^2} 
    +
    \alpha^2 \bar \kappa_*^2
    \;.
    \tagaligneq \label{eq:iso:EQs:tmp:sigma:alpha}
\end{align*}
Combining the two calculations, we obtain
\begin{align*}
    \theta \;=\; \mfrac{\alpha \kappa}{\gamma}\;,
    \qquad 
    \tau_2 \;=\; \mfrac{\kappa^{-1} \gamma}{\sigma_2(1 -  \gamma \lambda )}\;,
    \qquad 
    r_2 \;=\; \mfrac{\sigma_2 \sqrt{\kappa}}{\gamma}\;,
    \tagaligneq \label{eq:iso:EQs:partial}
\end{align*}
which gives the first three desired equations. Substituting these back into the first line of \eqref{eq:iso:EQs:tmp:tmp} gives
\begin{align*}
        \mfrac{\sigma^2 \kappa}{2}
        \;=\; 
        \mean\big[ \partial \rho( - \bar \kappa_* \bar Z_1 ) 
        \big( 
            \alpha \bar \kappa_* \bar Z_1 + \sigma_2 \bar Z_2 
            - 
            {\rm Prox}_{\gamma \rho(\argdot )}(  \alpha \bar \kappa_* \bar Z_1 + \sigma_2 \bar Z_2  ) \big)^2  
        \big] 
        \;,
        \tagaligneq \label{eq:iso:EQs:partial:two}
\end{align*}
which is the fourth desired equation. The first and second equations of \eqref{eq:iso:tmp} are handled similarly as appendix C.3 of \cite{salehi2019impact}. We recall that $1 - \ind_{\geq 0}\{ \bar \kappa_* \bar Z_1 - \varepsilon_1\} \,|\, \bar Z_1 \overset{d}{=} \ind_{\geq 0}\{ - \bar \kappa_* \bar Z_1 - \varepsilon_1\} \,|\, \bar Z_1 $ and $ \textrm{Prox}_{\gamma \rho(\argdot)}( x + \gamma) = -  \textrm{Prox}_{\gamma \rho(\argdot)}(-x)$ again to compute
\begin{align*}
    \mean\big[
            \bar Z_1 &\,
            u_{\bar Z, \varepsilon_1}
        \big]
    \;=\;
    \mean\big[ \bar Z_1 \textrm{Prox}_{\gamma \rho(\argdot)} 
    \big( \gamma \bar Y
    +
    \alpha \bar \kappa_* \bar Z_1 
    -
    \sigma_2 \bar Z_2  \big) \big]
    \\
    \;=&\;
    \mean\big[ \bar Z_1 \bar Y \textrm{Prox}_{\gamma \rho(\argdot)} 
    \big( \gamma
    +
    \alpha \bar \kappa_* \bar Z_1 
    -
    \sigma_2 \bar Z_2  \big) \big]
    +
    \mean\big[ \bar Z_1 (1 - \bar Y) \textrm{Prox}_{\gamma \rho(\argdot)} 
    \big( 
    \alpha \bar \kappa_* \bar Z_1 
    -
    \sigma_2 \bar Z_2  \big) \big]
    \\
    \;=&\;
    \mean\big[ \bar Z_1  \, \partial \rho( \bar \kappa_* \bar Z_1 ) \, \textrm{Prox}_{\gamma \rho(\argdot)} 
    \big( \gamma
    +
    \alpha \bar \kappa_* \bar Z_1 
    -
    \sigma_2 \bar Z_2  \big) \big]
    +
    \mean\big[ \bar Z_1 \, \partial \rho( - \bar \kappa_* \bar Z_1 ) \,   \textrm{Prox}_{\gamma \rho(\argdot)} 
    \big( 
    \alpha \bar \kappa_* \bar Z_1 
    -
    \sigma_2 \bar Z_2  \big) \big]
    \\
    \;=&\;
    - \mean\big[ \bar Z_1  \, \partial \rho( \bar \kappa_* \bar Z_1 ) \, \textrm{Prox}_{\gamma \rho(\argdot)} 
    \big( 
    -
    \alpha \bar \kappa_* \bar Z_1 
    +
    \sigma_2 \bar Z_2  \big) \big]
    +
    \mean\big[ \bar Z_1  \, \partial \rho( - \bar \kappa_* \bar Z_1 ) \,  \textrm{Prox}_{\gamma \rho(\argdot)} 
    \big( 
    \alpha \bar \kappa_* \bar Z_1 
    -
    \sigma_2 \bar Z_2  \big) \big]
    \\
    \;=&\;
    2 \, \mean\big[ \bar Z_1  \, \partial \rho( - \bar \kappa_* \bar Z_1 ) \,  \textrm{Prox}_{\gamma \rho(\argdot)} 
    \big( 
    \alpha \bar \kappa_* \bar Z_1 
    -
    \sigma_2 \bar Z_2  \big) \big]
    \\
    \;=&\;
    - 2 \, \mean\big[ \bar \kappa_* \partial^2 \rho\big( - \bar \kappa_* \bar Z_1 ) \,   \textrm{Prox}_{\gamma \rho(\argdot)} 
    \big( 
    \alpha \bar \kappa_* \bar Z_1 
    +
    \sigma_2 \bar Z_2  \big) \big] 
    +
    \bar \kappa_* \alpha \, \mean \bigg[ \mfrac{\partial \rho(- \bar \kappa_* \bar Z_1)}{1+ \gamma \partial^2 \rho\big(\textrm{Prox}_{\gamma \rho(\argdot)} 
    \big( 
    \alpha \bar \kappa_* \bar Z_1 
    +
    \sigma_2 \bar Z_2  \big) \big)} \bigg]
    \;,
    \tagaligneq \label{eq:iso:Z1:tmp}
\end{align*}
where the last line is exactly the same as (87)--(88) of \cite{salehi2019impact} via Stein's lemma and by noting that $\bar Z_2 \overset{d}{=} - \bar Z_2$. Similarly 
\begin{align*}
    \mean\big[
        \bar Z_2
        u_{\bar Z, \varepsilon_1}
    \big]
    =
    2 \, \mean\big[ \bar Z_2  \, \partial \rho( - \bar \kappa_* \bar Z_1 ) \,  \textrm{Prox}_{\gamma \rho(\argdot)} 
    \big( 
    \alpha \bar \kappa_* \bar Z_1 
    -
    \sigma_2 \bar Z_2  \big) \big]
    =
    2  \sigma_2 \, \mean \bigg[ \mfrac{\partial \rho(- \bar \kappa_* \bar Z_1)}{1+ \gamma \partial^2 \rho\big(\textrm{Prox}_{\gamma \rho(\argdot)} 
    \big( 
    \alpha \bar \kappa_* \bar Z_1 
    +
    \sigma_2 \bar Z_2  \big) \big)} \bigg] \;,
    \tagaligneq \label{eq:iso:Z2:tmp}
\end{align*}
where the last line is exactly the same as (83) of \cite{salehi2019impact} via Stein's lemma. Substituting \eqref{eq:iso:Z2:tmp} into the second equation of \eqref{eq:iso:tmp} gives 
\begin{align*}
    0
        \;=\; 
        - 
        \mfrac{1}{2 \tau_2}
        +
        \mfrac{\alpha^2 \bar \kappa_*^2}{2 \sigma_2^2 \tau_2}
        -
        \mfrac{1}{\sigma_2^2 \tau_2} \,
        \mfrac{r_2^2  \kappa + \theta^2 \bar \kappa_*^2 }{2 ( \lambda \kappa + \sigma_2^{-1} \tau_2^{-1} )^2}
        +
        \mfrac{2 \sigma_2}{\gamma}
        \, \mean \bigg[ \mfrac{\partial \rho(- \bar \kappa_* \bar Z_1)}{1+ \gamma \partial^2 \rho\big(\textrm{Prox}_{\gamma \rho(\argdot)} 
        \big( 
        \alpha \bar \kappa_* \bar Z_1 
        +
        \sigma_2 \bar Z_2  \big) \big)} \bigg]
        +
        \mfrac{\sigma_2 }{\gamma}
        \;.
\end{align*}
Upon rearranging and a substitution of $\sigma_2^2
    + \alpha^2 \bar \kappa_*^2
    =
    \frac{ r_2^2  \kappa + \theta^2 \bar \kappa_*^2 }{(\lambda \kappa + \sigma_2^{-1} \tau_2^{-1}   )^2} 
$ from \eqref{eq:iso:EQs:tmp:sigma:alpha} and $\tau_2\sigma_2 = \frac{\kappa^{-1} \gamma}{1 -  \gamma \lambda }$ from \eqref{eq:iso:EQs:partial} , we obtain
\begin{align*}
    1
    - 
    \mfrac{\gamma}{\tau_2 \sigma_2}
    \;=\;
    1 
    - 
    \kappa 
    +
    \gamma \lambda \kappa
    \;=\; 
    \mean \bigg[ \mfrac{\partial \rho(- \bar \kappa_* \bar Z_1)}{1+ \gamma \partial^2 \rho\big(\textrm{Prox}_{\gamma \rho(\argdot)} 
        \big( 
        \alpha \bar \kappa_* \bar Z_1 
        +
        \sigma_2 \bar Z_2  \big) \big)} \bigg]
    \;,
    \tagaligneq \label{eq:iso:EQs:partial:three}
\end{align*}
which gives the fifth desired equation. Substituting this into \eqref{eq:iso:Z1:tmp} implies 
\begin{align*}
    \mean\big[
            \bar Z_1
            u_{\bar Z, \varepsilon_1}
        \big]
    \;=&\;
    - 2 \, \mean\big[ \bar \kappa_* \partial^2 \rho\big( - \bar \kappa_* \bar Z_1 ) \,   \textrm{Prox}_{\gamma \rho(\argdot)} 
    \big( 
    \alpha \bar \kappa_* \bar Z_1 
    +
    \sigma_2 \bar Z_2  \big) \big] 
    +
    \bar \kappa_* \alpha 
    -
    \bar \kappa_* \alpha \,
    \mfrac{\gamma}{\tau_2 \sigma_2}
    \;,
\end{align*}
and substituting this into the first equation of \eqref{eq:iso:tmp} gives
\begin{align*}
    0
    \;=\;
    \theta \bar \kappa_*^2 
    - 
    \mfrac{ \alpha \bar \kappa_*^2}{\sigma_2 \tau_2}
    -
    \mfrac{ \bar \kappa_* }{\gamma}
    \bigg(
    - 2 \, \mean\big[ \bar \kappa_* \partial^2 \rho\big( - \bar \kappa_* \bar Z_1 ) \,   \textrm{Prox}_{\gamma \rho(\argdot)} 
    \big( 
    \alpha \bar \kappa_* \bar Z_1 
    +
    \sigma_2 \bar Z_2  \big) \big] 
    +
    \bar \kappa_* \alpha 
    -
    \bar \kappa_* \alpha \,
    \mfrac{\gamma}{\tau_2 \sigma_2}
    \bigg)
    +
    \mfrac{\alpha \bar \kappa_*^2 }{\gamma}
    \;,
\end{align*}
which simplifies to 
\begin{align*}
    -
    \mfrac{\gamma \theta}{2}
    \;=&\; 
    \mean[ \partial^2 \rho(- \bar \kappa_* \bar Z_1) {\rm Prox}_{\gamma \rho(\argdot)}( \bar \kappa_* \alpha \bar Z_1 + \sigma_2 \bar Z_2 ) ]
    \;.
\end{align*}
Replacing $\gamma \theta$ by $\alpha \kappa$ in view of \eqref{eq:iso:EQs:partial} gives the last desired equation. 

\jmlrQED

\subsubsection{Proof of \lemmaref{lem:CGMT:permute}: Random permutations}
Since $Z_1 \overset{d}{=} \phi_1(Z_1)$, \Cref{assumption:CGMT:var}(i) holds.  Now  note that by the total law of covariance followed by that $\phi_1$ and $\phi_2$ are i.i.d.,
\begin{align*}
    \Cov \,[ \phi_1(Z_1)\,,\, \phi_2(Z_1)]
    \;=&\;
    \Cov \,[ \mean[ \phi_1(Z_1) \,|\, Z_1] \,,\, \mean[ \phi_2(Z_1) \,|\, Z_1] ]
    \,+\,
    \mean \,[ \Cov[ \phi_1(Z_1) \,,\, \phi_2(Z_1) \,|\, Z_1] ]
    \\
    \;=&\;
    \Var\, \mean[\phi_1(Z_1) \,|\, Z_1]\;.
\end{align*}
Denote $\tilde p_t = \lceil r_{\rm perm} p_t \rceil$, the number of fixed entries of the $l$-th group to be permuted. We can WLOG suppose they are chosen as the first $\tilde p_t$ entries of the $t$-th group. Also write 
$Z^{(t)}_{t (\tilde p_t+1): tp_t} = (Z^{(t)}_{t(\tilde p_t+1)}, \ldots, Z^{(t)}_{tp_t})^\intercal$, the vector of un-permuted coordinates within the $t$-th group. Then we may compute 
\begin{align*}
    \Sigma_*
    \;=&\;
    ( \Sigma^\dagger )^{1/2} \, \Cov \,[ \phi_1(Z_1)\,,\, \phi_2(Z_1)] \,  ( \Sigma^\dagger )^{1/2}
    \\
    \;=&\;
    ( \Sigma^\dagger )^{1/2} \, \Var \, \mean[ \, \phi_1(Z_1)\,|\, Z_1 \,] \,  ( \Sigma^\dagger )^{1/2}
    \\
    \;=&\;
    ( \Sigma^\dagger )^{1/2}
    \Var \, \begin{pmatrix}
        \frac{1}{\tilde p_1} \sum_{l \leq p_1} Z^{(1)}_{1l} \times \bone_{\tilde p_1} 
        \\
        Z^{(1)}_{1 (\tilde p_1+1): 1p_1} 
        \\
        \vdots 
        \\
        \frac{1}{\tilde p_N} \sum_{l \leq\tilde p_N} Z^{(N)}_{Nl} \times \bone_{\tilde p_N} 
        \\
        Z^{(N)}_{N (\tilde p_N+1): Np_N} 
    \end{pmatrix}
    \,
    ( \Sigma^\dagger )^{1/2}
    \\
    \;=&\;
    ( \Sigma^\dagger )^{1/2}
    \begin{pmatrix}
        \frac{1}{\tilde p_1} 
        \, \Var[Z^{(1)}_{11}] \times \bone_{\tilde p_1 \times \tilde p_1} & & & & \\
        & \hspace{-2em} \Var[Z^{(1)}_{11}] \times I_{p-\tilde p_1} & & & \\
        &  &  \hspace{-2em} \ddots & & \\
        & & &  \hspace{-2em} \frac{1}{\tilde p_N} \, \Var[Z^{(N)}_{11}] \times \bone_{\tilde p_N \times \tilde p_N}  & \\
        & & & &  \hspace{-2em} \Var[Z^{(N)}_{11}] \times I_{p-\tilde p_N} 
    \end{pmatrix}
    ( \Sigma^\dagger )^{1/2}
    \;,
\end{align*}
whereas
\begin{align*}
    \Sigma 
    \;=\; 
    \Sigma_o
    \;=\;
    \Var[Z_1]
    \;=\;
    \begin{psmallmatrix}
        \Var[ Z^{(1)}_{11} ] \times I_{p_1} & & \\
        & \ddots & \\
        & & \Var[ Z^{(N)}_{11} ] \times I_{p_N} 
    \end{psmallmatrix}
    \;,
\end{align*}
and therefore 
\begin{align*}
    \Sigma_*
    \;=&\; 
    \begin{pmatrix}
        \frac{1}{\tilde p_1} 
        \, \bone_{\tilde p_1 \times \tilde p_1}  \,\ind\{ \Var[Z^{(1)}_{11}] > 0 \}  & & & & \\
        & \hspace{-2em} I_{p-\tilde p_1} \, \ind\{ \Var[Z^{(1)}_{11}] > 0 \}  & & & \\
        &  & \hspace{-2em}\ddots & & \\
        & & & \hspace{-2em}\frac{1}{\tilde p_N} \, \bone_{\tilde p_N \times \tilde p_N} \, \ind\{ \Var[Z^{(1)}_{11}] > 0 \}   & \\
        & & & & \hspace{-2em} I_{p-\tilde p_N} \, \ind\{ \Var[Z^{(1)}_{11}] > 0 \} 
    \end{pmatrix}
    \;,
\end{align*}
which satisfies $\Sigma_*^2 = \Sigma_*$. Thus \Cref{assumption:CGMT:var}(ii) holds. 
\jmlrQED

\subsubsection{Proof of \lemmaref{lem:CGMT:sign:flip}: Random sign flipping}
Since $\Sigma = \Sigma_o = \frac{1}{p} I_p$, we can write
\begin{align*}
    \Sigma_* 
    \;=&\; 
    ( \Sigma^\dagger )^{1/2} \, \Cov[ \phi_1(Z_1)\,,\, \phi_2(Z_1) ] \,  ( \Sigma^\dagger )^{1/2}
    \;=\; \mean[ \phi_1] \, \mean[\phi_2]
\end{align*}
and 
\begin{align*}
    ( \Sigma^\dagger )^{1/2} \, \Cov[ \phi_1(Z_1)\,,\, Z_1 ] \,  ( \Sigma_o^\dagger )^{1/2}
    \;=\; \mean[ \phi_1] 
    \;.
\end{align*}
WLOG we can suppose that the $\lceil r_{\rm flip} p \rceil$ entries are chosen as the first  $\lceil r_{\rm flip} p \rceil$ entries. Then each $\phi_{ij} = \textrm{diag}\{ {\rm Rad}_{ij1}, \ldots, {\rm Rad}_{ij\lceil r_{\rm flip} p \rceil}, 1, \ldots, 1 \}$, where ${\rm Rad}_{ijl}$'s are i.i.d.~Rademacher random variables. Therefore $\mean[\phi_1] = \textrm{diag}\{ 0, \ldots, 0, 1, \ldots, 1 \}$, where there are $\lceil r_{\rm flip} p \rceil$ zeros, and in particular $\mean[\phi_1] = \mean[\phi_1] \mean[\phi_1]  = \mean[\phi_1]   \mean[\phi_2^\intercal]$. This verifies both \Cref{assumption:CGMT:var}(i) and (ii).
\jmlrQED

\subsubsection{Proof of \lemmaref{lem:CGMT:crop}: Random cropping}

In the random cropping setup, $\Sigma_o = \Sigma_{\rm new}= \frac{1}{p} I_p$. Also note that each $\phi_i$ is a random projection matrix and independent of $Z_i$. Then by the total law of covariance,
\begin{align*}
    \Sigma 
    \;=\; 
    \Var[\phi_1(Z_1)]
    \;=&\; 
    \mean \Var[ \phi_1(Z_1) \,|\, \phi_1]
    +
    \Var\mean [ \phi_1(Z_1) \,|\, \phi_1]
    \\
    \;=&\; 
    \mean[ \phi_1 \Var[Z_1] \phi_1 ] + 0 \;=\; \mfrac{1}{p}  \mean[\phi_1] \;=\; \mfrac{1}{p}  \, \mfrac{p - \lceil r_{\rm crop} p \rceil}{p} \, I_p \;.
\end{align*}
This implies 
\begin{align*}
    \Sigma_* 
    \;=&\; 
    (\Sigma^\dagger)^{1/2} \, \Cov[ \phi_1(Z_1), \phi_2(Z_1)] \, (\Sigma^\dagger)^{1/2}
    \\
    \;=&\; 
    \Big(\mfrac{1}{p}  \, \mfrac{p - \lceil r_{\rm crop} p \rceil}{p}\Big)^{-1}
    \, \mean[ \phi_1  ] \, \Var[Z_1]  \, \mean[\phi_2] 
    \;=\; 
    \mfrac{p - \lceil r_{\rm crop} p \rceil}{p} \, I_p
    \;,
\end{align*}
and 
\begin{align*}
    (\Sigma^\dagger)^{1/2} \Cov[ \phi_1(Z_1), Z_1] (\Sigma_o^\dagger)^{1/2}
    \;=&\;
    (\Sigma^\dagger)^{1/2} \mean[\phi_1] \Var[Z_1] (\Sigma_o^\dagger)^{1/2}
    \;=\;
    I_p 
    \;.
\end{align*}
Therefore the desired statements hold with $a_1 =  a_2 =\frac{p - \lceil r_{\rm crop} p \rceil}{p}$.
\jmlrQED

\end{document}